\documentclass[12pt, hyperref]{article}
\usepackage{}
\usepackage{mathrsfs}
\usepackage{amsmath}
\usepackage{bbding}
\usepackage{bbding}
\usepackage{amsfonts}
\usepackage{mathrsfs}
\usepackage{amssymb,latexsym,amsmath, amsthm}
\usepackage{hyperref}
\usepackage{amscd}
\usepackage[all]{xy}
\usepackage{xcolor}
\usepackage{graphicx}
\usepackage{makeidx}
\usepackage{imakeidx}

\hypersetup{linkcolor=blue, pdfstartview=FitH} \hypersetup{colorlinks, linkcolor=blue,
pdfstartview=FitH}
\newtheorem{corollary}{Corollary}
\newtheorem{lemma}[corollary]{Lemma}
\newtheorem{definition}[corollary]{Definition}
\newtheorem{proposition}[corollary]{Proposition}
\newtheorem{theorem}[corollary]{Theorem}
\newtheorem{example}[corollary]{Example}
\newtheorem{remark}[corollary]{Remark}
\newtheorem{problem}[corollary]{Problem}

\newtheorem{condition}{Condition}[section]
\begin{document}

\author{Jiayang Yu\thanks{School of Mathematics, Sichuan University, Chengdu 610064, P. R. China. {\small\it E-mail:} {\small\tt
 jiayangyu@scu.edu.cn}.}\;\; and Xu Zhang\thanks{School of Mathematics, Sichuan University, Chengdu 610064, P. R. China. {\small\it E-mail:} {\small\tt
zhang$\_$xu@scu.edu.cn}.}}
\title{A convenient setting for infinite-dimensional analysis}
\date{}
\maketitle
\def\cc{\mathbb{C}}
\def\zz{\mathbb{Z}}
\def\nn{\mathbb{N}}
\def\rr{\mathbb{R}}
\def\qq{\mathbb{Q}}
\def\dd{\mathbb{D}}
\def\tt{\mathbb{T}}
\def\bb{\mathbb{B}}
\def\ff{\mathbb{F}}
\def\ll{\mathbb{L}}
\def\ds{\displaystyle}
\def\q{\quad}
\def\qq{\qquad}
\def\ns{\noalign{\smallskip} }
\def\nm{\noalign{\medskip} }
\def\ds{\displaystyle}
\def\a{\alpha}
\def\b{\beta}
\def\e{\varepsilon}
\def\lan{\mathop{\langle}}
\def\ran{\mathop{\rangle}}
\def\bel{\begin{equation}\label}
\def\ee{\end{equation}}
\def\cL{{\cal L}}
\def\deq{\mathop{\buildrel\D\over=}}
\def\resp{{\it resp. }}
\def\wt{\widetilde}
\def\span{\hbox{\rm span$\,$}}
\def\supp{\hbox{\rm supp$\,$}}
\def\dbA{{\mathbb{A}}}
\def\dbB{{\mathbb{B}}}
\def\dbC{{\mathbb{C}}}
\def\dbD{{\mathbb{D}}}
\def\dbE{{\mathbb{E}}}
\def\dbF{{\mathbb{F}}}
\def\dbG{{\mathbb{G}}}
\def\dbH{{\mathbb{H}}}
\def\dbI{{\mathbb{I}}}
\def\dbJ{{\mathbb{J}}}
\def\dbK{{\mathbb{K}}}
\def\dbL{{\mathbb{L}}}
\def\dbM{{\mathbb{M}}}
\def\dbN{{\mathbb{N}}}
\def\dbO{{\mathbb{O}}}
\def\dbP{{\mathbb{P}}}
\def\dbQ{{\mathbb{Q}}}
\def\dbR{{\mathbb{R}}}
\def\dbS{{\mathbb{S}}}
\def\dbT{{\mathbb{T}}}
\def\dbU{{\mathbb{U}}}
\def\dbV{{\mathbb{V}}}
\def\dbW{{\mathbb{W}}}
\def\dbX{{\mathbb{X}}}
\def\dbY{{\mathbb{Y}}}
\def\dbZ{{\mathbb{Z}}}

\def\A{\mathcal{A}}
\def\B{\mathcal{B}}
\def\D{\mathcal{D}}
\def\L{\mathcal{L}}
\def\M{\mathcal{M}}

\def\N{\mathcal{N}}
\def\K{\mathcal{K}}
\def\E{\mathcal{E}}
\def\F{\mathcal{F}}
\def\G{\mathcal{G}}
\def\R{\mathcal{R}}
\def\s{\mathcal{S}}
\def\p{\mathcal{P}}
\def\P{\mathcal{P}}
\def\T{\mathcal{T}}
\def\O{\mathcal{O}}
\def\Z{\mathcal{Z}}
\def\C{\mathcal{C}}
\def\SS{(S, \mathcal{S})}
\def\dbN{{\mathbb{N}}}
\def\dbR{{\mathbb{R}}}
\def\dbC{{\mathbb{C}}}
\def\={\buildrel \triangle \over =}

\def\al{\alpha}
\def\la{\lambda}
\def\ep{\epsilon}
\def\sig{\sigma}
\def\Sig{\Sigma}
\def\cd{\mathbb{C}^d}
\def\bm{\mathcal{M}}
\def\bn{\mathcal{N}}
\def\bc{\mathcal{C}}
\def\hN{H^2\otimes \mathbb{C}^N}
\def\ba{\mathcal{A}}

\def\bigpa#1{\biggl( #1 \biggr)}
\def\bigbracket#1{\biggl[ #1 \biggr]}
\def\bigbrace#1{\biggl\lbrace #1 \biggr\rbrace}

\def\papa#1#2{\frac{\partial #1}{\partial #2}}
\def\dbar{\bar{\partial}}

\def\oneover#1{\frac{1}{#1}}

\def\meihua{\bigskip \noindent $\clubsuit \ $}
\def\blue#1{\textcolor[rgb]{0.00,0.00,1.00}{#1}}
\def\red#1{\textcolor[rgb]{1.00,0.00,0.00}{#1}}
\def\xing{\heartsuit}
\def\tao{\spadesuit}
\def\lingxing{\blacklozenge}

\def\norm#1{||#1||}
\def\inner#1#2{\langle #1, \ #2 \rangle}

\def\divide{\bigskip \hrule \bigskip}

\def\bigno{\bigskip \noindent}
\def\medno{\medskip \noindent}
\def\smallno{\smallskip \noindent}
\def\bignobf#1{\bigskip \noindent \textbf{#1}}
\def\mednobf#1{\medskip \noindent \textbf{#1}}
\def\smallnobf#1{\smallskip \noindent \textbf{#1}}
\def\nobf#1{\noindent \textbf{#1}}
\def\nobfblue#1{\noindent \textbf{\textcolor[rgb]{0.00,0.00,1.00}{#1}}}
\def\purple#1{\textcolor[rgb]{1.00,0.00,0.50}{#1}}
\def\green#1{\textcolor[rgb]{0.00,1.00,0.00}{#1}}

\def\vector#1#2{\begin{pmatrix}  #1  \\  #2 \end{pmatrix}}

\def\be{\begin{eqnarray*}}
\def\bel{\begin{equation}\label}
\def\ee{\end{equation}}
\def\bea{\begin{eqnarray}}
\def\eea{\end{eqnarray}}
\def\bt{\begin{theorem}}
\def\et{\end{theorem}}
\def\bc{\begin{corollary}}
\def\ec{\end{corollary}}
\def\bl{\begin{lemma}}
\def\el{\end{lemma}}
\def\bp{\begin{proposition}}
\def\ep{\end{proposition}}
\def\br{\begin{remark}}
\def\er{\end{remark}}
\def\ba{\begin{array}}
\def\ea{\end{array}}
\def\bd{\begin{definition}}
\def\ed{\end{definition}}

\newtheorem{exa}{Example}[section]

\begin{abstract}
In this work, we propose a convenient framework for infinite-dimensional analysis (including both real and complex analysis in infinite dimensions), in which differentiation (in some weak sense) and integration operations can be easily performed, integration by parts can be conveniently established under rather weak conditions, and especially some nice properties and consequences obtained by convolution in Euclidean spaces can be extended to infinite-dimensional spaces in some sense by taking the limit. Compared to the existing tools in infinite-dimensional analysis, our setting enjoys more convenient and clearer links with that of finite dimensions, and hence it is more suitable for computation and studying some analysis problems in infinite-dimensional spaces.
\end{abstract}

\newpage
\tableofcontents

\newpage
{\Large\textbf{Acronyms}}
\begin{itemize}
\item[$\mathbb{N}$]{The collection of all positive integers}
\item[$\mathbb{N}_0$]{The collection of all non-negative integers}
\item[$\mathbb{N}_0^{(\mathbb{N})}$]{The set $ \{(\alpha_i)_{i\in\mathbb{N}}:\alpha_i\in\mathbb{N}_0,\,\forall\,i\in\mathbb{N},\,\text{and}\,\alpha_i\neq 0\,$only for finite many $i\in\mathbb{N} \}$}
\item[$\mathbb{N}_{2}$]{$\{(x_i,y_i)_{i\in\mathbb{N}}\in (\mathbb{R}\times\mathbb{R})^{\mathbb{N}}:$ there exists $i_0\in\mathbb{N}$such that $x_i=y_i=0,\,\forall\,i\in\mathbb{N}\setminus\{i_0\}$ and $(x_{i_0},y_{i_0})\in \{(i_0,0),(0,i_0)\} \}$ which can be viewed as $\mathbb{N} \sqcup \mathbb{N} $}
\item[$\mathbb{N}_{2,\widehat{n} }$]{$\{(x_i,y_i)_{i\in\mathbb{N}}\in (\mathbb{R}\times\mathbb{R})^{\mathbb{N}}:$ there exists $i_0\in\mathbb{N}\setminus\{1,\cdots,n\}$such that $x_i=y_i=0,\,\forall\,i\in\mathbb{N}\setminus\{i_0\}$ and $(x_{i_0},y_{i_0})\in \{(i_0,0),(0,i_0)\} \}$ where $n\in\mathbb{N}$ which can be viewed as $(\mathbb{N}\setminus\{1,\cdots,n\})\sqcup (\mathbb{N}\setminus\{1,\cdots,n\})$}
\item[$\mathbb{Z}$]{The collection of all integers}
\item[$\mathbb{R}$]{The collection of all real numbers}
\item[$\mathbb{C}$]{The collection of all complex numbers}
\item[$\mathbb{R}^{S}$]{The collection of all real-valued functions on non-empty set $S$}
\item[$\mathbb{R}^{\infty}$]{We abbreviate $\mathbb{R}^{ \mathbb{N}}$ to $\mathbb{R}^{\infty}$}
\item[$\mathbb{C}^{\infty}$]{We abbreviate $\mathbb{C}^{ \mathbb{N}}$ to $\mathbb{C}^{\infty}$}
\item[$\ell^p(S)$]{The set $
\left\{\textbf{x}=(x_i)_{i\in S}\in \mathbb{R}^{S}:\sum\limits_{i\in S}|x_i|^p<\infty\right\},
$ where $1\leqslant p<\infty$}
\item[$ ||\cdot||_{\ell^p(S)}$]{$||\textbf{x}||_{\ell^p(S)}\triangleq\left(\sum\limits_{i\in S}|x_i|^p\right)^{\frac{1}{p}} ,\quad\forall\,\textbf{x}=(x_i)_{i\in S}\in\ell^p(S)$, $1\leqslant p<\infty$}
\item[$\ell^{\infty}(S)$]{The set $\left\{(x_i)_{i\in S}\in \mathbb{R}^{S}:\sup\limits_{i\in S}|x_i|<\infty\right\}$}
\item[$\ell^{p}$]{We abbreviate $\ell^{p}(\mathbb{N})$ to $\ell^{p}$, where $1\leqslant p\leqslant\infty$}
\item[$B_r^p (\textbf{x})$]{$B_r (\textbf{x})\triangleq \left\{\textbf{y}\in\ell^p(S):||\textbf{y}-\textbf{x}||_{\ell^p(S)}<r \right\},\,\forall\,\textbf{x}\in\ell^2(S),\,r\in(0,+\infty)$, where $1\leqslant p\leqslant\infty$}
\item[$B_r^{p,S}$]{We abbreviate $B_r^p(\textbf{0})$ to $B_r^{p,S}$, where $\textbf{0}=(0)_{i\in S}\in \mathbb{R}^{S}$}
\item[$B_r$]{We abbreviate $B_r^{2,\mathbb{N}}$ to $B_r$ }
\item[$\mathcal{L}(X,Y)$]{The set of all continuous real linear maps from $X$ into $Y$ where $X$ and $Y$ are two real linear norm spaces}
\end{itemize}
\newpage
\section{Introduction}

Calculus was introduced by I. Newton and G. W. Leibniz in the second half of the 17th century and is one of the greatest scientific achievements in human history. The deepening and widespread application of calculus have led to the emergence of analytical mathematics, including complex analysis, real analysis, functional analysis, harmonic analysis, ordinary differential equations, partial differential equations, and so on. The intersection of analytical mathematics with other branches of mathematics has given rise to newer branches of mathematics such as geometric analysis, stochastic analysis, algebraic analysis, etc. At the beginning of this century, thanks to some tools including geometric analysis in particular, G. Y. Perelman solved the famous Poincar\'e's conjecture, a purely topological problem, which demonstrates the immense power of analytical mathematics.

Roughly speaking, calculus introduces another type of operation, called limit, in addition to the elementary arithmetic operations such as addition, subtraction, multiplication and division, essentially to deal with infinity. According to the mainstream theory of modern cosmology, the universe was formed by the expansion of a dense and hot singularity after a Big Bang 13.7 billion years ago, so time and space (and hence everything that arises from it) are finite. In this way, infinity is undoubtedly a creation of mathematicians, which, strictly speaking, only exists in the minds of mathematicians. The key in analytical mathematics is to handle various infinities, for which a basic idea is to quantify the object of study in order to calculate it precisely, achieving a ``clear understanding". After introducing the concept of infinity, mathematics underwent a revolutionary change, which provided mathematicians with a vast and boundless stage. Nowadays, the computing speed of supercomputers has already reached astronomical numbers, but no matter how fast they are, their speed is still finite and they can only perform limited tasks. In this sense, computers and artificial intelligence can never completely replace human's work, especially the work of mathematicians, because only mathematicians can truly handle infinity.

In the real world, time is $1$ dimensional and space is $3$ dimensional. Therefore, the domains of usually used functions in mathematics, especially that in analytical mathematics, are in finite dimensional spaces. The common feature of these functions is that their independent variables can be regarded as having only a finite number. The rich, profound, and beautiful results in modern mathematics are mostly limited to finite-dimensional spaces and their subspaces/submanifolds or function spaces on them. An essential reason for this is that there exist Lebesgue measures (with almost perfect properties) in finite-dimensional spaces (or more generally, Haar measures on locally compact Hausdorff topological groups), which seamlessly fit with operations such as differentiation, integration, convolution, and Fourier transform in finite dimensional spaces, and thus give rise to widely used Newton-Leibniz (as well as Gauss, Green, Stokes) formulas.

In mathematics, functions of infinitely many variables are also often used, for example, in both calculus of variations (which originated from Johann Bernoulli's research on the problem of brachistochrone shortly after the emergence of calculus) and functional analysis. In fact, functions of infinitely many variables have a wide range of application backgrounds, such as:
\begin{itemize}
\item In control theory, almost all continuous-time control problems are infinite-dimensional because any (continuous-time) control function space is actually an infinite-dimensional space. In this way, even continuous-time, deterministic finite-dimensional optimal control problems are essentially infinite-dimensional optimization problems. Typically, the value functions of optimal control problems for distributed parameter systems (including particularly controlled partial differential equations) are functions of infinitely many variables, and hence the corresponding Hamilton-Jacobi-Bellman equations are nonlinear partial differential equations with infinite number of variables (\cite[p. 225]{Li-Yong}). At present, people know very little about this type of equations, and the main difficulty is the lack of appropriate mathematical tools;

\item  In probability theory, many useful random variables and stochastic processes are actually defined on infinite-dimensional spaces (e.g., \cite{Gel64, Kol56, Wiener}), including typically the classical mathematical formulation of Brownian motion defined on the space of all continuous functions on
$[0, 1]$ which vanish at $0$ (by N. Wiener \cite{Wiener}). Also, it is well-known that one of the most fundamental concepts in probability theory is the so-called ``independence", for which a proper measure for constructing nontrivial sequences of independent random variables is, simply to use the (infinite) product measure of some given sequence of measures in finite dimensions (\cite{Kol56});

\item In physics, functions with infinite number of variables appear extensively in continuum mechanics, quantum mechanics, quantum field theory, statistical physics and so on (\cite{Glimm, GJ81, Popov, Neumann55}). The  most relevant example in this respect is the famous Feynman path integrals (\cite{Feynman}). As pointed by N.~Wiener (\cite[p. 131]{Wiener}), ``the physicist is equally concerned with systems the dimensionality of which, if not infinite, is so large that it invites the use of limit-process in which it is treated as infinite. These systems are the systems of statistical mechanics...".
\end{itemize}

It is worth noting that, the purpose of mathematicians to introduce infinity is NOT to complicate the problems, but rather to simplify them by focusing the ideal situations. Indeed, it is well-known that, the distribution of prime numbers in natural numbers, especially in small intervals, is very much irregular. But if we denote by $\pi (x) $ (for $x \geq 1 $) the number of prime numbers that do not exceed $x$, then the classical prime number theorem holds: $\pi (x) \sim x/\ln (x) $ as $x \to \infty $ (i.e., at a distance sufficiently far away, the number of prime numbers is quite regular). Interesting enough, natural numbers are discrete, but the first proof of this prime number theorem was given by means of a complex analysis method for dealing with continuous objects, which is also a manifestation of the enormous power of analytical mathematics.

Unlike the Pigeonhole Principle for finite sets, there exists the famous Hilbert's Grand Hotel for infinite sets. Therefore, there must exist delicate and wonderful mathematical structures in infinite-dimensional spaces that cannot be seen in finite-dimensional space. In \cite[p. 14]{Ati}, M. Atiyah remarked that ``{\it the 21st century might be the era of quantum mathematics or, if you like, of infinite-dimensional mathematics}". In the near future, the main object of mathematical research is, very likely, to be infinite-dimensional spaces, while the traditional setting of finite-dimensional spaces is just a special case, although very often this is the most important situation.

Similar to the finite-dimensional setting, infinite-dimensional analysis is the core and the basis of infinite-dimensional mathematics. Infinite-dimensional analysis has a very long history, dating back to early calculus of variations (with pioneering contributions from the Bernoulli family, L.~Euler, J.-L.~Lagrange and so on). In 1854, B.~Riemann (\cite{Rie}) proposed the concept of infinite-dimensional manifolds in his inaugural speech at G\"otttingen University. In 1887, V. Volterra (\cite{Vol1,Vol2,Vol3,Vol4,Vol5}) took a decisive step in creating infinite-dimensional analysis, despite being referred to as ``functions of functions" at that time. V. Volterra's work was further developed by H. von Koch (\cite{Koc}), D. Hilbert (\cite{Hil}), M. Fr\'echet (\cite{Fre1, Fre2, Fre3}), R. G\^ateaux (\cite{Gat1,Gat2,Gat3}) and so on.

One can find numerous works on infinite-dimensional analysis, including both real and complex analysis in infinite dimensions (See \cite{Aliprantis, AGS, AvSm, Bar, Bay, Becnel, BL00, Cha, Coe, Col, Din81, Din99, FHHSPZ, FrBu, Had, Her71, Her, Hille, IS85, KM, Lan, LT2021, Maz, Michal, Muj, Nac69, Nac, Nov, Schwartz, Uglanov, Yama} and the rich references therein).
Nevertheless, compared with the finite-dimensional case, the development of infinite-dimensional analysis is far from satisfactory. The in-depth study of infinite-dimensional analysis will encounter huge challenges:
\begin{itemize}
    \item Completely different from the finite-dimensional setting, unfortunately, there exist NO nontrivial translation invariant measures in infinite-dimensional spaces;

   \item Any useful topology in an infinite-dimensional space is usually much more complicated than its finite-dimensional counterpart, and in infinite dimensions the gap between local and
global is much wider than that in finite dimensions;

		\item As remarked in our paper \cite[p. 520]{YZ22}, another essential difficulty in infinite-dimensional spaces is how to treat the differential (especially the higher order differential) of (possibly vector-valued) functions. In this respect, the most popular notions are  Fr\'echet's derivative and G\^ateaux's derivative, which coincide whenever one of them is continuous. However, higher order (including the second order) Fr\'echet's derivatives are multilinear mappings, which are usually not convenient to handle analytically.
\end{itemize}
Because of the above difficulties, though after the works of great mathematicians such as V. Volterra (\cite{Vol6}),
D.~Hilbert (\cite{Hil}), A.~N.~Kolmogorov (\cite{Kol56}), N.~Wiener  (\cite{Wiener}), J.~von~Neumann (\cite{Neumann55}) and I.~M.~Gel'fand  (\cite{Gel64}), people have not yet found a suitable framework so that in-depth studies of infinite-dimensional analysis are always possible. In fact, although integrals can be defined in very general measure spaces and differentials can also be defined in very general topological linear spaces, differentials and integrals in infinite-dimensional spaces are studied separately in the vast majority of literature, unlike the Fundamental Theorem of Calculus in finite-dimensional cases where they can be organically connected:
\begin{itemize}
    \item The commonly used abstract space functions (or vector-valued functions) in functional analysis generally refer to functions whose domains are in finite dimensional spaces and whose values are in abstract spaces (of course, which can be infinite dimensional spaces). Regarding the differentiation and integration of such functions, their properties are quite similar to the classical calculus because their domains are still in finite dimensional spaces;

		\item Most nonlinear functional analysis literature (e.g., \cite{Zeidler1, Zeidler2, Zeidler3, Zeidler4, Zeidler5}) use differentials in infinite-dimensional spaces, but if integrals are also used, quite often they are not integrals in the same spaces, but integrals in other finite-dimensional spaces.
\end{itemize}
Roughly
speaking, (linear) functional analysis is in some sense the part of linear algebra in infinite-dimensional spaces; while
the part of calculus in infinite-dimensional spaces, i.e., infinite-dimensional analysis, to our opinion, very likely, requires the
efforts of mathematicians of several generations or even many generations to be satisfactory. As a reference, the long process of classical calculus (on finite dimensions) from its birth to maturity, from I. Newton and G.~W.~Leibniz to H.~L.~Lebesgue, has lasted more than 250 years.

For quite a long time, the development of infinite-dimensional analysis has been deeply influenced by research in probability theory and related fields (e.g., \cite{Berezanski, Ber-Kon, Bog10, Borodin, CDLL, Cerrai, Da-Za, Dalecky-Fomin, Fom68, FOT94, Gro65, Hairer, Hida, Huang-Yan, Kol56, Levy, Lyons, Ma-92, Mall, Obata, Skorohod, Ust, UZ, Wiener} and so on). Especially in recent decades, due to the works of outstanding probabilistic scientists such as L. Gross (\cite{Gro65}, for abstract Wiener spaces), T. Hida (\cite{Hida75}, about white noise analysis), and P.~Malliavin (\cite{Mall}, on Malliavin's analysis), significant progresses have been made in infinite-dimensional analysis. However, most of these probabilists' works consider infinite-dimensional analysis on quite general probability spaces (or even abstract probability spaces), whose results inevitably appear not powerful enough to solve some complicated problems in this respect that require particularly fine processing (such as the unsolved regularity problem for relaxation transposition solutions to the operator-valued backward stochastic evolution equations in \cite{Lu-Zhang}).

In our opinion, infinite-dimensional analysis has the significance of being independent of probability theory, and it should and can fully become an independent discipline and a mathematical branch. In particular, for infinite-dimensional analysis the ``randomness" should be removed to a certain extent (of course, probability theory, especially stochastic analysis,  will always be one of the most important application backgrounds of infinite-dimensional analysis), and the ``(mathematical) analysis nature" should be returned to. As a reference, functional analysis originated (to a large extent in history) from the study of integral equations, but as an independent branch of mathematics, the great development of functional analysis is possible only after getting rid of the shackles of integral equations.

In this work, based on our previous papers \cite{LZZ23, WYZ1, WYZ, Wang-Zhang1, Yu-Zhang0, YZ22, ZX}, we propose a new setting for infinite-dimensional analysis (including both real and complex analysis in infinite dimensions).  More precisely, as a reminiscence to the works \cite{Gat2, Levy} and so on, we choose the Hilbert space $\ell^2$ and/or its suitable (but quite general) subsets as basic spaces, for which typical Gaussian measures can be constructed as in \cite[pp. 8--9]{Da}. In some sense, $\ell^2$  is the simplest infinite-dimensional space, closest to Euclidean spaces (Nevertheless, infinite-dimensional analysis is far from well-understood even for $\ell^2$).  Note that the essential difficulties for analysis problems on $\ell^2$ are by no means reduced though sometimes working in this space does provide some convenience. Then, we refine our strategy adopted in the papers \cite{WYZ1, Wang-Zhang1, YZ22} to propose a convenient setting for infinite-dimensional analysis, in which differentiation (in some weak sense) and integration operations can be easily performed, integration by parts (including  Gauss, Green, or even Stokes type formulas in quite general forms) can be conveniently established under rather weak conditions,
 and especially some nice properties and consequences obtained by convolution (or even the Fourier transformation) in Euclidean spaces can be extended to infinite-dimensional spaces in some sense by taking the limit.

 Compared to the existing tools in infinite-dimensional analysis (including the above mentioned abstract Wiener space, white noise analysis and Malliavin's calculus, and so on), our setting enjoys more convenient and clearer links with that of finite dimensions, and hence it is more suitable for computation and studying some analysis problems in infinite-dimensional spaces.

 With the aid of the above mentioned new setting, we can solve the following longstanding problems in infinite-dimensional analysis:
 \begin{itemize}
     \item  Construction of naturally induced measures for general co-dimensional surfaces in infinite dimensions;
     
     \item   Differential forms and Stokes type theorems for general co-dimensional surfaces in infinite dimensions;
      
    \item  $L^2$ estimates for the $\overline{\partial}$ operators in infinite dimensions, particularly for general pseudo-convex domains therein;

   \item Holmgren type theorem with infinite number of variables;

		\item Equivalence of Sobolev spaces over general domains of infinite-dimensional spaces.
\end{itemize}
This work is intended to give a concise presentation of the above setting and results, as self-contained as possible. We hope that it will serve as a basis for further research in infinite-dimensional analysis.

\newpage

\section{Some Preliminaries}
In this section,  we shall collect some preliminary results that will be used in the rest.
Throughout this work, $\dbN$ is the set of positive integers, while
$\dbR$ and $\dbC$ stand for respectively
the fields of real numbers and complex numbers. Sometimes, $\dbR$ also stands for
one dimensional Euclidean space, while $\dbC$ is identified with $\dbR^2$.

\subsection{Topology}

We begin with some basic concepts and facts on topology without proofs (See \cite{Kel} for more materials).

In the rest of this subsection, we fix a nonempty set $X$.

\begin{definition}
A nonempty family $\mathscr{T}$ of subsets of $X$ is called a \index{topology}{\it topology} if $\emptyset, X\in \mathscr{T}$, the intersection of any two members of $\mathscr{T}$ is a member of
$\mathscr{T}$, and the union of the members of each subfamily of $\mathscr{T}$ is a
member of $\mathscr{T}$. The set $X$ is called the space of the topology $\mathscr{T}$ and
$\mathscr{T}$ is a topology for $X$. The pair $(X,\mathscr{T})$ is called a \index{topological space}{\it topological space}.
\end{definition}

In what follows, unless other stated, we fix  a topological space $(X,\mathscr{T})$ and a subset $U$ of $X$. When
no confusion seems possible we will not mention the
topology and simply write ``$X$ is a topological space".

One may
construct a topology $\mathscr{U}$ for the set $U$ which is called the \index{relative topology}{\it relative topology}, defined to be the family of all intersections of members of $\mathscr{T}$
with $U$, i.e., $\mathscr{U}\=\{V\cap U:\;V\in \mathscr{T}\}$.

\begin{definition}
The set $U$ is called \index{open set}{\it open} (\resp \index{closed set}{\it closed}) (relative to $\mathscr{T}$) if $U\in \mathscr{T}$ (\resp $X\setminus U\in \mathscr{T}$). $U$ is called a \index{neighborhood}{\it neighborhood} of a point $x\in X$ if $U$ contains an open set to which
$x$ belongs. If $(X,\mathscr{T}')$ is another topological space and each member of $\mathscr{T}$ is a member of $ \mathscr{T}'$, then $\mathscr{T}'$ is called {\it finer than} $\mathscr{T}$, and denoted by \index{$\mathscr{T}\subset \mathscr{T}'$}$\mathscr{T}\subset \mathscr{T}'$.
\end{definition}

\begin{definition}
The intersection of the members of the family of all closed
sets (in $X$) containing $U$ is called the \index{closure}{\it closure} (or more precisely, $\mathscr{T}$-closure) of $U$, denoted by \index{$\overline{U}$}$\overline{U}$. A point $x$ of $U$
is called an {\it interior point} of $U$ if $U$ is a neighborhood of $x$. The set
of all interior points of $U$ is called the \index{interior}{\it interior} of $U$, denoted by \index{$U^0$}$U^0$. The set of all points $x$ which are interior to
neither $U$ nor $X\setminus U$ is called the \index{boundary}{\it boundary} of $U$, denoted by \index{$\partial U$}$\partial U$.
\end{definition}

It is easy to see that  $\partial U=\partial (X\setminus U)$.

\begin{definition}
Two subsets $A$ and $B$ of $X$ are called \index{separated subsets}{\it separated} in the
topological space $(X,\mathscr{T})$ if $\overline{A}\cap B$ and $A \cap \overline{B}$ are both void.
\end{definition}

\begin{definition}
The
topological space $(X,\mathscr{T})$ is called \index{connected space}{\it connected} if $X$ is not the union
of two non-void separated subsets. A subset $Y$ of $X$ is called \index{connected subset}{\it connected}
if the topological space $Y$ with the relative topology is connected.
\end{definition}

\begin{definition}\label{241102def1}
The set $U$ is called \index{dense set}{\it dense} in $ X$ if $\overline{U}=X$. The topological space $X$ is called \index{separable space}{\it separable} if there is a countable
subset which is dense in $X$.
\end{definition}

\begin{definition}
A family $\mathscr{B}$ of sets of $X$ is called a \index{base}{\it base} for the topology $\mathscr{T}$ if $\mathscr{B}$ is a subfamily of
$\mathscr{T}$ and for each point $x$ of $X$, and each neighborhood $O$ of
$x$, there is a member $O'$ of $\mathscr{B}$ such that $x\in O' \subset O$.
\end{definition}

We shall use the following result.

\begin{theorem}\label{241026t1}
A family $\mathscr{B}$ of sets of $X$ with $X=\bigcup\{B:\; B\in \mathscr{B}\}$ is a base for some topology on $X$ if and only if for every two members $O_1$ and $O_2$
of $\mathscr{B}$ and each point $x \in O_1\cap O_2$ there is $O_0 \in \mathscr{B}$ such that
$O_0\subset O_1\cap O_2$.
\end{theorem}

A space whose topology has a countable base has many
pleasant properties.


\begin{definition}
A family $\mathscr{A}$ of sets of $X$ is called a \index{cover}{\it cover} of the set $U$ if $U$ is a subset of the union
$\bigcup\{A:\; A\in \mathscr{A}\}$. The family is called an \index{open cover}{\it open cover} of $U$ if each member
of $\mathscr{A}$ is an open set. A subfamily of $\mathscr{A}$ which is
also a cover (of $U$) is called a \index{subcover}{\it subcover} of $\mathscr{A}$.
\end{definition}

\begin{definition}
The set $U$ is called \index{compact set}{\it compact} if every cover of $U$ by sets
which are open in $X$ has a finite subcover. The topological space $X$ is called \index{compact space}{\it compact} if $X$ is a compact set; it is called \index{locally compact space}{\it locally compact} if each point in $X$ has at least one compact neighborhood.
.
\end{definition}

\begin{definition}
The topological space $X$ is called a \index{Lindel\"of space}{\it Lindel\"of space} if each open cover of
the space has a countable subcover.
\end{definition}

By \cite[Theorem 15, p. 49]{Kel}, it follows that

\begin{theorem}\label{Lindelof}
{\rm ({\bf Lindel\"of})} Any topological space whose topology has a countable
base is a Lindel\"of space.
\end{theorem}

There are many topological spaces in which the topology is
derived from a notion of distance. A \index{metric}{\it metric} for the set $X$ is a non-negative
function $d$ on the Cartesian product $X \times X$ such that for all points $x$, $y$, and $z$ of $X$,
\begin{itemize}
    \item[(a)] \q $d(x,y) = d(y,x)$;

    \item[(b)] \q (triangle inequality) $d(x,y) \leq d(x,z)+ d(y,z)$;

    \item[(c)] \q $d(x,y) = 0$ if and only $x = y$.
\end{itemize}
In this case, $(X,d)$ is called a \index{metric space}{\it metric space}. For any $r>0$, the open
sphere of radius $r$ about $x$ is denoted by
\bel{ballxr}
\index{$B(x;r)$}B(x;r)\=\{x'\in X:\;d(x,x')<r\}.
\ee
Clearly, the
family of all open spheres is a base for some topology for $X$ (see
Theorem \ref{241026t1}). This topology is the \index{metric topology}{\it metric topology} for $X$.
We call a sequence $\{y_k\}_{k=1}^\infty\subset X$ converges to some $y_0(\in X)$ in $X$ if $\ds\lim_{k\to\infty}y_k=y_0$ in $X$, i.e., $\ds\lim_{k\to\infty}d(y_k,y_0)=0$.
It is easy to see that, $(X,d)$ is separable if and only if there is a sequence $\{x_i\}_{i=1}^\infty$ in $X$ such that for each $x\in X$, one can find a subsequence $\{x_{i_k}\}_{k=1}^\infty$ satisfying $\ds\lim_{k\to\infty}x_{i_k}=x$ in $X$.

As a consequence of Theorem \ref{Lindelof}, one has the following result:
\begin{corollary}\label{Lindelof1}
Any separable metric space is a Lindel\"of space.
\end{corollary}

\begin{definition}\label{241102def2}
Let $(X,d)$ be a metric space.

1) We call $\{x_k\}_{k=1}^\infty\subset X$  a \index{Cauchy sequence}{\it Cauchy
    sequence} (in $(X,d)$) if for any $\e>0$, there is $k_0\in \dbN$ such that
$d(x_k,x_j)<\e$ for all $k,j\ge k_0$;

2) We call
$(X,||\cd||_X)$ a \index{complete metric space}{\it complete metric space} if for any Cauchy sequence
$\{x_k\}_{k=1}^\infty\subset X$,
there exists $x_0\in X$ so that $\ds\lim_{k\to\infty}x_k=x_0$ in $X$.
\end{definition}

Suppose that $(X',\mathscr{T}')$ is another topological space

\begin{definition}
A mapping $f:\;U\to X'$ is called \index{continuous}{\it continuous} if $f^{-1}(O')\in \mathscr{U},\;\forall\;O'\in \mathscr{T}'$.
\end{definition}

Denote by \index{$C(U;X')$}$C(U;X')$ the set of all continuous mapping from $X$ to $X'$. When $Y=\dbR$ (or sometimes $\dbC$, which is clear from the context), we simply denote it by \index{$C(U)$}$C(U)$. When the underlying topology $\mathscr{T}$ should be emphasized, we write these sets as \index{$C_{\mathscr{T}}(U;X')$}$C_{\mathscr{T}}(U;X')$ and \index{$C_{\mathscr{T}}(U)$}$C_{\mathscr{T}}(U)$, respectively. For any $f\in C(U)$, the set
$$
\index{$\supp f$}\supp f\triangleq\overline{\{x\in U:\;f(x)\neq 0\}}
$$
is called the \index{support}{\it support} of $f$.

\begin{definition}
A bijective mapping  $f:\;X\to X'$ is called a \index{homeomorphism}{\it homeomorphism} if both $f$ and the inverse mapping $f^{-1}$
are continuous. In this case, the topological spaces $X$ and $X'$ are called \index{homeomorphic}{\it homeomorphic}.
\end{definition}

\subsection{Functional Analysis}

In this work, $\dbK$ denotes either the field $\dbR$ or the field $\dbC$. For any $c\in\dbC$, denote by $\bar c$ its complex conjugate.

\subsubsection{Vector Spaces}


\begin{definition}
A nonempty set $X$ is called a \index{vector space}{\it vector space} (or \index{linear space}{\it linear space}) over $\dbK$ if there exist two mappings:
$$
(x,y)\in X\times X\to x+y\in X \hbox{ and }(\alpha,x)\in \dbK\times X\to \a x\in X,
$$
called respectively \index{addition}{\it addition} and \index{scalar multiplication}{\it scalar multiplication}, that together satisfy the following
properties:
$$
x + y= y + x\; \hbox{ and }\;x + (y + z) = (x + y) + z,\q \forall\; x, y, z \in X;
$$
there exists an element of $X$, denoted by $0$, such that $x + 0 = x$ for all $x \in X$; given any $x \in X$,
there exists an element of X, denoted by $-x$, such that $x + (-x)$ = 0; and for all $\a,\b\in\dbK$ and $x, y \in X$,
$$
\a(x + y) = \a x + \a y, \q (\a +\b)x = \a x +\b x, \q \a(\b x) = (\a\b)x, \q 1x = x.
$$
\end{definition}

A \index{real vector space}{\it real vector space} (or \index{real linear space}{\it real linear space}) is a vector space over $\dbR$. A \index{complex vector space}{\it complex vector space} (or \index{complex linear space}{\it complex linear space}) is a vector
space over $\dbC$. A vector space is either a real vector space or a complex vector space.  Clearly, $\dbR$ and $\dbC$ themselves are vector spaces, any vector space
over the field $\dbC$ is also a vector space
over the field $\dbR$.

In the rest of this sub-subsection, we fix a vector space $X$ over $\dbK$.

The elements of $X$ and $\dbK$ are respectively called \index{vectors}{\it vectors} and \index{scalars}{\it scalars}. The element $0 \in X$
is called the \index{origin}{\it origin}, or the \index{zero vector}{\it zero vector}, of $X$; in this respect, note that the same symbol
$0$ denotes both the zero vector of $X$ and the zero of $\dbK$. If $X \not= \{0\}$, any vector $x \in X$ such
that $x \not= 0$ is called a \index{nonzero vector}{\it nonzero vector} of $X$.

A \index{linear subspace}{\it linear subspace} of $X$ is any subset of $X$ that is also a vector space
over $\dbK$. In particular, $\{0\}$ is a linear subspace of $X$.

The linear subspace spanned by a subset $A$ of $X$, denoted by
\index{$\span A$}$\span A$, consists
of all finite linear combinations of vectors of $A$, i.e., vectors $x \in X$ of the form $\ds x =\sum_{j\in J}
 a_jx_j$, where the set $J $ of indices is finite, and $a_j \in\dbK$ and $x_j \in A$ for all $j \in J$.

Let $I$ be a nonempty index set.

\begin{definition}
 A family $(e_i)_{i\in I}$ of vectors $e_i \in X$ (for each $i\in I$) is called \index{linear independence}{\it linearly independent}, if for any nonempty finite subset $J$ of $I$ and any scalars $\a_j \in \dbK$ ($j \in J$) with $\ds\sum_{j\in J}\a_je_j = 0$,
it holds that $\a_j = 0$ for each $j \in J$. Further, $(e_i)_{i\in I}$ is called a \index{Hamel
basis}{\it Hamel
basis} in $X$ if it is linearly independent and $\span(e_i)_{i\in I} = X$.
\end{definition}

\begin{theorem}\label{Hamel}
If $X\not=\{0\}$, then there exists a Hamel basis of $X$. Further, if $E$ and $F$ are two Hamel bases of $X$, then $\hbox{\rm card$\,$} E = \hbox{\rm card$\,$} F$.
\end{theorem}

In view of Theorem \ref{Hamel}, the following notion makes sense.

\begin{definition}
We say the vector space $X$ to be \index{finite-dimensional vector space}{\it finite-dimensional} (\resp \index{infinite-dimensional vector space}{\it infinite-dimensional}), if there exists a
finite (\resp infinite) Hamel basis of $X$, and its dimension, denoted by \index{$\dim X$}$\dim X$, is the cardinal of anyone of its Hamel bases. A Hamel
basis of a finite-dimensional vector space $X$ is simply called a \index{basis}{\it basis}.
\end{definition}

\subsubsection{Banach Spaces and Hilbert Spaces}


\begin{definition}\label{2.1-def1} Let $X$ be a vector space.
A map $||\cd||_X:\,X\to\dbR$ is called a \index{norm}{\it norm}
on $X$ if it satisfies the following:
\begin{equation}\label{2.1-eq1}
\left\{
\begin{array}{ll}\ds
||x||_X \ge0,\q\forall\; x\in
X;\hbox{ and }||x||_X=0~\iff~x=0;\\
\ns\ds||\a x||_X=|\a|||x||_X,\q\forall\;\a\in \dbC,\,x\in
X;\\
\ns\ds||x+y||_X\le||x||_X+||y||_X,\q\forall\; x,y\in
X.
\end{array}
\right.
\end{equation}
With the above norm $||\cd||_X$, $X$ is called a \index{normed vector
space}{\it normed vector
space} and denoted by $(X, ||\cd||_X)$ (or simply by $X$ if the norm $||\cd||_X$ is clear from the context).
\end{definition}

In the rest of this sub-subsection, unless other stated, we fix a normed vector space $(X, ||\cd||_X)$. Clearly, $(X, ||\cd||_X)$ is a metric space with the norm given by
$$
d(x,y)=||x-y||_X,\q\forall\;x,y\in X.
$$
A subset $G$ of $X$ is called \index{bounded subset}{\it bounded}, if there is a constant $C$ such that
$||x||_X\le C$ for any $x\in G$. In view of Definition \ref{241102def1}, $G$ is said to be {\it dense}  in $X$ if for any $x\in X$, one can find a sequence $\{x_k\}_{k=1}^\infty\subset G$ such that $\ds\lim_{k\to\infty}x_k=x$ in $X$. Also, $X$ is called {\it separable} if there exists a countable dense subset of $X$. The set $B(0;1)$ (as defined in \eqref{ballxr}) is called the \index{unit ball}{\it unit ball} of $X$.

\begin{definition}\label{2.1-def3}
$(X,||\cd||_X)$ is called a \index{Banach space}{\it Banach space} if it
is {\it complete} in the sense of Definition \ref{241102def2}.
\end{definition}

\begin{definition}\label{2.1-def1.1} Let $H$ be a vector space.
A map $\lan\cd,\cd\ran_H:\,H\times H\to \dbC$ is called an
\index{inner product}{\it inner product} on $H$ if
$$
\left\{
\begin{array}{ll}\ds
{\lan x,x\ran}_H\ge0,\q\forall\; x\in H;\hbox{ and }{\lan x,x\ran}_H=0~\iff~x=0;
\\
\ns\ds{\lan x, y\ran}_H=\overline{{\lan y, x\ran}_H},\q\forall\; x,y\in
H; \\
\ns\ds {\lan \a x+\b
y,z\ran}_H=\a{\lan x,y\ran}_H+\b{\lan y,z\ran}_H, \q\forall\;\a,\b\in \dbC,\;x,y,z\in
H.
\end{array}
\right.
$$
A vector space $H$ with the above inner product ${\lan\cd,\cd\ran}_H$ is called an \index{inner product
space}{\it inner product
space} and denoted by $(H, {\lan\cd,\cd\ran}_H)$ (or simply by $H$ if the inner product ${\lan\cd,\cd\ran}_H$ is clear from the context).
\end{definition}

The following result gives a
relationship between norm and inner
products.

\begin{proposition}\label{2.1-prop4}
Let $(H, {\lan\cd,\cd\ran}_H)$  be an
inner product
space. Then, the map
$x\mapsto\sqrt{{\lan x,x\ran}_H}$, $\forall\; x\in H$, is a norm
on $H$.
\end{proposition}

By Proposition \ref{2.1-prop4}, any inner product space $(H, {\lan\cd,\cd\ran}_H)$ can be regarded as a
normed vector space. We call
$|x|_H=\sqrt{{\lan x,x\ran}_H}$ the norm {\it induced}
by ${\lan\cd,\cd\ran}_H$. Also, for any $x,y\in H$, we say that $x$ is {\it orthogonal to} $y$, denoted by \index{$\bot$}$x\bot y$, if ${\lan x,y\ran}_H=0$. For any $E\subset H$, the notation $x\bot E$ means that $x\bot z$ whenever $z \in E$. Also, denote by $E^\bot$
the set of all $h\in H$ that are orthogonal to every $z \in E$.  For any $E_1,E_2\subset H$, the notation $E_1\bot E_2$ means that $x_1\bot x_2$ whenever $x_1 \in E_1$ and $x_2 \in E_2$.

\begin{definition}\label{2.1-def5}
An inner product space $H$  is called
a \index{Hilbert space}{\it Hilbert space} if it is complete under
the norm induced by its inner product.
\end{definition}

In Definitions \ref{2.1-def1}, \ref{2.1-def3}, \ref{2.1-def1.1} and \ref{2.1-def5}, we recalled the notions of normed vector space, Banach space, inner product space and Hilbert space over the field $\dbC$. Similarly, one can define these spaces over the field $\dbR$.

For each $n\in \dbN$, write
$$
\mathbb{R}^n\=\prod\limits_{i=1}^{n}\mathbb{R}\equiv\overbrace{\mathbb{R}\times\cdots\times \mathbb{R}}^{n\hbox{ \tiny times}},
$$
which is the $n$-times Cartesian product of $\mathbb{R}$. It is an $n$-dimensional Hilbert space with inner product given by
$$
{\lan (x_i)_{i=1}^n,(y_i)_{i=1}^n\ran}_{\mathbb{R}^n}=\sum_{i=1}^nx_iy_i,\q\forall\;(x_i)_{i=1}^n,(y_i)_{i=1}^n\in \mathbb{R}^n.
$$
With the above inner product, $\mathbb{R}^n$ is called the $n$-dimensional \index{Euclidean space}{\it Euclidean space}.

Write $\mathbb{R}^{\infty}=\prod\limits_{i=1}^{\infty}\mathbb{R}$, which is the countable Cartesian product of $\mathbb{R}$, i.e., the space of sequences of real numbers (Nevertheless, in the sequel we shall write an element in $\mathbb{R}^{\infty }$ explicitly as $( x_{i})^{\infty }_{i=1}$ or $( x_{i})_{i\in \mathbb{N}}$ rather than $\{ x_{i}\}^{\infty }_{i=1}$, where $x_i\in\mathbb{R}$ for each $i\in \mathbb{N}$). One may check that $\mathbb{R}^{\infty}$ is a complete metric space with the following metric:
\bel{241102e1}
d(( x_{i})^{\infty }_{i=1},( y_{i})^{\infty }_{i=1})=\sum_{i=1}^\infty\frac{|x_i-y_i|}{2^i(1+|x_i-y_i|)},\q\forall\;( x_{i})^{\infty }_{i=1},( y_{i})^{\infty }_{i=1}\in \mathbb{R}^{\infty}.
\ee
Similarly, we have the notation $\mathbb{C}^{\infty}$, the countable Cartesian product of $\mathbb{C}$. Clearly, $\mathbb{C}^{\infty}$ can be identified with $(\mathbb{R}^2)^{\infty}=\prod\limits_{i=1}^{\infty}\mathbb{R}^2$, which is also a complete metric space with metric given similarly as that in \eqref{241102e1}.

In the sequel, unless other stated, for any $p\in[1,\infty)$, we denote by $\ell^p$ the vector space of all $p$-absolutely summable sequences of real
numbers, i.e.
\bel{241023e1}
\ell^p\triangleq \left\{(x_n)_{n=1}^\infty\in \mathbb{R}^{\infty}:\;\sum_{n=1}^{\infty}|x_n|^p<\infty\right\}.
\ee
For any $(x_n)_{n=1}^\infty\in \ell^p$, define
\bel{241023e2}
||(x_n)_{n=1}^\infty||_{\ell^p}=\left(\sum_{n=1}^{\infty}|x_n|^p\right)^{1/p}.
\ee
Then $\ell^p$ is a separable Banach space; particularly $\ell^2$ is a separable Hilbert space.

Also, let
$$
\ell^{\infty}\triangleq \left\{(x_n)_{n=1}^\infty\in \mathbb{R}^{\infty}:\;\sup_{1\leq n<\infty}|x_n|<\infty\right\}.
$$
Then $\ell^\infty$ is a Banach space (with the norm $\ds ||(x_n)_{n=1}^\infty||_{\ell^\infty}=\sup_{1\leq n<\infty}|x_n|$), which is not separable.

Both $\ell^p$ and $\ell^\infty$ are typically infinite-dimensional Banach spaces. In some sense, $\ell^2$ is the simplest infinite-dimensional space, closest to Euclidean spaces.
Nevertheless, there exist significant differences between finite-dimensional spaces and
infinite-dimensional spaces. To see this, let us recall the following Riesz's lemma (e.g., \cite[Lemma 1, p. 93]{Xia2010}).

\begin{lemma}\label{241016lem1}{\rm \textbf{(Riesz)}}
Suppose that $M$ is a closed, proper vector subspace of $X$. Then for any $\varepsilon\in(0,1)$, there exists $x\in X$ such that $||x||_X=1$ and $\sup\limits_{y\in M}||x-y||_X>\varepsilon$.
\end{lemma}

It is well-known that any bounded sequence in Euclidean spaces has a convergent subsequence. As a consequence of Lemma \ref{241016lem1}, it is easy to show the following characterization of finite-dimensional normed spaces.

\begin{theorem}\label{241103t1}
The normed vector space $(X, ||\cd||_X)$ is finite-dimensional if and only if any sequence in the unit ball of $X$ has a convergent subsequence.
\end{theorem}



\subsubsection{Bounded Linear Operators}


In the rest of this sub-subsection, unless otherwise stated,  $X_1$ and
$Y_1$ are normed vector spaces, and $X$ and
$Y$ are Banach spaces.

\begin{definition}
A map $A:X_1\to
Y_1$ is called a {\it linear operator} if
$$
A(\a x+\b y)=\a Ax+\b Ay,\qq\forall\; x,y\in X_1,\;\a,\b\in \dbC.
$$
Furthermore, $A$ is called a {\it bounded linear operator} if it is linear and it maps each bounded subset of $X_1$ into a bounded
subset of $Y_1$.
\end{definition}

Denote by $\cL(X_1;Y_1)$  the set of all bounded linear
operators from $X_1$ to $Y_1$. We simply write $\cL(X_1)$ instead
of $\cL(X_1;X_1)$. For any
$\a,\b\in \dbR$ and $A,B\in\cL(X_1;Y_1)$, we define $\a
A+\b B$ as follows:
\begin{equation}\label{2.1-eq9}
(\a A+\b B)(x)=\a Ax+\b Bx,\qq\forall\; x\in X_1.
\end{equation}
Then, $\cL(X_1;Y_1)$ is also a vector space. Let
\begin{equation}\label{2.1-eq10}
||A||_{\cL(X_1;Y_1)}\deq\sup_{x\in X_1\setminus\{0\}}{\frac{||Ax||_{Y_1}}{||x||_{X_1}}}.
\end{equation}
One can show that, $||\cd||_{\cL(X_1;Y_1)}$ defined by
\eqref{2.1-eq10} is a norm on $\cL(X_1;Y_1)$, and  $\cL(X;Y)$ is a
Banach space under such a norm (In the sequel, unless other stated, we shall always endow $\cL(X;Y)$ with such an operator norm topology).

Let us consider the special case $Y=\dbC$. Any $f\in \cL(X;\dbC)$ is
called a {\it bounded linear functional} on $X$.
Hereafter, we write $X'=\cL(X;\dbC)$ and call it
the {\it dual} (space) of $X$. We also
denote
\begin{equation}\label{2.1-eq15}
f(x)={\langle
f,x\rangle}_{X',X},\qq\forall\;
x\in X.
\end{equation}
The symbol
${\langle\cd\,,\cd\rangle}_{X',X}$
is referred to as the {\it duality pairing}
between $X'$ and $X$. It follows from
\eqref{2.1-eq10} that
\begin{equation}\label{2.1-eq16}
||f||_{X'}=\sup_{x\in X,||x||_X\le1}|f(x)|,\qq\forall\; f\in X'.
\end{equation}
Clearly, $X'$ is a Banach space. Particularly , $X$ is called {\it reflexive} if $X''=X$.

%
%

\subsubsection{Unbounded Linear Operators}

In this sub-subsection, $H_1$ and $H_2$ are Hilbert spaces.

\begin{definition}\label{1def8}
For a linear operator $T : D_T(\subset H_1) \to
H_2$, where $D_T$ is a linear
subspace of $ H_1$, we call $D_T$ the {\it
    domain} of $T$. The {\it graph} of $T$ is the subset
of $H_1\times H_2$ consisting of all elements of
the form $(x,Tx)$ with $x \in D_T$.
The operator $T$
is called {\it closed} (\resp {\it densely defined}) from $H_1$ into $H_2$ if its graph is a closed
subspace of $H_1\times H_2$ (\resp $D_T$ is dense in $H_1$).
\end{definition}

In the rest of this sub-subsection, we assume that $T$ is a linear, closed and densely defined operator from $H_1$ into $H_2$. Denote by $D_T$ the domain of $T$, and write $$R_T\triangleq\{Tx:x\in D_T\},\quad N_T\triangleq\{x\in D_T:Tx=0\},
$$
which are the range and the kernel of $T$, respectively.

The domain $D_{T^*}$ of the {\it adjoint operator} $T^*$ of $T$
is
defined as the set of all $f \in H_2$
such that, for some $g \in
H_1$,
$$
{\lan Ax,f\ran}_{H_2}={\lan x,g\ran}_{H_1},\q\forall\;x
\in D_T.
$$
In this case, we define $T^*f \deq g$.
One can show, $T^*$ is also a linear, closed and densely defined operator from $H_2$ into $H_1$, and $T^{**}=T$ (e.g., \cite[Section 4.2]{Kra}).

By \cite[Lemma 4.1.1]{Hor90}  and the proof of \cite[Lemma 3.3]{SlU}, it is easy to show the following result (See \cite[Lemma 2.1]{YZ22}):
\begin{lemma}\label{lower bounded lemma}
	Let $F$ be a closed subspace of $H_2$ and $R_T\subset F$. Then $F=R_T$ if and only if for some constant $C>0$,
	\begin{eqnarray}\label{lower bounded}
		||g||_{H_2}\leqslant C||T^*g||_{H_1},\quad \forall\; g\in D_{T^*}\cap F.
	\end{eqnarray}
	In this case, for any $f\in F$, there is a unique $u\in D_T\cap N_T^\bot$ such that
	\begin{eqnarray}\label{abs equa}
		Tu=f,
	\end{eqnarray}
	and
	\begin{eqnarray}\label{lower bounded1}
		||u||_{H_1}\leqslant C||f||_{H_2},
	\end{eqnarray}
	where the constant $C$ is the same as that in \eqref{lower bounded}.
\end{lemma}


\begin{remark}\label{estimation rem}
	Generally speaking, the solutions of \eqref{abs equa} are NOT unique. Indeed, the set of solutions of \eqref{abs equa} is as follows:
	$$
	U\equiv\{u+v:\,v\in N_T\},
	$$
	where $u$ is the unique solution (in the space $D_T\cap N_T^\bot$) to this equation (found in Lemma \ref{lower bounded lemma}). It is obvious that $u\bot N_T$.
\end{remark}

Similar to \cite[Lemma 3.3]{SlU}, as a consequence of Lemma \ref{lower bounded lemma}, we have the following result (See \cite[Corollary 2.1]{YZ22}).

\begin{corollary}\label{estimation lemma}
	Let $S$ be a linear, closed and densely defined operator from $H_2$ to another Hilbert space $H_3$, and $R_T\subset N_S$. Then,
	
	{\rm 1)} $R_T=N_S$ if there exists a constant $C>0$ such that
	\begin{eqnarray}\label{estimation lower bounded}
		||g||_{H_2}\leqslant C\sqrt{||T^*g||_{H_1}^2+||Sg||_{H_3}^2},\quad \forall\; g\in D_{T^*}\cap D_S;
	\end{eqnarray}
	
	{\rm 2)} $R_T=N_S$ if and only if for some constant $C>0$,
	\begin{eqnarray}\label{weakestimation lower bounded}
		||g||_{H_2}\leqslant C||T^*g||_{H_1},\quad \forall\; g\in D_{T^*}\cap N_S.
	\end{eqnarray}
	
	In each of the above two cases, for any $f\in N_S$, there exists a unique $u\in D_T\cap N_T^\bot$  such that $Tu=f$ and
	\begin{eqnarray}\label{Solution estation}
		||u||_{H_1}\leqslant C||f||_{H_2},
	\end{eqnarray}
	where the constant $C$ is the same as that in \eqref{estimation lower bounded} (or \eqref{weakestimation lower bounded}).
\end{corollary}

\begin{remark}\label{remark2-9-1}
	It seems that the estimate \eqref{estimation lower bounded} is equivalent to \eqref{weakestimation lower bounded}, i.e., \eqref{estimation lower bounded} is also a necessary condition for $R_T=N_S$ but, as far as we know, so far this is still unclear.
\end{remark}

\subsubsection{Differentiability of Vector-Valued Functions}\label{subsec0120}

In this sub-subsection, we recall the notion of differentiability for vector-valued functions. Let $X$ and $Y$ be Banach spaces over $\dbR$ (Recall that any Banach space over $\dbC$ is also a Banach space
over $\dbR$),  $X_0\subset X$, $X_0\not=\emptyset$ and let $F:\,X_0\to Y$ be a
    function (not necessarily linear).

\begin{definition}\label{def-con}
        {\rm 1)} We say
    that $F$ is {\it continuous at $x_0\in
        X_0$} if
    $$\lim_{x\in X_0,x\to
        x_0}||F(x)-F(x_0)||_Y=0.$$  If $F$ is
    continuous at any point of $X_0$, then we
    say that $F$ is {\it continuous on $X_0$}.

    \ss

    {\rm 2)}  We say that $F$ is  {\it Fr\'echet
        differentiable} at $x_0\in X_0$ if there
    exists $F_1\in\cL(X;Y)$ such that
    $$
    \lim_{x\in X_0,x\to
        x_0}\frac{||F(x)-F(x_0)-F_1(x-x_0)||_{Y}}{||x-x_0||_{X}}=0.
   $$
    We call $F_1$ the {\it Fr\'{e}chet derivative} of $F$ at
    $x_0$, and write $DF(x_0)=F_1$. If
    $F$ is Fr\'echet
    differentiable at each point of $
    X_0$, then we say that $F$ is {\it Fr\'echet differentiable on
    $X_0$}.  Moreover, when the map $DF:X_0\to \cL(X;Y)$ is continuous, we say that $F$ is {\it continuously Fr\'echet differentiable on
    $X_0$}.

    \ss

    {\rm 3)}  We say that $F$ is {\it  second order Fr\'echet
differentiable} at $x_0\in X_0$ if $F:\,X_0\to Y$ is continuous Fr\'echet differentiable and there
exists $F_2\in\cL(X;\cL(X;Y))$ such that
$$
\lim_{x\in X_0,x\to
    x_0}\frac{||DF(x)-DF(x_0)-F_2(x-x_0)||_{\cL(X;Y)}}{||x-x_0||_{X}}=0.
$$
We call $F_2$ the {\it second order Fr\'{e}chet derivative} of $F$ at
$x_0$, and write $D^2F(x_0)=F_2$. If
$F$ is  second order Fr\'echet
differentiable at each point of $
X_0$, then we say that $F$ is {\it  second order Fr\'echet differentiable on
$X_0$}. Moreover, when the map $D^2F:\;X_0\to \cL(X;\cL(X;Y))$ is continuous, we say that $F$ is {\it second order continuously Fr\'echet differentiable on
$X_0$}.
\end{definition}

Inductively, for each $k\geq 2$,  one can define the $k$-th order Fr\'echet derivative $D^kF(x_0)$ of $F$ at
$x_0$. The
set of all continuously (\resp continuously Fr\'echet
differentiable, $k$-th order continuously Fr\'echet differentiable) functions from $X_0$ to
$Y$ is denoted by $C(X_0;Y)$ (\resp $C^1(X_0; Y)$, $C^k(X_0; Y)$).
When $Y=\dbR$ (or sometimes $\dbC$, which is clear from the context), we simply denote it by
$C(X_0)$ (\resp $C^1(X_0)$, $C^k(X_0)$).

The above definition is essentially given by M.~Fr\'{e}chet in \cite{Fre1,Fre2}, which can be viewed as the strongest definition of derivatives on infinite dimensional spaces.
For each $p\in (1,\infty)$, one can show that the norm function $||\cdot||_{\ell^p}$ is Fr\'{e}chet differentiable at each point of $\ell^p\setminus \{0\}$.

As we shall see below, in terms of the Fr\'echet derivatives, people may establish an implicit function theorem in infinite dimensions.
More precisely, suppose that $Z$ is another Banach space over $\dbR$. Clearly, $X\times Y$ is also a Banach space with the following norm:
$$
||(x,y)||_{X\times Y}\triangleq ||x||_{X}+||y||_{Y},\quad \forall\; (x,y)\in X\times Y.
$$
Suppose that $O$ is a non-empty open subset of $X\times Y$, $(x_0,y_0)\in O$ and $f:\,O\to Z$ is a given function. We consider
the partial functions $x\to f(x,y_0)$ and $y\to f(x_0,y)$ of open subsets of
$X$ and $Y$ respectively into $Z$. We say that at $(x_0,y_0)$, $f$ is differentiable
with respect to the first (\resp second) variable if the partial function
$x\to f(x,y_0)$ (\resp $y\to f(x_0,y)$) is Fr\'echet differentiable at $x_0$ (\resp $y_0$); the
derivative of that function, which is an element of $\cL(X;Z)$ (\resp $\cL(Y;Z)$)
is called the {\it partial derivative} of $f$ at $(x_0,y_0)$ with respect to the first (\resp
second) variable, and written $D_1f(x_0,y_0)$ (\resp $D_2f(x_0,y_0)$). If $f\in C^1(O; Z)$, then one can show that $D_1f\in C(O; \cL(X;Z))$ and $D_2f\in C(O; \cL(Y;Z))$.

Now, we may state the following implicit function theorem (e.g., \cite[pp. 270--271]{Die}):

\begin{theorem}\label{230324the1}\index{The implicit function theorem}
{\rm (\textbf{Implicit Function Theorem})} Suppose that $f\in C^1(O;Z)$, $f(x_0,y_0)=0$, the partial derivative $D_2f(x_0,y_0)(\in \cL(Y;Z))$ is a bijection with a continuous inverse, i.e., $\left(D_2f(x_0,y_0)\right)^{-1}\in \cL(Z;Y)$. Then there exists an open
connected neighborhood $U$ of $x_0$ and a unique function $u\in C^1(U;Y)$ such that

{\rm 1)} $u(x_0)=y_0$, $(x,u(x))\in O$ and $f(x,u(x))=0$ for all $x\in U$;

{\rm 2)} For each $x\in U$, $D_2f(x,u(x))$ is a bijection with a continuous inverse and
$$
Du(x)=-(D_2f(x,u(x)))^{-1}\circ D_1f(x,u(x)).
$$
\end{theorem}

\begin{remark}\label{R-IFT}
For the above $u$ and $U$ (in Theorem \ref{230324the1}), write
$$
O_0\triangleq\big\{(x, u(x)):\; x\in U\big\}.
$$
Then, $O_0$ is a non-empty open subset of $X\times Y$. Define a map $\Pi:\; O_0\to U$ by $\Pi(x,y)=x$ for any $(x,y)\in O_0$. Then, by Theorem \ref{230324the1}, it is easy to see that $\Pi$ is a homeomorphism from $O_0$ onto $U$.
\end{remark}

Next,  we recall the definition of bounded  bilinear operator.

\begin{definition}\label{def-con-2}
A mapping $M:X\times Z\to Y$ is called a
{\it bounded bilinear operator} if  $M$ is
linear in each argument and there is a constant
$C>0$ such that
$$
||M(x,z)||_Y\leq C||x||_X||z||_Z,\qq \forall\; (x,z)\in X\times Z.
$$

\end{definition}

Denote by $\cL(X,Z;Y)$ the set of all
bounded bilinear operators from $X\times Z$
to $Y$. The norm of $M\in \cL(X,Z;Y)$ is
defined by
$$
||M||_{\cL(X,Z;Y)}=\sup_{x\in X\setminus\{0\},z\in Z\setminus\{0\}}\frac{||M(x,z)||_Y}{||x||_X||z||_Z}.
$$
One can show that $\cL(X,Z;Y)$ is a Banach space with respect to this norm.

Any $L\in \cL(X;\cL(Z;Y))$ defines a
bounded bilinear operator $\wt L\in
\cL(X,Z;Y)$ as follows:
$$
\wt L(x,z)=\big(L(x)\big)(z),\q \forall\; (x,z)\in X\times Z.
$$
Conversely, any  $\wt L\in \cL(X,Z;Y)$ defines an $L\in \cL(X;\cL(Z;Y))$ in the following way:
$$
\big(L(x)\big)(z)=\wt L(x,z),\q \forall\; (x,z)\in X\times Z.
$$
Hence, we may identify $\cL(X;\cL(Z;Y))$ with $\cL(X, Z;Y)$. In this way, for each $F\in C^2(X_0; Y)$, we have $D^2F\in C(X_0;\cL(X,X;Y))$.

Generally, for any $k\geq 2$, denote by $\cL(\underbrace{X,\cdots,X}_{k\hbox{ \tiny times}};Y)$ the Banach space of all
bounded $k$-multilinear operators from $\underbrace{X\times\cdots\times X}_{k\hbox{ \tiny times}}$
to $Y$. Then, for each
$k$-th order continuous Fr\'echet differentiable function $F$ from $X_0$ to
$Y$, $D^kF$ can be identified as a continuous function from $X_0$ to
$\cL(\underbrace{X,\cdots,X}_{k\hbox{ \tiny times}};Y)$. Then, for any $F\in C^k(X_0; Y)$, one has $D^kF\in C(X_0;\cL(\underbrace{X,\cdots,X}_{k\hbox{ \tiny times}};Y))$.

Nevertheless, as pointed in \cite[p. 520]{YZ22}, higher order (including the second order) Fr\'echet's derivatives are usually not easy to handle analytically because in general they are multilinear mappings. Indeed, for any given
infinite dimensional Hilbert space $H$, although ${\cal L}(H)$ is still a Banach space, it is neither
reflexive (needless to say to be a Hilbert space) nor separable anymore even if $H$ itself is
separable, and hence it may be quite difficult to handle some analytic problems related to ${\cal L}(H)$ (e.g., L\"u and Zhang \cite{Lu-Zhang}). Because of this, also as in \cite{YZ22}, we shall introduce a weak version of derivatives in the sequel.

\subsection{Measures and Integrals}\label{20240107subsection1}

We now recall some basic concepts and facts on the theory of measure and integral without proofs (See \cite{Hal, Rud87} for more materials).

In the rest of this subsection, we fix a nonempty set $X$ and a nonempty family $\mathscr{F}$ of subsets of $X$. For any $E\subset X$,
denote by \index{$\chi_E(\cdot)$} $\chi_E(\cdot)$ the (usual) characteristic
function of $E$, defined on $X$.

\subsubsection{Measures}

We call $\mathscr{F}$ a ring (on $X$) if $A\cup B\in \mathscr{F}$ and $A\setminus B\in \mathscr{F}$ whenever $A, B\in \mathscr{F}$. Further we call $\mathscr{F}$ an algebra (on $X$) if it is a ring and $X\in\mathscr{F}$.


A ring $\mathscr{R}$ on $X$ is called a $\sigma$-ring if
$
\bigcup\limits_{k=1}^{\infty}A_k\in \mathscr{R}$ for all $\{A_n\}_{n=1}^{\infty}\subset\mathscr{R}$.
A $\sigma$-ring $\mathscr{R}$ with $X\in\mathscr{R}$ is called a $\sigma$-algebra (on $X$), and in this case the pair $(X,\mathscr{R})$ is called a measurable space.

For any family $\mathscr{F}$ of subsets of $X$, there is a unique $\sigma$-algebra, denoted by $\sigma(\mathscr{F})$, containing $\mathscr{F}$ and contained in every
$\sigma$-algebra containing $\mathscr{F}$. Particularly, when $(X,\mathscr{T})$ is a topological space (and its topology is clear from the context), denote by $\mathscr{B}(X)$ the $\sigma$-algebra generated by $\mathscr{T}$, and each member of $\mathscr{B}(X)$ is called a Borel subset of $X$.

A set function (on $\mathscr{F}$) is a mapping from $\mathscr{F}$ into $[0,+\infty]$. A set function $\mu$ defined on $\mathscr{F}$ is called
finitely additive, if for every $n\in\mathbb{N}$ and $A_1,\cdots,A_n\in \mathscr{F}$ such that $A_j\cap A_i=\emptyset$ for any $1\leqslant j<i\leqslant n$ and
$\bigcup\limits_{k=1}^{n}A_k\in \mathscr{F}$,
it holds that $\mu\left(\bigcup\limits_{k=1}^{n}A_k\right)=\sum\limits_{k=1}^{n}\mu(A_k)$. It is called
countably additive, if for every $\{A_k\}_{k=1}^{\infty}\subset \mathscr{F}$ with $A_j\cap A_i=\emptyset$ for any $j\neq i$ and $\bigcup\limits_{k=1}^{\infty}A_k\in \mathscr{F}$,
it holds that $\mu\left(\bigcup\limits_{k=1}^{\infty}A_k\right)=\sum\limits_{k=1}^{\infty}\mu(A_k).$

Any countably additive set function $\mu$ on a ring $\mathscr{R}$ (on $X$) with $\mu(\emptyset)=0$ is called a measure on $(X,\mathscr{R})$. Furthermore, if $\mathscr{R}$ is an algebra and $\mu (X)<\infty$ (\resp there exists $\{E_n\}_{n=1}^{\infty}\subset\mathscr{R}$ so that $\mu(E_n)<\infty$ for all $n\in\mathbb{N})$ and
$X\subset \bigcup_{n=1}^{\infty}E_n$, then $\mu$ is called a totally finite measure (\resp a $\sigma$-finite measure).

Let $(X_1,\mathscr{S}_1),\cdots,(X_n,\mathscr{S}_n)$
be measurable spaces, $n\in\dbN$. Write\footnote{Note that here and henceforth the
$\sigma$-algebra $\prod\limits_{i=1}^{n} \mathscr{S}_i$ does not stand for the Cartesian product of $\mathscr{S}_1$, $\cdots$, $\mathscr{S}_n$.}
$$
\prod\limits_{i=1}^{n} \mathscr{S}_i\triangleq \sigma \left(\left\{(A_1,\cdots,A_n):\;A_i\in \mathscr{S}_i\hbox{ for }i=1,\cdots,n\right\}\right).
$$
Further, let $\mu_i$ be a measure  on $(X_i,\mathscr{S}_i)$, $i=1,\cdots,n$.
We call $\mu$ a product
measure on $(\prod\limits_{i=1}^{n} X_i,\prod\limits_{i=1}^{n} \mathscr{S}_i)$ induced
by $\mu_1,\cdots,\mu_n$ if
$$
\mu(A_1\times\cdots\times A_n)=\prod_{i=1}^n
\mu_i(A_i),\qquad \forall\; A_i\in\mathscr{S}_i.
$$
One can show that, there is
one and only one product measure $\mu$ (denoted
by $\mu_1\times\cdots\times \mu_n$) on
$(\prod\limits_{i=1}^{n} X_i,\prod\limits_{i=1}^{n} \mathscr{S}_i)$ induced
by $\{\mu_i\}_{i=1}^n$.

\begin{theorem}{\rm \textbf{(Carath\'eodory's Extension Theorem)}}\label{230520thm4}
If $\mathscr{A}$ is an algebra on $X$ and $\mu$ is a totally finite (\resp $\sigma$-finite) measure on $\mathscr{A}$, then there is a unique totally finite (\resp $\sigma$-finite)  measure $ \bar{\mu}$ on $\sigma(\mathscr{A})$ such that $\mu(E)=\bar{\mu}(E)$ for all $E\in\mathscr{A}$.
\end{theorem}

\begin{theorem}\label{theorem18}
Let $\mathscr{A}$ be an algebra on $X$ and $\mu$ be a finitely additive set function on $\mathscr{A}$ with $\mu(X)<\infty$. Then $\mu$ is a (totally finite) measure on $\mathscr{A}$ if and only if
for any sequence $\{E_n\}_{n=1}^{\infty}\subset \mathscr{A}$ with $E_1\supset E_2\supset E_3\supset\cdots$ and $
 \bigcap\limits_{n=1}^{\infty}E_n=\emptyset$, it holds that $\lim\limits_{n\to\infty}\mu(E_n)=0$.
\end{theorem}

\subsubsection{Integrals}

In the rest of this sub-subsection, we fix a $\sigma$-algebra $\mathscr{S}$ on $X$ and a $\sigma$-finite measure $\mu$ on $\mathscr{S}$. The triple $(X, \mathscr{S},\mu)$ is called a measure space. Particularly, if $\mu(X)=1$, then we call the triple $(X, \mathscr{S},\mu)$ a probability space and $\mu$ a probability measure.

A real-valued function $f(\cdot)$ on $X$ is called $\mathscr{S}$-measurable if
$\{x\in X:f(x)>r\}\in \mathscr{S}$ for all $ r\in \mathbb{R}$. In this case, we also call $f$ a measurable function on $(X, \mathscr{S})$. A complex-valued function $g(\cdot)$ on $X$ is called $\mathscr{S}$-measurable if both the real and the imaginary parts of $g(\cdot)$ are measurable on $(X, \mathscr{S})$.

A real-valued function $s(\cdot)$ on $X$ is called a simple function, if there exists $n\in\mathbb{N}$,
$c_i\in \mathbb{R}$  and $E_i\in \mathscr{S}$ for $i=1,\cdots,n$ so that
\begin{equation}\label{20230519for1}
 s(x)=\sum_{i=1}^{n}c_i \chi_{E_i}(x),\quad \forall\,x\in X.
\end{equation}
Clearly $s(\cdot)$ is $\mathscr{S}$-measurable.

For a non-negative simple function $s(\cdot)$ as in \eqref{20230519for1} (i.e., $c_i\geq 0$ for $i=1,\cdots,n$), write
$$
 \int_{X}s(x)\,\mathrm{d}\mu(x)\triangleq\sum_{i=1}^{n}c_i \mu(E_i)\in[0,+\infty].
$$
For a general nonnegative, $\mathscr{S}$-measurable function $g$ on $(X, \mathscr{S},\mu)$, write
$$
\begin{array}{ll}
\displaystyle
 \int_{X}g(x)\,\mathrm{d}\mu(x)\\[2mm]
 \displaystyle
 \triangleq \sup\bigg\{\int_{X}s(x)\,\mathrm{d}\mu(x):g(\cdot)\geqslant s(\cdot)\geqslant0,\,s(\cdot)\,\text{is a simple function}\bigg\}.
 \end{array}
$$
A real-valued, $\mathscr{S}$-measurable function $f$ on $(X, \mathscr{S},\mu)$ is called integrable on $(X, \mathscr{S},\mu)$, if $\int_{X}|f(x)|\,\mathrm{d}\mu(x)<\infty$. In this case, write
$$
 \int_{X}f(x)\,\mathrm{d}\mu(x)\triangleq  \int_{X}f^+(x)\,\mathrm{d}\mu(x)- \int_{X}f^-(x)\,\mathrm{d}\mu(x),
$$
where $f^+(\cdot)\triangleq \max\{f(\cdot),0\}$ and $f^-(\cdot)\triangleq -\min\{f(\cdot),0\}$. For $p\in [1,\infty)$, we denote by $L^p(X,\mu)$ the Banach space of all
(equivalence classes of) real-valued (or complex-valued, which is clear from the context), $\mathscr{S}$-measurable functions $f$ on $X$ for which $\int_X|f|^pd\mu<\infty$ (Particularly, $L^2(X,\mu)$ is a Hilbert space).

To end this sub-subsection, we fix two measures $\nu_1$ and $\nu_2$ on
$(X, \mathscr{S})$. We say that $\nu_1$ is
{\it absolutely continuous} with respect to $\nu_2$, denoted by
$\nu_1 \ll\nu_2$, if $\nu_1(E)=0$ for every $E\in\mathscr{S}$
with $\nu_2(E)=0$. These two measures $\nu_1$ and $\nu_2$  for which both $\nu_1 \ll\nu_2$
and $\nu_2 \ll\nu_1$ are called {\it equivalent}, in symbols $\nu_1 \equiv\nu_2$.
On the other hand, we say that $\nu_1$ and $\nu_2$ are {\it mutually singular}, in symbols $\nu_1\bot\nu_1$,
if there exist two disjoint sets $A$ and $B$ whose union is $X$ such
that, for every $E\in\mathscr{S}$, both $A \cap E$ and $B \cap E$ belong to $\mathscr{S}$ and $\nu_1 (A \cap E) = \nu_2 (B \cap E) = 0$.

We recall the following result (\cite[Theorem 4.2.2, p. 123]{Coh}):

\begin{theorem}{\rm \textbf{(Radon-Nikodym Theorem)}}\label{241014t1}
Let $\nu_1$ and $\nu_2$ be $\sigma$-finite measures on $\mathscr{S}$. If $\nu_1\ll\nu_2$, then there is a measurable function $g: X \to
[0,+\infty)$ such that $\nu_1(A) = \int_A
 gd\nu_2$ holds for each $A \in \mathscr{S}$. The function $g$ is unique
up to $\nu_2$-almost everywhere equality.
\end{theorem}

The function $g$ in Theorem \ref{241014t1} is
called the
Radon-Nikod\'ym derivative (of $\nu_1$ with respect to
$\nu_2$), and is denoted by
 $$
 g=\frac{d\nu_1}{d\nu_2}.
 $$
By Theorem \ref{241014t1}, it is easy to prove
the following result.

\begin{theorem}\label{1c1t4}
Let the assumptions in Theorem \ref{241014t1} hold
and $g:\ (X, \mathscr{S})\to \mathbb{C}$ be measurable. Then $g\in L^1(X,\nu_1)$ if and only if
$g\frac{d\nu_1}{d\nu_2}\in L^1(X,\nu_2)$. Furthermore,
 $$
 \int_A gd\nu_1=\int_A g\frac{d\nu_1}{d\nu_2}d\nu_2,\qq\forall \;A\in\mathscr{S}.
 $$
\end{theorem}

\subsubsection{Some Typical Measures}

\begin{exa}\label{20241013exa1}
Denote by $C_0$ the set of all real-valued, continuous functions $f(\cdot)$ on $[0,1]$ with $f(0)=0$. Then $C_0$ is a Banach space with the norm
$$
||f||\triangleq\sup\limits_{t\in[0,1]}|f(t)|,\quad \forall\;f\in C_0.
$$
Denote by $\mathcal {R}$ the collection of sets of the following form
$$
	I=\{f\in C_0:(f(t_1),\cdots,f(t_n))\in F\},
$$
where $n\in\mathbb{N},\,0<t_1<t_2<\cdots<t_n\leqslant 1$ and $F\in \mathscr{B}(\mathbb{R}^n)$. One can check that $\sigma (\mathcal {R})=\mathscr{B}(C_0)$. For the above $I\in \mathcal {R}$, write
$$
	W(I)\triangleq\int_F \frac{e^{-\frac{x_1^2}{2t_1}-\frac{(x_2-x_1)^2}{2(t_2-t_1)}-\cdots -\frac{(x_n-x_{n-1})^2}{2(t_n-t_{n-1})}}}{\sqrt{(2\pi)^nt_1(t_2-t_1)\cdots(t_n-t_{n-1})}}\,\mathrm{d}x_1\cdots\mathrm{d}x_n.
$$
Then $W$ is $\sigma$-additive on $\mathcal {R}$. By Theorem \ref{230520thm4}, $W$ can be uniquely extended to be a measure, called the Wiener measure, on $\mathscr{B}(C_0)$ (See \cite[pp. 86--91]{Kuo} for more details).
\end{exa}

\subsubsection{Hellinger's Integral}

An important tool to quantify the relationship between two measures on a measurable space $(X, \mathscr{S})$ is \index{Hellinger's integral} Hellinger's integral. In the rest of this sub-subsection, for any two probability measures $\mu$ and $\nu$ on $(X, \mathscr{S})$, we recall the definition of such an integral of $\mu$ and $\nu$ and apply it to the infinite product of probability measures (See \cite[Section 2.2]{Da}, \cite{Kak} and \cite[Section 1.4]{Xia} for more details).

\begin{definition}\label{230303def}
For any totally finite measure $\zeta$ on $(X, \mathscr{S})$ with $\mu\ll \zeta$ and $\nu\ll\zeta$, by Theorem \ref{241014t1}, $\frac{\mathrm{d}\mu}{\mathrm{d}\zeta},\frac{\mathrm{d}\nu}{\mathrm{d}\zeta}\in L^1(X, \zeta)$, both of which are nonnegative, and
\begin{eqnarray*}
\mu(A)=\int_A \frac{\mathrm{d}\mu}{\mathrm{d}\zeta}\,\mathrm{d}\zeta,\quad \nu(A)=\int_A\frac{\mathrm{d}\nu}{\mathrm{d}\zeta}\,\mathrm{d}\zeta,\q\forall\;A\in \mathscr{S}.
\end{eqnarray*}
\emph{Hellinger's integral} of $\mu$ and $\nu$ is defined by
\begin{eqnarray}\label{230306eq1}
H(\mu,\nu)\triangleq \int_{X}\sqrt{\frac{\mathrm{d}\mu}{\mathrm{d}\zeta}\cdot\frac{\mathrm{d}\nu}{\mathrm{d}\zeta}}\,\mathrm{d}\zeta.
\end{eqnarray}
\end{definition}

\begin{remark}\label{230307rem1}
By the Cauchy-Buniakowsky-Schwarz inequality, we see that
$$
0\leqslant H(\mu,\nu)=\int_{X}\sqrt{\frac{\mathrm{d}\mu}{\mathrm{d}\zeta}\cdot\frac{\mathrm{d}\nu}{\mathrm{d}\zeta}}\,\mathrm{d}\zeta
\leqslant  \left(\int_{X} \frac{\mathrm{d}\mu}{\mathrm{d}\zeta} \,\mathrm{d}\zeta\right)^{\frac{1}{2}}\cdot
\left(\int_{X} \frac{\mathrm{d}\nu}{\mathrm{d}\zeta} \,\mathrm{d}\zeta\right)^{\frac{1}{2}}=1.
$$
Suppose that $\zeta'$ is another totally finite measure on $(X, \mathscr{S})$ with $\mu\ll \zeta'$ and $\nu\ll\zeta'$. Let $\zeta^{''}\triangleq \zeta+\zeta'$. Then, $\zeta\ll \zeta''$ and $\zeta'\ll\zeta''$. It follows from Theorem \ref{241014t1} that
$$
\frac{\mathrm{d}\mu}{\mathrm{d}\zeta''}=\frac{\mathrm{d}\mu}{\mathrm{d}\zeta}\cdot \frac{\mathrm{d}\zeta}{\mathrm{d}\zeta''}=\frac{\mathrm{d}\mu}{\mathrm{d}\zeta'}\cdot \frac{\mathrm{d}\zeta'}{\mathrm{d}\zeta''},\,\quad\frac{\mathrm{d}\nu}{\mathrm{d}\zeta''}=\frac{\mathrm{d}\nu}{\mathrm{d}\zeta}\cdot \frac{\mathrm{d}\zeta}{\mathrm{d}\zeta''}=\frac{\mathrm{d}\nu}{\mathrm{d}\zeta'}\cdot \frac{\mathrm{d}\zeta'}{\mathrm{d}\zeta''}.
$$
Hence, by Theorem \ref{1c1t4},
$$
\int_{X}\sqrt{\frac{\mathrm{d}\mu}{\mathrm{d}\zeta''}\cdot\frac{\mathrm{d}\nu}{\mathrm{d}\zeta''}}\,\mathrm{d}\zeta''
=\int_{X}\sqrt{\frac{\mathrm{d}\mu}{\mathrm{d}\zeta}\cdot\frac{\mathrm{d}\nu}{\mathrm{d}\zeta}}\,\mathrm{d}\zeta
$$
and
$$
\int_{X}\sqrt{\frac{\mathrm{d}\mu}{\mathrm{d}\zeta'}\cdot\frac{\mathrm{d}\nu}{\mathrm{d}\zeta'}}\,\mathrm{d}\zeta'
=\int_{X}\sqrt{\frac{\mathrm{d}\mu}{\mathrm{d}\zeta}\cdot\frac{\mathrm{d}\nu}{\mathrm{d}\zeta}}\,\mathrm{d}\zeta.
$$
Therefore, the quantity in \eqref{230306eq1} is independent of $\zeta$.
\end{remark}

One can show the following two results:
\begin{lemma}\label{230303prop}
$H(\mu,\nu)=0$ if and only if $\mu\bot\nu$.
\end{lemma}

\begin{lemma}\label{230306lem1}
Suppose that $n\in\mathbb{N}$ and $\mu_k$ and $\nu_k$ are probability measures on a measurable space $(X_k, \mathscr{S}_k)$ for $ k=1,\cdots, n$. Then,
$$
H(\mu_1\times\cdots\times\mu_n,\nu_1\times\cdots\times\nu_n)=\prod_{i=1}^{n}H(\mu_i,\nu_i).
$$
\end{lemma}

Further, let us recall the countable product of probability measures (See \cite[$\S$ 38]{Hal} for more details). Suppose that $\{(X_i, \mathscr{S}_i,\mu_i)\}_{i\in \mathbb{N}}$ is a sequence of probability spaces.
Denote by $\mathcal {R}$ the following family of cylinder subsets (in $\prod\limits_{i=1}^{\infty}X_i$):
\begin{equation}\label{1qq1}
I_{n,A}\triangleq\left\{(x_i)_{i\in \mathbb{N}}\in \prod_{i=1}^{\infty}X_i:\;(x_1,\cdots,x_n)\in A\right\},
\end{equation}
where $n\in\mathbb{N}$ and $A\in \prod\limits_{i=1}^{n} \mathscr{S}_i$. Then, $\mathcal {R}$ is an algebra on $\prod\limits_{i=1}^{\infty}X_i$. Put
$$
\prod\limits_{i=1}^{\infty} \mathscr{S}_i\triangleq\sigma(\mathcal {R}).
$$

We define a set function $\mu_0$ on $\mathcal {R}$ by
 $$
\mu_0(I_{n,A})\triangleq\left( \mu_1\times\cdots \times \mu_n\right)(A),
$$
where $I_{n,A}$ is given by \eqref{1qq1}. It is easy to see that $\mu_0$ is finitely additive on $\mathcal {R}$ and $\mu_0(\prod\limits_{i=1}^{\infty} X_i)=1$. By Theorem \ref{theorem18}, one can prove that
$\mu_0$ is $\sigma$-additive on $\mathcal {R}$.
Then by Theorem \ref{230520thm4},  $\mu_0$ can be uniquely extended as a probability measure, denoted by $\prod\limits_{i=1}^{\infty}\mu_i$, on $\prod\limits_{i=1}^{\infty} \mathscr{S}_i$.
We call the triple $\left(\prod\limits_{i=1}^{\infty} X_i,\prod\limits_{i=1}^{\infty} \mathscr{S}_i,\prod\limits_{i=1}^{\infty}\mu_i\right)$ the infinite product measure space of probability spaces $\{(X_i, \mathscr{S}_i,\mu_i)\}_{i\in \mathbb{N}}$.

Based on Lemmas \ref{230303prop} and \ref{230306lem1}, we have the following result.
\begin{theorem}\label{230309the1}{\rm \textbf{(Kakutani)}}
Suppose that for each $k\in\mathbb{N}$, $\nu_k$ is a probability measure on $(X_k, \mathscr{S}_k)$, and $\mu_k$ and $\nu_k$ are equivalent. Then, for $X\triangleq \prod\limits_{i=1}^{\infty}X_i$, $\mu\triangleq \prod\limits_{i=1}^{\infty}\mu_i$ and $\nu\triangleq \prod\limits_{i=1}^{\infty}\nu_i$, the following assertions hold true:

{\rm 1)}
\begin{eqnarray}\label{230306eq2}
H(\mu,\nu)=\prod_{i=1}^{\infty}H(\mu_i,\nu_i);
\end{eqnarray}

{\rm 2)} If $H(\mu,\nu)>0$, then $\mu$ and $\nu$ are equivalent and
\begin{eqnarray}\label{230306eq3}
\frac{\mathrm{d}\mu}{\mathrm{d}\nu}(\textbf{x})=\lim_{n\to\infty}\prod_{i=1}^{n}\frac{\mathrm{d}\mu_i}{\mathrm{d}\nu_i}(x_i),\quad \text{in}\;L^1(X,\nu),
\end{eqnarray}
where $\textbf{x}=(x_i)_{i\in\mathbb{N}}\in X$;

{\rm 3)} If $H(\mu,\nu)=0$, then $\mu$ and $\nu$ are mutually singular.
\end{theorem}

\subsubsection{Nonexistence of two Nontrivial Measures in Infinite Dimensions}

 We fix a normed vector space $(X,||\cdot||)$. A measure $\mu$ on $\mathscr{B}(X)$ is called translation invariant if $\mu(E+z)=\mu(E)$ for all $z\in X$ and $E\in \mathscr{B}(X)$.

A fundamental property of Lebesgue's measures in finite dimensions is its translation-invariance. Unfortunately, as we shall see below, there does not exist a similar measure on infinite dimensional spaces which enjoys the same invariance.

By Lemma \ref{241016lem1}, one has the following simple result.
\begin{lemma}\label{230309lem1}
If $\dim X=\infty$, then for each non-empty open ball $B$ of $X$, there exists a sequence of non-empty open balls $\{B_n\}_{n=1}^{\infty}$ so that: 1) $B_n\cap B_m=\emptyset$ for any $m\neq n$; 2) $B_n$ and $B_m$ have the same radius for all $m,n\in\mathbb{N}$; and 3) $B_n\subset B$ for each $n\in\mathbb{N}$.
\end{lemma}

\begin{proof}
Choosing $x_1\in X$ such that $||x_1||=1$. Since $\dim X=\infty$, $\span\{x_1\}\varsubsetneqq X$ and by Lemma \ref{241016lem1}, one can find $x_2\in X$ such that $||x_2||=1$ and $||x_1-x_2||>\frac{1}{2}$. Note that $\span\{x_1,x_2\}$ is a proper subspace of $X$ and by Lemma \ref{241016lem1} again, there is $x_3\in X$ such that $||x_3||=1$, $||x_3-x_1||>\frac{1}{2}$ and $||x_3-x_2||>\frac{1}{2}$. Repeating the above process, we obtain a sequence $\{x_n\}_{n=1}^{\infty}\subset X$ such that $||x_n||=1$ for any $n\in\mathbb{N}$ and $||x_k-x_l||>\frac{1}{2}$ for any $k\neq l$.

Clearly, $B(x_n;\frac{1}{4})\cap B(x_m;\frac{1}{4})=\emptyset$ for $m\neq n$ and $B(x_k;\frac{1}{4})\subset B(0;2)$ for all $k\in\mathbb{N}$. Since $B$ is a non-empty open ball of $X$, there exists $x_0\in X$ and $r_0\in(0,+\infty)$ so that $B=B(x_0;r_0)$. Hence $B(x_0+\frac{r_0}{2}\cdot x_n;\frac{r_0}{8})\cap B(x_0+\frac{r_0}{2}\cdot x_m;\frac{r_0}{8})=\emptyset$ for all $m\neq n$ and $B(x_0+\frac{r_0}{2}\cdot x_k;\frac{r_0}{8})\subset B(x_0;r_0)$ for all $k\in\mathbb{N}$. Write $B_n\triangleq B(x_0+\frac{r_0}{2}\cdot x_n;\frac{r_0}{8})$ for each $n\in\mathbb{N}$. Then $\{B_n\}_{n=1}^{\infty}$ is the desired sequence. This completes the proof of Lemma \ref{230309lem1}.
\end{proof}

\begin{proposition}\label{230310prop1}
Suppose that $(X,||\cdot||)$ is separable and $\dim X=\infty$. If $\mu$ is a measure on $\mathscr{B}(X)$ such that $\mu(B(x,r))=\mu(B(x',r))$ for all $x,x'\in X$ and $r\in(0,+\infty)$, then $\mu\equiv 0$ or $\mu(O)=\infty$ for any non-empty open subset $O$ of $X$.
\end{proposition}

\begin{proof}
If $\mu(B(x,r))=0$ for all $x\in X$ and $r\in(0,+\infty)$, then by Corollary \ref{Lindelof1}, $X$ is a Lindel\"{o}f space, and hence there exists $\{x^0_m\}_{m=1}^{\infty}\subset X$ such that $X =\bigcup\limits_{m=1}^{\infty}B(x_m^0,1)$, which leads to $\mu(X)\leqslant \sum\limits_{m=1}^{\infty}\mu(B(x_m^0,1))=0$. Then we have $\mu\equiv 0$.

If there exits $x\in X$ and $r\in(0,+\infty)$ such that $\mu(B(x,r))>0$, then by Lemma \ref{230309lem1}, there exists $\{x_n\}_{n=1}^{\infty}\subset X$ and $r'\in(0,+\infty)$ such that $B(x_n,r')\cap B(x_m,r')=\emptyset$ for any $m\neq n$ and $B(x_k,r')\subset B(x,r)$ for any $k\in\mathbb{N}$. If $\mu(B(x_1,r'))=0$, then we have $\mu(B(y,r'))=0$ for all $y\in X$. Since $X$ is a Lindel\"{o}f space and $\{B(y,r'):y\in X\}$ is an open cover of $X$, there exists $\{y_n\}_{n=1}^{\infty}$ such that $X=\bigcup\limits_{n=1}^{\infty}B(y_n,r')$. Hence $\mu(X)\leqslant \sum\limits_{n=1}^{\infty}\mu(B(y_n,r')=0$ which contradicts to the fact that $\mu(X)\geqslant\mu(B(x,r))>0$. Therefore, $\mu(B(y,r'))>0$ for all $y\in X$ and hence $\mu(B(x,r))\geqslant \sum\limits_{n=1}^{\infty}\mu(B(x_n,r'))=\sum\limits_{n=1}^{\infty}\mu(B(x_1,r'))=\infty$. By the same argument, we see that $\mu(B(z,r''))=\infty$ for all $z\in X$ and $r''\in(0,+\infty)$. Then $\mu(O)=\infty$ for any non-empty open subset of $X$. This completes the proof of Proposition \ref{230310prop1}.
\end{proof}

Next, we will show that there does NOT exist an isotropic measure on any infinite-dimensional, real separable Hilbert space (e.g., \cite[pp. 54--55]{Kuo}).

Suppose that $H$ is a real separable Hilbert space with inner product ${\langle\cdot,\cdot\rangle}_H$, $\dim H=\infty$ and $\{e_n\}_{n=1}^{\infty}$ is an orthonormal basis of $H$. Denote by $\mathcal {R}$ the following family of subsets in $H$:
\begin{equation}\label{241016e1}
E=\{x\in H:\;({\langle x,e_1\rangle}_H,\cdots,{\langle x,e_n\rangle}_H)\in F\},
\end{equation}
where $n\in\mathbb{N}$ and $F\in\mathscr{B}(\mathbb{R}^n)$. Obviously, $\mathcal {R}$ is an algebra, but it is not a $\sigma$-algebra.
Define a set function $\mu$ from $\mathcal {R}$ into $[0,1]$ as follows:
$$
\mu(E)\triangleq \left(\frac{1}{\sqrt{2\pi}}\right)^n\cdot\int_{F}e^{-\frac{x_1^2+\cdots+x_n^2}{2}}\,\mathrm{d}x_1\cdots\mathrm{d}x_n,
$$
where $E$ is given by \eqref{241016e1}. It is easy to see that $\mu$ is finite-additive, but we have the following result:

\begin{proposition}\label{cha1prop2}
The above $\mu$ is not $\sigma$-additive.
\end{proposition}

\begin{proof}
Note that for each $k\in\mathbb{N}$, we have $0<\frac{1}{\sqrt{2\pi}}\cdot \int_{-k}^ke^{-\frac{x^2}{2}}\,\mathrm{d}x<1$. Hence there exists $n_k\in\mathbb{N}$ such that $n_k\geqslant k$ and $\left(\frac{1}{\sqrt{2\pi}}\cdot \int_{-k}^ke^{-\frac{x^2}{2}}\,\mathrm{d}x\right)^{n_k}<\frac{1}{2^{k+1}}$. Let
\begin{eqnarray*}
E_k\triangleq \{x\in H:\;|{\langle x,e_i\rangle}_H|\leqslant k,\,i=1,\cdots,n_k\}.
\end{eqnarray*}
Then, $H=\bigcup\limits_{k=1}^{\infty}E_k$ and
\begin{eqnarray*}
\mu(E_k)&= & \left(\frac{1}{\sqrt{2\pi}}\right)^{n_k}\cdot\int_{|x_1|\leqslant k}\cdots\int_{|x_{n_k}|\leqslant k}e^{-\frac{x_1^2+\cdots+x_{n_k}^2}{2}}\,\mathrm{d}x_1\cdots\mathrm{d}x_{n_k}\\
&=&\left(\frac{1}{\sqrt{2\pi}}\cdot \int_{-k}^ke^{-\frac{x^2}{2}}\,\mathrm{d}x\right)^{n_k}<\frac{1}{2^{k+1}}.
\end{eqnarray*}
 Note that $H=\{x\in H:\;{\langle x,e_1\rangle}_H\in\mathbb{R}\}$ and $\mu(H)=\frac{1}{\sqrt{2\pi}}\cdot \int_{\mathbb{R}}e^{-\frac{x^2}{2}}\,\mathrm{d}x=1$. If $\mu$ were  $\sigma$-additive, by Theorem \ref{230520thm4}, then $\mu$ could be extended to the $\sigma$-algebra generated by $\mathcal {R}$ as a probability measure. Thus we have $\mu(H)\leqslant \sum\limits_{k=1}^{\infty}\mu(E_k)<\frac{1}{2}$, which is a contradiction. Therefore, $\mu$ is not $\sigma$-additive. The proof of Proposition \ref{cha1prop2} is completed.
\end{proof}


\newpage

\section{A Basic Infinite-Dimensional Space: $\ell^2$}\label{20231108sect1}

In this section, we are concerned with how to introduce differentiation, various measures including particularly that for quite general surfaces (and hence integration) in a basic infinite-dimensional space, i.e., $\ell^2$.

\subsection{Partial Derivatives}

Our basic space in the sequel is the separable Hilbert space $\ell^2$, defined by \eqref{241023e1}--\eqref{241023e2}. In what follows, $O$ is a non-empty open subset of $\ell^2$ and $f$ is a real-valued function on $O$. Denote by $\mathscr{B}(O)$ the family of Borel subsets (of $\ell^2$) contained in $O$, which is a $\sigma$-algebra on $O$.

For each $n\in\mathbb{N}$, set $\textbf{e}_n\triangleq(\delta_{n,i})_{i\in\mathbb{N}}$ where $\delta_{n,i}\triangleq0$ if $n\neq i$ and $\delta_{n,i}\triangleq 1$ if $n=i$. As in \cite[p. 528]{YZ22}, for each $\textbf{x}=(x_i)_{i\in\mathbb{N}} \in O$, we define the partial derivative $D_n f$ of $f(\cdot)$ (at $\textbf{x}$) as follows:
\begin{eqnarray}\label{20240127def1}
(D_n f)(\textbf{x}) \triangleq\lim_{\mathbb{R}\ni t\to 0}\frac{f(\textbf{x}+t\textbf{e}_n)-f(\textbf{x})}{t},
\end{eqnarray}
provided that the above limits exist. As in calculus, sometimes we also use the notations $(D_{x_n}  f)(\textbf{x})$, $\frac{\partial f(\textbf{x})}{\partial x_n}$ or $\frac{\partial f}{\partial x_n}(\textbf{x})$ to denote the above limit.

Write
$$
O_{\textbf{x},n}\triangleq\{(s_1,\cdots,s_n)\in\mathbb{R}^n:\textbf{x}+s_1\textbf{e}_1+\cdots+s_n\textbf{e}_n\in O\}.
$$
Then $O_{\textbf{x},n}$ is a nonempty open subset of $\mathbb{R}^n$. We shall need the following class of functions on $O$:

\begin{definition}\label{230407def1}
We say that $f$ is an \index{$F$-continuous function}$F$-continuous function on $O$, if for each $\textbf{x}\in O$ and $n\in\mathbb{N}$ the following function
$$
O_{\textbf{x},n}\ni(s_1,\cdots,s_n)\mapsto f(\textbf{x}+s_1\textbf{e}_1+\cdots+s_n\textbf{e}_n)
$$
is continuous on $O_{\textbf{x},n}$.
\end{definition}

Denote by \index{$C_F(O)$}$C_F(O)$ and \index{$C_{F,b}(O)$}$C_{F,b}(O)$ the collection of all $F$-continuous functions on $O$ and uniformly bounded and $F$-continuous functions on $O$, respectively. Recall that $C(O)$ stands for the collection of all and continuous real-valued functions on $O$, endowed the usual $\ell^2$ norm topology.

We denote by \index{$C^{1 }_{F,b}(O)$}$C^{1 }_{F,b}(O)$ the following set
\bel{20231110for6}
 \left\{f\in C_{F,b}(O):\;\frac{\partial f}{\partial x_i}\in C_F(O),\,\forall\,i\in\mathbb{N},\,\sup_{\textbf{x}\in O}\sum_{i=1}^{\infty}\left|\frac{\partial f}{\partial x_i}(\textbf{x})\right|^2<\infty\right\}.
\ee

For each $n\in\mathbb{N}$, denote by $\mathscr{S}_n$ the collection of the following subsets (in $\ell^2$):
\begin{eqnarray*}
 \{\textbf{x}+s_1\textbf{e}_1+\cdots+s_n\textbf{e}_n:\;(s_1,\cdots,s_n)\in U\},
\end{eqnarray*}
where $U$ is an open subset of $\mathbb{R}^n$ and $\textbf{x}\in\ell^2$. Then, by Theorem \ref{241026t1}, $\mathscr{S}_n$ is a base for some topology space on $\ell^2$. Denote by $\mathscr{F}_n$ the topology generated by $\mathscr{S}_n$. Denote by $\mathscr{T}_{\ell^2}$ the usual topology on $\ell^2$ generated by its $\ell^2$ norm. Let $ \ds\mathscr{F}\triangleq \bigcap_{n=1}^{\infty}\mathscr{F}_{n}$. Then, it is easy to verify the following facts.

\begin{proposition}
$\mathscr{T}_{\ell^2}\subsetneqq \mathscr{F}\subsetneqq \mathscr{F}_{n+1}\subsetneqq \mathscr{F}_n$, for all $n\in\mathbb{N}$. Moreover, a function $f$ on $\ell^2$ is $F$-continuous if and only if it is continuous respect to $\mathscr{F}$.
\end{proposition}

Let us recall that, for each $n\in\mathbb{N}$, usually $C_c^{\infty}(\mathbb{R}^n)$ stands for the set of all real-valued $C^{\infty}$-functions on $\mathbb{R}^n$ with compact supports. Note that each function in $C_c^{\infty}(\mathbb{R}^n)$ can be viewed as a cylinder function on $\ell^2$ which depends only on the first $n$ variables. Set
$$
\mathscr {C}_c^{\infty}\triangleq \bigcup_{n=1}^{\infty}C_c^{\infty}(\mathbb{R}^n).
$$

The following notion (of uniform inclusion) will play a fundamental role in the sequel.
\begin{definition}\label{def of bounded contained}
A set $E\subset \ell^2$ is said to be \index{uniform inclusion}{\it uniformly included} in $O$, denoted by \index{$E\stackrel{\circ}{\subset} O$}$E\stackrel{\circ}{\subset} O$, if there exist $r,R\in(0,+\infty)$ such that $\ds\bigcup_{\textbf{x}\in E}B(\textbf{x};r)\subset O$ and $E\subset B(\textbf{0};R)$.
\end{definition}

Write $\mathbb{N}_0\triangleq \mathbb{N}\cup \{0\}$. Denote by $\mathbb{N}_0^{(\mathbb{N})}$ the set of all finitely supported sequences of nonnegative integers, i.e., for each $\alpha =(\alpha_{j} )_{j=1}^\infty\in \mathbb{N}^{\left( \mathbb{N} \right)  }_{0} $ with $\alpha_{m}\in\mathbb{N}_{0}$ for any $ m\in\mathbb{N}$, there exists $n\in\mathbb{N}$ such that $\alpha_j=0$ for all $j\geqslant n$. Suppose that $\alpha=(\alpha_1,\cdots,\alpha_k,0,\cdots)\in \mathbb{N}_0^{(\mathbb{N})}$, the corresponding higher partial derivative of $f$ is defined by $D^{\alpha}f\triangleq(D_1)^{\alpha_1}\cdots (D_k)^{\alpha_k}f.$
Denote by \index{$C_F^{\infty}(O)$}$C_F^{\infty}(O)$ the set of all real-valued $F$-continuous and $\mathscr{B}(O)$-measurable functions $f$ on $O$ such that all partial derivatives of $f$ exist and these partial derivatives are uniformly bounded $F$-continuous and $\mathscr{B}(O)$-measurable functions on $O$. Put
$$
C_F^{\infty}(O)\triangleq\left\{f\in C^{\infty}_F(O):\sup_{O}\left(|f|+\sum_{i=1}^{\infty}|D_if|^2 \right)<\infty,\,\text{supp}f \stackrel{\circ}{\subset}O \right\},
$$
and
$$
C_{F,b}^{\infty}(O)\triangleq\left\{f\in C^{\infty}_F(O):\sup_{O}\left( |f|+|D^{\alpha}f|\right)<\infty,\,\forall\,\alpha\in \mathbb{N}_0^{(\mathbb{N})}\right\}.
$$

\begin{remark}
Note that, since $O$ is assumed to be a nonempty open set of the infinite-dimensional Hilbert space $\ell^2$, any non-zero function $g$ in $C_F^{\infty}(O)$ does NOT need to have a compact support in $O$. Nevertheless, such a function $g$ is bounded and $F$-continuous on $O$ and $\supp g\stackrel{\circ}{\subset} O$, and therefore, in a suitable sense and to some extend, functions in $C_F^{\infty}(O)$ can be regarded as infinite-dimensional counterparts of compactly supported, infinitely differentiable functions on nonempty open sets of Euclidean spaces.
\end{remark}

\begin{lemma}
$C_F^{\infty}(O)$ is dense in $L^2(O,P)$.
\end{lemma}
\begin{proof}
The proof is a consequence of Dynkin's theorem.
\end{proof}

\subsection{A Borel Probability Measure}

For any nonempty set $I\subset \mathbb{N}$ and $p\in[1,+\infty)$, let
$$
\mathbb{R}^{I}=\{(x_i)_{i\in I}:\;x_i\in\mathbb{R}\hbox{ for each }i\in I\}
$$
and
$$
\ell^p(I)\triangleq \Big\{(x_i)_{i\in I}\in \mathbb{R}^{I}:\sum_{i\in I}|x_i|^p<\infty\Big\}.
$$
Then $\ell^p(I)$ is a Banach space with the following norm
$$||(x_i)_{i\in I}||_{\ell^p(I)}\triangleq  \Big(\sum_{i\in I}|x_i|^p\Big)^{\frac{1}{p}} ,\q\forall\;(x_i)_{i\in I}\in\ell^p(I).
$$
In particular, $\ell^2(I)$ is a Hilbert space.
For any $\textbf{x}\in\ell^2(I)$ and positive number $r$, put
$$B^I(\textbf{x};r)\triangleq \left\{\textbf{y}\in\ell^2(I):\;||\textbf{y}-\textbf{x}||_{\ell^2(I)}<r \right\}.
$$
As usual, we denote $\mathbb{R}^{\mathbb{N}}$, $\ell^2(\mathbb{N})$ and $B^{\mathbb{N}}(\textbf{x};r)$ by $\mathbb{R}^{\infty}$, $\ell^2$ and $B(\textbf{x};r)$, respectively. Also, we define a mapping $\pi_{I}:\; \ell^2\to \ell^2(I)$ by setting
$$\pi_{I}\textbf{x}\triangleq \textbf{x}^I,\q\forall\;\textbf{x}\in\ell^2,$$
where $\textbf{x}=(x_i)_{i\in\mathbb{N}}$ and $\textbf{x}^I=(x_i)_{i\in I}\in \ell^2(I)$.

For any given $a>0$, we define a probability measure $\bn_a$ on $(\mathbb{R},\mathscr{B}(\mathbb{R}))$ by
$$\bn_a(B)\triangleq \frac{1}{\sqrt{2\pi a^2}}\int_{B}e^{-\frac{x^2}{2a^2}}\,\mathrm{d}x,\quad\,\forall\;B\in \mathscr{B}(\mathbb{R}),$$
where $\mathrm{d}x$ is the usual Lebesgue measure on $(\mathbb{R},\mathscr{B}(\mathbb{R}))$.

Now we define a product measure on the space $\mathbb{R}^{\infty}$, endowed with the usual product topology $\mathscr{T}$. We will use the notation $\mathscr{B}(\mathbb{R}^{\infty})$ to denote the $\sigma$-algebra
generated by $\mathscr{T}$. Let $\prod\limits_{n=1}^{\infty}\mathscr{B}(\mathbb{R})$ denote the countable product $\sigma$-algebra of $\mathscr{B}(\mathbb{R})$ which is the $\sigma$-algebra generated by the following family of subsets in $\mathbb{R}^{\infty}$:
$$
B_1\times B_2\times \cdots \times B_k\times \left(\prod_{i=k+1}^{\infty}\mathbb{R}\right),
$$
where $k\in \mathbb{N}$ and $B_i\in \mathscr{B}(\mathbb{R})$ for $i=1,\cdots,k$.
\begin{lemma}\label{230326lem2}
$\mathscr{B}(\mathbb{R}^{\infty})=\prod\limits_{n=1}^{\infty}\mathscr{B}(\mathbb{R})$.
\end{lemma}
\begin{proof}
It is well known that $\mathbb{R}^{\infty}$ is a separable, complete metric space and the topology corresponding to its metric is precisely $\mathscr{T}$. Therefore, by Corollary \ref{Lindelof1}, $(\mathbb{R}^{\infty},\mathscr{T})$ is a Lindel\"{o}f space and hence each $O\in \mathscr{T}$ is a countable union of sets in the following family $\mathscr{T}_0$:
$$
O_1\times\cdots \times O_k\times \Bigg(\prod_{i=k+1}^{\infty}\mathbb{R}\Bigg),
$$
where $k\in \mathbb{N}$ and $O_i$ is an open subset of $\mathbb{R}$ for $i=1,\cdots,k$.
This implies that $\mathscr{B}(\mathbb{R}^{\infty})\subset\prod\limits_{n=1}^{\infty}\mathscr{B}(\mathbb{R})$.

Conversely, for each $k\in\mathbb{N}$ and $B_1,\cdots,B_k\in \mathscr{B}(\mathbb{R})$, it is easy to see that $B_1\times B_2\times \cdots \times B_k\times \Bigg(\prod\limits_{i=k+1}^{\infty}\mathbb{R}\Bigg)$ lies in the $\sigma$-algebra generated by the family $\mathscr{T}_0$ (which is a subfamily of $\mathscr{T}$).
Thus $\prod\limits_{n=1}^{\infty}\mathscr{B}(\mathbb{R})\subset\mathscr{B}(\mathbb{R}^{\infty})$. This completes the proof of Lemma \ref{230326lem2}.
\end{proof}

\begin{lemma}\label{230330lem1}
It holds that $\mathscr{B}(\ell^2)=\{B\cap\ell^2:\; B\in \mathscr{B}(\mathbb{R}^{\infty})\}$.
\end{lemma}
\begin{proof}
For any $\textbf{x}=(x_i)_{i\in\mathbb{N}}\in \ell^2$ and $r\in(0,+\infty)$, one has
\begin{eqnarray*}
&&\{\textbf{x}'=(x_i')_{i\in\mathbb{N}}\in \ell^2:||\textbf{x}'-\textbf{x}||_{\ell^2}<r\}\\
&&=\bigcup_{n=1}^{\infty}\left\{\textbf{x}'=(x_i')_{i\in\mathbb{N}}\in \ell^2:||\textbf{x}'-\textbf{x}||_{\ell^2}\leqslant r-\frac{1}{n}\right\}\\
&&=\bigcup_{n=1}^{\infty}\left(\bigcap_{m=1}^{\infty}\left\{\textbf{x}'=(x_i')_{i\in\mathbb{N}}\in \ell^2:\left(\sum_{i=1}^{m}|x_i'-x_i|^2\right)^{\frac{1}{2}}\leqslant r-\frac{1}{n}\right\}\right)\\
&&\in\{B\cap\ell^2:\; B\in \mathscr{B}(\mathbb{R}^{\infty})\}.
\end{eqnarray*}
Since $\ell^2$ is a Lindel\"{o}f space (by Corollary \ref{Lindelof1}), each open subset $O$ of $\ell^2$ is a countable union of above open balls and hence $O\in\{B\cap\ell^2:\; B\in \mathscr{B}(\mathbb{R}^{\infty})\}$. Therefore, $\mathscr{B}(\ell^2)\subset\{B\cap\ell^2:\; B\in \mathscr{B}(\mathbb{R}^{\infty})\}$.

Conversely, it is easy to see that the collection of all sets of the following form:
\begin{eqnarray*}
\left(O_1\times \cdots \times O_k\times \Bigg(\prod_{i=k+1}^{\infty}\mathbb{R}\Bigg)\right)\cap \ell^2,
\end{eqnarray*}
where $k\in\mathbb{N}$ and $O_i$ is an open subset of $\mathbb{R}$ for each $i=1,\cdots,k$, is a family of open subsets of $\ell^2$ and the $\sigma$-algebra generated by this family of sets is precisely $\{B\cap\ell^2:\; B\in \mathscr{B}(\mathbb{R}^{\infty})\}$. This implies that $\{B\cap\ell^2:\; B\in \mathscr{B}(\mathbb{R}^{\infty})\}\subset\mathscr{B}(\ell^2)$. Therefore, $\mathscr{B}(\ell^2)=\{B\cap\ell^2:\; B\in \mathscr{B}(\mathbb{R}^{\infty})\}$ and the proof of Lemma \ref{230330lem1} is completed.
\end{proof}

In what follows, we fix a sequence $\{a_i\}_{i=1}^{\infty}$ of positive numbers such that
\bel{20241023e3}
\sum\limits_{i=1}^{\infty}a_i<\infty.
\ee
For any $r\in(0,+\infty)$, let
 $$
 \bn^r\triangleq\prod\limits_{i=1}^{\infty}\bn_{ra_i}
 $$
be the product measure on $(\mathbb{R}^{\infty},\mathscr{B}(\mathbb{R}^{\infty}))$.
We shall need the following known simple but useful result (e.g, \cite[Proposition 1.11]{Da}):
\begin{proposition}\label{20240108prop1}
$\bn^r(\ell^2)=1$.
\end{proposition}

Thanks to Lemma \ref{230330lem1} and Proposition \ref{20240108prop1}, we obtain a Borel probability measure $P_r$ on $(\ell^2,\mathscr{B}(\ell^2))$ by setting
\begin{equation}\label{220817e1}
P_r(B)\triangleq \bn^r(B),\qquad\forall\,B\in \mathscr{B}(\ell^2).
\end{equation}
For abbreviation, we shall write $P$ instead of $P_1$ in Sections \ref{20240111chapter2} and \ref{20240111chapter1}.

For each $n\in\mathbb{N}$, we define a probability measure $\bn^{n,r}$ on $(\mathbb{R}^n,\mathscr{B}(\mathbb{R}^n))$ by setting
\begin{eqnarray*}\label{def of N^n}
\bn^{n,r}\triangleq \prod_{i=1}^{n}\bn_{ra_i}
\end{eqnarray*}
and a probability measure ${\widehat{ \bn}}^{n,r}$ on $(\mathbb{R}^{\mathbb{N}\setminus\{1,\cdots,n\}},\mathscr{B}(\mathbb{R}^{\mathbb{N}\setminus\{1,\cdots,n\}}))$ by setting
\begin{eqnarray*}\label{def of widetildeN^n}
\widehat{\bn}^{n,r}\triangleq \prod_{i=n+1}^{\infty}\bn_{ra_i}.
\end{eqnarray*}
Then, we obtain the following probability measure $P_{n,r}$ on $\ell^2(\mathbb{N}\setminus\{1,\cdots,n\})$:
$$
P_{n,r}(E)\triangleq \widehat{\bn}^{n,r}(E),\q\forall\;E\in \mathscr{B}(\ell^2(\mathbb{N}\setminus\{1,\cdots,n\})).
$$
Obviously, $P_r=\bn^{n,r}\times P_{n,r}.$
\begin{lemma}\label{dilation of Gaussian measure}
For any $s,r\in(0,+\infty)$ and $E\in \mathscr{B}(\ell^2)$, it holds that
\begin{eqnarray*}
P_{sr}(E)=P_r(s^{-1}E).
\end{eqnarray*}
\end{lemma}
\begin{proof}
Let $\mathscr{A}\triangleq \{E\in \mathscr{B}(\ell^2):\;P_{rs}(E)=P_r(s^{-1}E)\}$. Then $\mathscr{A}$ is a $\sigma$-algebra. Denote by $\mathscr{C}$ the following family of subsets of $\ell^2$:
$$
\{(x_i)_{i\in\mathbb{N}}\in\ell^2:\;b_i<x_i<d_i\hbox{ for }i=1,\cdots,n\}.
$$
where $n\in \mathbb{N}$, $b_i,d_i\in\mathbb{R}$ with $b_i<d_i$ for $i=1,\cdots,n$.
For each $E=\{(x_i)_{i\in\mathbb{N}}\in\ell^2:b_i<x_i<d_i\hbox{ for }i=1,\cdots,n\}\in \mathscr{C}$, it follows that
\begin{eqnarray*}
P_r(s^{-1}E)&=&\prod_{i=1}^{n}\left(\frac{1}{\sqrt{2\pi r^2a_i^2}}\int_{s^{-1}b_i}^{s^{-1}d_i}e^{-\frac{x_i^2}{2r^2a_i^2}}\,\mathrm{d}x_i\right)\\
&=&\prod_{i=1}^{n}\left(\frac{1}{\sqrt{2\pi s^2r^2a_i^2}}\int_{ b_i}^{ d_i}e^{-\frac{y_i^2}{2s^2r^2a_i^2}}\,\mathrm{d}y_i\right)
=P_{sr}(E).
\end{eqnarray*}
Thus $\mathscr{C}\subset \mathscr{A}$ and hence $\{\ell^2\cap \left(B_1\times B_2\times \cdots \times B_k\times  (\prod_{i=k+1}^{\infty}\mathbb{R} )\right):\;B_i\in \mathscr{B}(\mathbb{R}),\,1\leq i\leq k\}\in \mathscr{A}$ for each $k\in\mathbb{N}$. By Lemma \ref{230330lem1}, we have $\mathscr{B}(\ell^2)\subset \mathscr{A}$. Therefore, $\mathscr{B}(\ell^2)= \mathscr{A}$ and this completes the proof of Lemma \ref{dilation of Gaussian measure}.
\end{proof}

For each $\textbf{x}=(x_i)_{i\in\mathbb{N}}\in\ell^2$ and $r\in(0,+\infty)$, we define
$$
P_r(E,\textbf{x})\triangleq P_r(E+\textbf{x}),\q \bn^r(E,\textbf{x})\triangleq \bn^r(E+\textbf{x}),\q\forall\; E\in \mathscr{B}(\ell^2).
$$
Then, the following result holds true:

\begin{lemma}\label{cha2prop1}
Suppose that $r,s\in(0,+\infty)$ and $\textbf{x}_1=(x_{i,1})_{i\in\mathbb{N}},\textbf{x}_2=(x_{i,2})_{i\in\mathbb{N}}\in\ell^2$.
If $r=s$ and $\sum_{i=1}^{\infty}\frac{(x_{i,1}-x_{i,2})^2}{ a_i^2} <\infty$, then $P_r(\cdot,\textbf{x}_1)$ and $P_s(\cdot,\textbf{x}_2)$ are equivalent. Otherwise, $P_r(\cdot,\textbf{x}_1)$ and $P_s(\cdot,\textbf{x}_2)$ are singular. Furthermore, it holds that
\begin{eqnarray}\label{230702e1}
\frac{\mathrm{d} P_r(\cdot,s_1\textbf{e}_i)}{\mathrm{d}P_r(\cdot)}(\textbf{y})=e^{-\frac{2s_1 y_i +s_1^2}{2r^2a_i^2}},\quad\forall\;i\in\mathbb{N},\, s_1\in\mathbb{R},
\end{eqnarray}
i.e., $\mathrm{d}P_r(\textbf{y}+s_1\textbf{e}_i)=e^{-\frac{2s_1 y_i +s_1^2}{2r^2a_i^2}}\,\mathrm{d}P_r(\textbf{y})$ and for any $n\in\mathbb{N},\, s_1,\cdots,s_n\in\mathbb{R},$
\begin{eqnarray}\label{230702e2}
\frac{\mathrm{d} P_r(\cdot,s_1\textbf{e}_1+\cdots+s_n\textbf{e}_n)}{\mathrm{d}P_r(\cdot)}(\textbf{y})=\prod_{k=1}^{n}e^{-\frac{2s_k y_k +s_k^2}{2t^2a_k^2}},
\end{eqnarray}
i.e., $\mathrm{d}P_r(\textbf{y}+s_1\textbf{e}_1+\cdots+s_n\textbf{e}_n)=e^{-\frac{2s_1 y_1 +s_1^2}{2r^2a_1^2}}\cdots e^{-\frac{2s_n y_n +s_n^2}{2r^2a_n^2}}\,\mathrm{d}P_r(\textbf{y})$, where $\textbf{y}=(y_i)_{i\in\mathbb{N}}\in \ell^2$.
\end{lemma}

\begin{proof}
By Theorem \ref{230309the1},
\begin{eqnarray*}
H(\bn^r(\cdot,\textbf{x}_1),\bn^s(\cdot,\textbf{x}_2))=\prod_{i=1}^{\infty}H(\bn_{ra_i}(\cdot,x_{i,1}),\bn_{sa_i}(\cdot,x_{i,2})),
\end{eqnarray*}
where $\bn_a(E,x)\triangleq \bn_a(E+x),\quad\forall\, E\in \mathscr{B}(\mathbb{R}),\,a\in(0,+\infty),\,x\in\mathbb{R}$. Note that for each $i\in\mathbb{N}$, we have
\begin{eqnarray*}
&&H(\bn_{ra_i}(\cdot,x_{i,1}),\bn_{sa_i}(\cdot,x_{i,2}))\\
&&=\int_{\mathbb{R}}\frac{1}{\sqrt{2\pi sra_i^2}}e^{-\frac{(x+x_{i,1})^2}{4r^2a_i^2}-\frac{(x+x_{i,2})^2}{4s^2a_i^2}}\,\mathrm{d}x\\
&&=\int_{\mathbb{R}}\frac{1}{\sqrt{2\pi sra_i^2}}e^{-\frac{x^2}{4r^2a_i^2}-\frac{(x-x_{i,1}+x_{i,2})^2}{4s^2a_i^2}}\,\mathrm{d}x\\
&&=\int_{\mathbb{R}}\frac{1}{\sqrt{2\pi sra_i^2}}e^{-\frac{(r^2+s^2)x^2-2r^2(x_{i,1}-x_{i,2})x+r^2(x_{i,1}-x_{i,2})^2}{4r^2s^2a_i^2}}\,\mathrm{d}x\\
&&=\int_{\mathbb{R}}\frac{1}{\sqrt{2\pi sra_i^2}}e^{-\frac{ x^2-\frac{2t^2(x_{i,1}-x_{i,2})x}{r^2+s^2}+\frac{r^2(x_{i,1}-x_{i,2})^2}{r^2+s^2}}{\frac{4r^2s^2a_i^2}{r^2+s^2}}}\,\mathrm{d}x\\
&&=\frac{\sqrt{2\pi \frac{2r^2s^2a_i^2}{r^2+s^2}}}{\sqrt{2\pi sra_i^2}}\cdot e^{-\frac{ -\left(\frac{r^2(x_{i,1}-x_{i,2})}{r^2+s^2}\right)^2+\frac{t^2(x_{i,1}-x_{i,2})^2}{r^2+s^2}}{\frac{4r^2s^2a_i^2}{r^2+s^2}}}\\
&&=\sqrt{\frac{2rs}{r^2+s^2}}\cdot e^{-\frac{(x_{i,1}-x_{i,2})^2}{4(r^2+s^2)a_i^2} }.
\end{eqnarray*}
If $r\neq s$, then $H(\bn_{ra_i}(\cdot,x_{i,1}),\bn_{sa_i}(\cdot,x_{i,2}))=\sqrt{\frac{2rs}{r^2+s^2}}<1$ for each $i\in\mathbb{N}$ and hence $H(\bn^r(\cdot,\textbf{x}_1),\bn^s(\cdot,\textbf{x}_2))=0$. By Theorem \ref{230309the1}, $\bn^r(\cdot,\textbf{x}_1)$ and $\bn^s(\cdot,\textbf{x}_2)$ are singular.

If $r= s$, then $H(\bn_{ra_i}(\cdot,x_{i,1}),\bn_{sa_i}(\cdot,x_{i,2}))=e^{-\frac{(x_{i,1}-x_{i,2})^2}{8r^2a_i^2} }$ for each $i\in\mathbb{N}$ and hence $H(\bn^r(\cdot,\textbf{x}_1),\bn^s(\cdot,\textbf{x}_2))>0$ if and only if
\begin{eqnarray*}
\sum_{i=1}^{\infty}\frac{(x_{i,1}-x_{i,2})^2}{ a_i^2} <\infty.
\end{eqnarray*}
In this case, by Theorem \ref{230309the1} again, $\bn^r(\cdot,\textbf{x}_1)$ and $\bn^s(\cdot,\textbf{x}_2)$ are equivalent. For simplicity, we only give the proof of \eqref{230702e1}. Note that for any $E\in \mathscr{B}(\ell^2)$, we have
\bel{230702e3}
\begin{array}{ll}
\ds P_r(E,s_1\textbf{e}_i)
=P_r(E+s_1\textbf{e}_i)=\int_{\ell^2(\mathbb{N})}\chi_{E+s_1\textbf{e}_i}(\textbf{y})\,\mathrm{d}P_r(\textbf{y})\\[3mm]
\ds=\int_{\ell^2(\mathbb{N})}\chi_{E}(\textbf{y}-s_1\textbf{e}_i)\,\mathrm{d}P_r(\textbf{y})=\int_{\ell^2(\mathbb{N})}\chi_{E}(\textbf{y})\,\mathrm{d}P_r(\textbf{y}+s_1\textbf{e}_i)
\\[3mm]
\ds=\int_{\mathbb{R}}\int_{ \ell^{2}(\mathbb{N}\setminus\{i\})}\chi_{E}(\textbf{y})\,\mathrm{d}P_r^{\widehat{i}}(\textbf{y}^i)\mathrm{d}\bn_{ra_i}(y_i+s_1)\\[3mm]
\ds
=\int_{\mathbb{R}}\int_{\ell^{2}(\mathbb{N}\setminus\{i\})}\chi_{E}(\textbf{y})\,\mathrm{d}P_r^{\widehat{i}}(\textbf{y}^i)\frac{1}{\sqrt{2\pi r^2 a_i^2}}e^{-\frac{(y_i+s_1)^2}{2r^2a_i^2}}\mathrm{d}y_i\\[3mm]
\ds
=\int_{\mathbb{R}}\int_{ \ell^{2}(\mathbb{N}\setminus\{i\})}\chi_{E}(\textbf{y})\cdot e^{-\frac{2s_1 y_i +s_1^2}{2r^2a_i^2}}\,\mathrm{d}P_r^{\widehat{i}}(\textbf{y}^i)\mathrm{d}\bn_{ra_i}(y_i)\\[3mm]
\ds
=\int_{E} e^{-\frac{2s_1 y_i +s_1^2}{2r^2a_i^2}}\,\mathrm{d}P_r(\textbf{y}),
\end{array}
\ee
where $\textbf{y}^i=(y_j)_{j\in \mathbb{N}\setminus \{i\}}\in \ell^{2}(\mathbb{N}\setminus\{i\})$ and $P^{\widehat{i}}_r$ is the product measure without the $i$-th component, i.e., it is the restriction of the product measure $\Pi_{j\in\mathbb{N}\setminus\{i\}}\bn_{ra_j}$ on $\left(\ell^{2}(\mathbb{N}\setminus\{i\}),\mathscr{B}\big(\ell^{2}(\mathbb{N}\setminus\{i\})\big)\right)$ (as that in (\ref{220817e1})). Since we can view $\ell^2(\mathbb{N})$ as $\mathbb{R}\times \ell^{2}(\mathbb{N}\setminus\{i\})$, we have $P_r =\bn_{ra_i}\times P_r^{\widehat{i}}$ which implies the fifth equality in \eqref{230702e3}. The sixth equality follows from the one dimensional translation formula $\mathrm{d}\bn_{ra_i}(y_i+s_1)=\frac{1}{\sqrt{2\pi r^2 a_i^2}}e^{-\frac{(y_i+s_1)^2}{2r^2a_i^2}}\mathrm{d}y_i$.

The conclusion \eqref{230702e1} follows from \eqref{230702e3}. This completes the proof of Lemma \ref{cha2prop1}.
\end{proof}

\begin{proposition}\label{230410prop1}
For each non-empty open subset $O$ of $\ell^2$ and $r\in(0,+\infty)$, it holds that $P_r(O)>0$.
\end{proposition}
\begin{proof}
Let
\begin{eqnarray*}
\mathscr{Q}\triangleq \bigcup\limits_{n=1}^{\infty}\{(r_1,\cdots,r_{n},0,\cdots,0,\cdots): r_1,\cdots,r_n\text{ are rational numbers}\}.
\end{eqnarray*}
Then $\mathscr{Q}$ is dense in $\ell^2$. By Lemma \ref{cha2prop1}, for each $\textbf{r}\in \mathscr{Q}$, it holds that $P_r(O)=0$ if and only if $P_r(\textbf{r}+O)=0$, where $\textbf{r}+O=\{\textbf{r}+\textbf{x}:\textbf{x}\in O\}$. Suppose that $P_r(O)=0$. Then for each $\textbf{r}\in \mathscr{Q}$, we have $P_r(\textbf{r}+O)=0$. Since $O$ is a non-empty open subset, there exists $r\in(0,+\infty)$ and $\textbf{x}\in\ell^2$ such that
\begin{eqnarray*}
\{\textbf{x}'\in\ell^2:||\textbf{x}'-\textbf{x}||_{\ell^2}<r\}\subset O.
\end{eqnarray*}
Since $\mathscr{Q}$ is dense in $\ell^2$, for each $\textbf{x}''\in\ell^2$, there exists $\textbf{r}'\in \mathscr{Q}$ such that
\begin{eqnarray*}
 ||\textbf{r}'+\textbf{x}-\textbf{x}''||_{\ell^2}<r,
\end{eqnarray*}
which implies that
\begin{eqnarray*}
\textbf{x}''\in  \textbf{r}'+O.
\end{eqnarray*}
Therefore, $\ell^2=\bigcup\limits_{\textbf{r}\in \mathscr{Q}}(\textbf{r}+O)$ and $\mathscr{Q}$ is a countable set. Thus $1=P_r(\ell^2)\leqslant \sum\limits_{\textbf{r}\in \mathscr{Q}}P_r(\textbf{r}+O)=0$ which is a contradiction. Therefore, $P_r(O)>0$ and this completes the proof of Proposition \ref{230410prop1}.
\end{proof}

\begin{remark}
As in \cite[p. 125]{Com Neg}, a probability measure with the property stated in Proposition \ref{230410prop1} is called a strictly positive measure.
\end{remark}

As pointed in \cite[Proposition 9, p. 152]{Gro67}, the convolution of bounded $\mathscr{B}(\ell^2)$-measurable function on $\ell^2$ with the measure $P_r$ maybe not continuous with respect to the $\ell^2$ norm topology on $\ell^2$.
For any bounded, real-valued $\mathscr{B}(\ell^2)$-measurable $f$ on $\ell^2$ and $t\in(0,+\infty)$, put
$$
(P_tf)(\textbf{x})\triangleq \int_{\ell^2}f(\textbf{x}-\textbf{y})\,\mathrm{d}P_t(\textbf{y}),\quad \forall\,\textbf{x}\in\ell^2.
$$
As in \cite[Proposition 9, pp. 152--154]{Gro67}, we have the following properties for $P_tf$.

\begin{proposition}\label{partial derivative of Ptf}
All order of partial derivatives of $P_tf$ exist at every point of $\ell^2$. Moreover, for each $\textbf{x}\in\ell^2$ and $i\in\mathbb{N}$, it holds that

{\rm (1)} $\frac{\partial (P_tf)(\textbf{x})}{\partial x_i}=-\int_{\ell^2}f(\textbf{x}-\textbf{y})\cdot\frac{y_i}{t^2a_i^2}\,\mathrm{d}P_t(\textbf{y})$;

{\rm (2)} $\left|\frac{\partial (P_tf)(\textbf{x})}{\partial x_i}\right|\leqslant \frac{1}{ta_i} \sup\limits_{\textbf{y}\in\ell^2}|f(\textbf{y})|$;

{\rm (3)} All order of partial derivatives of $P_tf$ are $F$-continuous functions on $\ell^2$;

{\rm (4)} All order of partial derivatives of $P_tf$ are $\mathscr{B}(\ell^2)$-measurable functions on $\ell^2$.
\end{proposition}

\begin{proof}
(1) For any $\textbf{x}\in\ell^2$, $i\in\mathbb{N}$ and $s\in\mathbb{R}$, by \eqref{230702e1} in Lemma \ref{cha2prop1}, it holds that
\begin{eqnarray*}
(P_tf)(\textbf{x}+s\textbf{e}_i)
&=& \int_{\ell^2}f(\textbf{x}+s\textbf{e}_i-\textbf{y})\,\mathrm{d}P_t(\textbf{y})\\
&=& \int_{\ell^2}f(\textbf{x}-\textbf{y})\,\mathrm{d}P_t(\textbf{y}+s\textbf{e}_i)= \int_{\ell^2}f(\textbf{x}-\textbf{y})e^{-\frac{2s y_i +s^2}{2t^2a_i^2}}\,\mathrm{d}P_t(\textbf{y}).
\end{eqnarray*}
Put $J(s,y_i)\triangleq e^{-\frac{2s y_i +s^2}{2t^2a_i^2}},\,\forall \, s,y_i\in\mathbb{R}$. Then,
\begin{eqnarray*}
\frac{\partial J_i(rs,y_i)}{\partial r}= -\frac{s y_i +rs^2}{t^2a_i^2}\cdot J_i(rs,y_i),\quad\forall\,r\in\mathbb{R},
\end{eqnarray*}
and
\begin{eqnarray*}
(P_tf)(\textbf{x}+s\textbf{e}_i)-(P_tf)(\textbf{x})=-\int_{\ell^2}\int_0^1f(\textbf{x}-\textbf{y})\cdot\frac{s y_i +rs^2}{t^2a_i^2}\cdot J_i(rs,y_i)\,\mathrm{d}r \mathrm{d}P_t(\textbf{y}).
\end{eqnarray*}
Hence
\begin{eqnarray*}
&&\left|(P_tf)(\textbf{x}+s\textbf{e}_i)-(P_tf)(\textbf{x})+s\cdot\int_{\ell^2}f(\textbf{x}-\textbf{y})\cdot\frac{y_i}{t^2a_i^2}\,\mathrm{d}P_t(\textbf{y})\right|\\
&&\leqslant\left|\int_{\ell^2}\int_0^1f(\textbf{x}-\textbf{y})\cdot\frac{s y_i }{t^2a_i^2}\cdot (J_i(rs,y_i)-1)\,\mathrm{d}r \mathrm{d}P_t(\textbf{y})\right|\\
&&\q+\left|\int_{\ell^2}\int_0^1f(\textbf{x}-\textbf{y})\cdot\frac{-rs^2}{t^2a_i^2}\cdot J_i(rs,y_i)\,\mathrm{d}r \mathrm{d}P_t(\textbf{y})\right|\\
&&\leqslant \frac{ \sup\limits_{\textbf{y}\in\ell^2}|f(\textbf{y})|}{t^2a_i^2}\cdot\left|\int_{\ell^2}\int_0^1|y_i\cdot s\cdot (J_i(rs,y_i)-1)|\,\mathrm{d}r \mathrm{d}P_t(\textbf{y})\right|\\
&&\q+s^2\cdot\frac{\sup\limits_{\textbf{y}\in\ell^2}|f(\textbf{y})|}{t^2a_i^2}\left|\int_{\ell^2}\int_0^1  r\cdot J_i(rs,y_i)\,\mathrm{d}r \mathrm{d}P_t(\textbf{y})\right|\\
&&= \frac{ \sup\limits_{\textbf{y}\in\ell^2}|f(\textbf{y})|}{t^2a_i^2}\cdot\left|\int_0^1\int_{\ell^2}|y_i\cdot s\cdot (J_i(rs,y_i)-1)|\,\mathrm{d}P_t(\textbf{y})\mathrm{d}r \right|\\
&&\q+s^2\cdot\frac{\sup\limits_{\textbf{y}\in\ell^2}|f(\textbf{y})|}{t^2a_i^2}\left|\int_0^1\int_{\ell^2}  r\cdot J_i(rs,y_i)\,\mathrm{d}P_t(\textbf{y})\mathrm{d}r \right|\\
&&\leqslant |s|\cdot \frac{ \sup\limits_{\textbf{y}\in\ell^2}|f(\textbf{y})|}{t^2a_i^2}\cdot \int_0^1\left(\int_{\ell^2}|y_i|^2\,\mathrm{d}P_t(\textbf{y})\right)^{\frac{1}{2}}\\
&&\q\cdot \left(\int_{\ell^2}(J_i(rs,y_i)-1)^2\,\mathrm{d}P_t(\textbf{y})\right)^{\frac{1}{2}}\mathrm{d}r \\
&&\q+s^2\cdot\frac{\sup\limits_{\textbf{y}\in\ell^2}|f(\textbf{y})|}{t^2a_i^2} \int_0^1 r\cdot\int_{\ell^2}  J_i(rs,y_i)\,\mathrm{d}P_t(\textbf{y})\mathrm{d}r.
\end{eqnarray*}
Note that for each $r\in[0,1]$, we have
\begin{eqnarray*}
\int_{\ell^2}(J_i(rs,y_i)-1)^2\,\mathrm{d}P_t(\textbf{y})=e^{\frac{r^2s^2}{t^2a_i^2}}-1,\quad\int_{\ell^2}  J_i(rs,y_i)\,\mathrm{d}P_t(\textbf{y})=1.
\end{eqnarray*}
Thus it holds that
\begin{eqnarray*}
&&\left|(P_tf)(\textbf{x}+s\textbf{e}_i)-(P_tf)(\textbf{x})+s\cdot\int_{\ell^2}f(\textbf{x}-\textbf{y})\cdot\frac{y_i}{t^2a_i^2}\,\mathrm{d}P_t(\textbf{y})\right|\\
&&\leqslant |s|\cdot \frac{ \sup\limits_{\textbf{y}\in\ell^2}|f(\textbf{y})|}{t^2a_i^2}\cdot \left(\int_{\ell^2}|y_i|^2\,\mathrm{d}P_t(\textbf{y})\right)^{\frac{1}{2}}\cdot \int_0^1 \left(e^{\frac{r^2s^2}{t^2a_i^2}}-1\right)^{\frac{1}{2}}\mathrm{d}r \\
&&\q+s^2\cdot\frac{\sup\limits_{\textbf{y}\in\ell^2}|f(\textbf{y})|}{t^2a_i^2}\cdot \int_0^1 r\mathrm{d}r\\
&&\leqslant |s|\cdot \frac{ \sup\limits_{\textbf{y}\in\ell^2}|f(\textbf{y})|}{ta_i}\cdot\left(e^{\frac{s^2}{t^2a_i^2}}-1\right)^{\frac{1}{2}} +s^2\cdot\frac{\sup\limits_{\textbf{y}\in\ell^2}|f(\textbf{y})|}{t^2a_i^2}\cdot \frac{1}{2},
\end{eqnarray*}
which implies that
\begin{eqnarray*}
\left|(P_tf)(\textbf{x}+s\textbf{e}_i)-(P_tf)(\textbf{x})+s\cdot\int_{\ell^2}f(\textbf{x}-\textbf{y})\cdot\frac{y_i}{t^2a_i^2}\,\mathrm{d}P_t(\textbf{y})\right|=o(|s|),
\end{eqnarray*}
as $s\to 0$. Therefore, $\frac{\partial (P_tf)(\textbf{x})}{\partial x_i}=-\int_{\ell^2}f(\textbf{x}-\textbf{y})\cdot\frac{y_i}{t^2a_i^2}\,\mathrm{d}P_t(\textbf{y})$.

\medskip

(2) By the above conclusion (1), it holds that
\begin{eqnarray*}
 \left|\frac{\partial (P_tf)(\textbf{x})}{\partial x_i}\right|&\leqslant & \left|\int_{\ell^2}f(\textbf{x}-\textbf{y})\cdot\frac{y_i}{t^2a_i^2}\,\mathrm{d}P_t(\textbf{y})\right|\\
 &\leqslant&\frac{ \sup\limits_{\textbf{y}\in\ell^2}|f(\textbf{y})|}{t^2a_i^2}\cdot\left(\int_{\ell^2}y_i^2\,\mathrm{d}P_t(\textbf{y})\right)^{\frac{1}{2}}=\frac{ \sup\limits_{\textbf{y}\in\ell^2}|f(\textbf{y})|}{ta_i}.
 \end{eqnarray*}

\medskip

(3) For simplicity, we only prove that $P_tf$ is $F$-continuous on $\ell^2$. Note that for each $\textbf{x}\in \ell^2$, $n\in\mathbb{N}$ and $(s_1,\cdots,s_n)\in\mathbb{R}^n$, by \eqref{230702e2} in Lemma \ref{cha2prop1}, we have
\begin{eqnarray*}
&&(P_tf)(\textbf{x}+s_1\textbf{e}_1+\cdots+s_n\textbf{e}_n)
= \int_{\ell^2}f(\textbf{x}+s_1\textbf{e}_1+\cdots+s_n\textbf{e}_n-\textbf{y})\,\mathrm{d}P_t(\textbf{y})\\
&&= \int_{\ell^2}f(\textbf{x}-\textbf{y})\,\mathrm{d}P_t(\textbf{y}+s_1\textbf{e}_1+\cdots+s_n\textbf{e}_n)\\
&&= \int_{\ell^2}f(\textbf{x}-\textbf{y})e^{-\frac{2s_1 y_1 +s_1^2}{2t^2a_1^2}}\cdots e^{-\frac{2s_n y_n +s_n^2}{2t^2a_n^2}}\,\mathrm{d}P_t(\textbf{y})\\
&&=e^{\frac{-s_1^2}{2t^2a_1^2}}\cdots e^{\frac{-s_n^2}{2t^2a_n^2}} \int_{\ell^2}f(\textbf{x}-\textbf{y})e^{-\frac{2s_1 y_1}{2t^2a_1^2}}\cdots e^{-\frac{2s_n y_n }{2t^2a_n^2}}\,\mathrm{d}P_t(\textbf{y}),
\end{eqnarray*}
from which we deduce that the mapping
\begin{eqnarray*}
(s_1,\cdots,s_n)\mapsto (P_tf)(\textbf{x}+s_1\textbf{e}_1+\cdots+s_n\textbf{e}_n),\quad\forall\,(s_1,\cdots,s_n)\in\mathbb{R}^n,
\end{eqnarray*}
is continuous from $\mathbb{R}^n$ into $\mathbb{R}$. This implies that $P_tf$ is $F$-continuous.

\medskip

(4) For simplicity, we only prove that $P_tf$ is $\mathscr{B}(\ell^2)$-measurable. By the classic procedure in real analysis, there exits a sequence of uniformly bounded simple functions $\{f_n\}_{n=1}^{\infty}$ on $\ell^2$ such that $\lim\limits_{n\to\infty}f_n(\textbf{x})=f(\textbf{x})$ for any $\textbf{x}\in\ell^2.$ By bounded convergence theorem, we have $\lim\limits_{n\to\infty}(P_tf_n)(\textbf{x})=(P_tf)(\textbf{x})$ for any $\textbf{x}\in\ell^2.$ Therefore, we only need to prove that if $E$ is a Borel subset of $\ell^2$, then $P_t\chi_E$ is $\mathscr{B}(\ell^2)$-measurable. Let
\begin{eqnarray*}
\mathscr{M}\triangleq \left\{E\in\mathscr{B}(\ell^2):\; P_t\chi_E\text{ is $\mathscr{B}(\ell^2)$-measurable}\right\}.
\end{eqnarray*}
Then $\mathscr{M}$ is closed under monotone limits and complements. Denote by $\mathscr{C}$ the following family of subsets of $\ell^2$:
\begin{eqnarray*}
\left\{(x_i)_{i\in\mathbb{N}}\in\ell^2:(x_1,\cdots,x_n)\in B\right\},
\end{eqnarray*}
where $n\in \mathbb{N}$ and $B\in \mathscr{B}(\mathbb{R}^n)$. It is easy to see that $P_t\chi_E$ is continuous for any $E\in \mathscr{C}$ and hence $\mathscr{C}\subset\mathscr{M}$. Therefore, $\mathscr{M}=\mathscr{B}(\ell^2)$ which completes the proof of Proposition \ref{partial derivative of Ptf}.
\end{proof}

We shall construct below a function on $\ell^2$, for which all of its partial derivatives exist but it is not continuous respect to the $\ell^2$-norm topology. Let $A\triangleq \{\textbf{x}=(x_i)_{i\in\mathbb{N}}\in\ell^2:x_i\in (0,+\infty)\text{ for each }i\in\mathbb{N}\}$, $\textbf{x}_0\triangleq(\sqrt{a_i})_{i\in\mathbb{N}}\in \ell^2$ and
\bel{241118e1}
f(\textbf{x})\triangleq (P_1\chi_A)(\textbf{x}),\quad \forall\;\textbf{x}\in\ell^2.
\ee
\begin{proposition}\label{not continuous}
The function $f$ given by \eqref{241118e1} is $F$-continuous, and all of its partial derivatives exist at each $\textbf{x}=(x_i)_{i\in\mathbb{N}}\in\ell^2$. However, $f$ is not continuous at $\textbf{x}_0$ respect to the $\ell^2$-norm topology.
\end{proposition}
\begin{proof}
For each $n\in\mathbb{N}$, let
$$
x_{n,m}\triangleq\left\{\begin{array}{ll}
\ds\sqrt{a_m},\quad&\ds 1\leqslant m\leqslant n,\\[2mm]
\ds (-1)^{m+1}\sqrt{a_m},\q&\ds m>n.
\end{array}\right.
$$
For $n\in\mathbb{N}$, set $\textbf{x}_n\triangleq (x_{n,m})_{m\in\mathbb{N}}$. Then, $\{\textbf{x}_n\}_{n=1}^{\infty}\subset \ell^2$ and $\lim\limits_{n\to\infty}||\textbf{x}_n-\textbf{x}_0||_{\ell^2}=0$. Note that for each $n\in\mathbb{N}$, we have
\begin{eqnarray*}
f(\textbf{x}_n)=\int_{\ell^2} \chi_A(\textbf{x}_n-\textbf{y})\,\mathrm{d}P_1(\textbf{y})=\prod_{i=1}^{\infty}\frac{1}{\sqrt{2\pi a_i^2}}\int_{-\infty}^{x_{n,m}}e^{-\frac{y_i^2}{2a_i^2}}\,\mathrm{d}y_i=0,
\end{eqnarray*}
where the last equality follows from the fact that
\begin{eqnarray*}
\frac{1}{\sqrt{2\pi a_i^2}}\int_{-\infty}^{x_{n,i}}e^{-\frac{y_i^2}{2a_i^2}}\,\mathrm{d}y_i\leqslant \frac{1}{\sqrt{2\pi a_i^2}}\int_{-\infty}^{0}e^{-\frac{y_i^2}{2a_i^2}}\,\mathrm{d}y_i=\frac{1}{2}
\end{eqnarray*}
for all $i=2(n+k)$ and $k\in\mathbb{N}$. We also note that
\begin{eqnarray*}
f(\textbf{x}_0)&=&\int_{\ell^2} \chi_A(\textbf{x}_0-\textbf{y})\,\mathrm{d}P_1(\textbf{y})=\prod_{i=1}^{\infty}\frac{1}{\sqrt{2\pi a_i^2}}\int_{-\infty}^{\sqrt{a_i}}e^{-\frac{y_i^2}{2a_i^2}}\,\mathrm{d}y_i\\
&=& \prod_{i=1}^{\infty}\left(1-\frac{1}{\sqrt{2\pi a_i^2}}\int_{\sqrt{a_i}}^{+\infty} e^{-\frac{y_i^2}{2a_i^2}}\,\mathrm{d}y_i\right),
\end{eqnarray*}
and
\begin{eqnarray*}
&&\sum_{i=1}^{\infty} \frac{1}{\sqrt{2\pi a_i^2}}\int_{\sqrt{a_i}}^{+\infty} e^{-\frac{y_i^2}{2a_i^2}}\,\mathrm{d}y_i
=\sum_{i=1}^{\infty} \frac{1}{\sqrt{\pi }}\int_{\frac{1}{\sqrt{2a_i}}}^{+\infty} e^{- y_i^2 }\,\mathrm{d}y_i\\
&&\leqslant \sum_{i=1}^{\infty} \frac{1}{\sqrt{\pi }}\int_{\frac{1}{\sqrt{2a_i}}}^{+\infty} e^{- \frac{y_i}{\sqrt{2a_i}}  }\,\mathrm{d}y_i
=\frac{1}{\sqrt{\pi }}\sum_{i=1}^{\infty}\sqrt{2a_i}e^{-\frac{1}{2a_i}}
\leqslant\frac{1}{\sqrt{\pi }}\sum_{i=1}^{\infty}2\sqrt{2a_i} a_i <\infty.
\end{eqnarray*}
By \cite[Theorem 15.4]{Rud87}, we have $f(\textbf{x}_0)>0$.  Therefore, $f$ is not continuous at $\textbf{x}_0$ respect to the $\ell^2$-norm topology. This completes the proof of Proposition \ref{not continuous}.
\end{proof}

The following result can be viewed as a special version of Fernique's theorem for abstract Wiener spaces (\cite{Fer}).
\begin{proposition}\label{230522prop1}
For $r\in(0,+\infty)$, it holds that
\begin{eqnarray}\label{230522e1}
\int_{\ell^2}e^{c\cdot||\textbf{x}||_{\ell^2}^2}\,\mathrm{d}P_r(\textbf{x})<\infty,\quad \forall\,c\in \left(0,\frac{1}{2r^2\cdot \sup\limits_{k\in\mathbb{N}} a_k^2}\right),
\end{eqnarray}
\begin{eqnarray}\label{230522e2}
\int_{\ell^2}e^{c\cdot||\textbf{x}||_{\ell^2}^{2+\delta}}\,\mathrm{d}P_r(\textbf{x})=\infty,\quad \forall\,c,\delta\in(0,+\infty),
\end{eqnarray}
and
\begin{eqnarray}\label{230522e3}
\int_{\ell^2}e^{c\cdot||\textbf{x}||_{\ell^2}^2}\,\mathrm{d}P_r(\textbf{x})=\infty,\quad \forall\,c\in \left[\frac{1}{2r^2\cdot \sup\limits_{k\in\mathbb{N}} a_k^2},+\infty\right).
\end{eqnarray}
\end{proposition}
\begin{proof}
For $c\in \left(0,\frac{1}{2r^2\cdot \sup\limits_{k\in\mathbb{N}} a_k^2}\right)$, we have
\begin{eqnarray*}
&&\int_{\ell^2}e^{c\cdot||\textbf{x}||_{\ell^2}^2}\,\mathrm{d}P_r(\textbf{x})
=
\int_{\ell^2}e^{c\cdot\sum\limits_{k=1}^{\infty}x_k^2}\,\mathrm{d}P_r(\textbf{x})
=\prod_{k=1}^{\infty}\int_{\mathbb{R}}\frac{1}{\sqrt{2\pi r^2a_k^2}}e^{cx_k^2-\frac{x_k^2}{2r^2a_k^2}}\,\mathrm{d}x_k\\
&&=\prod_{k=1}^{\infty}\int_{\mathbb{R}}\frac{1}{\sqrt{2\pi r^2a_k^2}}e^{ -\frac{x_k^2}{2\frac{r^2a_k^2}{1-c2r^2a_k^2}}}\,\mathrm{d}x_k
=\prod_{k=1}^{\infty}\sqrt{\frac{1}{1-2r^2a_k^2c}}<\infty,
\end{eqnarray*}
where the second equality follows from the monotone convergence theorem and the last inequality follows from \cite[Theorem 15.4, p. 299]{Rud87} and the fact that $\sum\limits_{k=1}^{\infty}2r^2a_k^2c<\infty$. We have proved \eqref{230522e1}. Finally, \eqref{230522e2} and \eqref{230522e3} follows from the following two simple facts:
\begin{eqnarray*}
\int_{\mathbb{R}}e^{ax^2}\,\mathrm{d}x&=&\infty,\,\quad\forall\, a\in[0,+\infty),\\
\int_{\mathbb{R}}e^{c|x|^{2+\delta}-ax^2}\,\mathrm{d}x&=&\infty,\,\quad\forall\, a,c,\delta\in(0,+\infty).
\end{eqnarray*}
This completes the proof of Proposition \ref{230522prop1}.
\end{proof}

\subsection{General Co-dimensional Surfaces}

The rest of this section is based on \cite{WYZ}.

Suppose that $S$ is a nonempty subset of $\ell^2$ and $\textbf{x}\in S$. Furthermore, suppose that $\ell^2=M_1\oplus M_2$ and $\textbf{x}=\textbf{x}_1+\textbf{x}_2\in S$, where $M_1$ and $M_2$ are two closed linear subspaces of $\ell^2$, $M_1 \bot M_2$, $\textbf{x}_1\in M_1$ and $\textbf{x}_2\in M_2$. If there exists an open neighborhood $U_0$ of $\textbf{x}$ in $\ell^2$, an open neighborhood $U$ of $\textbf{x}_2$ in $M_2$ and $f\in C^1(U;M_1)$ such that $f(\textbf{x}_2)=\textbf{x}_1$ and
$$
S\cap U_0=\{f(\textbf{x}_2')+\textbf{x}_2':\textbf{x}_2'\in U\},
$$
then for $\textbf{x}'=f(\textbf{x}_2')+\textbf{x}_2'\in S\cap U_0$,
\begin{equation}\label{240323for1}
\begin{array}{ll}
\ds\textbf{x}'-\textbf{x}= f(\textbf{x}_2')-f(\textbf{x}_2)+\textbf{x}_2'-\textbf{x}_2\\[2mm]
\ds=(Df(\textbf{x}_2))(\textbf{x}_2'-\textbf{x}_2)+\textbf{x}_2'-\textbf{x}_2+o(||\textbf{x}_2'-\textbf{x}_2||_{\ell^2}),\hbox{ as }\textbf{x}_2'\to\textbf{x}_2\hbox{ in }U.
\end{array}
\end{equation}
We will use the notations $P$ and $M$ to denote respectively the orthogonal normal projections in $\ell^2$ whose ranges are
\begin{eqnarray}\label{20240705for5}
\{\textbf{x}\in\ell^2:\;{\lan \textbf{x},(Df(\textbf{x}_2))\textbf{x}_2'+\textbf{x}_2'\ran}_{\ell^2}=0,\,\forall\;\textbf{x}_2'\in M_2\}
\end{eqnarray}
and
\begin{eqnarray}\label{20240705for4}
\overline{\{(Df(\textbf{x}_2))\textbf{x}_2'+\textbf{x}_2':\;\textbf{x}_2'\in M_2\}}.
\end{eqnarray}

The following simple result will play a key role in the sequel.
\begin{proposition}\label{240618prop1}
For any orthogonal normal projection $P'$ in $\ell^2$, if
\begin{eqnarray}\label{240326for1}
P'(\textbf{x}'-\textbf{x})=o(||\textbf{x}'-\textbf{x}||_{\ell^2}),\quad\,\text{ as }\textbf{x}'\to\textbf{x}\hbox{ in }S,
\end{eqnarray}
then $P'\subset P$.
\begin{proof}
We use the contradiction argument. Suppose that $P'\nsubseteq P$, then there would exist $\textbf{x}''\in P'\setminus P$ and $\textbf{x}''=\textbf{x}''_1+\textbf{x}''_2$, where $\textbf{x}''_1\in P$ and $\textbf{x}''_2\in M\setminus\{0\}$. Note that there exists $\textbf{x}'=(Df(\textbf{x}_2))\textbf{x}_2'+\textbf{x}_2'\in \{(Df(\textbf{x}_2))\textbf{x}_2'+\textbf{x}_2':\textbf{x}_2'\in M_2\}(\subset M)$ such that ${\lan\textbf{x}''_2,\textbf{x}'\ran}_{\ell^2}\neq 0$. Thus for sufficiently small $t\in\mathbb{R}$, we have $\textbf{x}_2+t\cdot \textbf{x}_2'\in U$ and as $t\to 0$, it holds that
\begin{eqnarray*}
&&||P'(f(\textbf{x}_2+t\cdot \textbf{x}_2')+\textbf{x}_2+t\cdot \textbf{x}_2'-f(\textbf{x}_2)-\textbf{x}_2)||_{\ell^2}\\
&&=||P'(f(\textbf{x}_2+t\cdot \textbf{x}_2')-f(\textbf{x}_2)+ t\cdot \textbf{x}_2')||_{\ell^2}\\
&&=||P'(Df(\textbf{x}_2)(t\cdot \textbf{x}_2')+o(t)+ t\cdot \textbf{x}_2') ||_{\ell^2}\\
&&=||P'(Df(\textbf{x}_2)(t\cdot \textbf{x}_2')+ t\cdot \textbf{x}_2')+o(t) ||_{\ell^2}
=||P'(t\cdot \textbf{x}')+o(t) ||_{\ell^2}\\
&&\geq||P'(t\cdot \textbf{x}')||_{\ell^2}-o(t)
\geq\left|\frac{1}{||\textbf{x}''||_{\ell^2}}{\lan\textbf{x}'',t\cdot \textbf{x}'\ran}_{\ell^2}\right|-o(t)\\
&&
=\frac{|t|}{||\textbf{x}''||_{\ell^2}}\cdot|{\lan\textbf{x}''_2,\textbf{x}'\ran}_{\ell^2}|-o(t),
\end{eqnarray*}
which contradicts to \eqref{240326for1}. This completes the proof of Proposition \ref{240618prop1}.
\end{proof}
\end{proposition}

As a consequence of Proposition \ref{240618prop1}, one can deduce the following result.
\begin{corollary}\label{20240705cor1}
The spaces in \eqref{20240705for4} and \eqref{20240705for5} are independent of $M_1,M_2$ and $f$.
\end{corollary}
\begin{remark}
Roughly speaking, we can view \eqref{20240705for4} and \eqref{20240705for5} as the tangent space and the normal space of $S$ at $\textbf{x}\in S$, respectively. Corollary \ref{20240705cor1} says that these two spaces are independent of the coordinates and the coordinate functions of $S$, and hence they are geometric invariants.
\end{remark}

For each $i\in\mathbb{N}$, write $\textbf{e}_i\triangleq (\delta_{i,k})_{k\in\mathbb{N}}$ where $\delta_{i,k}\triangleq 0$ if $i\neq k$ and $\delta_{i,k}\triangleq 1$ if $i=k$. Then $\{\textbf{e}_i:i\in\mathbb{N}\}$ is an orthonormal basis of $\ell^2$. For a subset $I$ of $\mathbb{N}$, if $I=\emptyset$, denote the linear spaces spanned by the zero vector of $\ell^2$ by $P_I$, if $I\neq\emptyset$, denote the closed linear space spanned by $\{\textbf{e}_i:i\in I\}$ by $P_I$.

Suppose that $I_1$ and $I_2$ are two non-empty subsets of $\mathbb{N}$, if there exists $k\in\mathbb{N}_0$, $K_1 \subset I_1$ and $K_2\subset \mathbb{N}\setminus I_1$ such that $|K_1|=|K_2|=k$ and $I_2=(I_1\setminus K_1)\cup K_2$, then we write $I_1\sim I_2$.
\begin{lemma}\label{20241011lem2}
$\sim$ is an equivalent relation.
\end{lemma}
\begin{proof}
We only need to prove that if $I_1\sim I_2$ and $I_2\sim I_3$, then $I_1\sim I_3$. Note that $I_1\sim I_3$ is equivalent to $|I_1\setminus I_3|=|I_3\setminus I_1|\in \mathbb{N}_0$. Let
$$
S_1\triangleq I_1\setminus(I_2\cup I_3),\,S_2\triangleq I_2\setminus(I_1\cup I_3),\,S_3\triangleq I_3\setminus(I_1\cup I_2).
$$
and
$$
S_4\triangleq(I_1\cap I_2)\setminus I_3,\,S_5\triangleq(I_1\cap I_3)\setminus I_2,\,S_6\triangleq (I_2\cap I_3)\setminus I_1.
$$
Since $I_3\setminus I_1=S_3\sqcup S_6$, $S_1,S_2,S_3,S_4,S_5,S_6$ are disjoint finite subsets of $\mathbb{N}$, $I_1\setminus I_3=S_1\sqcup S_4$, we only to prove that $|S_1\sqcup S_4\sqcup S_5|=|S_3\sqcup S_6\sqcup S_5|$. By the facts that $|I_1\setminus I_2|=|I_2\setminus I_1|\in \mathbb{N}_0$, $I_1\setminus I_2=S_1\sqcup S_5$ and $I_2\setminus I_1=S_2\sqcup S_6$, it holds that
$|S_1\sqcup S_4\sqcup S_5|=|S_2\sqcup S_4\sqcup S_6|$. Similarly, $I_2\sim I_3$ implies that $|S_2\sqcup S_4\sqcup S_6|=|S_3\sqcup S_6\sqcup S_5|$ which completes the proof of Lemma \ref{20241011lem2}.
\end{proof}

For any non-empty subset $I$ of $\mathbb{N}$, write $J=\mathbb{N}\setminus I$ and
$$
\Gamma_I\triangleq \{I':I'\subset \mathbb{N}, \, I'\sim I\}.
$$

We now introduce the following notion.
\begin{definition}
We say that $S$ is a surface of $\ell^2$ with codimension $\Gamma_J$, if for each $\textbf{x}\in S$ there exists an open neighborhood $U$ of $\textbf{x}$ and $I_1\in\Gamma_I$ such that $P_{I_1}|_{S\cap U}$ is a homeomorphism from $S\cap U$ onto an open subset $U_{I_1}$ of $P_{I_1}$. In this case there exists a continuous function $f$ from $U_{I_1}$ into $P_{\mathbb{N}\setminus I_1}$ such that
$$
S\cap U=\left\{f(\textbf{x}_{I_1})+\textbf{x}_{I_1}:\;\textbf{x}_{I_1}=\sum_{i\in I_1}x_i\textbf{e}_i\in U_{I_1}\right\},
$$
and we call $(S\cap U, P_{I_1})$ a coordinate pairing of $S$.
\end{definition}

\subsection{Construction of Surface Measures}
Before proceeding, we need the determinant tools for some special operators.
\begin{definition}\label{20241113def1}
Suppose that $I_1,I_2\subset \mathbb{N}$ and $|I_1\setminus I_2|=|I_2\setminus I_1|=s\in\mathbb{N}_0$. If $T$ is bounded linear operator from $P_{I_1}$ into $P_{I_2}$ and there exists $I_0\subset I_1\cap I_2$ such that $|I_0|=t\in\mathbb{N}_0$ and
\begin{eqnarray}\label{20241011for3}
T|_{P_{(I_1\cap I_2)\setminus I_0}}=id_{P_{(I_1\cap I_2)\setminus I_0}},
\end{eqnarray}
where the right side means the identity operator on $P_{(I_1\cap I_2)\setminus I_0}$. Then there exists $i_1,\cdots,i_{s+t},j_1,\cdots,j_{s+t}\in\mathbb{N}$ such that $i_1<\cdots<i_{s+t},j_1<\cdots<j_{s+t}\in\mathbb{N}$, $I_0\cup (I_1\setminus I_2)=\{i_1,\cdots,i_{s+t}\}$ and $I_0\cup (I_2\setminus I_1)=\{j_1,\cdots,j_{s+t}\}$. If $s+t=0$, we define det$T=1$, otherwise we define
\begin{eqnarray}
\text{det}\,T\triangleq  \left|\begin{array}{cccc}
(T\textbf{e}_{i_1},\textbf{e}_{j_1})&(T\textbf{e}_{i_2},\textbf{e}_{j_1})&\cdots&(T\textbf{e}_{i_{s+t}},\textbf{e}_{j_1})\\
(T\textbf{e}_{i_1},\textbf{e}_{j_2})&(T\textbf{e}_{i_2},\textbf{e}_{j_2})&\cdots&(T\textbf{e}_{i_{s+t}},\textbf{e}_{j_2})\\
\vdots&\vdots&\ddots&\vdots\\
(T\textbf{e}_{i_1},\textbf{e}_{j_{s+t}})&(T\textbf{e}_{i_2},\textbf{e}_{j_{s+t}}) &\cdots& (T\textbf{e}_{i_{s+t}},\textbf{e}_{j_{s+t}})
\end{array}\right|.
\end{eqnarray}
By \eqref{20241011for3}, one can see that the above definition is well defined. We denote all bounded linear operators from $P_{I_1}$ into $P_{I_2}$ which satisfy the assumptions in this definition by $F(P_{I_1},P_{I_2})$.
\end{definition}
\begin{proposition}\label{20241011prop1}
Suppose that $I_1,I_2,I_3\subset \mathbb{N}$ such that $|I_1\setminus I_2|=|I_2\setminus I_1|\in\mathbb{N}_0$ and $|I_3\setminus I_2|=|I_2\setminus I_3|\in\mathbb{N}_0$. If $T_1\in F(P_{I_1},P_{I_2})$ and $T_2\in F(P_{I_2},P_{I_3})$, then $T_2T_1\in F(P_{I_1},P_{I_3})$ and det$(T_2T_1)=$det$(T_2)\cdot$det$(T_1)$.
\end{proposition}
\begin{proof}
Since $I_1,I_3\in\Gamma_{I_2}$, Lemma \ref{20241011lem2} implies that $|I_1\setminus I_3|=|I_3\setminus I_1|\in\mathbb{N}_0$.
By the assumptions $T_1\in F(P_{I_1},P_{I_2})$ and $T_2\in F(P_{I_2},P_{I_3})$, there exists $I_0\subset I_1\cap I_2$ and $I_0'\subset I_2\cap I_3$ such that $|I_0|=t_1\in\mathbb{N}_0$, $|I_0'|=t_2\in\mathbb{N}_0$, $T_1|_{P_{(I_1\cap I_2)\setminus I_0}}=id_{P_{(I_1\cap I_2)\setminus I_0}}$ and $T_2|_{P_{(I_2\cap I_3)\setminus I_0'}}=id_{P_{(I_2\cap I_3)\setminus I_0'}}$.
Let $I_0''\triangleq ((I_1\cap I_2\cap I_3)\cap (I_0\cup I_0'))\cup ((I_1\cap I_3)\setminus I_2)$, since $I_1\cap I_2\cap I_3=(I_1\cap I_3)\setminus (\mathbb{N}\setminus I_2)=(I_1\cap I_3)\setminus ((I_1\cap I_3)\setminus I_2)$ and $(I_1\cap I_2\cap I_3)\setminus   (I_0\cup I_0')=(I_1\cap I_3)\setminus   I_0''$, then $I_0''$ is a finite subset of $I_1\cap I_3$ and $T_2T_1|_{P_{(I_1\cap I_3)\setminus   I_0''}}=id_{P_{(I_1\cap I_3)\setminus   I_0''}}$. Note that $T_1|_{P_{(I_1\cap I_2\cap I_3)\setminus   I_0''}}=T_2|_{P_{(I_1\cap I_2\cap I_3)\setminus   I_0''}}=T_2T_1|_{P_{(I_1\cap I_2\cap I_3)\setminus   I_0''}}=id_{P_{(I_1\cap I_2\cap I_3)\setminus   I_0''}}$ and write $I_0'''= (I_1\cap I_2\cap I_3)\cap I_0''$. Then it is easy to see that $|(I_1\setminus (I_1\cap I_2\cap I_3))\sqcup I_0'''|=|(I_2\setminus (I_1\cap I_2\cap I_3))\sqcup I_0'''|=|(I_3\setminus (I_1\cap I_2\cap I_3))\sqcup I_0'''|=r\in\mathbb{N}_0$. Hence there exists $i_1,\cdots,i_r,i_1',\cdots,i_r',i_1'',\cdots,i_r''\in\mathbb{N}$ such that
\begin{eqnarray*}
&&i_1<\cdots<i_r,\quad i_1'<\cdots <i_r',\quad i_1''<\cdots <i_r'',\\
&&(I_1\setminus (I_1\cap I_2\cap I_3))\sqcup I_0''' =\{i_1,\cdots,i_r\}, \\
&&(I_2\setminus (I_1\cap I_2\cap I_3))\sqcup I_0''' = \{i_1',\cdots,i_r'\},\\
&&(I_3\setminus (I_1\cap I_2\cap I_3))\sqcup I_0'''=\{i_1'',\cdots,i_r''\}.
\end{eqnarray*}
Then
\begin{eqnarray*}
&&\text{det}(T_2T_1)\\
&&= \left |\begin{array}{cccc}
(T_2T_1\textbf{e}_{i_1},\textbf{e}_{i_1''})&\cdots&(T_2T_1\textbf{e}_{i_{r}},\textbf{e}_{i_1''})\\
\vdots&\ddots&\vdots\\
(T_2T_1\textbf{e}_{i_1},\textbf{e}_{i_{r}''})&\cdots& (T_2T_1\textbf{e}_{i_{r}},\textbf{e}_{i_{r}''})
\end{array}\right |\\
&&=  \left|\begin{array}{cccc}
\sum\limits_{j\in\mathbb{N}} (T_2 \textbf{e}_{i_j'} ,\textbf{e}_{i_1''} )\cdot(T_1\textbf{e}_{i_1},\textbf{e}_{i_j'})&\cdots&\sum\limits_{j\in\mathbb{N}} (T_2 \textbf{e}_{i_j'} ,\textbf{e}_{i_1''} )\cdot(T_1\textbf{e}_{i_r},\textbf{e}_{i_j'})\\
\vdots&\ddots&\vdots\\
\sum\limits_{j\in\mathbb{N}} (T_2 \textbf{e}_{i_j'} ,\textbf{e}_{i_r''} )\cdot(T_1\textbf{e}_{i_1},\textbf{e}_{i_j'})&\cdots& \sum\limits_{j\in\mathbb{N}} (T_2 \textbf{e}_{i_j'} ,\textbf{e}_{i_r''} )\cdot(T_1\textbf{e}_{i_r},\textbf{e}_{i_j'})
\end{array}\right |\\
&&=  \left|\begin{array}{cccc}
\sum\limits_{j=1}^{r} (T_2 \textbf{e}_{i_j'} ,\textbf{e}_{i_1''} )\cdot(T_1\textbf{e}_{i_1},\textbf{e}_{i_j'})&\cdots&\sum\limits_{j=1}^{r}  (T_2 \textbf{e}_{i_j'} ,\textbf{e}_{i_1''} )\cdot(T_1\textbf{e}_{i_r},\textbf{e}_{i_j'})\\
\vdots&\ddots&\vdots\\
\sum\limits_{j=1}^{r}  (T_2 \textbf{e}_{i_j'} ,\textbf{e}_{i_r''} )\cdot(T_1\textbf{e}_{i_1},\textbf{e}_{i_j'})&\cdots& \sum\limits_{j=1}^{r}  (T_2 \textbf{e}_{i_j'} ,\textbf{e}_{i_r''} )\cdot(T_1\textbf{e}_{i_r},\textbf{e}_{i_j'})
\end{array}\right |\\
&&= \left|\begin{array}{cccc}
  (T_2 \textbf{e}_{i_1'} ,\textbf{e}_{i_1''} )&\cdots&  (T_2 \textbf{e}_{i_r'} ,\textbf{e}_{i_1''} )\\
\vdots&\ddots&\vdots\\
  (T_2 \textbf{e}_{i_1'} ,\textbf{e}_{i_r''} )&\cdots&   (T_2 \textbf{e}_{i_r'} ,\textbf{e}_{i_r''} )
\end{array}\right |\cdot
\left|\begin{array}{cccc}
  (T_1 \textbf{e}_{i_1} ,\textbf{e}_{i_1'} )&\cdots&  (T_1 \textbf{e}_{i_r} ,\textbf{e}_{i_1'} )\\
\vdots&\ddots&\vdots\\
  (T_1 \textbf{e}_{i_1} ,\textbf{e}_{i_r'} )&\cdots&   (T_1 \textbf{e}_{i_r} ,\textbf{e}_{i_r'} )
\end{array}\right |\\
&&=\text{det}(T_2)\cdot \text{det}(T_1).
\end{eqnarray*}
This completes the proof of Proposition \ref{20241011prop1}.
\end{proof}

\begin{corollary}\label{20241015cor1}
Suppose that $S$ is a surface of $\ell^2$ with codimension $\Gamma_{\mathbb{N}\setminus I}$, for $\textbf{x}\in S$, if there exists three coordinate pairings $(S\cap U_1, P_{I_1})$, $(S\cap U_2, P_{I_2})$ and $(S\cap U_3, P_{I_3})$ of $S$ such that $\textbf{x}\in S\cap U_1\cap U_2\cap U_3$. Then,
$$
\text{det}\left(D(P_{I_3}P_{I_1}^{-1})(\textbf{x}_{I_1})\right)
=\text{det}\left(D(P_{I_3}P_{I_2}^{-1})(\textbf{x}_{I_2})\right)\cdot \text{det}\left(D(P_{I_2}P_{I_1}^{-1})(\textbf{x}_{I_1})\right),
$$
where $\textbf{x}_{I_1}=P_{I_1}\textbf{x}$ and $\textbf{x}_{I_2}=P_{I_2}\textbf{x}$.
\end{corollary}

\begin{proof}
By $P_{I_3}P_{I_1}^{-1}=(P_{I_3}P_{I_2}^{-1})(P_{I_2}P_{I_1}^{-1})$, it follows that
$$
 D(P_{I_3}P_{I_1}^{-1})(\textbf{x}_{I_1})
= D(P_{I_3}P_{I_2}^{-1})(\textbf{x}_{I_2}) \circ  D(P_{I_2}P_{I_1}^{-1})(\textbf{x}_{I_1}).
$$
This completes the proof of Corollary \ref{20241015cor1}.
\end{proof}

\begin{definition}\label{20241012def1}
Suppose that $S$ is a surface of $\ell^2$ with codimension $\Gamma_{\mathbb{N}\setminus I}$, for each $\textbf{x}\in S$, there exists a coordinate pairing $(S\cap U, P_{I})$ of $S$ such that $\textbf{x}\in S\cap U$. Then there exists a smooth function $f$ from $U_{I}$ into $P_{\mathbb{N}\setminus I}$ such that $\textbf{x}=\textbf{x}_{I}+f(\textbf{x}_{I})$ and
$$
S\cap U=\left\{\textbf{y}_{I}+f(\textbf{y}_{I}):\textbf{y}_{I}=\sum_{i\in I}y_i\textbf{e}_i\in U_{I}\right\}.
$$
Then we define
\begin{eqnarray*}
n_{I}(\textbf{x})\triangleq \sqrt{ 1+\sum_{J_0 \subset \mathbb{N}\setminus I, I_0\subset I, |I_0|=|J_0|\in\mathbb{N}} \left|\text{det}\left( D_{x_j} f_i(\textbf{x}_{I})\right)_{i\in I_0,j\in J_0}\right|^2},
\end{eqnarray*}
where $f_i(\textbf{x}_{I})\triangleq (f(\textbf{x}_{I}),\textbf{e}_i)$ for each $i\in \mathbb{N}\setminus I$.
\end{definition}
\begin{remark}\label{20241016rem1}
Note that in the Definition \ref{20241012def1}, $D (P_{I'}P_{I}^{-1})(\textbf{x}_{I})\in F(P_I,P_{I'})$ for each $I'\in \Gamma_I$ and hence
\begin{eqnarray*}
n_{I}(\textbf{x})\triangleq \sqrt{ \sum_{I'\in \Gamma_I} \left|\text{det}(D (P_{I'}P_{I}^{-1})(\textbf{x}_{I}))\right|^2}.
\end{eqnarray*}
\end{remark}
\begin{lemma}\label{20241012lem3}
In the Definition \ref{20241012def1}, suppose that there exists another coordinate pairing $(S\cap U_1, P_{I_1})$ of $S$ such that $\textbf{x}\in S\cap U_1$. Then it holds that
$$
\begin{array}{ll}\ds
n_{I}(\textbf{x})=\left|\text{det}(D(P_{I_1}P_{I}^{-1})(\textbf{x}_{I}))\right|\cdot n_{I_1}(\textbf{x}),\\[2mm]
\ds
n_{I_1}(\textbf{x})=\left|\text{det}(D(P_{I}P_{I_1}^{-1})(\textbf{x}_{I_1}))\right|\cdot n_{I}(\textbf{x}).
\end{array}
$$
\end{lemma}
\begin{proof}
Note that
\begin{eqnarray*}
n_{I}(\textbf{x})=\sqrt{ \sum_{I'\in \Gamma_I} \left|\text{det}(D (P_{I'}P_{I}^{-1})(\textbf{x}_{I}))\right|^2},
\end{eqnarray*}
and for each $I'\in \Gamma_I$, it holds that
\begin{eqnarray*}
(D (P_{I'}P_{I}^{-1})(\textbf{x}_{I}))&=&(D (P_{I'}P_{I_1}^{-1}P_{I_1}P_{I}^{-1})(\textbf{x}_{I}))\\
&=&(D (P_{I'}P_{I_1}^{-1})(\textbf{x}_{I_1}))(D(P_{I_1}P_{I}^{-1})(\textbf{x}_{I})).
\end{eqnarray*}
Hence Proposition \ref{20241011prop1} implies that
\begin{eqnarray*}
\text{det}(D (P_{I'}P_{I}^{-1})(\textbf{x}_{I}))
&=&\text{det}(D (P_{I'}P_{I_1}^{-1})(\textbf{x}_{I_1}))\cdot\text{det}(D(P_{I_1}P_{I}^{-1})(\textbf{x}_{I})).
\end{eqnarray*}
Therefore,
\begin{equation}
\begin{array}{ll}
\ds n_{I}(\textbf{x})&\ds=\left|\text{det}(D(P_{I_1}P_{I}^{-1})(\textbf{x}_{I}))\right|\cdot\sqrt{ \sum_{I'\in \Gamma_I} \left|\text{det}(D (P_{I'}P_{I_1}^{-1})(\textbf{x}_{I_1}))\right|^2}\label{20241012for5}\\[2mm]
&\ds=\left|\text{det}(D(P_{I_1}P_{I}^{-1})(\textbf{x}_{I}))\right|\cdot n_{I_1}(\textbf{x}).
\end{array}
\end{equation}
Note that $(P_{I_1}P_{I}^{-1})(P_{I}P_{I_1}^{-1})=\text{id}_{P_{I_1}(S\cap U_1\cap U_2)}$ and hence
$$(D (P_{I_1}P_{I}^{-1})(\textbf{x}_{I}))(D(P_{I}P_{I_1}^{-1})(\textbf{x}_{I_1}))=\text{id}_{P_{I_1}}.
$$
Combining Proposition \ref{20241011prop1}, we have
\begin{eqnarray}\label{20241013for1}
 \text{det}(D(P_{I_1}P_{I}^{-1})(\textbf{x}_{I})) \cdot  \text{det}(D(P_{I}P_{I_1}^{-1})(\textbf{x}_{I_1}))=1.
\end{eqnarray}
By \eqref{20241012for5}, we have
\begin{eqnarray*}
n_{I_1}(\textbf{x})=\left|\text{det}(D(P_{I}P_{I_1}^{-1})(\textbf{x}_{I_1}))\right|\cdot n_{I}(\textbf{x}).
\end{eqnarray*}
This completes the proof of Lemma \ref{20241012lem3}.
\end{proof}

Suppose that $S$ is a surface of $\ell^2$ with codimension $\Gamma_{\mathbb{N}\setminus I}$, $(S\cap U_1, P_{I_1})$ and $(S\cap U_2, P_{I_2})$ are two coordinate pairings of $S$ such that $W\triangleq(S\cap U_1)\cap (S\cap U_2)\neq \emptyset$. Then there exists two non-empty sets $I_0\subset I_1$ and $J_0\subset \mathbb{N}\setminus I_1$ such that $|I_0|=|J_0|\in\mathbb{N}$ and $I_2=(I_1\setminus I_0)\cup J_0$.
Then there exists open sets $U_{I_1}$ of $P_{I_1}$ and $U_{I_2}$ of $P_{I_2}$, continuous maps $f'$ from $U_{I_1}$ into $P_{\mathbb{N}\setminus I_1}$ and $f''$ from $U_{I_2}$ into $P_{\mathbb{N}\setminus I_2}$ such that
$$
S\cap U_1=\left\{\textbf{x}_{I_1}+f'(\textbf{x}_{I_1}): \textbf{x}_{I_1}=\sum_{i\in I_1}x_i\textbf{e}_i\in U_{I_1}\right\}
$$
and
$$
S\cap U_2=\left\{\textbf{x}_{I_2}+f''(\textbf{x}_{I_2}): \textbf{x}_{I_2}=\sum_{i\in I_2}x_i\textbf{e}_i\in U_{I_2}\right\}.
$$
Clearly, $T\triangleq P_{I_1}P_{I_2}^{-1}$ is a homeomorphism from $P_{I_2}W$ into $P_{I_1}W$. Note that  for any Borel subset $F$ of $P_{I_2}W$, Borel measure $\mu$ on $P_{I_2}W$ and non-negative Borel measurable function $h$ on $P_{I_2}W$, it holds that
\begin{eqnarray*}
\int_F h \,\mathrm{d}\mu =\int_{T^{-1}F} h(T\omega)\,\mathrm{d}\mu(T\omega).
\end{eqnarray*}
Then for any $\textbf{x}_{I_2}\in P_{I_2}W$, it holds that
\begin{equation}\label{220916e1}
\begin{array}{ll}
T\textbf{x}_{I_2}
&\ds=P_{I_1}P_{I_2}^{-1}\textbf{x}_{I_2}=P_{I_1}(\textbf{x}_{I_2}+f''(\textbf{x}_{I_2}))\\[2mm]
&\ds=\textbf{x}_{I_1\setminus I_0}+\sum_{i\in I_0}f''_i(\textbf{x}_{I_2}) \textbf{e}_i
=\textbf{x}_{I_1\setminus I_0}+\sum_{i\in I_0}f''_i(\textbf{x}_{I_1\setminus I_0}+\textbf{x}_{J_0}) \textbf{e}_i,
\end{array}
\end{equation}
where $f''_i(\textbf{x}_{I_2})=(f''(\textbf{x}_{I_2}),\textbf{e}_i)$ for each $i\in I_0$. Then for any $\textbf{x}_{I_1}\in P_{I_1}W$, it holds that
\begin{equation}\label{220916e2}
\begin{array}{ll}
T^{-1}\textbf{x}_{I_1}
=\ds P_{I_2}P_{I_1}^{-1}\textbf{x}_{I_1}=P_{I_2}(\textbf{x}_{I_1}+f'(\textbf{x}_{I_1}))\\[2mm]
=\ds\textbf{x}_{I_1\setminus I_0}+\sum_{j\in J_0}f'_j(\textbf{x}_{I_1}) \textbf{e}_j=\textbf{x}_{I_1\setminus I_0}+\sum_{j\in J_0}f'_j(\textbf{x}_{I_1\setminus I_0}+\textbf{x}_{I_0}) \textbf{e}_j,
\end{array}
\end{equation}
where $f'_j(\textbf{x}_{I_1})=(f'(\textbf{x}_{I_1}),\textbf{e}_j)$ for each $j\in J_0$.

We assume that $f'$ and $f''$ are Fr\'{e}chet differentiable.  Then it holds that
\begin{eqnarray*}
(DT(\textbf{x}_{I_2}))\textbf{y}_{I_2}&=&\textbf{y}_{I_1\setminus I_0}+\sum_{i\in I_0}\left(\sum_{j\in I_2}(D_{x_j}f''_i(\textbf{x}_{I_2} ))\cdot y_j \right)\textbf{e}_i\in P_{I_1},\\&&\forall \;\textbf{y}_{I_2}=\sum_{i\in I_2}y_i \textbf{e}_i\in P_{I_2},\\
(DT^{-1}(\textbf{x}_{I_1}))\textbf{y}_{I_1}&=&\textbf{y}_{I_1\setminus I_0}+\sum_{i\in J_0}\left(\sum_{j\in I_1}(D_{x_j}f'_i(\textbf{x}_{I_1} ))\cdot y_j \right)\textbf{e}_i\in P_{I_2},\\&&\forall \;\textbf{y}_{I_1}=\sum_{i\in I_1}y_i \textbf{e}_i\in P_{I_1},
\end{eqnarray*}
where $\textbf{y}_{I_1\setminus I_0}=\sum\limits_{i\in I_1\setminus I_0}y_i \textbf{e}_i$.

Since $TT^{-1}\textbf{y}_{I_1}=\textbf{y}_{I_1},\,\forall\, \textbf{y}_{I_1}\in P_{I_1}$ and $T^{-1}T\textbf{y}_{I_2}=\textbf{y}_{I_2},\,\forall\, \textbf{y}_{I_2}\in P_{I_2}$. By the chain rule, we have
\begin{eqnarray*}
(DT(\textbf{x}_{I_2}))(DT^{-1}(\textbf{x}_{I_1}))\textbf{y}_{I_1}&=&\textbf{y}_{I_1},\quad\forall\, \textbf{y}_{I_1}\in P_{I_1}\\
 (DT^{-1}(\textbf{x}_{I_1}))(DT(\textbf{x}_{I_2}))\textbf{y}_{I_2}&=&\textbf{y}_{I_2},\quad\forall\, \textbf{y}_{I_2}\in P_{I_2},
\end{eqnarray*}
where $T\textbf{x}_{I_2}=\textbf{x}_{I_1}\in P_{I_1}W$.

Note that for $\textbf{y}_{I_0}=\sum\limits_{i\in I_0}y_i \textbf{e}_i\in P_{I_0}$, it holds that
\begin{equation}
\begin{array}{ll}
\ds (DT(\textbf{x}_{I_2}))(DT^{-1}(\textbf{x}_{I_1}))\textbf{y}_{I_0}
=(DT(\textbf{x}_{I_2}))\sum_{i\in J_0}\sum_{j\in I_0}(D_{x_j}f'_i(\textbf{x}_{I_1} ))\cdot y_j \textbf{e}_i\label{20241011for1}\\[2mm]\ds
=\sum_{k\in I_0}\sum_{i\in J_0} \sum_{j\in I_0}(D_{x_i}f''_k(\textbf{x}_{I_2} ))\cdot(D_{x_j}f'_i(\textbf{x}_{I_1} ))\cdot y_j    \textbf{e}_k
=\sum_{k\in I_0}  y_k\textbf{e}_k.
\end{array}
\end{equation}
Write $|I_0|=|J_0|\triangleq r\in\mathbb{N}$ and there exists $i_1,\cdots,i_r,j_1,\cdots,j_r\in \mathbb{N}$ such that $i_1<i_2<\cdots<i_r$, $j_1<j_2<\cdots<j_r$ and $I_0=\{i_1,i_2,\cdots,i_r\},\,J_0=\{j_1,j_2,\cdots,j_r\}$. Set
\begin{eqnarray*}
(D_{I_1\bigtriangleup I_2} P_{I_1}P_{I_2}^{-1})(\textbf{x}_{I_2})\triangleq\left(\begin{array}{cccc}
D_{x_{j_1}}f_{i_1}''(\textbf{x}_{I_2} )&D_{x_{j_2}}f_{i_1}''(\textbf{x}_{I_2} )&\cdots&D_{x_{j_r}}f_{i_1}''(\textbf{x}_{I_2} )\\
D_{x_{j_1}}f_{i_2}''(\textbf{x}_{I_2} )&D_{x_{j_2}} f_{i_2}''(\textbf{x}_{I_2} )&\cdots&D_{x_{j_{r}}} f_{i_2}''(\textbf{x}_{I_2} )\\
\vdots&\vdots&\ddots&\vdots\\
D_{x_{i_1}}f_{i_r}''(\textbf{x}_{I_2} )&D_{x_{j_2}} f_{j_r}''(\textbf{x}_{I_2} ) &\cdots& D_{x_{j_r}} f_{i_r}''(\textbf{x}_{I_2} )
\end{array}\right)_{r\times r}
\end{eqnarray*}
and
\begin{eqnarray*}
(D_{I_1\bigtriangleup I_2} P_{I_2}P_{I_1}^{-1})(\textbf{x}_{I_1})\triangleq
 \left(\begin{array}{cccc}
D_{x_{i_1}}f_{j_1}'(\textbf{x}_{I_1} )&D_{x_{i_2}}f_{j_1}'(\textbf{x}_{I_1} )&\cdots&D_{x_{i_r}}f_{j_1}'(\textbf{x}_{I_1} )\\
D_{x_{i_1}}f_{j_2}'(\textbf{x}_{I_1} )&D_{x_{i_2}} f_{j_2}'(\textbf{x}_{I_1} )&\cdots&D_{x_{i_{r}}} f_{j_2}'(\textbf{x}_{I_1} )\\
\vdots&\vdots&\ddots&\vdots\\
D_{x_{i_1}}f_{j_r}'(\textbf{x}_{I_1} )&D_{x_{i_2}} f_{j_r}'(\textbf{x}_{I_1} ) &\cdots& D_{x_{i_r}} f_{j_r}'(\textbf{x}_{I_1} )
\end{array}\right)_{r\times r}
\end{eqnarray*}
Then \eqref{20241011for1} implies that
\begin{equation}\label{20241011for2}
(D_{I_1\bigtriangleup I_2} P_{I_1}P_{I_2}^{-1})(\textbf{x}_{I_2})(D_{I_1\bigtriangleup I_2} P_{I_2}P_{I_1}^{-1})(\textbf{x}_{I_1})=
 \left(\begin{array}{cccc}
1&0&\cdots&0\\
0&1&\cdots&0\\
\vdots&\vdots&\ddots&\vdots\\
0&0&\cdots& 1
\end{array}\right)_{r\times r}.
\end{equation}
For each non-empty set $I\subset\mathbb{N}$, we denote the restriction of the product measure$\prod\limits_{i\in I}\frac{e^{-\frac{x_i^2}{2a_i^2}}}{\sqrt{2\pi a_i^2}}$ on $P_{I}$ by $\mu_I$.
\begin{lemma}\label{20241011lem1}
It holds that
$$
\begin{array}{ll}
\ds\mathrm{d}\mu_{I_1}(T\textbf{x}_{I_2})= &\ds\left(\prod\limits_{i\in I_0}\frac{e^{-\frac{\left(f''_i(\textbf{x}_{I_2})\right)^2}{2a_i^2}}}{\sqrt{2\pi a_i^2}}\right)\cdot\left|\text{det}(D_{I_1\bigtriangleup I_2} P_{I_1}P_{I_2}^{-1})(\textbf{x}_{I_2})\right|\\[2mm]
&\ds \cdot\left(\prod_{i\in J_0}\,\mathrm{d} x_{i}\right)\times\mathrm{d}\mu_{I_1\setminus I_0}(\textbf{x}_{I_1\setminus I_0}),
\end{array}
$$
on $\mathscr{B}\big(P_{I_2} W \big)$.
\end{lemma}
\begin{proof}
It suffices to prove the above two measures coincide on connect subsets of $P_{I_2} W $ in the following form:
\begin{eqnarray*}
A=\{  \textbf{x}_{I_1\setminus I_0}+\textbf{x}_{J_0}:\;\textbf{x}_{I_1\setminus I_0}\in U_{I_1\setminus I_0},\,\textbf{x}_{J_0}\in U_{J_0}\}\subset  P_{I_2} W,
\end{eqnarray*}
where $U_{I_1\setminus I_0}$ is an open subset of $P_{I_1\setminus I_0}$ and $U_{I_0}$ is an open subset of $P_{I_0}$. By (\ref{220916e1}), it is easy to see that
$$TA=\left\{\textbf{x}_{I_1\setminus I_0}+\sum_{i\in I_0}f''_i(\textbf{x}_{I_1\setminus I_0}+\textbf{x}_{J_0}) \textbf{e}_i:\;\textbf{x}_{I_1\setminus I_0}\in U_{I_1\setminus I_0},\,\textbf{x}_{J_0}\in U_{J_0}\right\}.
$$
Then we have
\begin{eqnarray*}
&&\int_A \mathrm{d}\mu_{I_1}(T\textbf{x}_{I_2})= \mu_{I_1}(TA)=\int_{TA} \mathrm{d}\mu_{I_1}(\textbf{x}_{I_1})
=\int_{TA}\mathrm{d}\mu_{I_0}(\textbf{x}_{I_0})\mathrm{d}\mu_{I_1\setminus I_0}(\textbf{x}_{I_1\setminus I_0}) \\
&&=\int_{TA}\left(\prod\limits_{i\in I_0}\frac{e^{-\frac{x_i^2}{2a_i^2}}}{\sqrt{2\pi a_i^2}}\right)\cdot\left(\prod_{i\in I_0}\,\mathrm{d} x_{i}\right)\mathrm{d}\mu_{I_1\setminus I_0}(\textbf{x}_{I_1\setminus I_0}) \\
&&=\int_{U_{I_1\setminus I_0}}\int_{(TA)_{\textbf{x}_{I_1\setminus I_0}}}\left(\prod\limits_{i\in I_0}\frac{e^{-\frac{x_i^2}{2a_i^2}}}{\sqrt{2\pi a_i^2}}\right)\cdot\left(\prod_{i\in I_0}\,\mathrm{d} x_{i}\right)\mathrm{d}\mu_{I_1\setminus I_0}(\textbf{x}_{I_1\setminus I_0}) \\
&&=\int_{U_{I_1\setminus I_0}}\int_{ U_{J_0}}\left(\prod\limits_{i\in I_0}\frac{e^{-\frac{\left(f''_i(\textbf{x}_{I_2})\right)^2}{2a_i^2}}}{\sqrt{2\pi a_i^2}}\right)\cdot \left|\text{det}(D_{I_1\bigtriangleup I_2} P_{I_1}P_{I_2}^{-1})(\textbf{x}_{I_2})\right|\\
&&\qq\qq\qq\qq\qq\cdot \left(\prod_{j\in J_0}\,\mathrm{d} x_{i}\right)\mathrm{d}\mu_{I_1\setminus I_0}(\textbf{x}_{I_1\setminus I_0}) \\
&&=\int_{A}\left(\prod\limits_{i\in I_0}\frac{e^{-\frac{\left(f''_i(\textbf{x}_{I_2})\right)^2}{2a_i^2}}}{\sqrt{2\pi a_i^2}}\right)\cdot \left|\text{det}(D_{I_1\bigtriangleup I_2} P_{I_1}P_{I_2}^{-1})(\textbf{x}_{I_2})\right|\\
&&\qq\qq\qq\qq\qq\cdot \left(\prod_{j\in J_0}\,\mathrm{d} x_{i}\right)\times\mathrm{d}\mu_{I_1\setminus I_0}(\textbf{x}_{I_1\setminus I_0}),
\end{eqnarray*}
where
$$
\begin{array}{ll}\ds
(TA)_{\textbf{x}_{I_1\setminus I_0}}&\ds\triangleq \{\textbf{x}_{ I_0}\in P_{I_0}:\textbf{x}_{I_1\setminus I_0}+\textbf{x}_{I_0}\in TA\}\\[2mm]
&\ds
=\left\{\sum\limits_{i\in I_0}f''_i(\textbf{x}_{I_1\setminus I_0}+\textbf{x}_{J_0}) \textbf{e}_i:\;\textbf{x}_{J_0}\in U_{J_0}\right\},
\end{array}
 $$
the third equality follows from the finite dimensional change of variable formula for the following transformation
\begin{eqnarray*}
\textbf{x}_{J_0}\mapsto\sum_{i\in I_0}f''_i(\textbf{x}_{I_1\setminus I_0}+\textbf{x}_{J_0}) \textbf{e}_i,\quad \forall\,\textbf{x}_{J_0}=\sum_{j\in J_0}x_j\textbf{e}_j\in U_{J_0}.
\end{eqnarray*}
This completes the proof of Lemma \ref{20241011lem1}.
\end{proof}

Suppose $I$ is a non-empty subset of $\mathbb{N}$, let
\begin{eqnarray*}
C_I\triangleq\left\{\sum_{i\in I}x_i\textbf{e}_i\in P_I:\prod_{i\in I}\frac{1}{\sqrt{2\pi a_i^2}}e^{-\frac{x_i^2}{2a_i^2}}\in(0,+\infty)\right\}.
\end{eqnarray*}
Then for $\sum\limits_{i\in I}x_i\textbf{e}_i\in P_I$, $\sum\limits_{i\in I}x_i\textbf{e}_i\in C_I$ if and only if
\begin{eqnarray*}
\sum_{i\in I}\left|\ln \frac{1}{\sqrt{2\pi} a_i}-\frac{x_i^2}{2a_i^2}\right|<\infty.
\end{eqnarray*}
For any
$\textbf{x}^I=(x_i)_{i\in I}\in C_I$, let
\begin{equation}\label{20241012for4}
G_I(\textbf{x}^I)\triangleq \sum\limits_{i\in I}\left(\ln \frac{1}{\sqrt{2\pi} a_i}-\frac{x_i^2}{2a_i^2}\right),\quad F_I(\textbf{x}^I)\triangleq e^{G_I(\textbf{x}^I)}
\end{equation}
and
\begin{eqnarray*}
C_I(\textbf{x}^I)\triangleq\left\{\Delta \textbf{x}^I=\sum_{i\in I}\Delta x_i\textbf{e}_i\in P_I:\sum_{i\in I}\left|\ln \frac{1}{\sqrt{2\pi} a_i}-\frac{(x_i+\Delta x_i)^2}{2a_i^2}\right|<\infty\right\}.
\end{eqnarray*}
It is easy to see that if $\Delta \textbf{x}^I=\sum\limits_{i\in I}\Delta x_i\textbf{e}_i\in P_I$ satisfies that
\begin{eqnarray*}
\sum_{i\in I}\frac{|\Delta x_i|^2}{a_i^4}<\infty,
\end{eqnarray*}
then $\textbf{x}^I+\Delta \textbf{x}^I\in C_I(\textbf{x}^I)$ and
\begin{eqnarray}
G_I(\textbf{x}_I+\Delta \textbf{x}_I)-G_I(\textbf{x}_I)=\sum_{i\in I}x_i\cdot \frac{\Delta x_i}{a_i^2}+\sum_{i\in I}a_i^2\cdot \frac{(\Delta x_i)^2}{a_i^4}.\label{20241007for1}
\end{eqnarray}
Let
\begin{eqnarray}
H^1_I&\triangleq&\left\{\sum_{i\in I}y_i\textbf{e}_i\in P_I:\sum_{i\in I}\frac{y_i^2}{a_i^2}<\infty\right\},\nonumber\\
H^2_I&\triangleq&\left\{\sum_{i\in I}y_i\textbf{e}_i\in P_I:\sum_{i\in I}\frac{y_i^2}{a_i^4}<\infty\right\},\label{20241020for1}\\
T\textbf{y}_I&\triangleq&\sum_{i\in I}x_i\cdot \frac{y_i}{a_i^2},\qquad \forall\, \textbf{y}_I=\sum_{i\in I}y_i\textbf{e}_i\in  H^2_I.\nonumber
\end{eqnarray}
Note that for $\textbf{y}_I=\sum\limits_{i\in I}y_i\textbf{e}_i\in  H^2_I$, we have
\begin{eqnarray*}
\sum_{i\in I}a_i^2\cdot \frac{y_i^2}{a_i^4} \leq \left(\sup_{i\in I}a_i^2\right)\cdot ||\textbf{y}_I||_{H^2_I}^2.
\end{eqnarray*}
Then it holds that $T\in (H^2_I)^*$, $\textbf{x}+H^2_I\subset C_I(\textbf{x}_I)$ for any $\textbf{x}_I=\sum\limits_{i\in I}x_i\textbf{e}_i\in C_I$ and
\begin{eqnarray*}
G_I(\textbf{x}_I+\Delta \textbf{x}_I)-G_I(\textbf{x}_I)=T(\Delta \textbf{x}_I)+o(||\Delta \textbf{x}_I||_{H^2_I}),
\end{eqnarray*}
as $\Delta \textbf{x}_I\in H^2_I$ and $||\Delta \textbf{x}_I||_{H^2_I}$ tends to zero.
Motivated by above observations, we have the following definition.
\begin{definition}\label{20241115def1}
Suppose that $X,\, Y$ and $Z$ are three norm spaces, if $X_0\subset X$ and $f$ is a mapping from $X_0$ into $Z$ such that for any $\textbf{x}\in X_0$, there exists an open neighborhood $U_{\textbf{x}}$ of 0 in $Y$ such that $\textbf{x}+U_{\textbf{x}}\subset X_0$ and the mapping $f_{\textbf{x}}$ defined by
$$
f_{\textbf{x}}(\textbf{x}')\triangleq f(\textbf{x}+\textbf{x}'),\quad \forall\, \textbf{x}'\in U_{\textbf{x}},
$$
is Fr\'{e}chet differentiable from $U_{\textbf{x}}$ into $Z$, then we say that $f$ is Fr\'{e}chet differentiable respect to the $Y$ direction.
\end{definition}
Therefore, $G_I$ is Fr\'{e}chet differentiable respect to the $H^2_I$ direction. Then we have the following result.
\begin{proposition}\label{20241020prop1}
$F_I$ is Fr\'{e}chet differentiable respect to the $H_2^I$ direction.
\end{proposition}

Suppose that $f$ is a mapping from $(-1,1)$ into $\ell^2$ such that $f((-1,1))\subset C$, $f(t)-f(0)\in H^2_{I}$ for any $t\in (-1,1)$ and let $g(t)\triangleq f(t)-f(0),\,\forall \,t\in (-1,1)$. We also assume that $g$ is Fr\'{e}chet differentiable from $\mathbb{R}$ into $H^2_I$. Then by the chain rule, $F_I(f)$ is Fr\'{e}chet differentiable from $\mathbb{R}$ into $\mathbb{R}$.

Meanwhile, we note that if a real-valued function is Fr\'{e}chet differentiable from $H_I^1$ into $\mathbb{R}$, then it is Fr\'{e}chet differentiable from $H_I^2$ into $\mathbb{R}$.
\begin{definition}\label{20241012def2}
Suppose that $S$ is a surface of $\ell^2$ with codimension $\Gamma_{\mathbb{N}\setminus I}$, $(S\cap U, P_{I})$ is a coordinate pairing of $S$. Then there exists a smooth function $f$ from $U_{I}$ into $P_{\mathbb{N}\setminus I}$ such that
$$
S\cap U=\left\{\textbf{x}_{I}+f(\textbf{x}_{I}):\textbf{x}_{I}=\sum_{i\in I}x_i\textbf{e}_i\in U_{I}\right\},
$$
where $U_I$ is an open subset of $P_I$. If
\begin{eqnarray}\label{20241013cond1}
\sup_{\textbf{x}_I\in U_{I}} n_{I}(\textbf{x}_I+f(\textbf{x}_I))\cdot F_{\mathbb{N}\setminus I}(f(\textbf{x}_I))<\infty,
\end{eqnarray}
Then we call $\Gamma_{\mathbb{N}\setminus I}$, $(S\cap U, P_{I})$ is a locally finite coordinate pairing of $S$ and define
\begin{eqnarray*}
\mu(E,S\cap U,I,f)\triangleq&&\int_{P_IE}n_{I}(\textbf{x}_I+f(\textbf{x}_I))\cdot F_{\mathbb{N}\setminus I}(f(\textbf{x}_I))\,\mathrm{d}\mu_I(\textbf{x}_I),\\
&&\qq\qq\qq\forall\;E\in \mathscr{B}(S\cap U).
\end{eqnarray*}
\end{definition}
\begin{lemma}\label{20241013lem1}
As in Definition \ref{20241012def2}, suppose that $(S\cap U_1, P_{I_1})$ is another locally finite coordinate pairing of $S$ such that $(S\cap U)\cap (S\cap U_1)\neq \emptyset$. Then there exists a smooth function $f'$ from $U_{I_1}$ into $P_{\mathbb{N}\setminus I_1}$ such that
$$
S\cap U_1=\left\{\textbf{x}_{I_1}+f'(\textbf{x}_{I_1}):\textbf{x}_{I_1}=\sum_{i\in I_1}x_i\textbf{e}_i\in U_{I_1}\right\},
$$
where  $U_{I_1}$ is an open subset of $P_{I_1}$. Then
\begin{eqnarray*}
\mu(E,S\cap U_1,I_1,f')=\int_{P_{I_1}E}n_{I_1}(\textbf{x}_{I_1}+f'(\textbf{x}_{I_1}))\cdot&& F_{\mathbb{N}\setminus I_1}(f'(\textbf{x}_{I_1}))\,\mathrm{d}\mu_{I_1}(\textbf{x}_{I_1}),\\
&&\forall\;E\in \mathscr{B}(S\cap U_1).
\end{eqnarray*}
Furthermore,
\begin{eqnarray*}
\mu(E,S\cap U_1,I_1,f')=\mu(E,S\cap U,I,f),\quad\forall\,E\in \mathscr{B}(S\cap U\cap U_1).
\end{eqnarray*}
\end{lemma}
\begin{proof}
Let $W\triangleq S\cap U\cap U_1$. There exists $I_0\subset I_1$ and $J_0\subset \mathbb{N}\setminus I_1$ such that $|I_0|=|J_0|\in\mathbb{N}_0$ and $I=(I_1\setminus I_0)\cup J_0$.
Clearly, $T\triangleq P_{I_1}P_{I}^{-1}$ is a homeomorphism from $P_{I}W$ into $P_{I_1}W$. Note that  for any Borel subset $F$ of $P_{I_1}W$, Borel measure $\mu$ on $P_{I_1}W$ and non-negative Borel measurable function $h$ on $P_{I_1}W$, it holds that
\begin{eqnarray*}
\int_F h \,\mathrm{d}\mu =\int_{T^{-1}F} h(T\textbf{x}_{I})\,\mathrm{d}\mu(T\textbf{x}_{I}).
\end{eqnarray*}
Thus
\begin{eqnarray*}
&&\int_{P_{I_1}E}n_{I_1}(\textbf{x}_{I_1}+f'(\textbf{x}_{I_1}))\cdot F_{\mathbb{N}\setminus I_1}(f'(\textbf{x}_{I_1}))\,\mathrm{d}\mu_{I_1}(\textbf{x}_{I_1})\\
&&=\int_{P_{I}E}n_{I_1}(P_{I_1}^{-1}(T\textbf{x}_{I}))\cdot F_{\mathbb{N}\setminus I_1}(f'(T\textbf{x}_{I}))\,\mathrm{d}\mu_{I_1}(T\textbf{x}_{I})\\
&&=\int_{P_{I}E}n_{I_1}( \textbf{x}_{I}+ f(\textbf{x}_{I}))\cdot F_{\mathbb{N}\setminus I_1}(f'(T\textbf{x}_{I}))\,\mathrm{d}\mu_{I_1}(T\textbf{x}_{I}).
\end{eqnarray*}
By Lemma \ref{20241012lem3},
$
n_{I_1}( \textbf{x}_{I}+ f(\textbf{x}_{I}))=\left|\text{det}(D(P_{I}P_{I_1}^{-1})(T\textbf{x}_{I}))\right|\cdot n_{I}(\textbf{x}_{I}+ f(\textbf{x}_{I})).
$
Then for any $\textbf{x}_{I}\in P_{I}W$, it holds that
\begin{equation}\label{220916e1zzz}
\begin{array}{ll}
\ds T\textbf{x}_{I}=P_{I_1}P_{I}^{-1}\textbf{x}_{I}
=P_{I_1}(\textbf{x}_{I}+f(\textbf{x}_{I}))\\[2mm]
=\ds\textbf{x}_{I_1\setminus I_0}+\sum_{i\in I_0}f_i(\textbf{x}_{I_2}) \textbf{e}_i=\textbf{x}_{I_1\setminus I_0}+\sum_{i\in I_0}f_i(\textbf{x}_{I_1\setminus I_0}+\textbf{x}_{J_0}) \textbf{e}_i,
\end{array}
\end{equation}
where $f_i(\textbf{x}_{I_2})=(f(\textbf{x}_{I_2}),\textbf{e}_i)$ for each $i\in I_0$. Then for any $\textbf{x}_{I_1}\in P_{I_1}W$, it holds that
\begin{equation}\label{220916e2zzz}
\begin{array}{ll}
\ds T^{-1}\textbf{x}_{I_1}
= P_{I}P_{I_1}^{-1}\textbf{x}_{I_1}=P_{I}(\textbf{x}_{I_1}+f'(\textbf{x}_{I_1}))\\[2mm]
\ds=\textbf{x}_{I_1\setminus I_0}+\sum_{j\in J_0}f'_j(\textbf{x}_{I_1}) \textbf{e}_j=\textbf{x}_{I_1\setminus I_0}+\sum_{j\in J_0}f'_j(\textbf{x}_{I_1\setminus I_0}+\textbf{x}_{I_0}) \textbf{e}_j,
\end{array}
\end{equation}
where $f'_j(\textbf{x}_{I_1})=(f'(\textbf{x}_{I_1}),\textbf{e}_j)$ for each $j\in J_0$. Thus
\begin{eqnarray*}
F_{\mathbb{N}\setminus I_1}(f'(T\textbf{x}_{I}))&=&F_{\mathbb{N}\setminus I_1}\left(f'\left(P_{I_1}(\textbf{x}_{I}+f(\textbf{x}_{I}))\right) \right)\\
&=&\left(\prod_{i\in J_0}\frac{e^{-\frac{x_i^2}{2a_i^2}}}{\sqrt{2\pi a_i^2}}\right)\cdot \left( \prod_{i\in (\mathbb{N}\setminus I_1)\setminus J_0}\frac{e^{-\frac{(f_i(\textbf{x}_{I}))^2}{2a_i^2}}}{\sqrt{2\pi a_i^2}}\right).
\end{eqnarray*}
By Lemma \ref{20241011lem1},
\begin{eqnarray*}
\mathrm{d}\mu_{I_1}(T\textbf{x}_{I})= &&\left(\prod\limits_{i\in I_0}\frac{e^{-\frac{\left(f_i(\textbf{x}_{I})\right)^2}{2a_i^2}}}{\sqrt{2\pi a_i^2}}\right)\cdot\left|\text{det}(D_{I_1\bigtriangleup I} P_{I_1}P_{I}^{-1})(\textbf{x}_{I})\right|\\
&&\cdot\left(\prod_{i\in J_0}\,\mathrm{d} x_{i}\right)\times\mathrm{d}\mu_{I_1\setminus I_0}(\textbf{x}_{I_1\setminus I_0}).
\end{eqnarray*}
Therefore,
\begin{eqnarray*}
&&\int_{P_{I_1}E}n_{I_1}(\textbf{x}_{I_1}+f'(\textbf{x}_{I_1}))\cdot F_{\mathbb{N}\setminus I_1}(f'(\textbf{x}_{I_1}))\,\mathrm{d}\mu_{I_1}(\textbf{x}_{I_1})\\
&&=\int_{P_{I}E}n_{I_1}( \textbf{x}_{I}+ f(\textbf{x}_{I}))\cdot F_{\mathbb{N}\setminus I_1}(f'(T\textbf{x}_{I}))\,\mathrm{d}\mu_{I_1}(T\textbf{x}_{I})\\
&&=\int_{P_{I}E} \left|\text{det}(D(P_{I}P_{I_1}^{-1})(T\textbf{x}_{I}))\right|\cdot n_{I}(\textbf{x}_{I}+ f(\textbf{x}_{I}))\cdot \left(\prod_{i\in J_0}\frac{e^{-\frac{x_i^2}{2a_i^2}}}{\sqrt{2\pi a_i^2}}\right)\\
&&\q\cdot \left( \prod_{i\in (\mathbb{N}\setminus I_1)\setminus J_0}\frac{e^{-\frac{(f_i(\textbf{x}_{I}))^2}{2a_i^2}}}{\sqrt{2\pi a_i^2}}\right)\left(\prod\limits_{i\in I_0}\frac{e^{-\frac{\left(f_i(\textbf{x}_{I})\right)^2}{2a_i^2}}}{\sqrt{2\pi a_i^2}}\right)\\
&&\q\cdot\left|\text{det}(D_{I_1\bigtriangleup I} P_{I_1}P_{I}^{-1})(\textbf{x}_{I})\right|\left(\prod_{i\in J_0}\,\mathrm{d} x_{i}\right)\times\mathrm{d}\mu_{I_1\setminus I_0}(\textbf{x}_{I_1\setminus I_0})\\
&&=\int_{P_{I}E}  n_{I}(\textbf{x}_{I}+ f(\textbf{x}_{I}))\cdot  F_{\mathbb{N}\setminus I}(f(\textbf{x}_{I}))\mathrm{d}\mu_{I}(\textbf{x}_{I}),
\end{eqnarray*}
where the last equality follows from \eqref{20241013for1}, which completes the proof of Lemma \ref{20241013lem1}.
\end{proof}
\begin{definition}\label{20241013def1}
The quantity $\mu(E,S\cap U,I,f)$ defined in Definition \ref{20241012def2} are independent of $I$ and $f$, hence we will use the notation $\mu(E,S\cap U)$ to denote it. Furthermore, if $(S\cap U_1)\cap (S\cap U)\neq \emptyset$, then
$$
\mu(E,S\cap U)=\mu(E,S\cap U_1),\quad \forall\, E\in \mathscr{B}(S\cap U\cap U_1).
$$
Then we call $\mu(\cdot,S\cap U)$ a local version of surface measure of $S$. If there exists a family of locally finite coordinate pairings of $S$ that covers $S$, then we call $S$ a measurable surface of $\ell^2$ with codimension $\Gamma_{\mathbb{N}\setminus I}$.
\end{definition}
\begin{proposition}\label{20241013prop1}
Suppose that $S$ is a measurable surface of $\ell^2$ with codimension $\Gamma_{\mathbb{N}\setminus I}$, $\{(S\cap U_i, P_{I_i}):i\in \Lambda\}$ is a family of locally finite coordinate pairings of $S$ that covers $S$. Then there exists a Borel measure $\mu_S$ on $S$ such that
$$
\mu_S(E)=\mu(E,S\cap U_i),\quad\forall E\in \mathscr{B}(S\cap U_i),\,\,\forall i\in \Lambda.
$$
\end{proposition}
\begin{proof}
Since $\ell^2$ has a countable base, by \cite[Theorem 15, p. 49]{Kel}, we can find a sequence of sets $\{U_{i_k}\}_{k=1}^{
\infty}$ such that $S\subset\bigcup_{k=1}^{\infty}U_{i_k}$. Set $W_n\triangleq S\cap U_{i_n}$ for $n\in\mathbb{N}$, then $W_n$ is a coordinate neighborhood and $\mu(W_n,W_n)<\infty$. Write $Y_1=W_1$ and $Y_n=W_n\setminus  \left(\bigcup _{k=1}^{n-1}W_k\right)$ for $n>1$. Clearly, $\{Y_n\}_{n=1}^{\infty}$  is a sequence of pairwise disjoint Borel subsets of $S$ and $Y_n\subset W_n$ for each $n\in\mathbb{N}$. Now we define a measure $\mu_S$ on $(S,\mathscr{B}(S))$ by
\begin{eqnarray*}
\mu_S(E)\triangleq\sum_{n=1}^{\infty}\mu(E\cap Y_n, W_n),\quad \forall\,E\in\mathscr{B}(S).
\end{eqnarray*}
Suppose that $\mu(\cdot, S\cap U)$ is a local version of surface measure defined in Definition \ref{20241013def1}. Then for any Borel subset $E$ of $S\cap U$, by Lemma \ref{20241013lem1}, we have
\begin{eqnarray*}
\mu_S(E)=\sum_{n=1}^{\infty}\mu(E\cap Y_n, W_n)
=\sum_{n=1}^{\infty}\rho_t(E\cap Y_n, S\cap U)=\mu(E,S\cap U).
\end{eqnarray*}
Note that
\begin{eqnarray*}
\mu_S(S)=\mu_S\left(\bigsqcup\limits_{n=1}^{\infty}Y_n\right)=\sum_{n=1}^{\infty}\mu_S(Y_n),
\end{eqnarray*}
and $\mu_S(Y_n)=\mu_S(Y_n, W_n)\leqslant \mu_S(W_n,W_n)<\infty$ for each $n\in\mathbb{N}$, from which we see that $\mu_S$ is a $\sigma$-finite measure. This completes the proof of Proposition \ref{20241013prop1}.
\end{proof}

\begin{definition}\label{20241009for1}
Suppose that $\overline{S}=S\sqcup (\partial S)$ is a non-empty subset of $\ell^2$, $I$ is a non-empty subset of $\mathbb{N}$ and $J=\mathbb{N}\setminus I$. If $S$ is a surface of $\ell^2$ with codimension $\Gamma_J$ and for each for each $\textbf{x}\in \partial S$, there exists $I_1\in \Gamma_I$, $I_1'\subset I_1$ and $i_1\in I_1$ satisfying $ I_1 =I_1'\cup\{i_1\}$.
There exists an open neighborhood $U_0$ of $\textbf{x}$ in $\ell^2$, an open neighborhood $U_{I_1'}$ of $P_{I_1'}\textbf{x}$ in $P_{I_1'}$ and two mappings $f'\in C^1_{P_{I_1'}}(U_{I_1'};P_{\mathbb{N}\setminus I_1})$ and $f''\in C^1_{P_{I_1'}}(U_{I_1'};\mathbb{R})$ such that
\begin{equation}\label{20241015for1}
(	\partial S)\cap U_0=\left\{\textbf{y}_{I_1'}+f''(\textbf{y}_{I_1'})\textbf{e}_{i_1}+f'(\textbf{y}_{I_1'}):
\textbf{y}_{I_1'}=\sum_{i\in I_1'}y_i\textbf{e}_i\in U_{I_1'}\right\},
\end{equation}
$\textbf{x}=\textbf{x}_{I_1'}+f''(\textbf{x}_{I_1'})\textbf{e}_{i_1}+f'(\textbf{x}_{I_1'})$ and
\begin{eqnarray*}
S\cap U_0=\Big\{\textbf{y}_{I_1'}+y_{i_1}\textbf{e}_{i_1}+f'(\textbf{y}_{I_1'}):\;
&&\textbf{y}_{I_1'}=\sum_{i\in I_1'}y_i\textbf{e}_i\in U_{I_1'},\\
&&y_{i_1}\in(f''(\textbf{y}_{I_1'}),f''(\textbf{y}_{I_1'})+\delta_0)\Big\},
\end{eqnarray*}
or
\begin{eqnarray*}
S\cap U_0=\Big\{\textbf{y}_{I_1'}+y_{i_1}\textbf{e}_{i_1}+f'(\textbf{y}_{I_1'}):\;
&&\textbf{y}_{I_1'}=\sum_{i\in I_1'}y_i\textbf{e}_i\in U_{I_1'},\\
&&y_{i_1}\in(f''(\textbf{y}_{I_1'})-\delta_0,f''(\textbf{y}_{I_1'}))\Big\},
\end{eqnarray*}
for some positive number $\delta_0$. Therefore, there exists $k\in\mathbb{N}$ such that  if we define
\begin{eqnarray*}
g(\textbf{y}_{I_1'}+y_{i_1}\textbf{e}_{i_1})\triangleq (-1)^k(y_{i_1}-f''(\textbf{y}_{I_1'})),\qquad
\end{eqnarray*}
for $\textbf{y}_{I_1'}+y_{i_1}\textbf{e}_{i_1}$ lies in an open neighborhood of $P_{I_1}\textbf{x}$ in $P_{I_1}$, then
\begin{eqnarray*}
S\cap U_0=\Big\{\textbf{y}_{I_1'}+y_{i_1}\textbf{e}_{i_1}+f'(\textbf{y}_{I_1'}):\;
&&\textbf{y}_{I_1'}=\sum_{i\in I_1'}y_i\textbf{e}_i\in U_{I_1'},\\
&&g(\textbf{y}_{I_1'}+y_{i_1}\textbf{e}_{i_1})<0\Big\}.
\end{eqnarray*}
If for each for each $\textbf{x}\in  S$, there exists $I_2\in \Gamma_I$, an open neighborhood $U_1$ of $\textbf{x}$ in $\ell^2$, an open neighborhood $U_{I_2}$ of $\textbf{x}_{I_2}=P_{I_2}\textbf{x}$ in $P_{I_2}$ and $f\in C^1_{P_{I_2}}(U_{I_2};P_{\mathbb{N}\setminus I_2})$ such that $\textbf{x}=\textbf{x}_{I_2}+f(\textbf{x}_{I_2})$ and
\begin{eqnarray}\nonumber
\overline{S}\cap U_1=S\cap U_1=\left\{\textbf{y}_{I_2}+f(\textbf{y}_{I_2}):
\textbf{y}_{I_2}=\sum_{i\in I_2}y_i\textbf{e}_i\in U_{I_2}\right\}.
\end{eqnarray}
Then we call $S$ a surface of $\ell^2$ with codimension $\Gamma_J$ whose boundary is a surface of $\ell^2$ with codimension $\Gamma_{J\cup\{i_1\}}$.
\end{definition}
\begin{remark}
Roughly speaking, the Definition \ref{20241009for1} says that for each $\textbf{x}\in \partial S$, there exists $I_1\in \Gamma_I$ and an open neighborhood $U_0$ of $\textbf{x}$ such that $P_{I_1}(S\cap U_0)$ is locally a surface of $P_{I_1}$ with codimension 1 whose boundary is $P_{I_1}((\partial S)\cap U_0)$.
\end{remark}


\subsection{Sequences of Compactly Supported Functions}

Sequences of compactly supported functions play very important roles in finite-dimensional analysis. In this subsection we shall construct such sort of function sequences on $\ell^2$, which will be of crucial importance in the sequel.

By \eqref{20241023e3}, it follows that
\begin{eqnarray*}
\sum_{i=1}^{\infty}a_i^2\cdot \left|\ln \frac{1}{\sqrt{2\pi} a_i}\right|<\infty.
\end{eqnarray*}
Choosing a sequence of positive numbers $\{c_i\}_{i=1}^{\infty}$ such that $\lim\limits_{i\to\infty}c_i=0$,
 and
\begin{eqnarray}\label{20240920for1}
\sum_{i=1}^{\infty}\frac{a_i^2}{c_i}\cdot \left|\ln \frac{1}{\sqrt{2\pi} a_i}\right|<\infty,\qquad \sum_{i=1}^{\infty}\frac{a_i^2}{c_i}<\infty.
\end{eqnarray}
For each non-empty subset $I$ of $\mathbb{N}$, we also note that for $(x_i)_{i\in I}\in\ell^2(I)$ satisfying
$$
\sum_{i\in I}\left|\ln \frac{1}{\sqrt{2\pi} a_i}-\frac{x_i^2}{2a_i^2}\right|<\infty,
$$
write $d_i=\ln \frac{1}{\sqrt{2\pi} a_i}-\frac{x_i^2}{2a_i^2}$ for $i\in I$, then we have
\begin{equation}\label{20241013for2}
\sum_{i\in I}\frac{x_i^2}{c_i}=\sum_{i\in I}\frac{2a_i^2}{c_i}\left(\ln \frac{1}{\sqrt{2\pi} a_i}-d_i\right)
\leq \sum_{i\in I}\frac{2a_i^2}{c_i}\left|\ln \frac{1}{\sqrt{2\pi} a_i} \right|+\sum_{i\in I}\frac{2a_i^2}{c_i}\cdot |d_i|<\infty.
\end{equation}
Setting $c\triangleq \sup\limits_{i\in\mathbb{N}}c_i$.
Write
\begin{eqnarray}\label{230708e1}
K_n\triangleq \left\{(x_i)_{i\in\mathbb{N}}\in \ell^2:\sum\limits_{i=1}^{\infty}\frac{x_i^2}{c_i}\leqslant n^2\right\},\quad \forall\,n\in\mathbb{N},\quad K\triangleq \bigcup\limits_{n=1}^{\infty}K_n.
\end{eqnarray}

\begin{lemma}\label{basic propeties of Knm}
For any $n\in\mathbb{N}$, $K_n$ is a compact subset of $\ell^2$. If $\sum\limits_{k=1}^{\infty}\frac{a_k^2}{c_k}<\infty,$ then $\lim\limits_{n\to\infty}P_t(K_n)=1$ for any $t\in(0,+\infty)$.
\end{lemma}

\begin{proof}
Firstly, we shall prove that $K_n$ is a closed subset of $\ell^2$.

For any sequence $\{\textbf{x}_m=(x_{i,m})_{i\in\mathbb{N}}\}_{m=1}^{\infty}\subset K_n$ and $\textbf{x}=(x_i)_{i\in\mathbb{N}}\in \ell^2$ such that $\lim\limits_{m\to\infty}\textbf{x}_m=\textbf{x}$ in $\ell^2$. Then for each $i\in\mathbb{N}$, it holds that $\lim\limits_{m\to\infty}x_{i,m}=x_i$. As a consequence, for each $N\in\mathbb{N}$, we have
\begin{eqnarray*}
\sum_{i=1}^{N}\frac{x_i^2}{c_i}=\lim_{m\to\infty}\left(\sum_{i=1}^{N}\frac{x_{i,m}^2}{c_i}\right)\leqslant n^2.
\end{eqnarray*}
Letting $N\to\infty$, we have $\sum\limits_{i=1}^{\infty}\frac{x_i^2}{c_i}\leqslant n^2$ which implies that $\textbf{x}=(x_i)_{i\in\mathbb{N}}\in K_n$.

Secondly, we shall prove that $K_n$ is a compact subset of $\ell^2$.

For any sequence $\{\textbf{y}_m=(y_{i,m})_{i\in\mathbb{N}}\}_{m=1}^{\infty}\subset K_n$. Note that for any $i,m\in\mathbb{N}$, it holds that $y_{i,m}^2\leqslant c_i \left(\sum\limits_{j=1}^{\infty}\frac{y_{j,m}^2}{c_j}\right)\leqslant c_i n^2\leqslant cn^2.$ By the diagonal process, there exists a subsequence $\{\textbf{y}_{m_k}\}_{k=1}^{\infty}$ of $\{\textbf{y}_m\}_{m=1}^{\infty}$ such that $\lim\limits_{k\to\infty}y_{i,m_k}$ exists for each $i\in\mathbb{N}$. Denote $y_i\triangleq\lim\limits_{k\to\infty}y_{i,m_k}$ for each $i\in\mathbb{N}$ and $\textbf{y}\triangleq (y_i)_{i\in\mathbb{N}}$. For each $N_1\in\mathbb{N}$, we have
\begin{eqnarray*}
\sum_{i=1}^{N_1}\frac{y_i^2}{c_i}=\lim_{k\to\infty}\left(\sum_{i=1}^{N_1}\frac{y_{i,m_k}^2}{c_i}\right)\leqslant n^2.
\end{eqnarray*}
Letting $N_1\to\infty$, we have $\sum\limits_{i=1}^{\infty}\frac{y_i^2}{c_i}\leqslant n^2$. Since $\sum\limits_{i=1}^{\infty} y_i^2\leqslant c\cdot\left(\sum\limits_{i=1}^{\infty}\frac{y_i^2}{c_i}\right)\leqslant cn^2<\infty,$ we have $\textbf{y}\in\ell^2$.  These imply that $\textbf{y}\in K_n$. Since $\lim\limits_{i\to\infty}c_i=0$, for any $\varepsilon>0$, there exists $N_2\in\mathbb{N}$ such that $4c_i n^2<\frac{\varepsilon}{2}$ for all $i\geqslant N_2.$ Note that $\lim\limits_{k\to\infty}\sum\limits_{i=1}^{N_2}|y_{i,m_k}-y_i|^2=0$. Thus there exists $N_3\in\mathbb{N}$ such that $\sum\limits_{i=1}^{N_2}|y_{i,m_k}-y_i|^2<\frac{\varepsilon}{2}$ for all $k\geqslant N_3.$ Therefore, for any $k\geqslant N_3$, it holds that
\begin{eqnarray*}
\sum_{i=1}^{\infty}|y_{i,m_k}-y_i|^2&=&\sum_{i=1}^{N_2}|y_{i,m_k}-y_i|^2+\sum_{i=N_2 +1}^{\infty}|y_{i,m_k}-y_i|^2\\
&<&\frac{\varepsilon}{2}+\sum_{i=N_2 +1}^{\infty}2(|y_{i,m_k}|^2+|y_i|^2)\\
&=&\frac{\varepsilon}{2}+\sum_{i=N_2 +1}^{\infty}c_i\cdot\frac{2(|y_{i,m_k}|^2+|y_i|^2)}{c_i}\\
&\leqslant&\frac{\varepsilon}{2}+4\left(\sup_{i>N_2}c_{i}\right)\cdot n^2\\
&\leqslant&\varepsilon,
\end{eqnarray*}
which implies that $\lim\limits_{k\to\infty}\textbf{y}_{m_k}=\textbf{y}$ in $\ell^2$. Thus $K_n$ is a compact subset of $\ell^2$.

Finally, we come to prove that if $\sum\limits_{k=1}^{\infty}\frac{a_k^2}{c_k}<\infty,$ then $\lim\limits_{n\to\infty}P_t(K_n)=1$. Note that
 \begin{eqnarray*}
\int_{\ell^2}\sum_{i=1}^{\infty}\frac{x_i^2}{c_i}\,\mathrm{d}P_t(\textbf{x})
=\sum_{i=1}^{\infty}\int_{\ell^2}\frac{x_i^2}{c_i}\,\mathrm{d}P_t(\textbf{x})
=t^2\sum_{i=1}^{\infty}\frac{a_i^2}{c_i}<\infty,
\end{eqnarray*}
which implies that
\begin{eqnarray*}
P_t\left(\left\{\textbf{x}=(x_i)_{i\in\mathbb{N}}\in\ell^2:\sum_{i=1}^{\infty}\frac{x_i^2}{c_i}<\infty\right\}\right)=1.
\end{eqnarray*}
Since
\begin{eqnarray*}
\bigcup_{n=1}^{\infty}K_n=\left\{\textbf{x}=(x_i)_{i\in\mathbb{N}}\in\ell^2:\sum_{i=1}^{\infty}\frac{x_i^2}{c_i}<\infty\right\},
\end{eqnarray*}
we have $\lim\limits_{n\to\infty}P_t(K_n)=1$. This completes the proof of Lemma \ref{basic propeties of Knm}.
\end{proof}

The proof of the following result is motivated by \cite[pp. 419--421]{Goo}, but our approach is simpler and more readable.
\begin{theorem}\label{230215Th1}
There is a sequence of real-valued, $F$-continuous and $\mathscr{B}(\ell^2)$-measurable functions $\{f_n\}_{n=1}^{\infty}$ on $\ell^2$ with the following properties:

{\rm 1)} For each $n\in\mathbb{N}$, $f_n\notin C(\ell^2)$ and $\supp f_n$ is a compact subset of $\ell^2$;

{\rm 2)}   There exists a positive number $C$ such that $\sum\limits_{i=1}^{\infty}a_i^2\left|\frac{\partial f_n(\textbf{x})}{\partial x_i}\right|^2\leqslant C^2$ for any $\textbf{x}\in\ell^2$ and $n\in\mathbb{N}$.
\end{theorem}

\begin{proof}
Choosing a sequence of positives $\{c_i\}_{i=1}^{\infty}$ as in \eqref{230708e1} such that $\lim\limits_{i\to \infty}c_i=0$ and $\sum\limits_{k=1}^{\infty}\frac{a_k^2}{c_k}<\infty$. Let
\begin{eqnarray*}
g_n(\textbf{x})\triangleq \int_{\ell^2}\chi_{K_n}(\textbf{x}-\textbf{y})\,\mathrm{d}P_1(\textbf{y})=P_1(\textbf{x}-K_n),\quad \forall\,n\in\mathbb{N},\quad\textbf{x}\in\ell^2.
\end{eqnarray*}
(1). By Lemma \ref{basic propeties of Knm}, we have $\lim\limits_{n\to\infty}P_1(K_n)=1$. Choosing $N_1\in\mathbb{N}$ such that $P_1(K_{N_1})>\frac{4}{5}$. For any $n>N_1$. If $\textbf{x}\in K_{n-N_1}$, then we have $K_{N_1}\subset K_n-\textbf{x}$ and hence
\begin{eqnarray*}
g_n(\textbf{x})=P_1(\textbf{x}-K_n)=P_1(K_n-\textbf{x})\geqslant P_1(K_{N_1})>\frac{4}{5}.
\end{eqnarray*}
Similarly, if $\textbf{x}\notin K_{n+N_1}$, then we have $K_n-\textbf{x}\subset \ell^2\setminus K_{N_1}$ and hence
\begin{eqnarray*}
g_n(\textbf{x})=P_1(\textbf{x}-K_n)=P_1(K_n-\textbf{x})\leqslant 1-P_1(K_{N_1})<\frac{1}{5}.
\end{eqnarray*}
Let
$$h(t)\triangleq\left\{\begin{array}{ll}
 e^{\frac{1}{(t-\frac{1}{16})(t-\frac{9}{16})}},\quad&\frac{1}{16}<t<\frac{9}{16},\\
0,&t\geqslant\frac{9}{16}\text{ or }t\leqslant\frac{1}{16}.\end{array}\right.$$
Set $H(x)\triangleq \frac{\int_{-\infty}^{x^2}h(t)\,\mathrm{d}t}{\int_{-\infty}^{+\infty}h(t)\,\mathrm{d}t},\,\forall\,x\in\mathbb{R}$. Then $H\in C^{\infty}(\mathbb{R};[0,1])$, $H(x)=0$ for $|x|<\frac{1}{4}$ and $H(x)=1$ for $|x|>\frac{3}{4}$. Let $f_n\triangleq H\circ g_{n+N_1}$ for each $n\in\mathbb{N}$. By Proposition \ref{partial derivative of Ptf}, $\{f_n\}_{n=1}^{\infty}$ is a sequence of real-valued F continuous and Borel measurable functions. Then for any $n,i\in\mathbb{N}$, we have
\begin{equation}\label{230409eq1}
f_n(\textbf{x})=1,\quad\frac{\partial f_n(\textbf{x})}{\partial x_i}=0,\quad\textbf{x}\in K_{n},\quad\supp f_n\subset K_{n+2N_1},
\end{equation}
and hence $\supp f_n$ is a compact subset of $\ell^2$.

Suppose that $f_n\in C_{\ell^2}(\ell^2)$, then $f_n^{-1}(\mathbb{R}\setminus \{0\})$ is an open subset of $\ell^2$ and $f_n^{-1}(\mathbb{R}\setminus \{0\})\subset$ supp$f_n\subset K_{n+2 N_1}$. Since every non-empty open subset of $\ell^2$ contains countable pairwise disjoint non-empty open balls with the same radius, the interior of any compact subset of $\ell^2$ is empty, we have $f_n^{-1}(\mathbb{R}\setminus \{0\})=\emptyset$ which is a contradiction. Therefore, $f_n\notin C_{\ell^2}(\ell^2)$.

(2). Let $C\triangleq \sup\limits_{x\in\mathbb{R}}|H'(x)|<\infty$. For any $i,n\in\mathbb{N}$, by (2) of Proposition \ref{partial derivative of Ptf}, we have
\begin{eqnarray*}
\left|\frac{\partial f_n(\textbf{x})}{\partial x_i}\right|=|H'(g_{n+N_1}(\textbf{x}))|\cdot\left|\frac{\partial g_{n+N_1}(\textbf{x})}{\partial x_i}\right|\leqslant \frac{C}{a_i},\quad \forall\,\textbf{x}\in\ell^2.
\end{eqnarray*}
Therefore, we have
\begin{eqnarray*}
\sum\limits_{i=1}^{\infty}a_i^2\cdot\left|\frac{\partial f_n(\textbf{x})}{\partial x_i}\right|^2\leqslant \sum\limits_{i=1}^{\infty}a_i\cdot C^2<C^2,\quad \forall\,\textbf{x}\in\ell^2,\,n\in\mathbb{N}.
\end{eqnarray*}
This completes the proof of Theorem \ref{230215Th1}.
\end{proof}

To end this subsection, for the surface $S$, the measure $\mu_S$ on $S$ and the measure $\mu_{\partial S}$ on $\partial S$ constructed in the last subsection, we have the following result.

\begin{theorem}\label{20241013thm1}
Suppose that $\{f_n\}_{n=1}^{\infty}$ is the sequence of functions on $\ell^2$ defined at \eqref{230409eq1}. Then it holds that
\begin{itemize}
\item[$(1)$]$\lim\limits_{n\to\infty}f_n=1$ and $\lim\limits_{n\to\infty}\sum\limits_{i=1}^{\infty}a_i^2\left|\frac{\partial f_n}{\partial x_i}\right|^2=0$ almost everywhere on $S$ with respect to $\mu_S$;
\item[$(2)$]$\lim\limits_{n\to\infty}f_n=1$ and $\lim\limits_{n\to\infty}\sum\limits_{i=1}^{\infty}a_i^2\left|\frac{\partial f_n}{\partial x_i}\right|^2=0$ almost everywhere on $\partial S$ with respect to $\mu_{\partial S}$.
\end{itemize}
\end{theorem}
\begin{proof}
Choosing a sequence of positives $\{c_i\}_{i=1}^{\infty}$ as in \eqref{230708e1} such that $\lim\limits_{i\to \infty}c_i=0$ and $\sum\limits_{k=1}^{\infty}\frac{a_k^2}{c_k}<\infty$. Let
\begin{eqnarray*}
g_n(\textbf{x})\triangleq \int_{\ell^2}\chi_{K_n}(\textbf{x}-\textbf{y})\,\mathrm{d}P_1(\textbf{y})=P_1(\textbf{x}-K_n),\quad \forall\,n\in\mathbb{N},\quad\textbf{x}\in\ell^2.
\end{eqnarray*}
As the assumption \eqref{20241013cond1} in the Definition \ref{20241012def2}, let
\begin{eqnarray*}
C_1\triangleq \sup_{\textbf{x}_I\in U_{I}} n_{I}(\textbf{x}_I+f(\textbf{x}_I))\cdot F_{\mathbb{N}\setminus I}(f(\textbf{x}_I))<\infty.
\end{eqnarray*}
Note that
\begin{equation}
\begin{array}{ll}\ds
\int_{S\cap U}\left(\sum_{i\in I}\frac{x_i^2}{c_i}\right)\,\mathrm{d}\mu_S(\textbf{x},S\cap U)\\[2mm]
\ds= \int_{P_{I}(S\cap U)}  \left(\sum_{i\in I}\frac{x_i^2}{c_i}\right)\cdot n_{I}(\textbf{x}_I+f(\textbf{x}_I))\cdot F_{\mathbb{N}\setminus I}(f(\textbf{x}_I))\,\mathrm{d}\mu_I(\textbf{x}_I)\\[2mm]
\ds\leqslant C_1\cdot \int_{P_{I}(S\cap U)} \left(\sum_{i\in I}\frac{x_i^2}{c_i}\right)\,\mathrm{d}\mu_I(\textbf{x}_I)
\leqslant C_1\cdot \int_{P_{I}} \left(\sum_{i\in I}\frac{x_i^2}{c_i}\right)\,\mathrm{d}\mu_I(\textbf{x}_I)\label{230413eq2}\\[2mm]\ds
=C_1\cdot  \left(\sum_{i\in I}\frac{a_i^2}{c_i}\right)
<C_1\cdot   \left(\sum_{i=1}^{\infty}\frac{a_i^2}{c_i}\right)<\infty.
\end{array}
\end{equation}
From \eqref{230413eq2}, we see that the series $\sum\limits_{i\in I}\frac{x_i^2}{c_i}<\infty$ almost everywhere on $S\cap U$ respect to $\mu_{S}(\cdot,S\cap U)$, i.e., there exists a Borel subset $E$ of $S\cap U$ such that $\mu_{S}((S\cap U)\setminus E,S\cap U)=0$ and $\sum\limits_{i\in I}\frac{x_i^2}{c_i}<\infty$ for each $\textbf{x}=(x_i)_{i\in\mathbb{N}}\in E$. Combing assumption \eqref{20241013cond1} and the arguments around \eqref{20240920for1} and \eqref{20241013for2}, we have $\sum\limits_{i=1}^{\infty}\frac{x_i^2}{c_i}<\infty$ for each $\textbf{x}=(x_i)_{i\in\mathbb{N}}\in E$, which implies that $E\subset K$ which is defined at \eqref{230708e1}. Suppose that $\{S\cap U_n, P_{I_n}\}_{n=1}^{\infty}$ is a family of locally coordinate pairings that covers $S$. For each $n\in\mathbb{N}$, choosing a Borel subset $E_n$ of $S\cap U_n$ such that $\mu_{S}((S\cap U_n)\setminus E_n,S\cap U_n)=0$ and $E_n\subset K$. Set $E\triangleq \bigcup\limits_{n=1}^{\infty}E_n$. Then it holds that $E$ is a Borel subset of $S$,
\begin{eqnarray*}
\mu_{S}(S\setminus E)&=&\mu_{S}\left(\bigcup\limits_{n=1}^{\infty}((S\cap U_n)\setminus E)\right)
\leqslant \sum_{n=1}^{\infty}\mu_{S}((S\cap U_n)\setminus E)\\
&\leqslant  &\sum_{n=1}^{\infty}\mu_{S}((S\cap U_n)\setminus E_n)= \sum_{n=1}^{\infty}\mu_{S}((S\cap U_n)\setminus E_n,S\cap U_n)=0,
\end{eqnarray*}
and for each $\textbf{x}\in E\subset K$, by \eqref{230409eq1}, we have  $\lim\limits_{n\to\infty}f_n(\textbf{x})=1$ and $\lim\limits_{n\to\infty}\sum\limits_{i=1}^{\infty}a_i^2\left|\frac{\partial f_n(\textbf{x})}{\partial x_i}\right|^2=0$. Hence,
$\lim\limits_{n\to\infty}f_n=1$ and $\lim\limits_{n\to\infty}\sum\limits_{i=1}^{\infty}a_i^2\left|\frac{\partial f_n}{\partial x_i}\right|^2=0$ almost everywhere on $S$ with respect to $\mu_S$. This completes the proof of conclusion (1). The proof of conclusion  (2) is similar and hence we omit the details. The proof of Theorem \ref{20241013thm1} is completed.
\end{proof}

\newpage

\section{Infinite Dimensional Stokes
Type Theorems}
\label{20240126chapter1}
In calculus, the integration by parts formula is of fundamental importance. Many important results in finite-dimensional spaces are based on this formula and its generalizations including the Gauss-Green formula, the divergence theorem and the Stokes formula. We will establish its infinite dimensional counterpart in this section which will be frequently used in the rest sections of this work. This section is mainly based on \cite{WYZ}.

\subsection{Local Version of Gauss-Green Type Theorem}

Throughout this section, we fix a non-empty subset $I$ of $\mathbb{N}$ and a surface $S$ of $\ell^2$ with codimension $\Gamma_{\mathbb{N}\setminus I}$. We shall borrow some idea in \cite{Goo}. Nevertheless, we consider $C^1$ surfaces instead of $H$-$C^1$ surfaces in \cite{Goo} and we use the measure $P_r$ (defined by (\ref{220817e1})) instead of the abstract Wiener space there.

For $\textbf{x}\in S$, if $(S\cap U_1, P_{I_1})$ and $(S\cap U_2, P_{I_2})$ are two coordinate pairings of $S$ such that $\textbf{x}\in S\cap U_1\cap U_2$. By Corollary \ref{20241015cor1}, we see that
$$
\text{det}\left(D(P_{I_1}P_{I_2}^{-1})(\textbf{x}_{I_2})\right)\cdot \text{det}\left(D(P_{I_2}P_{I_1}^{-1})(\textbf{x}_{I_1})\right)=1,
$$
where $\textbf{x}_{I_k}=P_{I_k}\textbf{x}$ for $k=1,2$. Therefore, both $\text{det}\left(D(P_{I_1}P_{I_2}^{-1})(\textbf{x}_{I_2})\right)$ and $\text{det}\left(D(P_{I_2}P_{I_1}^{-1})(\textbf{x}_{I_1})\right)$ are positive numbers or negative numbers. Motivated by this fact, we have the following definition.
\begin{definition}\label{20241015def1}
We call the surface $S$ is orientable, if there exists an index set $\Lambda$ and a family of coordinate pairings $\{(S\cap U_i,P_{I_i}):i\in \Lambda\}$ of $S$ such that $S\subset\bigcup\limits_{i\in \Lambda}U_i$ and for any $i_1,i_2\in \Lambda$ with $S\cap U_1\cap U_2\neq \emptyset$, it holds that
\begin{eqnarray}\label{20241015for3}
\text{det}\left(D(P_{I_{i_1}}P_{I_{i_2}}^{-1})(\textbf{x}_{I_{i_2}})\right)>0,\quad\forall\, \textbf{x}_{I_{i_2}}\in P_{I_{i_2}}(S\cap U_1\cap U_2).
\end{eqnarray}
\end{definition}

We will use the notation $S_{\mathbb{N}}$ to stand for the family of all ordered non-empty subsets of $\mathbb{N}$, i.e., each element of $S_{\mathbb{N}}$ can be written as
$I=(i_1,\cdots,i_n)$ for some $n\in \mathbb{N}$ and distinct positive integers $i_1,\cdots,i_n$, or $I=(i_k)_{k=1}^{\infty}$ for some sequence $\{i_k\}_{k=1}^{\infty}$ of distinct positive integers.
Clearly, for each $I\in S_{\mathbb{N}}$, there exists a unique $\overline{I}\in S_{\mathbb{N}}$ such that if $I=(i_1,\cdots,i_n)$ for some $n\in\mathbb{N}$, then $\overline{I}=(i_1',\cdots,i_n')$, $\{i_1,\cdots,i_n\}=\{i_1',\cdots,i_n'\}$ and $i_1'<\cdots<i_n'$, if $I=(i_k)_{k=1}^{\infty}$, then $\overline{I}=(i_k')_{k=1}^{\infty}$, $\{i_k:k\in\mathbb{N}\}=\{i_k':k\in\mathbb{N}\}$ and $i_1'<\cdots<i_k'<\cdots$.

For each $I=(i_k)_{k\in \Delta}\in S_{\mathbb{N}}$ where $\Delta=\{1,2,\cdots,n\}$ for some $n\in\mathbb{N}$ or $\Delta=\mathbb{N}$, if $\sigma$ is a bijection from $\{i_k:k\in \Delta\}$ onto itself, then we call $\sigma$ a permutation on $I$ which is the same name as its finite dimensional counterpart in \cite[p. 8]{Lan1}. Write $\sigma(I)\triangleq(\sigma(i_k))_{k\in \Delta}$. If $\sigma$ is a permutation and there exists two distinct $i_{k_1},i_{k_2}\in\Delta$ such that $\sigma(i_{k_1})=i_{k_2},\,\sigma(i_{k_2})=i_{k_1}$ and $\sigma(i_k)=i_k$ for $i_k\in \Delta\setminus\{i_{k_1},i_{k_2}\}$, then we call $\sigma$ a transposition which is the same name as its finite dimensional counterpart in \cite[p. 13]{Lan1}. We will use the notation $S_{\mathbb{N}}^F$ to denote the collection of all $I\in S_{\mathbb{N}}$ such that there exists finite transpositions $\sigma_1,\cdots,\sigma_{n}$ such that $(\sigma_1\circ\cdots\circ\sigma_n)(I)=\overline{I}$. For $I\in S_{\mathbb{N}}^F$, if there are even transpositions $\sigma_1,\cdots,\sigma_{2n}$ such that $(\sigma_1\circ\cdots\circ\sigma_{2n})(I)=\overline{I}$, then let $s(I)\triangleq 1$, otherwise, let $s(I)\triangleq -1$ and we call it the sign of $I$.

\begin{remark}
The subsets of $\mathbb{N}$ we used before can be viewed as ordered non-empty subset of $\mathbb{N}$ with strictly increasing order, we will not distinguish them if there is no confusion arise.
\end{remark}
Now, we can easily construct some non-trivial coordinate independent functions on orientable surfaces.
\begin{exa}
We adapt all the notations and assumptions as in the Definition \ref{20241015def1}. For each $I_0\in S_{\mathbb{N}}^F$ such that $\overline{I_0}\in \Gamma_{I}$ and $\textbf{x}\in S$, there exists a coordinate pairing $(S\cap U_1,P_{I_1})$ such that $\textbf{x}\in S\cap U_1$, define
\begin{eqnarray*}
\textbf{n}_{I_0}^S(\textbf{x})\triangleq \frac{s(I_0)\cdot\text{det}\left(D(P_{\overline{I_{0}}}P_{I_{1}}^{-1})(P_{I_{1}}\textbf{x})\right)}{n_{I_1}(\textbf{x})}.
\end{eqnarray*}
To see it is well-defined, we only need to see that if $(S\cap U_2,P_{I_2})$is another coordinate pairing such that $\textbf{x}\in S\cap U_2$. Combing Corollary \ref{20241015cor1} and Lemma \ref{20241012lem3}, we have
\begin{eqnarray*}
\frac{\text{det}\left(D(P_{\overline{I_{0}}}P_{I_{1}}^{-1})(P_{I_{1}}\textbf{x})\right)}{n_{I_1}(\textbf{x})}
&=&\frac{\text{det}\left(D(P_{\overline{I_{0}}}P_{I_{2}}^{-1})(P_{I_{2}}\textbf{x})\right)\cdot \text{det}\left(D(P_{I_{2}}P_{I_{1}}^{-1})(P_{I_{1}}\textbf{x})\right)}{n_{I_2}(\textbf{x})\cdot |\text{det}\left(D(P_{I_{2}}P_{I_{1}}^{-1})(P_{I_{1}}\textbf{x})\right)|}\\
&=&\frac{\text{det}\left(D(P_{\overline{I_{0}}}P_{I_{2}}^{-1})(P_{I_{2}}\textbf{x})\right)  }{n_{I_2}(\textbf{x}) },
\end{eqnarray*}
where the second equality follows from \eqref{20241015for3}. By Remark \ref{20241016rem1}, one may see that $0\leq |\textbf{n}_{I_0}^S(\textbf{x})|\leq 1$.
\end{exa}
\begin{definition}\label{20241009for1wb}
Suppose that $\overline{S}=S\sqcup (\partial S)$ is a non-empty closed subset of $\ell^2$, $I$ is a non-empty subset of $\mathbb{N}$ and $J=\mathbb{N}\setminus I$. We call $S$ a surface of $\ell^2$ with co-dimension $\Gamma_J$ and boundary $\partial S$, if $S$ is a surface of $\ell^2$ with co-dimension $\Gamma_J$ and
\begin{itemize}
\item[$(1)$] For each $\textbf{x}^0\in \partial S$, there exists $K=\{k_{1},k_2,\cdots\}\in S_{\mathbb{N}}^F$, $k_{i_0}\in \overline{K}$, an open neighborhood $U$ of $\textbf{x}^0$ in $\ell^2$, an open neighborhood $U_{K }$ of $P_{K }\textbf{x}^0$ in $P_{K }$, an open neighborhood $U_{K\setminus\{k_{i_0}\}}$ of $P_{K\setminus\{k_{i_0}\}}\textbf{x}^0$ in $P_{K\setminus\{k_{i_0}\}}$ and two mappings $f'\in C^1_{P_{K\setminus\{k_{i_0}\}}}(U_{K\setminus\{k_{i_0}\}};P_{\mathbb{N}\setminus \overline{K}})$ and $f''\in C^1_{P_{K\setminus\{k_{i_0}\}}}(U_{K\setminus\{k_{i_0}\}};\mathbb{R})$,
$\overline{K}\in \Gamma_I$, $\textbf{x}^0=\textbf{x}^0_{K\setminus\{k_{i_0}\}}+f''(\textbf{x}^0_{K\setminus\{k_{i_0}\}})\textbf{e}_{k_{i_0}}+f'(\textbf{x}^0_{K\setminus\{k_{i_0}\}})$,
\begin{equation}\label{20241015for1zx}
\begin{array}{ll}
\displaystyle (	\partial S)\cap U=\big\{\textbf{x}_{K\setminus\{k_{i_0}\}}+f''(\textbf{x}_{K\setminus\{k_{i_0}\}})\textbf{e}_{k_{i_0}}+f'(\textbf{x}_{K\setminus\{k_{i_0}\}}):\\[2mm]
\displaystyle\qq\qq\qq\qq\qq\qq\qq\qq\qq
\textbf{x}_{K\setminus\{k_{i_0}\}}\in U_{K\setminus\{k_{i_0}\}}\big\},
\end{array}
\end{equation}
and
\begin{eqnarray*}
S\cap U=\big\{\textbf{x}_{K\setminus\{k_{i_0}\}}+x_{k_{i_0}}\textbf{e}_{k_{i_0}}+f'(\textbf{x}_{K\setminus\{k_{i_0}\}}):
\textbf{x}_{K\setminus\{k_{i_0}\}}+x_{k_{i_0}}\textbf{e}_{k_{i_0}}\in U_{K},\\
g(\textbf{x}_{K\setminus\{k_{i_0}\}}+x_{k_{i_0}}\textbf{e}_{k_{i_0}})<0\big\},
\end{eqnarray*}
where for any $\textbf{x}_{K\setminus\{k_{i_0}\}}+x_{k_{i_0}}\textbf{e}_{k_{i_0}}\in
U_{K }$,
\begin{eqnarray*}
g(\textbf{x}_{K\setminus\{k_{i_0}\}}+x_{k_{i_0}}\textbf{e}_{k_{i_0}})\triangleq (-1)^{i_0-1}(x_{k_{i_0}}-f''(\textbf{x}_{K\setminus\{k_{i_0}\}}));
\end{eqnarray*}
\item[$(2)$] For each $\textbf{x}^0\in  S$, there exists $I_2\in \Gamma_I$, an open neighborhood $U_1$ of $\textbf{x}^0$ in $\ell^2$, an open neighborhood $U_{I_2}$ of $\textbf{x}_{I_2}^0=P_{I_2}\textbf{x}^0$ in $P_{I_2}$ and $f\in C^1_{P_{I_2}}(U_{I_2};P_{\mathbb{N}\setminus I_2})$ such that $\textbf{x}^0=\textbf{x}_{I_2}^0+f(\textbf{x}_{I_2}^0)$ and
\begin{eqnarray}\nonumber
\overline{S}\cap U_1=S\cap U_1=\left\{\textbf{x}_{I_2}+f(\textbf{x}_{I_2}):
\textbf{x}_{I_2}\in U_{I_2}\right\}.
\end{eqnarray}
\end{itemize}
\end{definition}

\begin{definition}\label{20241013def2}
We adapt the same notations and assumptions as in the Definition \ref{20241009for1wb}. If $S$ a locally finite surface of $\ell^2$ with codimension $\Gamma_J$ whose boundary $\partial S$ is a locally finite surface of $\ell^2$ with codimension $\Gamma_{(\mathbb{N}\setminus\overline{K})\cup\{k_{i_0}\}}$. Then by the arguments from Definition \ref{20241113def1} to Proposition \ref{20241013prop1}, we can define two $\sigma$-finite Borel measures $\mu_S$ and $\mu_{\partial S}$ on $S$ and $\partial S$ respectively.
\end{definition}

In order to obtain a global version of Gauss-Green type theorem, we need a version of local Gauss-Green type theorem with coordinate pairing independent integral functions.

\begin{proposition}\label{20241102prop1}
(\textbf{Local Version of Gauss-Green Type Theorem}).We adapt the same assumptions and notations as in Definitions \ref{20241009for1wb} and \ref{20241013def2}, let $f$ be a real-valued function defined on $U$ and $I_0\in S_{\mathbb{N}}^F$ such that $\overline{I_0}\in \Gamma_{\overline{K}\setminus\{k_{i_0}\}}$. If $f$ is Borel measurable on $U$, $f\equiv 0$ on $ U\setminus B_0$ where $B_0$ is a non-empty closed ball contained in $U$, $f\in C_F^1(U)$, $f$ is Fr\'{e}chet differentiable respect to the $H_{ \mathbb{N}}^2$-direction (Recall the definitions at \eqref{20241020for1} and Definition \ref{20241115def1}),
$$
 \sum\limits_{i=1}^{\infty} \left( |D_{x_i}f  |+\left|\frac{x_i}{a_i^2}\cdot f \right| \right) \cdot |\textbf{n}_{\left( i,I_0 \right)}^{S} |
$$
is integrable on $S\cap U$ with respect to the measure $\mu_S$, and $f \cdot \textbf{n}_{I_0}^{\partial S}$ is integrable on $\partial S$ with respect to the surface measure $\mu_{\partial S}$,
 then
$$
\begin{aligned}
-\sum\limits_{i=1}^{\infty} \int_{U\cap S}  \delta_{i} f\left( \textbf{x} \right) \cdot \textbf{n}_{\left( i,I_0 \right)}^{S}\left( \textbf{x} \right) \mathrm{d} \mu_{S} \left( \textbf{x} \right)=\int_{ (\partial S)\cap U  }  f \cdot    \textbf{n}_{I_0}^{\partial S}
 \mathrm{d} \mu_{\partial S}.
\end{aligned}
$$
If there exists $0<r<R<\infty$ and a real-valued Borel measurable function $g$ on $U_1$ such that $B_R(\textbf{x}^0)\subset U_1$, supp$g\subset B_r(\textbf{x}^0)$, $g$ is Fr\'{e}chet differentiable respect to the $H_{ \mathbb{N}}^2$ direction, and
$$
 \sum\limits_{i=1}^{\infty} \left( |D_{x_i}g  |+\left|\frac{x_i}{a_i^2}\cdot g \right| \right) \cdot |\textbf{n}_{\left( i,I_0 \right)}^{S} |
$$
is integrable on $S\cap U_1$ with respect to the measure $\mu_S$, then
$$
\begin{aligned}
-\sum\limits_{i=1}^{\infty} \int_{U_1\cap S}  \delta_{i} g\left( \textbf{x} \right) \cdot \textbf{n}_{\left( i,I_0 \right)}^{S}\left( \textbf{x} \right) \mathrm{d} \mu_{S} \left( \textbf{x} \right)=0.
\end{aligned}
$$
\end{proposition}

\begin{proof}
By assumptions, there exists $n\in\mathbb{N}$ such that
$$
\ba{ll}\ds
I_0 =\left\{ i_{1},i_{2},\cdots ,i_{n},i_{n+1},i_{n+2},\cdots \right\},\\[2mm]
\ds K=\left\{k_0, k_{1},k_{2},\cdots ,k_{n},k_{n+1},k_{n+2},\cdots \right\} ,
\ea
$$
and $i_{s}=k_{s},s\geqslant n+1.$ Let $\textbf{n}_{\left( i, I_0 \right)}^{S}\triangleq 0,i\in \overline{I_0}$. Note that
$$
\begin{aligned}
&\sum\limits_{i=1}^{\infty} \int_{U\cap S}  \delta_{i} f\left( \textbf{x} \right) \cdot \textbf{n}_{\left( i,I_0 \right)}^{S}\left( \textbf{x} \right) \mathrm{d} \mu_{S} \left( \textbf{x} \right)\\
&=\int_{U\cap S} \sum\limits_{i=1}^{\infty} \left( D_{x_i}f\left( \textbf{x} \right) -\frac{x_i}{a_i^2}\cdot f\left( \textbf{x} \right) \right) \cdot \textbf{n}_{\left( i,I_0 \right)}^{S}\left( \textbf{x} \right) \mathrm{d} \mu_{S} \left( \textbf{x} \right)\\
& =\int_{P_{K}\left( U \right)} \sum\limits_{i=1}^{\infty} D_{x_i}f\left( P_{K}^{-1}\textbf{x}_K \right)\cdot \left( -1 \right)^{s\left( \left( i,I_0 \right) \right)}\cdot \det \left( D \left( P_{\overline{\left( i,I_0 \right)}} P_{K}^{-1} \right) \right) \left( \textbf{x}_{K}  \right)\\
&\qq\qq\qq\qq\qq\qq\qq\qq\qq\qq\cdot F_{\mathbb{N} \setminus K} \left( P_{K}^{-1}\textbf{x}_{K}  \right) \mathrm{d} \mu_{K} \left( \textbf{x}_{K}  \right)\\
&\q-\int_{P_{K}\left( U \right)} \sum\limits_{i=1}^{\infty} \frac{x_i}{a_i^2}\cdot f\left( P_{K}^{-1}\textbf{x}_{K}  \right)\cdot \left( -1 \right)^{s \left( \left( i,I_0 \right) \right)}\cdot \det \left( D\left( P_{\overline{\left( i,I_0 \right)}}P_{K}^{-1} \right) \right) \left( \textbf{x}_{K}  \right)\\
&\qq\qq\qq\qq\qq\qq\qq\qq\qq\qq\cdot F_{\mathbb{N} \setminus K} \left( P_{K}^{-1}\textbf{x}_{K}  \right) \mathrm{d} \mu_{K} \left( \textbf{x}_{K}  \right)\\
&=\int_{P_{K}\left( U \right)} \sum\limits_{i=1}^{\infty} D_{x_i}f(P_K^{-1}\cdot)\cdot
\begin{vmatrix}D_{x_{k_{0}}}\left( P_iP_{K}^{-1} \right) &\cdots&D_{x_{k_{n}}}\left( P_{i}P_{K}^{-1} \right)
\\ D_{x_{k_{0}}}\left( P_{i_{1}}P_{K}^{-1} \right) &\cdots&D_{x_{k_{n}}}\left( P_{i_{1}}P_{K}^{-1} \right)
\\ \vdots&\ddots&\vdots
\\ D_{x_{k_{0}}}\left( P_{i_{n}}P_{K}^{-1} \right)&\cdots&D_{x_{k_{n}}}\left( P_{i_{n }}P_{K}^{-1} \right)\end{vmatrix}
\cdot F_{\mathbb{N} \setminus K} \mathrm{d} \mu_{K}\\
&\q-\int_{P_{K}\left( U \right)} \sum\limits_{i=1}^{\infty} \frac{x_i}{a_i^2}\cdot f\left( P_{K}^{-1}\cdot  \right) \cdot\begin{vmatrix}D_{x_{k_{0}}}\left( P_iP_{K}^{-1} \right) &\cdots&D_{x_{k_{n}}}\left( P_{i}P_{K}^{-1} \right)\\ D_{x_{k_{0}}}\left( P_{i_{1}}P_{K}^{-1} \right) &\cdots&D_{x_{k_{n}}}\left( P_{i_{1}}P_{K}^{-1} \right)\\ \vdots&\ddots&\vdots\\ D_{x_{k_{0}}}\left( P_{i_{n}}P_{K}^{-1} \right)&\cdots&D_{x_{k_{n}}}\left( P_{i_{n }}P_{K}^{-1} \right)\end{vmatrix}\\
&\qq\qq\qq\qq\qq\qq\qq\qq\qq\qq\qq\qq\qq\qq\cdot F_{\mathbb{N} \setminus K} \mathrm{d} \mu_{K}.
\end{aligned}
$$
Since
$$
\begin{aligned}
&\sum\limits_{i=1}^{\infty} D_{x_i}f(P_K^{-1}\cdot)\cdot
\begin{vmatrix}D_{x_{k_{0}}}\left( P_iP_{K}^{-1} \right) &\cdots&D_{x_{k_{n}}}\left( P_{i}P_{K}^{-1} \right)\\ D_{x_{k_{0}}}\left( P_{i_{1}}P_{K}^{-1} \right) &\cdots&D_{x_{k_{n}}}\left( P_{i_{1}}P_{K}^{-1} \right)\\ \vdots&\ddots&\vdots\\ D_{x_{k_{0}}}\left( P_{i_{n}}P_{K}^{-1} \right)&\cdots&D_{x_{k_{n}}}\left( P_{i_{n }}P_{K}^{-1} \right)
\end{vmatrix}
\\
&=
\begin{vmatrix}
 \sum\limits_{i=1}^{\infty} D_{x_i}f(P_K^{-1}\cdot)\cdot D_{x_{k_{0}}}\left( P_iP_{K}^{-1} \right) &\cdots& \sum\limits_{i=1}^{\infty} D_{x_i}f(P_K^{-1}\cdot)\cdot D_{x_{k_{n}}}\left( P_{i}P_{K}^{-1} \right)\\ D_{x_{k_{0}}}\left( P_{i_{1}}P_{K}^{-1} \right)&\cdots&D_{x_{k_{n}}}\left( P_{i_{1}}P_{K}^{-1} \right)\\ \vdots&\ddots&\vdots\\ D_{x_{k_{0}}}\left( P_{i_{n}}P_{K}^{-1} \right)&\cdots&D_{x_{k_{n}}}\left( P_{i_{n }}P_{K}^{-1} \right)
\end{vmatrix}
\\
&=
\begin{vmatrix}
 D_{x_{k_{0}}}f(P_K^{-1}\cdot) &\cdots& D_{x_{k_{n}}}f(P_K^{-1}\cdot)\\ D_{x_{k_{0}}}\left( P_{i_{1}}P_{K}^{-1} \right)&\cdots&D_{x_{k_{n}}}\left( P_{i_{1}}P_{K}^{-1} \right)\\ \vdots&\ddots&\vdots\\ D_{x_{k_{0}}}\left( P_{i_{n}}P_{K}^{-1} \right)&\cdots&D_{x_{k_{n}}}\left( P_{i_{n }}P_{K}^{-1} \right)
\end{vmatrix},
\end{aligned}
$$
and
$$
\begin{aligned}
&-\sum\limits_{i=1}^{\infty} \frac{x_i}{a_i^2}\cdot f\left( P_{K}^{-1}\cdot  \right) \cdot\begin{vmatrix}D_{x_{k_{0}}}\left( P_iP_{K}^{-1} \right) &D_{x_{k_{1}}}\left( P_{i}P_{K}^{-1} \right)&\cdots&D_{x_{k_{n}}}\left( P_{i}P_{K}^{-1} \right)\\ D_{x_{k_{0}}}\left( P_{i_{1}}P_{K}^{-1} \right)&D_{x_{k_{1}}}\left( P_{i_{1}}P_{K}^{-1} \right)&\cdots&D_{x_{k_{n}}}\left( P_{i_{1}}P_{K}^{-1} \right)\\ \vdots&\vdots&\ddots&\vdots\\ D_{x_{k_{0}}}\left( P_{i_{n}}P_{K}^{-1} \right)&D_{x_{k_{1}}}\left( P_{i_{n }}P_{K}^{-1} \right)&\cdots&D_{x_{k_{n}}}\left( P_{i_{n }}P_{K}^{-1} \right)\end{vmatrix}\\
&=-\sum\limits_{j=0}^{n} \frac{x_{k_j}}{a_{k_j}^2}\cdot f\left( P_{K}^{-1}\cdot  \right) \cdot\begin{vmatrix}D_{x_{k_{0}}}\left( P_{k_j}P_{K}^{-1} \right) &D_{x_{k_{1}}}\left( P_{k_j}P_{K}^{-1} \right)&\cdots&D_{x_{k_{n}}}\left( P_{k_j}P_{K}^{-1} \right)\\ D_{x_{k_{0}}}\left( P_{i_{1}}P_{K}^{-1} \right)&D_{x_{k_{1}}}\left( P_{i_{1}}P_{K}^{-1} \right)&\cdots&D_{x_{k_{n}}}\left( P_{i_{1}}P_{K}^{-1} \right)\\ \vdots&\vdots&\ddots&\vdots\\ D_{x_{k_{0}}}\left( P_{i_{n}}P_{K}^{-1} \right)&D_{x_{k_{1}}}\left( P_{i_{n }}P_{K}^{-1} \right)&\cdots&D_{x_{k_{n}}}\left( P_{i_{n }}P_{K}^{-1} \right)\end{vmatrix}\\
&\quad -\sum\limits_{i\in \mathbb{N}\setminus K} \frac{x_i}{a_i^2}\cdot f\left( P_{K}^{-1}\cdot  \right) \cdot\begin{vmatrix}D_{x_{k_{0}}}\left( P_iP_{K}^{-1} \right) &D_{x_{k_{1}}}\left( P_{i}P_{K}^{-1} \right)&\cdots&D_{x_{k_{n}}}\left( P_{i}P_{K}^{-1} \right)\\ D_{x_{k_{0}}}\left( P_{i_{1}}P_{K}^{-1} \right)&D_{x_{k_{1}}}\left( P_{i_{1}}P_{K}^{-1} \right)&\cdots&D_{x_{k_{n}}}\left( P_{i_{1}}P_{K}^{-1} \right)\\ \vdots&\vdots&\ddots&\vdots\\ D_{x_{k_{0}}}\left( P_{i_{n}}P_{K}^{-1} \right)&D_{x_{k_{1}}}\left( P_{i_{n }}P_{K}^{-1} \right)&\cdots&D_{x_{k_{n}}}\left( P_{i_{n }}P_{K}^{-1} \right)\\
\end{vmatrix}\\
&=-\sum\limits_{j=0}^{n} \frac{x_{k_j}}{a_{k_j}^2}\cdot f\left( P_{K}^{-1}\cdot  \right) \cdot\begin{vmatrix}D_{x_{k_{0}}}\left( P_{k_j}P_{K}^{-1} \right) &D_{x_{k_{1}}}\left( P_{k_j}P_{K}^{-1} \right)&\cdots&D_{x_{k_{n}}}\left( P_{k_j}P_{K}^{-1} \right)\\ D_{x_{k_{0}}}\left( P_{i_{1}}P_{K}^{-1} \right)&D_{x_{k_{1}}}\left( P_{i_{1}}P_{K}^{-1} \right)&\cdots&D_{x_{k_{n}}}\left( P_{i_{1}}P_{K}^{-1} \right)\\ \vdots&\vdots&\ddots&\vdots\\ D_{x_{k_{0}}}\left( P_{i_{n}}P_{K}^{-1} \right)&D_{x_{k_{1}}}\left( P_{i_{n }}P_{K}^{-1} \right)&\cdots&D_{x_{k_{n}}}\left( P_{i_{n }}P_{K}^{-1} \right)\end{vmatrix}\\
&\quad +\begin{vmatrix}\frac{D_{x_{k_{0}}}\left( F_{\mathbb{N}\setminus K}(P_K^{-1}) \right)}{F_{\mathbb{N}\setminus K}(P_K^{-1})} &\frac{D_{x_{k_{1}}}\left( F_{\mathbb{N}\setminus K}(P_K^{-1}) \right)}{F_{\mathbb{N}\setminus K}(P_K^{-1})}&\cdots&\frac{D_{x_{k_{n}}}\left( F_{\mathbb{N}\setminus K}(P_K^{-1}) \right)}{F_{\mathbb{N}\setminus K}(P_K^{-1})}\\ D_{x_{k_{0}}}\left( P_{i_{1}}P_{K}^{-1} \right)&D_{x_{k_{1}}}\left( P_{i_{1}}P_{K}^{-1} \right)&\cdots&D_{x_{k_{n}}}\left( P_{i_{1}}P_{K}^{-1} \right)\\ \vdots&\vdots&\ddots&\vdots\\ D_{x_{k_{0}}}\left( P_{i_{n}}P_{K}^{-1} \right)&D_{x_{k_{1}}}\left( P_{i_{n }}P_{K}^{-1} \right)&\cdots&D_{x_{k_{n}}}\left( P_{i_{n }}P_{K}^{-1} \right)\\
\end{vmatrix},
\end{aligned}
$$
we have
$$
\begin{aligned}
&\int_{P_{K}\left( U \right)} \sum\limits_{i=1}^{\infty} D_{x_i}f(P_K^{-1}\cdot)\cdot
\begin{vmatrix}D_{x_{k_{0}}}\left( P_iP_{K}^{-1} \right) &\cdots&D_{x_{k_{n}}}\left( P_{i}P_{K}^{-1} \right)\\ D_{x_{k_{0}}}\left( P_{i_{1}}P_{K}^{-1} \right)&\cdots&D_{x_{k_{n}}}\left( P_{i_{1}}P_{K}^{-1} \right)\\ \vdots&\ddots&\vdots\\ D_{x_{k_{0}}}\left( P_{i_{n}}P_{K}^{-1} \right)&\cdots&D_{x_{k_{n}}}\left( P_{i_{n }}P_{K}^{-1} \right)\end{vmatrix}
\cdot F_{\mathbb{N} \setminus K} \mathrm{d} \mu_{K}\\
&\quad -\int_{P_{K}\left( U \right)} \sum\limits_{i=1}^{\infty} \frac{x_i}{a_i^2}\cdot f\left( P_{K}^{-1}\cdot  \right) \cdot
\begin{vmatrix}D_{x_{k_{0}}}\left( P_iP_{K}^{-1} \right) &\cdots&D_{x_{k_{n}}}\left( P_{i}P_{K}^{-1} \right)\\ D_{x_{k_{0}}}\left( P_{i_{1}}P_{K}^{-1} \right)&\cdots&D_{x_{k_{n}}}\left( P_{i_{1}}P_{K}^{-1} \right)\\ \vdots&\ddots&\vdots\\
D_{x_{k_{0}}}\left( P_{i_{n}}P_{K}^{-1} \right)&\cdots&D_{x_{k_{n}}}\left( P_{i_{n }}P_{K}^{-1} \right)\end{vmatrix}\\
&\q\qq\qq\qq\qq\qq\qq\qq\qq\qq\qq\qq\qq\qq\cdot F_{\mathbb{N} \setminus K} \mathrm{d} \mu_{K}\\
&=\int_{P_{K}\left( U \right)}
\begin{vmatrix}
\delta_{k_{0}}\left( f(P_{K}^{-1})\cdot  F_{\mathbb{N} \setminus K}(P_{K}^{-1}) \right) &\cdots&\delta_{k_{n}}\left( f(P_{K}^{-1})\cdot  F_{\mathbb{N} \setminus K}(P_{K}^{-1}) \right)\\ D_{x_{k_{0}}}\left( P_{i_{1}}P_{K}^{-1} \right)&\cdots&D_{x_{k_{n}}}\left( P_{i_{1}}P_{K}^{-1} \right)\\ \vdots&\ddots&\vdots\\ D_{x_{k_{0}}}\left( P_{i_{n}}P_{K}^{-1} \right)&\cdots&D_{x_{k_{n}}}\left( P_{i_{n }}P_{K}^{-1} \right)\end{vmatrix}
 \mathrm{d} \mu_{K}\\
&=\sum_{i=0}^{n}(-1)^i\cdot \int_{P_{K}\left( U \right)}\delta_{k_{i}}\left( f(P_{K}^{-1})\cdot  F_{\mathbb{N} \setminus K}(P_{K}^{-1}) \right)
\\
&\quad\times\begin{vmatrix}
D_{x_{k_{0}}}\left( P_{i_{1}}P_{K}^{-1} \right)&\cdots&\widehat{D_{x_{k_{i}}}\left( P_{i_{1}}P_{K}^{-1} \right)}&\cdots&D_{x_{k_{n}}}\left( P_{i_{1}}P_{K}^{-1} \right)\\ \vdots&\ddots&\vdots&\ddots&\vdots\\ D_{x_{k_{0}}}\left( P_{i_{n}}P_{K}^{-1} \right)&\cdots&\widehat{D_{x_{k_{i}}}\left( P_{i_{n }}P_{K}^{-1} \right)}&\cdots&D_{x_{k_{n}}}\left( P_{i_{n }}P_{K}^{-1} \right)\end{vmatrix}
 \mathrm{d} \mu_{K}.
\end{aligned}
$$
Let $K_n'=\{k_{i_0},k_{i_1},\cdots,k_{i_n}\}, K_n=K\setminus K_n'$ and
$$
\ba{ll}
\ds F\triangleq f(P_{K}^{-1})\cdot  F_{\mathbb{N} \setminus K_n}(P_{K}^{-1})\\[2mm]
\ds\qq
\cdot\sum_{i=0}^{n}(-1)^{i}\cdot \begin{vmatrix}
D_{x_{k_{0}}}\left( P_{i_{1}}P_{K}^{-1} \right)&\cdots&\widehat{D_{x_{k_{i}}}\left( P_{i_{1}}P_{K}^{-1} \right)}&\cdots&D_{x_{k_{n}}}\left( P_{i_{1}}P_{K}^{-1} \right)\\ \vdots&\ddots&\vdots&\ddots&\vdots\\ D_{x_{k_{0}}}\left( P_{i_{n}}P_{K}^{-1} \right)&\cdots&\widehat{D_{x_{k_{i}}}\left( P_{i_{n }}P_{K}^{-1} \right)}&\cdots&D_{x_{k_{n}}}\left( P_{i_{n }}P_{K}^{-1} \right)\end{vmatrix}.
\ea
$$
For $\textbf{x}_{K_n}\in P_{K_n}U$, let $U_{\textbf{x}_{K_n}}\triangleq\{\textbf{x}_{K_n'}:\textbf{x}_{K_n'}\in P_{K_n'}U,\,\textbf{x}_{K_n'}+\textbf{x}_{K_n}\in U\}$,
$$
F_{\textbf{x}_{K_n}}(\textbf{x}_{K_n'})\triangleq F(\textbf{x}_{K_n'}+\textbf{x}_{K_n}),\qquad g_{\textbf{x}_{K_n}}(\textbf{x}_{K_n'})\triangleq g(\textbf{x}_{K_n'}+\textbf{x}_{K_n}),\quad\forall\, \textbf{x}_{K_n'}\in U_{\textbf{x}_{K_n}}.
$$
Then for $i=0,1,\cdots,n$, it holds that
\begin{eqnarray*}
&&\int_{P_KU}D_{x_{k_i}}F(\textbf{x}_{K_n'}+\textbf{x}_{K_n})\,\mathrm{d}\textbf{x}_{K_n'}\mathrm{d}\mu_{K_n}(\textbf{x}_{K_n})\\
&&=\int_{P_{K_n}U}\left(\int_{U_{\textbf{x}_{K_n}}}D_{x_{k_i}}F(\textbf{x}_{K_n'}+\textbf{x}_{K_n})\,\mathrm{d}\textbf{x}_{K_n'}\right)\mathrm{d}\mu_{K_n}(\textbf{x}_{K_n})
\end{eqnarray*}
and by the classic Gauss-Green formula, for $i=0,1,\cdots,n,$ we have
\begin{eqnarray*}
&&\int_{U_{\textbf{x}_{K_n}}}D_{x_{k_i}}F(\textbf{x}_{K_n'}+\textbf{x}_{K_n})\,\mathrm{d}\textbf{x}_{K_n'}\\
&&=\int_{U_{\textbf{x}_{K_n}}\cap \{g_{\textbf{x}_{K_n}}=0\}}\frac{D_{x_{k_i}}g_{\textbf{x}_{K_n}}}{\sqrt{\sum_{j=0}^{n}|D_{x_{k_j}}g_{\textbf{x}_{K_n}}|^2}} \\
&&\q\qq\qq\cdot F(\textbf{x}_{K_n'\setminus\{k_{i_0}\}}+(P_{k_{i_0}}P^{-1}_{K\setminus\{k_{i_0}\}}(\textbf{x}_{K_n'\setminus\{k_{i_0}\}}+\textbf{x}_{K_n}))\textbf{e}_{k_{i_0}}+\textbf{x}_{K_n})
\\
&&\q\qq\qq\cdot\frac{\sqrt{\sum_{j=0}^{n}|D_{x_{k_j}}g_{\textbf{x}_{K_n}}|^2}}{|D_{x_{k_{i_0}}}g_{\textbf{x}_{K_n}}|}\,\mathrm{d}\textbf{x}_{K_n'\setminus\{k_{i_0}\}}\\
&&=\int_{U_{\textbf{x}_{K_n}}\cap \{g_{\textbf{x}_{K_n}}=0\}}\frac{D_{x_{k_i}}g_{\textbf{x}_{K_n}}}{\sqrt{\sum_{j=0}^{n}|D_{x_{k_j}}g_{\textbf{x}_{K_n}}|^2}} \\
&&\q\qq\qq \cdot F(\textbf{x}_{K_n'\setminus\{k_{i_0}\}}+(P_{k_{i_0}}P^{-1}_{K\setminus\{k_{i_0}\}}(\textbf{x}_{K_n'\setminus\{k_{i_0}\}}+\textbf{x}_{K_n}))\textbf{e}_{k_{i_0}}+\textbf{x}_{K_n})\\
&&\q\qq\qq\cdot\frac{\sqrt{\sum_{j=0}^{n}|D_{x_{k_j}}g_{\textbf{x}_{K_n}}|^2}}{1}\,\mathrm{d}\textbf{x}_{K_n'\setminus\{k_{i_0}\}}\\
&&=\int_{U_{\textbf{x}_{K_n}}\cap \{g_{\textbf{x}_{K_n}}=0\}} (D_{x_{k_i}}g_{\textbf{x}_{K_n}} )\\
&&\q\cdot F(\textbf{x}_{K_n'\setminus\{k_{i_0}\}}+(P_{k_{i_0}}P^{-1}_{K\setminus\{k_{i_0}\}}(\textbf{x}_{K_n'\setminus\{k_{i_0}\}}+\textbf{x}_{K_n}))\textbf{e}_{k_{i_0}}+\textbf{x}_{K_n})
\,\mathrm{d}\textbf{x}_{K_n'\setminus\{k_{i_0}\}}.
\end{eqnarray*}
Thus
$$
\begin{aligned}
&\sum_{i=0}^{n}(-1)^i\cdot \int_{P_{K}\left( U \right)}\delta_{k_{i}}\left( f(P_{K}^{-1})\cdot  F_{\mathbb{N} \setminus K}(P_{K}^{-1}) \right)
\\
&\quad\times\begin{vmatrix}
D_{x_{k_{0}}}\left( P_{i_{1}}P_{K}^{-1} \right)&\cdots&\widehat{D_{x_{k_{i}}}\left( P_{i_{1}}P_{K}^{-1} \right)}&\cdots&D_{x_{k_{n}}}\left( P_{i_{1}}P_{K}^{-1} \right)\\ \vdots&\ddots&\vdots&\ddots&\vdots\\ D_{x_{k_{0}}}\left( P_{i_{n}}P_{K}^{-1} \right)&\cdots&\widehat{D_{x_{k_{i}}}\left( P_{i_{n }}P_{K}^{-1} \right)}&\cdots&D_{x_{k_{n}}}\left( P_{i_{n }}P_{K}^{-1} \right)\end{vmatrix}
 \mathrm{d} \mu_{K}\\
&=\int_{P_{K}\left( U \right)} f(P_{K}^{-1})\cdot  F_{\mathbb{N} \setminus K}(P_{K}^{-1})\cdot\sum_{i=0}^{n}(-1)^{i+1}
\\
&\quad\cdot D_{x_{k_{i}}}\begin{vmatrix}
D_{x_{k_{0}}}\left( P_{i_{1}}P_{K}^{-1} \right)&\cdots&\widehat{D_{x_{k_{i}}}\left( P_{i_{1}}P_{K}^{-1} \right)}&\cdots&D_{x_{k_{n}}}\left( P_{i_{1}}P_{K}^{-1} \right)\\ \vdots&\ddots&\vdots&\ddots&\vdots\\ D_{x_{k_{0}}}\left( P_{i_{n}}P_{K}^{-1} \right)&\cdots&\widehat{D_{x_{k_{i}}}\left( P_{i_{n }}P_{K}^{-1} \right)}&\cdots&D_{x_{k_{n}}}\left( P_{i_{n }}P_{K}^{-1} \right)\end{vmatrix}
 \mathrm{d} \mu_{K}\\
&\q+\sum_{i=0}^{n}(-1)^i\cdot \int_{P_{K\setminus\{k_{i_0}\}}\left( \partial S\cap U \right)}  f(P_{K\setminus\{k_{i_0}\}}^{-1})\cdot  F_{\mathbb{N} \setminus (K\setminus\{k_{i_0}\})}(P_{K\setminus\{k_{i_0}\}}^{-1})\cdot D_{x_{k_i}}g
\\
&\quad\times\begin{vmatrix}
D_{x_{k_{0}}}\left( P_{i_{1}}P_{K}^{-1} \right)&\cdots&\widehat{D_{x_{k_{i}}}\left( P_{i_{1}}P_{K}^{-1} \right)}&\cdots&D_{x_{k_{n}}}\left( P_{i_{1}}P_{K}^{-1} \right)\\ \vdots&\ddots&\vdots&\ddots&\vdots\\ D_{x_{k_{0}}}\left( P_{i_{n}}P_{K}^{-1} \right)&\cdots&\widehat{D_{x_{k_{i}}}\left( P_{i_{n }}P_{K}^{-1} \right)}&\cdots&D_{x_{k_{n}}}\left( P_{i_{n }}P_{K}^{-1} \right)\end{vmatrix}
 \mathrm{d} \mu_{K\setminus\{k_{i_0}\}}.
\end{aligned}
$$
Firstly, it is easy to see that
$$
\begin{aligned}
&\sum_{i=0}^{n}(-1)^{i+1}\cdot D_{x_{k_{i}}}\begin{vmatrix}
D_{x_{k_{0}}}\left( P_{i_{1}}P_{K}^{-1} \right)&\cdots&\widehat{D_{x_{k_{i}}}\left( P_{i_{1}}P_{K}^{-1} \right)}&\cdots&D_{x_{k_{n}}}\left( P_{i_{1}}P_{K}^{-1} \right)\\ \vdots&\ddots&\vdots&\ddots&\vdots\\ D_{x_{k_{0}}}\left( P_{i_{n}}P_{K}^{-1} \right)&\cdots&\widehat{D_{x_{k_{i}}}\left( P_{i_{n }}P_{K}^{-1} \right)}&\cdots&D_{x_{k_{n}}}\left( P_{i_{n }}P_{K}^{-1} \right)\end{vmatrix}\\
&=\sum_{i=0}^{n}(-1)^{i+1}\cdot\sum_{j=1}^{n}\begin{vmatrix}
D_{x_{k_{0}}}\left( P_{i_{1}}P_{K}^{-1} \right)&\cdots&\widehat{D_{x_{k_{i}}}\left( P_{i_{1}}P_{K}^{-1} \right)}&\cdots&D_{x_{k_{n}}}\left( P_{i_{1}}P_{K}^{-1} \right)\\
\vdots&\ddots&\vdots&\ddots&\vdots\\
 D_{x_{k_{i}}}D_{x_{k_{0}}}\left( P_{i_{j}}P_{K}^{-1} \right)&\cdots&\widehat{ D_{x_{k_{i}}}D_{x_{k_{i}}}\left( P_{i_{j}}P_{K}^{-1} \right)}&\cdots& D_{x_{k_{i}}}D_{x_{k_{n}}}\left( P_{i_{j}}P_{K}^{-1} \right)\\
 \vdots&\ddots&\vdots&\ddots&\vdots\\
D_{x_{k_{0}}}\left( P_{i_{n}}P_{K}^{-1} \right)&\cdots&\widehat{D_{x_{k_{i}}}\left( P_{i_{n }}P_{K}^{-1} \right)}&\cdots&D_{x_{k_{n}}}\left( P_{i_{n }}P_{K}^{-1} \right)\end{vmatrix}\\
&=\sum_{j=1}^{n}\sum_{i=0}^{n}\sum_{r=0}^{i-1}(-1)^{i+1}\cdot(-1)^{j+r+1}\cdot  D_{x_{k_{i}}}D_{x_{k_{r}}}\left( P_{i_{j}}P_{K}^{-1} \right)\\
&\q\times\begin{vmatrix}
D_{x_{k_{0}}}\left( P_{i_{1}}P_{K}^{-1} \right)&\cdots&\widehat{D_{x_{k_{r}}}\left( P_{i_{1}}P_{K}^{-1} \right)}&\cdots&\widehat{D_{x_{k_{i}}}\left( P_{i_{1}}P_{K}^{-1} \right)}&\cdots&D_{x_{k_{n}}}\left( P_{i_{1}}P_{K}^{-1} \right)\\
\vdots&\ddots&\vdots&\ddots&\vdots\\
\widehat{ D_{x_{k_{0}}}\left( P_{i_{j}}P_{K}^{-1} \right)}&\cdots&\widehat{D_{x_{k_{r}}}\left( P_{i_{j}}P_{K}^{-1} \right)}&\cdots&\widehat{ D_{x_{k_{i}}}\left( P_{i_{j}}P_{K}^{-1} \right)}&\cdots& \widehat{D_{x_{k_{n}}}\left( P_{i_{j}}P_{K}^{-1} \right)}\\
 \vdots&\ddots&\vdots&\ddots&\vdots\\
D_{x_{k_{0}}}\left( P_{i_{n}}P_{K}^{-1} \right)&\cdots&\widehat{D_{x_{k_{r}}}\left( P_{i_{n}}P_{K}^{-1} \right)}&\cdots&\widehat{D_{x_{k_{i}}}\left( P_{i_{n }}P_{K}^{-1} \right)}&\cdots&D_{x_{k_{n}}}\left( P_{i_{n }}P_{K}^{-1} \right)\end{vmatrix}\\
&\q-\sum_{j=1}^{n}\sum_{i=0}^{n}\sum_{r=i+1}^{n}(-1)^{i+1}\cdot(-1)^{j+r+1}\cdot  D_{x_{k_{i}}}D_{x_{k_{r}}}\left( P_{i_{j}}P_{K}^{-1} \right)\\
&\q\times\begin{vmatrix}
D_{x_{k_{0}}}\left( P_{i_{1}}P_{K}^{-1} \right)&\cdots&\widehat{D_{x_{k_{i}}}\left( P_{i_{1}}P_{K}^{-1} \right)}&\cdots&\widehat{D_{x_{k_{r}}}\left( P_{i_{1}}P_{K}^{-1} \right)}&\cdots&D_{x_{k_{n}}}\left( P_{i_{1}}P_{K}^{-1} \right)\\
\vdots&\ddots&\vdots&\ddots&\vdots\\
\widehat{ D_{x_{k_{0}}}\left( P_{i_{j}}P_{K}^{-1} \right)}&\cdots&\widehat{D_{x_{k_{i}}}\left( P_{i_{j}}P_{K}^{-1} \right)}&\cdots&\widehat{ D_{x_{k_{r}}}\left( P_{i_{j}}P_{K}^{-1} \right)}&\cdots& \widehat{D_{x_{k_{n}}}\left( P_{i_{j}}P_{K}^{-1} \right)}\\
 \vdots&\ddots&\vdots&\ddots&\vdots\\
D_{x_{k_{0}}}\left( P_{i_{n}}P_{K}^{-1} \right)&\cdots&\widehat{D_{x_{k_{i}}}\left( P_{i_{n}}P_{K}^{-1} \right)}&\cdots&\widehat{D_{x_{k_{r}}}\left( P_{i_{n }}P_{K}^{-1} \right)}&\cdots&D_{x_{k_{n}}}\left( P_{i_{n }}P_{K}^{-1} \right)\end{vmatrix}\\
&=0.
\end{aligned}
$$

Note that
$$
\begin{aligned}
&\sum_{i=0}^{n}(-1)^i\cdot \int_{P_{K\setminus\{k_{i_0}\}}\left( \partial S\cap U \right)}  f(P_{K\setminus\{k_{i_0}\}}^{-1})\cdot  F_{\mathbb{N} \setminus (K\setminus\{k_{i_0}\})}(P_{K\setminus\{k_{i_0}\}}^{-1})\cdot D_{x_{k_i}}g
\\
&\quad\times\begin{vmatrix}
D_{x_{k_{0}}}\left( P_{i_{1}}P_{K}^{-1} \right)&\cdots&\widehat{D_{x_{k_{i}}}\left( P_{i_{1}}P_{K}^{-1} \right)}&\cdots&D_{x_{k_{n}}}\left( P_{i_{1}}P_{K}^{-1} \right)\\ \vdots&\ddots&\vdots&\ddots&\vdots\\ D_{x_{k_{0}}}\left( P_{i_{n}}P_{K}^{-1} \right)&\cdots&\widehat{D_{x_{k_{i}}}\left( P_{i_{n }}P_{K}^{-1} \right)}&\cdots&D_{x_{k_{n}}}\left( P_{i_{n }}P_{K}^{-1} \right)\end{vmatrix}
 \mathrm{d} \mu_{K\setminus\{k_{i_0}\}}\\
&= (-1)^{i_0}\cdot \int_{P_{K\setminus\{k_{i_0}\}}\left( \partial S\cap U \right)}  f(P_{K\setminus\{k_{i_0}\}}^{-1})\cdot  F_{\mathbb{N} \setminus (K\setminus\{k_{i_0}\})}(P_{K\setminus\{k_{i_0}\}}^{-1})\cdot D_{x_{k_{i_0}}}g\\
&\quad\times\begin{vmatrix}
D_{x_{k_{0}}}\left( P_{i_{1}}P_{K}^{-1} \right)&\cdots&\widehat{D_{x_{k_{i_0}}}\left( P_{i_{1}}P_{K}^{-1} \right)}&\cdots&D_{x_{k_{n}}}\left( P_{i_{1}}P_{K}^{-1} \right)\\ \vdots&\ddots&\vdots&\ddots&\vdots\\ D_{x_{k_{0}}}\left( P_{i_{n}}P_{K}^{-1} \right)&\cdots&\widehat{D_{x_{k_{i_0}}}\left( P_{i_{n }}P_{K}^{-1} \right)}&\cdots&D_{x_{k_{n}}}\left( P_{i_{n }}P_{K}^{-1} \right)\end{vmatrix}
 \mathrm{d} \mu_{K\setminus\{k_{i_0}\}}\\
 &\quad+\sum_{0 \leq i\leq n,\,i\neq i_0}(-1)^i\cdot \int_{P_{K\setminus\{k_{i_0}\}}\left( \partial S\cap U \right)}  f(P_{K\setminus\{k_{i_0}\}}^{-1})\cdot  F_{\mathbb{N} \setminus (K\setminus\{k_{i_0}\})}(P_{K\setminus\{k_{i_0}\}}^{-1})
\\
&\quad\cdot D_{x_{k_i}}g\cdot\begin{vmatrix}
D_{x_{k_{0}}}\left( P_{i_{1}}P_{K}^{-1} \right)&\cdots&\widehat{D_{x_{k_{i}}}\left( P_{i_{1}}P_{K}^{-1} \right)}&\cdots&D_{x_{k_{n}}}\left( P_{i_{1}}P_{K}^{-1} \right)\\ \vdots&\ddots&\vdots&\ddots&\vdots\\ D_{x_{k_{0}}}\left( P_{i_{n}}P_{K}^{-1} \right)&\cdots&\widehat{D_{x_{k_{i}}}\left( P_{i_{n }}P_{K}^{-1} \right)}&\cdots&D_{x_{k_{n}}}\left( P_{i_{n }}P_{K}^{-1} \right)\end{vmatrix}
 \mathrm{d} \mu_{K\setminus\{k_{i_0}\}}\\
&= (-1)^{2i_0-1}\cdot \int_{P_{K\setminus\{k_{i_0}\}}\left( \partial S\cap U \right)}  f(P_{K\setminus\{k_{i_0}\}}^{-1})\cdot  F_{\mathbb{N} \setminus (K\setminus\{k_{i_0}\})}(P_{K\setminus\{k_{i_0}\}}^{-1}) \\
&\quad\times\begin{vmatrix}
D_{x_{k_{0}}}\left( P_{i_{1}}P_{K}^{-1} \right)&\cdots&\widehat{D_{x_{k_{i_0}}}\left( P_{i_{1}}P_{K}^{-1} \right)}&\cdots&D_{x_{k_{n}}}\left( P_{i_{1}}P_{K}^{-1} \right)\\ \vdots&\ddots&\vdots&\ddots&\vdots\\ D_{x_{k_{0}}}\left( P_{i_{n}}P_{K}^{-1} \right)&\cdots&\widehat{D_{x_{k_{i_0}}}\left( P_{i_{n }}P_{K}^{-1} \right)}&\cdots&D_{x_{k_{n}}}\left( P_{i_{n }}P_{K}^{-1} \right)\end{vmatrix}
 \mathrm{d} \mu_{K\setminus\{k_{i_0}\}}\\
 &\quad+\sum_{0 \leq i\leq n,\,i\neq i_0}(-1)^{i+i_0}\cdot \int_{P_{K\setminus\{k_{i_0}\}}\left( \partial S\cap U \right)}  f(P_{K\setminus\{k_{i_0}\}}^{-1})\cdot  F_{\mathbb{N} \setminus (K\setminus\{k_{i_0}\})}(P_{K\setminus\{k_{i_0}\}}^{-1})\cdot D_{x_{k_i}}(P_{k_{i_0}}P^{-1}_{K\setminus\{k_{i_0}\}})
\\
&\quad\times\begin{vmatrix}
D_{x_{k_{0}}}\left( P_{i_{1}}P_{K}^{-1} \right)&\cdots&\widehat{D_{x_{k_{i}}}\left( P_{i_{1}}P_{K}^{-1} \right)}&\cdots&D_{x_{k_{n}}}\left( P_{i_{1}}P_{K}^{-1} \right)\\ \vdots&\ddots&\vdots&\ddots&\vdots\\ D_{x_{k_{0}}}\left( P_{i_{n}}P_{K}^{-1} \right)&\cdots&\widehat{D_{x_{k_{i}}}\left( P_{i_{n }}P_{K}^{-1} \right)}&\cdots&D_{x_{k_{n}}}\left( P_{i_{n }}P_{K}^{-1} \right)\end{vmatrix}
 \mathrm{d} \mu_{K\setminus\{k_{i_0}\}}\\
& =(-1)^{i_0-1}\int_{P_{K\setminus\{k_{i_0}\}}\left( \partial S\cap U \right)}  f(P_{K\setminus\{k_{i_0}\}}^{-1})\cdot  F_{\mathbb{N} \setminus (K\setminus\{k_{i_0}\})}(P_{K\setminus\{k_{i_0}\}}^{-1}) \\
&\quad\times\begin{vmatrix}
-D_{x_{k_{0}}}(P_{k_{i_0}}P^{-1}_{K\setminus\{k_{i_0}\}})&\cdots& 1 &\cdots&-D_{x_{k_{n}}}(P_{k_{i_0}}P^{-1}_{K\setminus\{k_{i_0}\}})\\
D_{x_{k_{0}}}\left( P_{i_{1}}P_{K}^{-1} \right)&\cdots& D_{x_{k_{i_0}}}\left( P_{i_{1}}P_{K}^{-1} \right) &\cdots&D_{x_{k_{n}}}\left( P_{i_{1}}P_{K}^{-1} \right)\\ \vdots&\ddots&\vdots&\ddots&\vdots\\ D_{x_{k_{0}}}\left( P_{i_{n}}P_{K}^{-1} \right)&\cdots& D_{x_{k_{i_0}}}\left( P_{i_{n }}P_{K}^{-1} \right)&\cdots&D_{x_{k_{n}}}\left( P_{i_{n }}P_{K}^{-1} \right)\end{vmatrix}
 \mathrm{d} \mu_{K\setminus\{k_{i_0}\}}.
\end{aligned}
$$
If $k_{i_0}\notin\{i_1,\cdots,i_n\}$, then
$$
\begin{aligned}
&\begin{vmatrix}
-D_{x_{k_{0}}}(P_{k_{i_0}}P^{-1}_{K\setminus\{k_{i_0}\}})&\cdots& 1 &\cdots&-D_{x_{k_{n}}}(P_{k_{i_0}}P^{-1}_{K\setminus\{k_{i_0}\}})\\
D_{x_{k_{0}}}\left( P_{i_{1}}P_{K}^{-1} \right)&\cdots& D_{x_{k_{i_0}}}\left( P_{i_{1}}P_{K}^{-1} \right) &\cdots&D_{x_{k_{n}}}\left( P_{i_{1}}P_{K}^{-1} \right)\\ \vdots&\ddots&\vdots&\ddots&\vdots\\ D_{x_{k_{0}}}\left( P_{i_{n}}P_{K}^{-1} \right)&\cdots& D_{x_{k_{i_0}}}\left( P_{i_{n }}P_{K}^{-1} \right)&\cdots&D_{x_{k_{n}}}\left( P_{i_{n }}P_{K}^{-1} \right)
\end{vmatrix}\\
&=\begin{vmatrix}
-D_{x_{k_{0}}}(P_{k_{i_0}}P^{-1}_{K\setminus\{k_{i_0}\}})&\cdots& 1 &\cdots&-D_{x_{k_{n}}}(P_{k_{i_0}}P^{-1}_{K\setminus\{k_{i_0}\}})\\
D_{x_{k_{0}}}\left( P_{i_{1}}P_{K}^{-1} \right)&\cdots& 0 &\cdots&D_{x_{k_{n}}}\left( P_{i_{1}}P_{K}^{-1} \right)\\ \vdots&\ddots&\vdots&\ddots&\vdots\\ D_{x_{k_{0}}}\left( P_{i_{n}}P_{K}^{-1} \right)&\cdots& 0&\cdots&D_{x_{k_{n}}}\left( P_{i_{n }}P_{K}^{-1} \right)
\end{vmatrix}\\
&=\begin{vmatrix}
0&\cdots& 1 &\cdots&0\\
D_{x_{k_{0}}}\left( P_{i_{1}}P_{K}^{-1} \right)&\cdots& 0 &\cdots&D_{x_{k_{n}}}\left( P_{i_{1}}P_{K}^{-1} \right)\\ \vdots&\ddots&\vdots&\ddots&\vdots\\ D_{x_{k_{0}}}\left( P_{i_{n}}P_{K}^{-1} \right)&\cdots& 0&\cdots&D_{x_{k_{n}}}\left( P_{i_{n }}P_{K}^{-1} \right)
\end{vmatrix}\\
&=\begin{vmatrix}
0&\cdots& 1 &\cdots&0\\
D_{x_{k_{0}}}\left( P_{i_{1}}P_{K\setminus\{k_{i_0}\}}^{-1} \right)&\cdots& 0 &\cdots&D_{x_{k_{n}}}\left( P_{i_{1}}P_{K\setminus\{k_{i_0}\}}^{-1} \right)\\ \vdots&\ddots&\vdots&\ddots&\vdots\\ D_{x_{k_{0}}}\left( P_{i_{n}}P_{K\setminus\{k_{i_0}\}}^{-1} \right)&\cdots& 0&\cdots&D_{x_{k_{n}}}\left( P_{i_{n }}P_{K\setminus\{k_{i_0}\}}^{-1} \right)
\end{vmatrix}\\
&=(-1)^{i_0}\begin{vmatrix}
D_{x_{k_{0}}}\left( P_{i_{1}}P_{K\setminus\{k_{i_0}\}}^{-1} \right)&\cdots& \widehat{D_{x_{k_{i_0}}}\left( P_{i_{1 }}P_{K\setminus\{k_{i_0}\}}^{-1} \right)} &\cdots&D_{x_{k_{n}}}\left( P_{i_{1}}P_{K\setminus\{k_{i_0}\}}^{-1} \right)\\ \vdots&\ddots&\vdots&\ddots&\vdots\\ D_{x_{k_{0}}}\left( P_{i_{n}}P_{K\setminus\{k_{i_0}\}}^{-1} \right)&\cdots& \widehat{D_{x_{k_{i_0}}}\left( P_{i_{n }}P_{K\setminus\{k_{i_0}\}}^{-1} \right)} &\cdots&D_{x_{k_{n}}}\left( P_{i_{n }}P_{K\setminus\{k_{i_0}\}}^{-1} \right)
\end{vmatrix}\\
&= (-1)^{i_0}\cdot (-1)^{s(I_0)}\cdot \det(D(P_{\overline{I_0}}P_{K\setminus\{k_{i_0}\}}^{-1})).
\end{aligned}
$$
If $k_{i_0}\in\{i_1,\cdots,i_n\}$, then $k_{i_0}=i_k$ for some $k\in\{1,\cdots,n\}$ and
$$
\begin{aligned}
&\begin{vmatrix}
-D_{x_{k_{0}}}(P_{k_{i_0}}P^{-1}_{K\setminus\{k_{i_0}\}})&\cdots& 1 &\cdots&-D_{x_{k_{n}}}(P_{k_{i_0}}P^{-1}_{K\setminus\{k_{i_0}\}})\\
D_{x_{k_{0}}}\left( P_{i_{1}}P_{K}^{-1} \right)&\cdots& D_{x_{k_{i_0}}}\left( P_{i_{1}}P_{K}^{-1} \right) &\cdots&D_{x_{k_{n}}}\left( P_{i_{1}}P_{K}^{-1} \right)\\ \vdots&\ddots&\vdots&\ddots&\vdots\\ D_{x_{k_{0}}}\left( P_{i_{n}}P_{K}^{-1} \right)&\cdots& D_{x_{k_{i_0}}}\left( P_{i_{n }}P_{K}^{-1} \right)&\cdots&D_{x_{k_{n}}}\left( P_{i_{n }}P_{K}^{-1} \right)
\end{vmatrix}\\
&=\begin{vmatrix}
-D_{x_{k_{0}}}(P_{k_{i_0}}P^{-1}_{K\setminus\{k_{i_0}\}})&\cdots& 1 &\cdots&-D_{x_{k_{n}}}(P_{k_{i_0}}P^{-1}_{K\setminus\{k_{i_0}\}})\\
\vdots&\ddots&\vdots&\ddots&\vdots\\
D_{x_{k_{0}}}\left( P_{i_{k}}P_{K}^{-1} \right)&\cdots& D_{x_{k_{i_0}}}\left( P_{i_{k}}P_{K}^{-1} \right) &\cdots&D_{x_{k_{n}}}\left( P_{i_{k}}P_{K}^{-1} \right)\\ \vdots&\ddots&\vdots&\ddots&\vdots\\ D_{x_{k_{0}}}\left( P_{i_{n}}P_{K}^{-1} \right)&\cdots& D_{x_{k_{i_0}}}\left( P_{i_{n }}P_{K}^{-1} \right)&\cdots&D_{x_{k_{n}}}\left( P_{i_{n }}P_{K}^{-1} \right)
\end{vmatrix}\\
&=\begin{vmatrix}
-D_{x_{k_{0}}}(P_{k_{i_0}}P^{-1}_{K\setminus\{k_{i_0}\}})&\cdots& 1 &\cdots&-D_{x_{k_{n}}}(P_{k_{i_0}}P^{-1}_{K\setminus\{k_{i_0}\}})\\
\vdots&\ddots&\vdots&\ddots&\vdots\\
0&\cdots& 1 &\cdots&0\\ \vdots&\ddots&\vdots&\ddots&\vdots\\ D_{x_{k_{0}}}\left( P_{i_{n}}P_{K}^{-1} \right)&\cdots& D_{x_{k_{i_0}}}\left( P_{i_{n }}P_{K}^{-1} \right)&\cdots&D_{x_{k_{n}}}\left( P_{i_{n }}P_{K}^{-1} \right)
\end{vmatrix}\\
&=-\begin{vmatrix}
0&\cdots& 1 &\cdots&0\\
\vdots&\ddots&\vdots&\ddots&\vdots\\
-D_{x_{k_{0}}}(P_{k_{i_0}}P^{-1}_{K\setminus\{k_{i_0}\}})&\cdots& 1 &\cdots&-D_{x_{k_{n}}}(P_{k_{i_0}}P^{-1}_{K\setminus\{k_{i_0}\}})\\ \vdots&\ddots&\vdots&\ddots&\vdots\\ D_{x_{k_{0}}}\left( P_{i_{n}}P_{K}^{-1} \right)&\cdots& D_{x_{k_{i_0}}}\left( P_{i_{n }}P_{K}^{-1} \right)&\cdots&D_{x_{k_{n}}}\left( P_{i_{n }}P_{K}^{-1} \right)
\end{vmatrix}\\
&=\begin{vmatrix}
0&\cdots& 1 &\cdots&0\\
\vdots&\ddots&\vdots&\ddots&\vdots\\
D_{x_{k_{0}}}(P_{k_{i_0}}P^{-1}_{K\setminus\{k_{i_0}\}})&\cdots& -1 &\cdots&D_{x_{k_{n}}}(P_{k_{i_0}}P^{-1}_{K\setminus\{k_{i_0}\}})\\ \vdots&\ddots&\vdots&\ddots&\vdots\\ D_{x_{k_{0}}}\left( P_{i_{n}}P_{K}^{-1} \right)&\cdots& D_{x_{k_{i_0}}}\left( P_{i_{n }}P_{K}^{-1} \right)&\cdots&D_{x_{k_{n}}}\left( P_{i_{n }}P_{K}^{-1} \right)
\end{vmatrix}\\
&=(-1)^{i_0}\begin{vmatrix}
D_{x_{k_{0}}}\left( P_{i_{1}}P_{K\setminus\{k_{i_0}\}}^{-1} \right)&\cdots& \widehat{D_{x_{k_{i_0}}}\left( P_{i_{1 }}P_{K\setminus\{k_{i_0}\}}^{-1} \right)} &\cdots&D_{x_{k_{n}}}\left( P_{i_{1}}P_{K\setminus\{k_{i_0}\}}^{-1} \right)\\ \vdots&\ddots&\vdots&\ddots&\vdots\\ D_{x_{k_{0}}}\left( P_{i_{n}}P_{K\setminus\{k_{i_0}\}}^{-1} \right)&\cdots& \widehat{D_{x_{k_{i_0}}}\left( P_{i_{n }}P_{K\setminus\{k_{i_0}\}}^{-1} \right)} &\cdots&D_{x_{k_{n}}}\left( P_{i_{n }}P_{K\setminus\{k_{i_0}\}}^{-1} \right)
\end{vmatrix}\\
&=(-1)^{i_0}\cdot (-1)^{s(I_0)}\cdot \det(D(P_{\overline{I_0}}P_{K\setminus\{k_{i_0}\}}^{-1})).
\end{aligned}
$$
Thus
$$
\begin{aligned}
&\sum_{i=0}^{n}(-1)^i\cdot \int_{P_{K\setminus\{k_{i_0}\}}\left( \partial S\cap U \right)}  f(P_{K\setminus\{k_{i_0}\}}^{-1})\cdot  F_{\mathbb{N} \setminus (K\setminus\{k_{i_0}\})}(P_{K\setminus\{k_{i_0}\}}^{-1})\cdot D_{x_{k_i}}g
\\
&\quad\times\begin{vmatrix}
D_{x_{k_{0}}}\left( P_{i_{1}}P_{K}^{-1} \right)&\cdots&\widehat{D_{x_{k_{i}}}\left( P_{i_{1}}P_{K}^{-1} \right)}&\cdots&D_{x_{k_{n}}}\left( P_{i_{1}}P_{K}^{-1} \right)\\ \vdots&\ddots&\vdots&\ddots&\vdots\\ D_{x_{k_{0}}}\left( P_{i_{n}}P_{K}^{-1} \right)&\cdots&\widehat{D_{x_{k_{i}}}\left( P_{i_{n }}P_{K}^{-1} \right)}&\cdots&D_{x_{k_{n}}}\left( P_{i_{n }}P_{K}^{-1} \right)\end{vmatrix}
 \mathrm{d} \mu_{K\setminus\{k_{i_0}\}}\\
& =(-1)^{i_0-1}\int_{P_{K\setminus\{k_{i_0}\}}\left( \partial S\cap U \right)}  f(P_{K\setminus\{k_{i_0}\}}^{-1})\cdot  F_{\mathbb{N} \setminus (K\setminus\{k_{i_0}\})}(P_{K\setminus\{k_{i_0}\}}^{-1})\cdot (-1)^{i_0}\\
 &\q\qq\qq\qq\cdot (-1)^{s(I_0)}\cdot \det(D(P_{\overline{I_0}}P_{K\setminus\{k_{i_0}\}}^{-1}))
 \mathrm{d} \mu_{K\setminus\{k_{i_0}\}}\\
& =-\int_{P_{K\setminus\{k_{i_0}\}}\left( \partial S\cap U \right)}  f(P_{K\setminus\{k_{i_0}\}}^{-1})\cdot  \frac{(-1)^{s(I_0)}\cdot \det(D(P_{\overline{I_0}}P_{K\setminus\{k_{i_0}\}}^{-1}))}{n_{K\setminus\{k_{i_0}\}}}\\
 &\q\qq\qq\qq\cdot  F_{\mathbb{N} \setminus (K\setminus\{k_{i_0}\})}(P_{K\setminus\{k_{i_0}\}}^{-1})\cdot n_{K\setminus\{k_{i_0}\}}
 \mathrm{d} \mu_{K\setminus\{k_{i_0}\}}\\
& =-\int_{ \partial S\cap U  }  f \cdot    \textbf{n}_{I_0}^{\partial S}
 \mathrm{d} \mu_{\partial S}.
\end{aligned}
$$
Therefore,
$$
\begin{aligned}
-\sum\limits_{i=1}^{\infty} \int_{U\cap S}  \delta_{i} f\left( \textbf{x} \right) \cdot \textbf{n}_{\left( i,I_0 \right)}^{S}\left( \textbf{x} \right) \mathrm{d} \mu_{S} \left( \textbf{x} \right)=\int_{ \partial S\cap U  }  f \cdot    \textbf{n}_{I_0}^{\partial S}
 \mathrm{d} \mu_{\partial S}.
\end{aligned}
$$
The proof of the second half of Proposition \ref{20241102prop1} is similar and simpler, and therefore we omit the details. This completes the proof of Proposition \ref{20241102prop1}.
\end{proof}

\subsection{Global Version of Gauss-Green Type Theorem}
We begin with the following notion of partition of unity on $\ell^2$.
\begin{definition}\label{230415def1}
A sequence of real-valued functions $\{\gamma_n\}_{n=1}^{\infty}\subset C^1(\ell^2; [0,1])$ is called a $C^1$-partition of unity on $\ell^2$, if the following conditions hold:
\begin{itemize}
\item[$(1)$]  For each $\textbf{x}\in \ell^2$, there exists a non-empty open neighborhood $U$ of $\textbf{x}$ such that only finite many functions in $\{\gamma_n\}_{n=1}^{\infty}$  are non-zero on $U;$
\item[$(2)$] $\sum\limits_{n=1}^{\infty}\gamma_n\equiv1$ on $\ell^2$.
\end{itemize}
\end{definition}

The existence of the $C^1$-partition of unity on $\ell^2$ is based on the following simple separating property.
\begin{lemma}\label{230708lem1}
For any $\textbf{x}\in \ell^2$ and positive numbers $r_1$ and $r_2$ with $r_1<r_2$, there exists $f\in C^1(\ell^2; [0,1])$ such that
$f\equiv 1$ on $B_{r_1}(\textbf{x})$ and $f\equiv 0$ outside $B_{r_2}(\textbf{x})$.
\end{lemma}
\begin{proof}
Choosing $h\in C^{\infty}(\mathbb{R};[0,1])$ such that $h(t)=1$ for all $t\leqslant r_1^2$ and $h(t)=0$ for all $t\geqslant r_2^2$. Note that $||\cdot-\textbf{x}||_{\ell^2}^2\in  C^1(\ell^2; \mathbb{R})$. Let $f(\textbf{y})\triangleq h(|| \textbf{y}-\textbf{x}||_{\ell^2}^2),\,\forall\,\textbf{y}\in \ell^2.$  Then it holds that $f\in C^1(\ell^2; [0,1])$, $f\equiv 1$ on $B_{r_1}(\textbf{x})$ and $f\equiv 0$ outside $B_{r_2}(\textbf{x})$. This completes the proof of Lemma \ref{230708lem1}.
\end{proof}

The following lemma must be a known result but we do not find an exact reference.
\begin{lemma}\label{second lemma for PU l^2}
 For any open covering $\{U_{\alpha}\}_{\alpha\in A}$ of $\ell^2$, there exists four countable locally finite open coverings $\{V_i^1\}_{i=1}^{\infty},\{V_i^2\}_{i=1}^{\infty},\{V_i^3\}_{i=1}^{\infty}$ and $\{V_i^4\}_{i=1}^{\infty}$ of $\ell^2$, and a sequence of functions $\{g_i\}_{i=1}^{\infty}\subset C^1(\ell^2;[0,1])$ such that
\begin{itemize}
\item[$(1)$] For each $i\in \mathbb{N}$, $\overline{V_i^1}\subset V_i^2\subset\overline{V_i^2}\subset V_i^3\subset\overline{V_i^3}\subset V_i^4$, $g_i\equiv1$ on $\overline{V_i^1}$ and supp$g_i\subset V_i^3$;
\item[$(2)$]  $\{V_i^4\}_{i=1}^{\infty}$ is a refinement of $\{U_{\alpha}\}_{\alpha\in A}$.
\end{itemize}
\end{lemma}
\begin{proof}
Since $\{U_{\alpha}\}_{\alpha\in A}$ is an open covering of $\ell^2$, for each $\textbf{y}\in \ell^2$, we may find some $\alpha\in A$ so that $\textbf{y}\in U_{\alpha}$. Hence, $\textbf{y}\in B_{r}(\textbf{y})\subset B_{r'}(\textbf{y})\subset U_{\alpha}$ for some real numbers $r$ and $r'$ with $0<r<r'<\infty$.
By Lemma \ref{230708lem1}, there exists $\phi^{\textbf{y}}\in C^1_{\ell^2}(\ell^2; [0,1])$ such that
$\phi^{\textbf{y}}\equiv 1$ on $B_{r_1}(\textbf{y})$ and $\phi^{\textbf{y}}\equiv 0$ outside $B_{r_2}(\textbf{y})$.
Hence,  $\phi^{\textbf{y}}(\textbf{y})=1$ and supp$\phi^{\textbf{y}}\subset U_{\alpha}$. Write
$$A_{\textbf{y}}\triangleq \left\{{\textbf{x}}\in \ell^2:\;\phi^{\textbf{y}}({\textbf{x}})>\frac{1}{2}\right\}.$$
Then $\{A_{\textbf{y}}:\;{\textbf{y}}\in \ell^2\}$ is an open covering of $\ell^2$. Since $\ell^2$ is a Lindel\"{o}ff space, there is a countable subset $\{A_{{\textbf{y}}_i}\}_{i=1}^{\infty}$ (of $\{A_{\textbf{y}}:\;\textbf{y}\in \ell^2\}$) which covers $\ell^2$.

Choosing $f_1\in C^{\infty}(\mathbb{R};[0,1])$ such that $f_1(t)=1$ for all $ t\geqslant\frac{1}{2}$, and $f_1(t)=0$ for all $ t\leqslant 0$. For $j\geqslant 2$, we can find $f_j\in C^{\infty}(\mathbb{R}^j;[0,1])$ such that:
\begin{itemize}
\item[$(a)$]  $f_j(t_1,\cdots,t_j)=1$, if $t_j\geqslant \frac{1}{2}$ and $t_i\leqslant \frac{1}{2}+\frac{1}{j}$ for all $i\in\{1,\cdots,j-1\}$;
\item[$(b)$]  $f_j(t_1,\cdots,t_j)=0$, if $t_j\leqslant \frac{1}{2}-\frac{1}{j}$ or $t_i\geqslant \frac{1}{2}+\frac{2}{j}$ for some $i\in\{1,\cdots,j-1\}$.
\end{itemize}
Now let $\psi_1({\textbf{x}})\triangleq f_1(\phi^{\textbf{y}_1}({\textbf{x}}))$ and $\psi_j({\textbf{x}})\triangleq f_j(\phi^{\textbf{y}_1}({\textbf{x}}),\cdots,\phi^{\textbf{y}_j}({\textbf{x}}))$ if $j\geqslant 2$, $\forall\, \textbf{x}\in\ell^2$. For $i\in \mathbb{N}$, define
\begin{eqnarray*}
V_i^k&\triangleq&\bigg\{{\textbf{x}}\in \ell^2:\;\psi_i({\textbf{x}})>\frac{4-k}{4}\bigg\},\qquad
k=1,2,3,4.
\end{eqnarray*}
Clearly, $\overline{V_i^1}\subset V_i^2\subset\overline{V_i^2}\subset V_i^3\subset\overline{V_i^3}\subset V_i^4$. Since for each $i\in \mathbb{N}$, $V_i^4\subset \text{supp}\,\phi^{{\textbf{y}}_i}$, $\{V_i^4\}_{i=1}^{\infty}$ is a refinement of $\{U_{\alpha}\}_{\alpha\in A}$.

Since $\{A_{{\textbf{y}}_i}\}_{i=1}^{\infty}$ covers $\ell^2$, for any ${\textbf{x}}\in \ell^2$ we denote by $i({\textbf{x}})$ the smallest positive integer $i$ such that $\phi^{{\textbf{y}}_i}({\textbf{x}})\geqslant \frac{1}{2}$. Then, $\psi_{i({\textbf{x}})}({\textbf{x}})=1$ and ${\textbf{x}}\in V_{i({\textbf{x}})}^1$. Hence, $\{V_i^1\}_{i=1}^{\infty}$ is an open covering of $\ell^2$.

For each ${\textbf{x}}\in \ell^2$, one can find a positive integer $n({\textbf{x}})$ such that $\phi^{y_{n({\textbf{x}})}}({\textbf{x}})>\frac{1}{2}$. By the continuity of $\phi^{y_{n({\textbf{x}})}}(\cdot)$, there exists a neighborhood ${\cal N}_{\textbf{x}}$ of ${\textbf{x}}$ and a real number $a_{\textbf{x}}>\frac{1}{2}$ such that
$\inf_{{\textbf{x}}'\in {\cal N}_{\textbf{x}}}\phi^{y_{n({\textbf{x}})}}({\textbf{x}}')\geqslant a_{\textbf{x}}.$
Choosing a sufficiently large $k\in \mathbb{N}$ such that $a_{\textbf{x}}>\frac{1}{2}+\frac{2}{k}$. Hence $\phi^{y_{n({\textbf{x}})}}({\textbf{x}}')\geqslant \frac{1}{2}+\frac{2}{k},\,$ for all $ {\textbf{x}}'\in {\cal N}_{\textbf{x}}.$
Therefore, for any ${\textbf{x}}'\in {\cal N}_{\textbf{x}}$ and $j\geqslant k+n({\textbf{x}})$, we have $\psi_j({\textbf{x}}')=0$. Thus for $j\geqslant k+n({\textbf{x}})$, it holds that ${\cal N}_{\textbf{x}}\cap V_j^4=\emptyset$ and hence $\{V_i^4\}_{i=1}^{\infty}$ is locally finite.

Choosing $h\in  C^{\infty}(\mathbb{R};[0,1])$ satisfying $h(t)=0$ for all $ t\leqslant \frac{1}{2}$, and $h(t)=1$ for all $t\geqslant \frac{3}{4}$.  Now we define $g_i({\textbf{x}})\triangleq h(\psi_i({\textbf{x}}))$ for all $\,i\in\mathbb{N}$ and ${\textbf{x}}\in \ell^2.$
Then $\{g_i\}_{i=1}^{\infty}$ is the desired sequence of functions so that $\{g_i\}_{i=1}^{\infty}\subset C^1_{\ell^2}(\ell^2;[0,1])$, $g_i\equiv1$ on $\overline{V_i^1}$ and supp$g_i\subset V_i^3$. This completes the proof of Lemma \ref{second lemma for PU l^2}.
\end{proof}

Now we are ready to prove the existence of $C^1$-partition of unity on $\ell^2$.
\begin{proposition}\label{H-C1 parition of unity l^2}
Suppose that $\{U_{\alpha}\}_{\alpha\in \mathscr{A}}$ is an open covering of $\ell^2$. Then there exists a $C^1$-partition of unity $\{\gamma_n\}_{n=1}^{\infty}$ subordinated to $\{U_{\alpha}\}_{\alpha\in A}$, i.e., for each $n\in\mathbb{N}$, supp $\gamma_n\subset U_{\alpha}$ for some $\alpha\in \mathscr{A}$.
\end{proposition}
\begin{proof}
We find four open coverings $\{V_i^j\}_{i=1}^{\infty}\; (j=1,2,3,4)$ (which are refinements of $\{U_{\alpha}\}_{\alpha\in A}$) and a sequence of functions $\{g_i\}_{i=1}^{\infty}\subset C^1(\ell^2;[0,1])$ as that in Lemma \ref{second lemma for PU l^2}. Let $\gamma_1\triangleq g_1$ and $\gamma_i\triangleq g_i\Pi_{j=1}^{i-1}(1-g_j)$ for $i\geqslant 2.$
Then for each $n\in\mathbb{N}$, $\gamma_n\in C^1_{\ell^2}(\ell^2;[0,1])$ and supp $\gamma_n\subset $supp $g_n\subset V_n^3$. Recall that $\{V_i^3\}_{i=1}^{\infty}$ is a locally finite open covering of $\ell^2$. Thus for each point in $\ell^2$, there exists an open neighborhood of that point so that only finite many $\gamma_i$'s are non-zero on this neighborhood.  On the other hand, since
$V_i^1\subset \{\textbf{y}\in \ell^2:\;g_i(\textbf{y})=1\}$ for each $i\in\mathbb{N}$,
and $\{V_i^1\}_{i=1}^{\infty}$ covers $\ell^2$, for each $\textbf{x}\in \ell^2$, there exists $n\in \mathbb{N}$ such that
$\prod_{i=1}^{n}(1-g_i(\textbf{x}))=0.$
Hence
\begin{eqnarray*}
\sum_{i=1}^{\infty}\gamma_i(\textbf{x})=\lim_{n\to\infty}\sum_{i=1}^n\gamma_i(\textbf{x})=1-\lim_{n\to\infty}\prod_{i=1}^{n}(1-g_i(\textbf{x}))=1,
\end{eqnarray*}
which implies that $\{\gamma_n\}_{n=1}^{\infty}$ is a $C^1$-partition of unity subordinated to $\{U_{\alpha}\}_{\alpha\in A}$. The proof of Proposition \ref{H-C1 parition of unity l^2} is completed.
\end{proof}
\begin{remark}
From the above proof of the existence of $C^1$-partition of unity, it is easy to see that the smooth condition ``$\{\gamma_n\}_{n=1}^{\infty} \subset C^1(\ell^2; [0,1]) $" in Definition \ref{230415def1} can be changed into other smooth conditions for which the separating property in Lemma \ref{230708lem1} holds and the corresponding smooth functions are closed under composition with finite dimensional smooth functions. Then the corresponding partition of unity also exists.
\end{remark}
\begin{theorem}\label{20241015thm1}
(\textbf{Global Version of Gauss-Green Type Theorem}). We adapt the same assumptions and notations as in Definition \ref{20241009for1wb}, suppose that $I_0\in S_{\mathbb{N}}^F$ satisfies the conditions in Proposition \ref{20241102prop1}, the function
$$
\sqrt{ \sum\limits_{i=1}^{\infty} \frac{|\textbf{n}_{\left( i,I_0 \right)}^{S} |^2}{a_i^2}}
$$
is a real-valued continuous function on $S$, $O$ is an open subset of $\ell^2$ so that $S\cup \partial S\subset O$, and $f$ is a real-valued Borel measurable function on $O$. If $f$ is Fr\'{e}chet differentiable respect to the $H_{ \mathbb{N}}^2$-direction, $f\in C_F^1(O)$, the function
$$
 \sum\limits_{i=1}^{\infty} \left( |D_{x_i}f  |+\left|\frac{x_i}{a_i^2}\cdot f \right| \right) \cdot |\textbf{n}_{\left( i,I_0 \right)}^{S} |
$$
is integrable on $S$ with respect to the measure $\mu_S$ and $f \cdot    \textbf{n}_{I_0}^{\partial S}$ is integrable on $\partial S$ with respect to the surface measure $\mu_{\partial S}$, then
$$
\begin{aligned}
-\sum\limits_{i=1}^{\infty} \int_{S}  \delta_{i} f\left( \textbf{x} \right) \cdot \textbf{n}_{\left( i,I_0 \right)}^{S}\left( \textbf{x} \right) \mathrm{d} \mu_{S} \left( \textbf{x} \right)= \int_{ \partial S }  f \cdot    \textbf{n}_{I_0}^{\partial S}
 \mathrm{d} \mu_{\partial S}.
\end{aligned}
$$
\end{theorem}
\begin{proof}
Firstly, we consider the case that $\supp f$ is a compact subset of $\ell^2$. By
Proposition \ref{20241102prop1}, for each $\textbf{x}^0\in\ell^2$, there exists an open neighborhood $U_{\textbf{x}^0}$ of $\textbf{x}^0$ such that if supp$f\subset U_{\textbf{x}^0}$, then it holds that
$$
\begin{aligned}
-\sum\limits_{i=1}^{\infty} \int_{S\cap U_{\textbf{x}^0}}  \delta_{i} f\left( \textbf{x} \right) \cdot \textbf{n}_{\left( i,I_0 \right)}^{S}\left( \textbf{x} \right) \mathrm{d} \mu_{S} \left( \textbf{x} \right)= \int_{ (\partial S)\cap U_{\textbf{x}^0} }  f \cdot    \textbf{n}_{I_0}^{\partial S}
 \mathrm{d} \mu_{\partial S}.
\end{aligned}
$$
Since $\{U_{\textbf{x}^0}   : \textbf{x}^0\in\ell^2  \}$ is a family of open covering of $\ell^2$, by Proposition \ref{H-C1 parition of unity l^2}, there exists a partition of unity $\{\gamma_n\}_{n=1}^{\infty}$ subordinated to $\{U_{\textbf{x}^0}   : \textbf{x}^0\in\ell^2 \}$. Therefore, for each $k\in\mathbb{N}$, it holds that
$$
\begin{aligned}
-\sum\limits_{i=1}^{\infty} \int_{S }  \delta_{i} (f\cdot \gamma_k)\left( \textbf{x} \right) \cdot \textbf{n}_{\left( i,I_0 \right)}^{S}\left( \textbf{x} \right) \mathrm{d} \mu_{S} \left( \textbf{x} \right)= \int_{  \partial S }  f\cdot \gamma_k \cdot    \textbf{n}_{I_0}^{\partial S}
 \mathrm{d} \mu_{\partial S}.
\end{aligned}
$$
Since supp$f$ is a compact subset of $\ell^2,$ the formal infinite sum $f=\sum\limits_{k=1}^{\infty}\gamma_k\cdot f$ is a finite non-zero sum, and hence
$$
\begin{aligned}
&- \sum\limits_{i=1}^{\infty} \int_{S }  \delta_{i} f\left( \textbf{x} \right) \cdot \textbf{n}_{\left( i,I_0 \right)}^{S}\left( \textbf{x} \right) \mathrm{d} \mu_{S} \left( \textbf{x} \right)\\
&=-\sum\limits_{i=1}^{\infty} \int_{S }  \delta_{i} \left(\sum_{k=1}^{\infty}f\cdot \gamma_k\right)\left( \textbf{x} \right) \cdot \textbf{n}_{\left( i,I_0 \right)}^{S}\left( \textbf{x} \right) \mathrm{d} \mu_{S} \left( \textbf{x} \right)\\
&=-\sum_{k=1}^{\infty}\sum\limits_{i=1}^{\infty} \int_{S }  \delta_{i} (f\cdot \gamma_k)\left( \textbf{x} \right) \cdot \textbf{n}_{\left( i,I_0 \right)}^{S}\left( \textbf{x} \right) \mathrm{d} \mu_{S} \left( \textbf{x} \right)\\
&= \sum_{k=1}^{\infty}\int_{  \partial S }  f\cdot \gamma_k \cdot    \textbf{n}_{I_0}^{\partial S}
 \mathrm{d} \mu_{\partial S}=  \int_{  \partial S }  f\cdot  \textbf{n}_{I_0}^{\partial S}
 \mathrm{d} \mu_{\partial S}.
\end{aligned}
$$

Secondly, if supp$f$ is not a compact subset of $\ell^2$, then note that for each $k\in\mathbb{N}$, $f\cdot X_k$ satisfies the above conditions and it holds that
\begin{eqnarray*}
&&-\sum\limits_{i=1}^{\infty} \int_{S}  \delta_{i} (f\cdot X_k)  \cdot \textbf{n}_{\left( i,I_0 \right)}^{S} \mathrm{d} \mu_{S}\\
&&=-\sum\limits_{i=1}^{\infty} \int_{S}  X_k\cdot \delta_{i} f  \cdot \textbf{n}_{\left( i,I_0 \right)}^{S} \mathrm{d} \mu_{S}
-\sum\limits_{i=1}^{\infty} \int_{S} f  \cdot (D_{x_i}X_k)\cdot \textbf{n}_{\left( i,I_0 \right)}^{S} \mathrm{d} \mu_{S}\\
&&=\int_{  \partial S }  f\cdot X_k \cdot \textbf{n}_{I_0}^{\partial S}
 \mathrm{d} \mu_{\partial S}.
\end{eqnarray*}
Noting that
$$
\sum\limits_{i=1}^{\infty}  | f|  \cdot |(D_{x_i}X_k)|\cdot |\textbf{n}_{\left( i,I_0 \right)}^{S}|
\leq |f|\cdot \sqrt{\sum\limits_{i=1}^{\infty} a_i^2 |(D_{x_i}X_k)|^2} \cdot \sqrt{\sum\limits_{i=1}^{\infty} \sup_{\supp X_k}\frac{|\textbf{n}_{\left( i,I_0 \right)}^{S}|}{a_i^2}},
$$
combing Theorem \ref{20241013thm1} and the Lebesgue Dominated Convergence Theorem, we have
\begin{eqnarray*}
\lim_{k\to\infty}\sum\limits_{i=1}^{\infty} \int_{S} f  \cdot (D_{x_i}X_k)\cdot \textbf{n}_{\left( i,I_0 \right)}^{S} \mathrm{d} \mu_{S}=0.
\end{eqnarray*}
By the Lebesgue Dominated Convergence Theorem again,  we have
\begin{eqnarray*}
\lim_{k\to\infty}\sum\limits_{i=1}^{\infty} \int_{S}  X_k\cdot \delta_{i} f  \cdot \textbf{n}_{\left( i,I_0 \right)}^{S} \mathrm{d} \mu_{S}&=&\sum\limits_{i=1}^{\infty} \int_{S}  \delta_{i} f  \cdot \textbf{n}_{\left( i,I_0 \right)}^{S} \mathrm{d} \mu_{S},\\
\lim_{k\to\infty}\int_{  \partial S }  f\cdot X_k \cdot \textbf{n}_{I_0}^{\partial S}
 \mathrm{d} \mu_{\partial S}&=&\int_{  \partial S }  f \cdot \textbf{n}_{I_0}^{\partial S}
 \mathrm{d} \mu_{\partial S}.
\end{eqnarray*}
Thus we arrives at
$$
\begin{aligned}
-\sum\limits_{i=1}^{\infty} \int_{S}  \delta_{i} f\left( \textbf{x} \right) \cdot \textbf{n}_{\left( i,I_0 \right)}^{S}\left( \textbf{x} \right) \mathrm{d} \mu_{S} \left( \textbf{x} \right)= \int_{ \partial S }  f \cdot    \textbf{n}_{I_0}^{\partial S}
 \mathrm{d} \mu_{\partial S},
\end{aligned}
$$
which completes the proof of Theorem \ref{20241015thm1}.
\end{proof}

\begin{remark}\label{20241126rem1}
From the above proof, one may see that for the case of surface with finite co-dimension , the smooth assumption that ``$f\in C_F^1(U)$" in  Theorem \ref{20241015thm1} for the function is enough and the assumption that ``$f$ is Fr\'{e}chet differentiable respect to the $H_{ \mathbb{N}}^2$ direction" is redundant. Especially for the case of surface with co-dimension 1 which was considered by Goodman in [Theorem 2, p. 421]\cite{Goo}, the smooth conditions here are weaker than its counterpart in \cite[Theorem 2, p. 421]{Goo}.
\end{remark}

The following is a non-trivial example that satisfies the key assumptions in this paper.
\begin{exa}
Let $I=\{2n-1:n\in\mathbb{N}\}$ and $J=\mathbb{N}\setminus I$. Choosing
$\textbf{x}^0_J=\sum\limits_{j\in J}x_j^0\textbf{e}_j,\,\Delta\textbf{x}^0_J=\sum\limits_{j\in J}(\Delta x_j^0)\textbf{e}_j\in P_J$ such that
\begin{eqnarray*}
\sum_{j\in J}\left|\ln \frac{1}{\sqrt{2\pi} a_j}-\frac{x_j^2}{2a_j^2}\right|<\infty,\qquad\sum_{j\in J}\frac{|\Delta x_j|^2}{a_j^4}<\infty,\qquad \Delta x_j\neq 0,\,\forall\, j\in J.
\end{eqnarray*}
Let
\begin{eqnarray*}
S=\left\{\textbf{x}_I+\textbf{x}^0_J+x_1\cdot\Delta\textbf{x}^0_J:\textbf{x}_I=\sum\limits_{i\in I}x_i\textbf{e}_i\in P_I\right\}.
\end{eqnarray*}
Then $S$ is a surface of $\ell^2$ with co-dimension $\Gamma_J$ and $\{(S,P_I)\}\cup\{(S,P_{(I\setminus\{1\})\cup\{j\}}):j\in J\}$ is a family of coordinate pairing of $S$. It is easy to see that
\begin{eqnarray*}
n_{I}(\textbf{x})=\sqrt{ 1+\sum_{j\in J} \left|\Delta x_j\right|^2}=\sqrt{ 1+ ||\Delta\textbf{x}^0_J||^2},
\quad\forall \, \textbf{x}=\textbf{x}_I+\textbf{x}^0_J+x_1\cdot\Delta\textbf{x}^0_J\in S,
\end{eqnarray*}
and
\begin{eqnarray*}
F_{J}(\textbf{x}^0_J+x_1\cdot\Delta\textbf{x}^0_J)
&=&\prod_{j\in J}\frac{1}{\sqrt{2\pi a_j^2}}e^{-\frac{(x_j^0+x_1\cdot (\Delta x_j^0))^2}{2a_j^2}}\\
&=&e^{\sum\limits_{j\in J}\left(\ln \frac{1}{\sqrt{2\pi} a_j}-\frac{(x_j^0)^2}{2a_j^2}\right) -x_1\cdot \left(\sum\limits_{j\in J}\frac{x_j^0\cdot (\Delta x_j^0)}{a_j^2}\right)-x_1^2\cdot \left(\sum\limits_{j\in J}\frac{(\Delta x_j^0)^2}{a_j^2}\right)}.
\end{eqnarray*}
Thus
$$
\sup_{\textbf{x}_I\in U_{I}} n_{I}(\textbf{x}_I+f(\textbf{x}_I))\cdot F_{\mathbb{N}\setminus I}(f(\textbf{x}_I))<\infty,
$$
for any bounded subset $U_{I}$ of $P_I$.  One can also see that
$$
\sqrt{ \sum\limits_{i=1}^{\infty} \frac{|\textbf{n}_{\left( i,I_0 \right)}^{S} |^2}{a_i^2}}
$$
is a constant function which satisfies the assumption in Theorem \ref{20241015thm1}.
\end{exa}

For the special case of surface with co-dimension 1, as a consequence of Theorem \ref{20241015thm1}, we have the following result.
\begin{corollary}\label{20241127cor1}
Let $V$ be an open subset (of $\ell^2$) with the boundary $\partial V$ being a $C^1$-surface of $\ell^2$ with co-dimension 1, $\widetilde{V}$ be an open subset of $\ell^2$ so that $V\cup \partial V\subset \widetilde{V}$, and $f$ be a function defined on $\widetilde{V}$. For each $i\in\mathbb{N}$, if $D_{x_i}f$ exists in $\widetilde{V}$, $f\in C_F^1(\widetilde{V})\cap B_{\ell^2}(\widetilde{V})$, both $D_{x_i}f$ and $x_i\cdot f$ are integrable on $V$ with respect to the measure $P $, and $f\cdot \textbf{n}_{\mathbb{N}\setminus\{i\}}^{\partial V}$ is integrable on $\partial V$ with respect to the surface measure $\mu_{\partial V}$, then
\begin{eqnarray*}
\int_{V}D_{x_i}f\,\mathrm{d}P
=\int_{V}\frac{x_i\cdot f}{a_i^2} \,\mathrm{d}P +\int_{\partial V}f\cdot \textbf{n}_{\mathbb{N}\setminus\{i\}}^{\partial V}\,\mathrm{d}\mu_{\partial V}.
\end{eqnarray*}
\end{corollary}

As another consequence of Theorem \ref{20241015thm1}, we see that the infinite-dimensional Gauss-Green theorem also holds for open subset (of $\ell^2$) which is covered by two $C^1$-surfaces.
\begin{corollary}\label{20241127cor2}
Let $V$ be an open subset of $\ell^2$ such that $\partial V= S_1\sqcup S_2\sqcup I$, where $S_k$ is an open subset of a $C^1$-surface $S^k$ for
$k=1,2$ and $I=S^1\cap S^2$. For each $\textbf{x}_0\in I$, there exists an open neighborhood $U$ of $\textbf{x}_0$ and $g_1,g_2\in C^1(U)$ such that
\begin{itemize}
\item[$(1)$] $S^i\cap U=\{\textbf{x}\in U:g_i(\textbf{x})=0\}$ for $i=1,2$;
\item[$(2)$] $S_1=\{\textbf{x}\in U:g_1(\textbf{x})=0, g_2(\textbf{x})<0\}$;
\item[$(3)$] $S_2=\{\textbf{x}\in U:g_1(\textbf{x})<0, g_2(\textbf{x})=0\}$;
\item[$(3)$] $I\cap U=\{\textbf{x}\in U:g_1(\textbf{x})=g_2(\textbf{x})=0\}$, $V\cap U=\{\textbf{x}\in U:g_1(\textbf{x})<0, \text{and } g_2(\textbf{x})<0\}$;
\item[$(4)$] There exists $i_0,i_0'\in\mathbb{N}$ such that $i_0\neq i_0'$, $D_{i_0}g_1(\textbf{x})\neq 0$ and $D_{i_0'}g_2(\textbf{x})\neq 0$ for all $\textbf{x}\in U$.
\end{itemize}
Suppose that $\widetilde{V}$ be an open subset of $\ell^2$ so that $V\cup \partial V\subset \widetilde{V}$, and $f$ be a function defined on $\widetilde{V}$. For each $i\in\mathbb{N}$, if $D_{x_i}f$ exists in $\widetilde{V}$, $f\in C_F^1(\widetilde{V})\cap B_{\ell^2}(\widetilde{V})$, both $D_{x_i}f$ and $x_i\cdot f$ are integrable on $V$ with respect to the measure $P$, and $f\cdot \textbf{n}_{\mathbb{N}\setminus\{i\}}^{S^k}$ is integrable on $S_k$ with respect to the surface measure $\mu_{S^k}$ for $k=1,2$, then
\begin{equation}\label{20241127eq1}
\int_{V}D_{x_i}f\,\mathrm{d}P
=\int_{V}\frac{x_i\cdot f}{ a_i^2} \,\mathrm{d}P +\int_{S_1}f\cdot \textbf{n}_{\mathbb{N}\setminus\{i\}}^{S^1}\,\mathrm{d}\mu_{S^1}+\int_{S_2}f\cdot \textbf{n}_{\mathbb{N}\setminus\{i\}}^{S^2}\,\mathrm{d}\mu_{S^2}.
\end{equation}
\end{corollary}

\subsection{Differential Forms and Stokes Type Theorems}
For each $\textbf{x}\in S$, there exists $I_0\in S_{\mathbb{N}}^F$, an open neighborhood $U_{\textbf{x}}$ of $\textbf{x}$ such that $(S\cap U_{\textbf{x}}, P_{\overline{I_0}})$ is a coordinate pairing of $S$, $\overline{I_0}\in \Gamma_I$, we define
\begin{eqnarray*}
\left( \bigwedge_{I_0}\mathrm{d}\textbf{X}_{I_0}\right)(E)\triangleq \int_{P_{\overline{I_0}}E}F_{\mathbb{N}\setminus \overline{I_0}}(P_{\overline{I_0}}^{-1}\textbf{x}_{\overline{I_0}})\,\mathrm{d}\mu_{\overline{I_0}}(\textbf{x}_{\overline{I_0}}),
\end{eqnarray*}
for any Borel subset of $S\cap U_{\textbf{x}}$.
\begin{definition}
Suppose that $I$ is a non-empty subset of $\mathbb{N}$, $S$ is a surface of $\ell^2$ with codimension $\Gamma_{\mathbb{N}\setminus I}$ and $\mu$ is complex Borel measure on $S$. If for each $\textbf{x}\in S$, there exists $I_0\in S_{\mathbb{N}}^F$, an open neighborhood $U_{\textbf{x}}$ of $\textbf{x}$ and a complex Borel measurable function $f$ on $S\cap U_{\textbf{x}}$ such that $(S\cap U_{\textbf{x}}, P_{\overline{I_0}})$ is a coordinate pairing of $S$ and
\begin{equation}\label{202041121for1}
\mu(E)=\int_{P_{\overline{I_0}E}}f( P_{\overline{I_0}}^{-1}\textbf{x}_{\overline{I_0}})\cdot F_{\mathbb{N}\setminus \overline{I_0}}(P_{\overline{I_0}}^{-1}\textbf{x}_{\overline{I_0}})\,\mathrm{d}\mu_{\overline{I_0}}(\textbf{x}_{\overline{I_0}}),\qquad\forall\,
E\in \mathscr {B}(S\cap U_{\textbf{x}}),
\end{equation}
then we say that $\mu$ is a $\Gamma_I$-differential form on $S$. Meanwhile, we denote $\mu(S)$ by $\int_{S}\,\mu$.
\end{definition}

\begin{remark}
We should note that \eqref{202041121for1} means that $\mu|_{S\cap U_{\textbf{x}}}=f\cdot\left( \bigwedge\limits_{I_0}\mathrm{d}\textbf{X}_{I_0}\right)$ which is similar to property of $n$-differential forms on manifolds of dimension $n$ for each $n\in\mathbb{N}$.
\end{remark}
\begin{exa}
The measure $\mu_S$ on $S$ which appeared in Proposition \ref{20241013prop1} is a $\Gamma_I$-differential form on $S$.
\end{exa}
Now we come to define the differential operator $d$ on differential forms. Inspired by Theorem \ref{20241015thm1}, one maybe try to define

\begin{eqnarray*}
d \left(f \cdot    \textbf{n}_{I_0}^{\partial S}
\mu_{\partial S}\right)\triangleq-\sum\limits_{i=1}^{\infty}  \delta_{i} f \cdot \textbf{n}_{\left( i,I_0 \right)}^{S}\mu_{S},
\end{eqnarray*}
but, even in the finite dimensional case, this map is not well-defined, for example, one can see the smooth functions which are 0 on the unit circle and non-zero on the unit open disc. In order to overcome this problem, we will use the tools of equivalent class.
 Suppose that $f_1$ and $f_2$ are two complex-valued Borel measurable functions on $S$ such that $\int_S|f_1|\,\mathrm{d}\mu_S+\int_S|f_2|\,\mathrm{d}\mu_S<+\infty$, if $\int_Sf_1\mu_S=\int_Sf_2\mu_S$, then we write $f_1\mu_S\sim f_2\mu_S.$ It is easy to see that $\sim$ is an equivalent relation and for each complex-valued Borel measurable function $f$ on $S$ such that $\int_S|f|\,\mathrm{d}\mu_S<+\infty$, we denote the collection of all complex-valued Borel measurable function $g$ on $S$ such that $\int_S|g|\,\mathrm{d}\mu_S<+\infty$ and $\int_Sg\mu_S=\int_Sf\mu_S$ by $[f\mu_S]$. By the way, we define $\int_S[f\mu_S]\triangleq \int_Sg\,\mathrm{d}\mu_S$, where $g\mu_S\in [f\mu_S]$. Meanwhile, for each complex Borel measure $\mu$ on $\partial S$, we define $\int_{\partial S}\mu\triangleq\mu(\partial S)$.
\begin{definition}
Suppose that $f$ is a complex-valued function that satisfy the assumptions in Theorem \ref{20241015thm1}, then we define
\begin{eqnarray*}
d \left(f \cdot    \textbf{n}_{I_0}^{\partial S}
\mu_{\partial S}\right)\triangleq -\sum\limits_{i=1}^{\infty}  \delta_{i} f \cdot \textbf{n}_{\left( i,I_0 \right)}^{S}\mu_{S},
\end{eqnarray*}
and one can easily extend it by linearity. Then we denote the definition domain of $d$ by $D_d$.
\end{definition}
As a direct consequence of the definition, one can show the following Stokes type theorem.
\begin{theorem}{\rm (\textbf{Stokes Type Theorem})}
For each $\mu\in D_d$, it holds that 
$$\int_{\partial S} \mu=\int_S d\mu.$$
\end{theorem}

\newpage

\section{Infinite Dimensional Cauchy-Riemann Equation }
\label{20240111chapter2}

In this section, we consider functions on the following space
$$
\left\{(x_i+\sqrt{-1}y_i)_{i\in\mathbb{N}}\in\mathbb{C}^{\infty}:x_j,y_j\in \mathbb{R},\,\forall\,j\in\mathbb{N},\,\sum_{j=1}^{\infty}(x_j^2+y_j^2)<\infty\right\}.
$$
For simplicity of notation, we use the same letter $\ell^2$ to denote the above space when no confusion can arise. If we write $\mathbb{N}_2\triangleq \{(x_i,y_i)_{i\in\mathbb{N}}\in (\mathbb{R}\times\mathbb{R})^{\mathbb{N}}: \text{ there exists }i_0\in\mathbb{N}\text{ such that }x_i=y_i=0,\,\forall\,i\in\mathbb{N}\setminus\{i_0\}\text{ and }(x_{i_0},y_{i_0})\in \{(i_0,0),(0,i_0)\} \}$ which can be viewed as $\mathbb{N}\sqcup \mathbb{N}$. Thus, the $\ell^2$ here is $\ell^2(\mathbb{N}_2)$, where the second notation is the same as in other sections. We will write them simply $\ell^2$ and $\mathbb{N}_2$ in this section when no confusion can arise.

We will fixed a positive number $r$. By the same arguments as in Section \ref{20231108sect1}, we continue to write $P_r$ for the restriction of the product measure $\prod\limits_{i\in \mathbb{N}}\bn_{ra_i}\times \bn_{ra_i}$ on $\left(\ell^{2},\mathscr{B}\big(\ell^{2}\big)\right)$. The counterparts of Sections \ref{20231108sect1} and \ref{20240126chapter1} holds for $P_r$.

This section is mainly based on \cite{WYZ1}.

\subsection{Extension of the Cauchy-Riemann Operator to (s,t)-forms}\label{230916sec1}

Suppose that $f$ is a complex-valued function defined on $\ell^2 $. For each $\textbf{z}=(z_i)_{i\in\mathbb{N}}=(x_i+\sqrt{-1}y_i)_{i\in\mathbb{N}}\in \ell^2$, where $x_j,y_j\in \mathbb{R}$ for each $j\in\mathbb{N}$, we define
\begin{eqnarray*}
&&\displaystyle D_{x_j} f(\textbf{z}) \triangleq\lim_{\mathbb{R}\ni t\to 0}\frac{f(z_1,\cdots,z_{j-1},z_j+t,z_{j+1},\cdots)-f(\textbf{z})}{t},\\[3mm]
&&\displaystyle D_{y_j}f(\textbf{z}) \triangleq\lim_{\mathbb{R}\ni t\to 0}\frac{f(z_1,\cdots,z_{j-1},z_j+\sqrt{-1}t,z_{j+1},\cdots)-f(\textbf{z})}{t},
\end{eqnarray*}
provided that the limits exist.
Just like the case of finite dimensions, for each $j\in\mathbb{N}$, we set $\overline{\textbf{z}}=(\overline{z_i})_{i\in\mathbb{N}}=(x_i-\sqrt{-1}y_i)_{i\in\mathbb{N}}$, and
\begin{eqnarray}
\partial_{j} f(\textbf{z}) &\triangleq&\frac{1}{2}\left(D_{x_j} f(\textbf{z})-\sqrt{-1}D_{y_j} f(\textbf{z})\right),\nonumber\\
\overline{\partial_{j}} f(\textbf{z})&\triangleq &\frac{1}{2}\left(D_{x_j} f(\textbf{z})+\sqrt{-1}D_{y_j} f(\textbf{z})\right),\nonumber\\
\delta_{j} f(\textbf{z})&\triangleq &\partial_{j} f(\textbf{z})-\frac{\overline{z_{j}} }{2r^2a^{2}_{j}} \cdot f(\textbf{z}),\, \overline{\delta_{j} } f(\textbf{z})\triangleq \overline{\partial_{j} } f(\textbf{z})-\frac{z_{j}}{2r^2a^{2}_{j}}\cdot f(\textbf{z}),\label{230629def1}
\end{eqnarray}
where $\{a_i\}_{i=1}^{\infty}$ is the sequence of positive numbers given in Section \ref{20231108sect1}.
\begin{corollary}\label{integration by Parts or deltai}
Suppose that $f, g$ are Borel measurable functions on $\ell^2$ such that $f,g\in C_{S,b}^1(\ell^2)$ and supp$f\cup$supp$g$  is a bounded subset of $\ell^2$. It holds that
$$
\int_{\ell^2} \overline{\partial}_{i}  f\cdot \bar{g}  \,\mathrm{d}P_r=-\int_{\ell^2} f\cdot\overline{ \delta_{i} g}  \;\mathrm{d}P_r,\quad \forall\; i\in\mathbb{N}.
$$
\end{corollary}
\begin{proof}
By assumption, there exists $r,R\in(0,+\infty)$ such that
$$
\displaystyle\bigcup_{\textbf{z}\in  \supp f\bigcup \supp g}B_r(\textbf{z})\subset B_R.
$$
Then both $f$ and $g$ can be viewed as elements in $ C^{1}_{0} ( \ell^2 )$ by extending their values to $\ell^2\setminus (\supp f\cup \supp g)$ by 0.

Thus, for each $i\in\mathbb{N}$, it holds that
\begin{eqnarray*}
\int_{\ell^2} \overline{\partial}_{i}  f\cdot \overline{g} \,\mathrm{d}P_r=\int_{B_R}\frac{1}{2}( D_{x_{i}}f +\sqrt{-1}D_{y_{i}}f)\cdot \overline{g}  \,\mathrm{d}P_r.
\end{eqnarray*}
By Corollary \ref{20241127cor1}, we have
$$
\begin{array}{ll}
\displaystyle \int_{B_R}\frac{1}{2}( D_{x_{i}}f +\sqrt{-1}D_{y_{i}}f)\cdot \overline{g}  \,\mathrm{d}P_r\\[2mm]
\displaystyle
= -\int_{B_R}f\cdot \overline{ \frac{1}{2}( D_{x_{i}}g -\sqrt{-1}D_{y_{i}}g)}  \,\mathrm{d}P_r
+\int_{B_R}f\cdot \overline{ \frac{x_i-\sqrt{-1}y_i}{2r^2a_i^2} g}  \,\mathrm{d}P_r\\
\displaystyle
= -\int_{\ell^2}f\cdot \overline{ \frac{1}{2}( D_{x_{i}}g -\sqrt{-1}D_{y_{i}}g)}  \,\mathrm{d}P_r
+\int_{\ell^2}f\cdot \overline{ \frac{x_i-\sqrt{-1}y_i}{2r^2a_i^2} g}  \,\mathrm{d}P_r\\
\displaystyle= \int_{\ell^2}f\cdot \overline{ \left(-\partial_{i}g  +\frac{\overline{z_i}}{2r^2a_i^2}\cdot g\right) }\,\mathrm{d}P_r
=-\int_{\ell^2}f\cdot \overline{\delta_i g}\,\mathrm{d}P_r.
\end{array}
$$
This completes the proof of Corollary \ref{integration by Parts or deltai}.
\end{proof}
Suppose that $s,t$ are non-negative integers. If $s+t\geqslant 1$, write
$$
(\ell^2)^{s+t}=\underbrace{\ell^2\times\ell^2\times\cdots\times\ell^2}_{s+t\hbox{ \tiny times}}.
$$
Suppose that $I=(i_1,\cdots,i_s)$ and $J=(j_1,\cdots,j_t)$ are multi-indices, where $i_1,\cdots,i_s,$ $j_1,\cdots,j_t\in \mathbb{N}$, write
\begin{eqnarray}\label{200208t8}
\textbf{a}^{I,J}\triangleq\prod_{l=1}^{s}a_{i_l}^2\cdot \prod_{r=1}^{t}a_{j_r}^2.
\end{eqnarray}
and define $I\cup J\triangleq\{i_1,\cdots,i_s\}\cup \{j_1,\cdots,j_t\}$. For $j\in \mathbb{N}$, we define $J\cup\{j\}=\{j_1,\cdots,j_t\}\cup\{j\}$.
As in \cite[p. 530]{YZ22}, we define a complex-valued function $\mathrm{d}z^I\wedge \mathrm{d}\overline{z}^J$ on $(\ell^2)^{s+t}$ by
$$
(\mathrm{d}z^I\wedge \mathrm{d}\overline{z}^J)(\textbf{z}^1,\cdots,\textbf{z}^{s+t})\triangleq \frac{1}{\sqrt{(s+t)!}}\sum_{\sigma\in S_{s+t}}(-1)^{s(\sigma)}\cdot\prod_{k=1}^{s}z_{i_k}^{\sigma_k}\cdot\prod_{l=1}^{t}\overline{z_{j_l}^{\sigma_{s+l}}},
$$
where $\textbf{z}^l=(z_j^l)_{j\in\mathbb{N}} \in\ell^2,\,\,l=1,\cdots, s+t$, $S_{s+t}$ is the permutation group of $\{1,\cdots,s+t\}$, $s(\sigma)$ is the sign of $\sigma=(\sigma_1,\cdots,\sigma_s,\sigma_{s+1},\cdots,\sigma_{s+t})$, and we agree that $0!=1$. We call the following (formal) summation an $(s,t)$-form on $\ell^2$:
 \begin{equation}\label{230320e1}
 \sum^{\prime }_{\left| I\right|  =s} \sum^{\prime }_{\left| J\right|  =t} f_{I,J}\,\mathrm{d}z_{I}\wedge \,\mathrm{d}\overline{z_{J}},
 \end{equation}
where the sum $\sum\limits^{\prime }_{\left| I\right|  =s} \sum\limits_{\left| J\right|  =t}^{\prime } $ is taken only over strictly increasing multi-indices $I$ and $J$ with $\left| I\right|  =s$ and $\left| J\right|  =t$ (for which $\left| I\right|$ and $\left| J\right|$ stand for respectively the cardinalities of sets $\{i_1,\cdots,i_s\}$ and $\{j_1,\cdots,j_t\}$), and $f_{I,J}$ is a function on $\ell^2$ for any  strictly
increasing multi-indices $I$ and $J$. Clearly, each $f_{I,J}\,\mathrm{d}z_{I}\wedge \,\mathrm{d}\overline{z_{J}}$ can be viewed as a function on $\ell^2\times (\ell^2)^{s+t}$. Nevertheless, in the present setting of infinite dimensions (\ref{230320e1}) is an infinite series for which the convergence is usually not guaranteed, and therefore it is a formal summation (unless further conditions are imposed).

We need to introduce some working spaces which will play key roles in the sequel. For any $n\in\mathbb{N}$, we denote by $C_c^{\infty}(\mathbb{C}^n)$ the set of all $\mathbb{R}$-valued, $C^{\infty}$-functions on $\mathbb{C}^n(\equiv \mathbb{R}^{2n})$ with compact supports. Clearly, each function in $C_c^{\infty}(\mathbb{C}^n)$ can also be viewed as a cylinder function on $\ell^2$. Set
$$
\mathscr {C}_c^{\infty}(\mathbb{C}^n)\triangleq \left\{f+\sqrt{-1}g:\;f,g\in C_c^{\infty}(\mathbb{R}^{2n})\right\},\quad\mathscr {C}_c^{\infty}\triangleq \bigcup_{n=1}^{\infty}\mathscr {C}_c^{\infty}(\mathbb{C}^n).
$$
By the similar arguments as in \cite[Proposition 2.4, p. 528]{YZ22}, $\mathscr {C}_c^{\infty}$ is dense in $L^2(\ell^2,P_r)$. For each $r\in(0,+\infty)$, denote by
by $L^{2}_{\left( s,t\right)  }\left( \ell^2, P_r\right)$ the set of all $(s,t)$-forms $$\sum^{\prime }_{\left| I\right|  =s} \sum^{\prime }_{\left| J\right|  =t} f_{I,J}\,\mathrm{d}z_{I}\wedge \,\mathrm{d}\overline{z_{J}}$$ on $\ell^2$ for which
$$
f_{I,J}\in L^{2}\left(\ell^2, P_r\right)\hbox{ for each } I,J,\hbox{ and }\sum^{\prime }_{\left| I\right|  =s} \sum^{\prime }_{\left| J\right|  =t}\textbf{a}^{I,J} \int_{\ell^2} \left| f_{I,J}\right|^{2}\,\mathrm{d}P_r<\infty.
$$
Suppose that
\begin{eqnarray}\label{230629for1}
f=\sum^{\prime }_{\left| I\right|  =s} \sum^{\prime }_{\left| J\right|  =t} f_{I,J}\,\mathrm{d}z_{I}\wedge \,\mathrm{d}\overline{z_{J}}\in L^{2}_{\left( s,t\right)  }\left( \ell^2, P_r\right),
\end{eqnarray}
$j\in\mathbb{N}, I=(i_1,\cdots,i_s), J=(j_1,\cdots,j_t)$ and $K=(k_1,\cdots,k_{t+1})$ are multi-indices with strictly increasing order.
If $K\neq J\cup \{j\}$, set $\varepsilon_{j, J}^{K}=0$. If $K= J\cup \{j\}$, we denote the sign of the permutation taking $(j,j_1,\cdots,j_t)$ to $K$ by $\varepsilon_{j, J}^{K}$. If for each $I$ and $K$, there exists $g_{I,K}\in L^2\left( \ell^2, P_r\right)$ such that
\begin{eqnarray}\label{weak def of st form}
\int_{\ell^2}g_{I,K}\cdot\overline{\varphi}\,\mathrm{d}P_r=-\int_{\ell^2}\sum_{1\leqslant i<\infty}\sum_{|J|=t}'\varepsilon_{i, J}^{K}f_{I,J}\cdot\overline{\delta_i\varphi} \,\mathrm{d}P_r,
\end{eqnarray}
for any $\varphi\in \mathscr{C}_c^{\infty}$ and
\begin{eqnarray}\label{230929for1}
\sum_{|I|=s,|K|=t+1}'\textbf{a}^{I,K}\cdot \int_{\ell^2}|g_{I,K}|^2\,\mathrm{d}P_r<\infty,
\end{eqnarray}
then, we define
\begin{eqnarray}\label{d-dar-f-defi}
\overline{\partial}f\triangleq (-1)^s \sum_{|I|=s,|K|=t+1}'g_{I,K} \mathrm{d}z_I\wedge \mathrm{d}\overline{z_K}.
\end{eqnarray}
We will use the notation $D_{\overline{\partial}}$ to denote all $f\in L^{2}_{\left( s,t\right)  }\left( \ell^2, P_r\right)$ such that \eqref{weak def of st form}-\eqref{230929for1} holds.
Clearly,
$$\overline{\partial}f\in L^2_{(s,t+1)}\left( \ell^2, P_r\right).
 $$
Similarly to the proof of \cite[Lemma 2.5, pp. 529-530]{YZ22}, we can prove that the operator $\overline{\partial}$ given in the above is a densely defined, closed linear operator from $ L^2_{(s,t)}\left( \ell^2, P_r\right)$ to $L^2_{(s,t+1)}\left( \ell^2, P_r\right)$.

Similarly to the proof of \cite[Lemma 2.6, pp. 531-532]{YZ22}, we have the following result:

\begin{lemma}\label{lem2.5.1}
Suppose that $f\in L^2_{(s,t)}\left( \ell^2, P_r\right)$ is in the form of \eqref{230629for1} so that $\overline{\partial}f$ can be defined as that in \eqref{d-dar-f-defi}. Then,
\begin{eqnarray}\label{kernal-range-P}
\overline{\partial}(\overline{\partial}f)=0.
\end{eqnarray}
\end{lemma}

\begin{proof}
Suppose $I=(i_1,\cdots,i_s)$, $K=(k_1,\cdots,k_{t+1})$ and\\ $M=(m_1,\cdots,m_{t+2})$ are multi-indices with strictly increasing order. We claim that, for each $j\in\mathbb{N}$ and $\varphi\in \mathscr{C}_c^{\infty}$,
\begin{eqnarray}\label{noncompact test}
\int_{\ell^2}g_{I,K}\cdot\overline{\delta_j\varphi}\,\mathrm{d}P_r
=-\int_{\ell^2}\sum_{1\leq i<\infty}\sum_{|J|=t}'\varepsilon_{i, J}^{K}f_{I,J} \cdot\overline{\delta_i(\delta_j\varphi)} \,\mathrm{d}P_r.
\end{eqnarray}
Here we should note that $\delta_j\varphi$ (defined by \eqref{230629def1}) may not lie in $\mathscr{C}_c^{\infty}$. Because of this, we choose a function $\varsigma\in C_c^\infty(\mathbb{R})$ so that $\varsigma\equiv 1$ on $[-1,1]$. For each $n\in\mathbb{N}$, write $\varphi_n\equiv \varsigma\big(\frac{\sum^{j}_{k=1}|z_k|^2}{n^2}\big)$. Then, $\delta_j(\varphi\varphi_n)\in \mathscr{C}_c^{\infty}$. Applying (\ref{weak def of st form}), we arrive at
\begin{eqnarray*}
\int_{\ell^2}g_{I,K}\cdot\overline{\delta_j(\varphi\varphi_n)}\,\mathrm{d}P_r
=-\int_{\ell^2}\sum_{1\leq i<\infty}\sum_{|J|=t}'\varepsilon_{i, J}^{K}f_{I,J}\cdot\overline{\delta_i(\delta_j(\varphi\varphi_n))} \,\mathrm{d}P_r.
\end{eqnarray*}
Letting $n\to\infty$ in the above, we obtain \eqref{noncompact test}.

Thanks to \eqref{noncompact test}, it follows that
\begin{eqnarray*}
&&-\int_{\ell^2}\sum_{1\leq j<\infty}\sum_{|K|=t+1}'\varepsilon_{j, K}^{M}g_{I,K}\cdot\overline{\delta_j\varphi}\,\mathrm{d}P_r\\
&=&\sum_{1\leq i,j<\infty}\sum_{|K|=t+1}'\sum_{|J|=t}'\varepsilon_{i, J}^{K}\cdot\varepsilon_{j, K}^{M}\cdot\int_{\ell^2}f_{I,J}\cdot\overline{\delta_i(\delta_j\varphi)}\,\mathrm{d}P_r\\
&=&\sum_{1\leq i,j<\infty}\sum_{|K|=t+1}'\sum_{|J|=t}'\varepsilon_{j,i, J}^{j,K}\cdot\varepsilon_{j, K}^{M}\cdot\int_{\ell^2}f_{I,J}\cdot\overline{\delta_i(\delta_j\varphi)}\,\mathrm{d}P_r\\
&=&\sum_{1\leq i,j<\infty}\sum_{|J|=t}'\varepsilon_{j,i, J}^{M}\cdot\int_{\ell^2}f_{I,J}\cdot\overline{\delta_i(\delta_j\varphi)}\,\mathrm{d}P_r\\
&=&-\sum_{1\leq i,j<\infty}\sum_{|J|=t}'\varepsilon_{i,j, J}^{M}\cdot\int_{\ell^2}f_{I,J}\cdot\overline{\delta_i(\delta_j\varphi)}\,\mathrm{d}P_r\\
&=&-\sum_{1\leq i,j<\infty}\sum_{|J|=t}'\varepsilon_{i,j, J}^{M}\cdot\int_{\ell^2}f_{I,J}\cdot\overline{\delta_j(\delta_i\varphi)}\,\mathrm{d}P_r\\
&=&-\sum_{1\leq i,j<\infty}\sum_{|J|=t}'\varepsilon_{j,i, J}^{M}\cdot\int_{\ell^2}f_{I,J}\cdot\overline{\delta_i(\delta_j\varphi)}\,\mathrm{d}P_r,
\end{eqnarray*}
where the fourth equality follows from the fact that $\varepsilon_{j,i, J}^{M}=-\varepsilon_{i,j, J}^{M}$. The above equalities imply that
\begin{eqnarray*}
-\int_{\ell^2}\sum_{1\leq j<\infty}\sum_{|K|=t+1}'\varepsilon_{j, K}^{M}g_{I,K}\cdot\overline{\delta_j\varphi}\,\mathrm{d}P_r=0.
\end{eqnarray*}
The arbitrariness of $I$ and $M$ implies the desired property \eqref{kernal-range-P}. This completes the proof of Lemma \ref{lem2.5.1}.
\end{proof}

\subsection{Reduction of Dimension and Mollification}
Similar to the proof of \cite[Proposition 2.4, p. 528]{YZ22}, it holds that
\begin{proposition}\label{20230916prop1}
$\mathscr {C}_c^{\infty}$ is a dense subspace of $L^2\left( \ell^2, P_r\right)$.
\end{proposition}


Motivated by  Proposition \ref{20230916prop1},  we have the following definition for support of functions in $ L^2\left( \ell^2, P_r\right)$.
\begin{definition}
For $f\in L^2(\ell^2, P_r)$, we will use the symbol supp$f$ to denote $\ell^2$ minus the union of all open balls $B_{r_1}(\textbf{z})$ such that $\int_{B_{r_1}(\textbf{z})} f\phi \mathrm{d}P_r=0$ for all $\phi \in \mathscr {C}_c^{\infty}.$
	
As in Subsection \ref{230916sec1}, for
\begin{eqnarray*}
f=\sum^{\prime }_{\left| I\right|  =s} \sum^{\prime }_{\left| J\right|  =t} f_{I,J}\,\mathrm{d}z_{I}\wedge \,\mathrm{d}\overline{z_{J}}\in L^{2}_{\left( s,t\right)  }\left( \ell^2, P_r\right),
\end{eqnarray*}
write supp$f\triangleq \overline{\bigcup\limits^{\prime }_{\left| I\right|  =s,\left| J\right|  =t} \text{supp}  f_{I,J}}$ where the union $\bigcup\limits^{\prime }$ is taken only over strictly increasing multi-indices $I=\left( i_{1},i_{2},\cdots ,i_{s}\right)$ and $J=\left( j_{1},j_{2},\cdots ,j_{t}\right)$, i.e., $i_{1}<i_{2}<\cdots <i_{s}$ and $j_{1}<j_{2}<\cdots <j_{t}$ and the overline means the closure in $\ell^2$.
\end{definition}
For each $n\in\mathbb{N}$, write $\mathbb{N}_{2,\widehat{n} }\triangleq \{(x_i,y_i)_{i\in\mathbb{N}}\in (\mathbb{R}\times\mathbb{R})^{\mathbb{N}}: $ there exists $i_0\in\mathbb{N}\setminus\{1,\cdots,n\}$ such that $x_i=y_i=0,\,\forall\,i\in\mathbb{N}\setminus\{i_0\}$ and \\ $(x_{i_0},y_{i_0})\in \{(i_0,0),(0,i_0)\} \}$ which can be viewed as $(\mathbb{N}\setminus\{1,\cdots,n\})\sqcup (\mathbb{N}\setminus\{1,\cdots,n\})$, we define a probability measure $\bn^{n,r}$ in $(\mathbb{C}^n,\mathscr{B}(\mathbb{C}^n))$ by setting
$$
\bn^{n,r}\triangleq \prod_{i=1}^{n}\bn_{ra_i}\times \bn_{ra_i}
$$
and a probability measure ${\widehat{ \bn}^{n,r}}$ in $(\mathbb{C}^{\mathbb{N}\setminus\{1,\cdots,n\}},\mathscr{B}(\mathbb{C}^{\mathbb{N}\setminus\{1,\cdots,n\}}))$ by setting
$$
\widehat{\bn}^{n,r}\triangleq \prod_{i=n+1}^{\infty}\bn_{ra_i}\times \bn_{ra_i}.
$$
Let
\begin{equation}\label{220817e1zx}
P_{n,r}(E)\triangleq \widehat{\bn}^{n,r}(E),\quad\forall\;E\in \mathscr{B}(\ell^2(\mathbb{N}_{2,\widehat{n} })).
\end{equation}
Then, by a similar arguments around \eqref{220817e1}, we obtain a Borel probability measure $P_{n,r}$ on $\ell^2(\mathbb{N}_{2,\widehat{n} })$. Obviously, $\ell^2=\mathbb{C}^n\times \ell^2(\mathbb{N}_{2,\widehat{n} })$ and
$$P_r=\bn^{n,r}\times P_{n,r}\quad\forall\; n\in \mathbb{N}.$$

\begin{definition}\label{Integration reduce dimension}
Suppose that $f\in L^2(\ell^2, P_r)$ and $n\in\mathbb{N}$. Then let
\begin{eqnarray*}
f_n(\textbf{z}_n)&\triangleq &\int_{\ell^2(\mathbb{N}_{2,\widehat{n} })} f(\textbf{z}_n,\textbf{z}^n)\,\mathrm{d}P_{n,r}(\textbf{z}^n),
\end{eqnarray*}
where $\textbf{z}^n=(x_{i} +\sqrt{-1}y_{i})_{i=n+1}^{\infty}\in \ell^2(\mathbb{N}_{2,\widehat{n} }),\,\,\, \textbf{z}_n=(x_{i} +\sqrt{-1}y_{i})_{i=1}^{n}\in \mathbb{C}^n.$
\end{definition}

\begin{proposition}\label{Reduce diemension}
For $f\in L^{2} ( \ell^{2},P_r )$ and $n\in\mathbb{N}$, then $f_n$ can also be viewed as a cylinder function on $\ell^2$ with the following properties.
\begin{itemize}
\item[(1)]\label{z1}
		$\left| \left| f_{n}\right|  \right|_{L^{2} ( \ell^{2} ,P_r )  }  \leqslant \left| \left| f\right|  \right|_{L^{2} ( \ell^{2} ,P_r )  }  .$
\item[(2)] \label{z2}
		$\lim\limits_{n\rightarrow \infty } \left| \left| f_{n}-f\right|  \right|_{L^{2} ( \ell^{2},P_r)  }  =0.$
\item[(3)] \label{z3}
	$\lim\limits_{n\rightarrow \infty } \int_{\ell^2}\big|| f_{n}|^{2}  -| f|^{2}\big|\,\mathrm{d}P_r =0.$	
\end{itemize}
\end{proposition}
\begin{proof}
(1) Note that
\begin{eqnarray*}
\int_{\ell^2}|f_{n}|^2\,\mathrm{d}P&=&\int_{\mathbb{C}^n}|f_n(\textbf{z}_n)|^2\,\mathrm{d}\mathcal{N}^{n,r}(\textbf{z}_n)\\
&=&\int_{\mathbb{C}^n}\bigg|\int_{\ell^2(\mathbb{N}_{2,\widehat{n} })} f(\textbf{z}_n,\textbf{z}^n)\cdot
\mathrm{d}P_{n,r}(\textbf{z}^n)\bigg|^2\mathrm{d}\mathcal{N}^{n,r}(\textbf{z}_n)\\
&\leqslant&\int_{\mathbb{C}^n} \int_{\ell^2(\mathbb{N}_{2,\widehat{n} })} |f(\textbf{z}_n,\textbf{z}^n)|^2\cdot
\mathrm{d}P_{n,r}(\textbf{z}^n)\mathrm{d}\mathcal{N}^{n,r}(\textbf{z}_n)\\
&=&\int_{\ell^2}|f|^2\,\mathrm{d}P_r,
\end{eqnarray*}
where the inequality follows form the Jensen's inequality and the first and the third equalities follows form the fact that $P_r=\bn^{n,r}\times P_{n,r}.$\\
(2) Since $\mathscr {C}_c^{\infty}$ is dense in $L^2(\ell^2,P_r)$, for any $\epsilon>0$, there exists $m\in\mathbb{N}$ and $g\in \mathscr{C}_c^{\infty}(\mathbb{C}^m)$ such that
\begin{eqnarray*}
\bigg(\int_{\ell^2} |f-g|^2\mathrm{d}P_r\bigg)^{\frac{1}{2}}<\frac{\epsilon}{2}.
\end{eqnarray*}
Then for any $n\geqslant m$, we have $g_n=g$
and
\begin{eqnarray*}
\bigg(\int_{\ell^2} |f_n-f|^2\mathrm{d}P_r\bigg)^{\frac{1}{2}}
&=&\bigg(\int_{\ell^2} |f_n-g_n+g-f|^2\mathrm{d}P_r\bigg)^{\frac{1}{2}}\\
&\leqslant&\bigg(\int_{\ell^2} |f_n-g_n|^2\mathrm{d}P_r\bigg)^{\frac{1}{2}}+\bigg(\int_{\ell^2 } |g-f|^2\mathrm{d}P_r\bigg)^{\frac{1}{2}}\\
&\leqslant&2\bigg(\int_{\ell^2} |g-f|^2\mathrm{d}P_r\bigg)^{\frac{1}{2}}<\epsilon,
\end{eqnarray*}
where the second inequality follows from (1) of Proposition \ref{Reduce diemension} that we have proved.\\
(3) Note that
\begin{eqnarray*}
\int_{\ell^2}\big|| f_{n}|^{2}  -| f|^{2}\big|\,\mathrm{d}P_r  &=&
\int_{\ell^2}\big|| f_{n}|-| f|\big|\cdot(| f_{n}|+| f|)\,\mathrm{d}P_r\\
&\leqslant& \big|\big| |f_{n}|-| f|\big|\big|_{L^{2}\left(\ell^2 ,P_r\right)}\cdot \big|\big| |f_{n}|+| f|\big|\big|_{L^{2}\left(\ell^2,P_r\right)}\\
&\leqslant&  || f_{n}- f||_{L^{2}(\ell^2,P_r)}\cdot(|| f_n||_{L^{2}\left( \ell^2,P_r\right)}+ || f||_{L^{2}\left(\ell^2 ,P_r\right)})\\
&\leqslant& 2|| f_{n}- f||_{L^{2}\left(\ell^2,P_r\right)}\cdot || f||_{L^{2}\left(\ell^2,P_r\right)},
\end{eqnarray*}
where the first inequality follows from the Cauchy-Schwarz inequality. By (2) of Proposition \ref{Reduce diemension} that we have proved, we see that
\begin{eqnarray*}
\lim_{n\to\infty}2|| f_{n}- f||_{L^{2}\left( \ell^2,P_r\right)}\cdot || f||_{L^{2}\left( \ell^2,P_r\right)}=0,
\end{eqnarray*}
which implies that (3) of Proposition \ref{Reduce diemension} holds. This completes the proof of Proposition \ref{Reduce diemension}.
\end{proof}
For each $n\in\mathbb{N}$, choose $\chi_{n} \in C^{\infty }_{c}\left( \mathbb{C}^{n} \right)$ such that
\begin{itemize}
\item[(1)]
$\chi_{n} \geqslant 0$;
\item[(2)]
$\int_{\mathbb{C}^{n} } \chi_{n} \left( \textbf{z}_n\right)  d\textbf{z}_n=1$ where $d\textbf{z}_n$ is the Lebesuge meausre on $\mathbb{C}^n$;
\item[(3)]
 supp$\chi_{n} \subset \{ \textbf{z}_n \in \mathbb{C}^{n} :\left| \left| \textbf{z}_n\right|  \right|_{\mathbb{C}^{n} }  \leqslant 1\} ;$
\item[(4)] $\chi_{n} ( \textbf{z}_n)  =\chi_{n} ( \textbf{z}_n')$ for all $\textbf{z}_n,\textbf{z}_n'\in\mathbb{C}^{n}$ such that $| | \textbf{z}_n  |  |_{\mathbb{C}^{n}}  =\left| \left| \textbf{z}_n'\right|  \right|_{\mathbb{C}^{n}}  $.
\end{itemize}
For $\delta\in(0,+\infty)$, set
$\chi_{n,\delta } \left( \textbf{z}_n\right)  \triangleq \frac{1}{\delta^{2n} } \chi \left( \frac{\textbf{z}_n  }{\delta } \right),\,\textbf{z}_n  \in \mathbb{C}^{n}$.
\begin{definition}\label{Convolution after reduce dimension}
Suppose that $f\in L^2(\ell^2,P_r)$ and there exists $r_1\in(0,+\infty)$ such that supp$f\subset B_{r_1}$. Then for $n\in\mathbb{N}$, supp$f_n\subset \{ \textbf{z}_n \in \mathbb{C}^{n} :\left| \left| \textbf{z}_n\right|  \right|_{\mathbb{C}^{n} }  \leqslant r_1\}$. By (1) of Proposition \ref{Reduce diemension}, $f_n\in L^2(\mathbb{C}^n,\mathcal{N}^n)$. Thus $f_n\in L^2(\mathbb{C}^n, d\textbf{z}_n)\cap L^1(\mathbb{C}^n, d\textbf{z}_n)$ and for $\delta\in(0,+\infty)$, let
\begin{eqnarray*}
f_{n,\delta}(\textbf{z}_n)\triangleq  \int_{\mathbb{C}^n}f_n(\textbf{z}_n')\chi_{n,\delta }  ( \textbf{z}_n -\textbf{z}_n')\mathrm{d}\textbf{z}_n',\,\textbf{z}_n  \in \mathbb{C}^{n}.
\end{eqnarray*}
Then $f_{n,\delta}\in C_c^{\infty}(\mathbb{C}^n)$ and $f_{n,\delta}$ is the convolution of $f_n$ and $\chi_{n,\delta}.$
\end{definition}
\begin{definition}\label{def for Gauss weight}
For each $n\in\mathbb{N}$, set
$$
\varphi_n(\textbf{z}_n)\triangleq\prod_{i=1}^{n}\left(\frac{1}{2\pi r^2a_i^2}\cdot e^{-\frac{|z_i|^2}{2r^2a_i^2}}\right),\quad \; \textbf{z}_n=(z_1,\cdots,z_n)\in \mathbb{C}^n.
$$
\end{definition}
\begin{proposition}\label{convolution properties}
Suppose that $f\in L^2(\ell^2,P_r)$ and supp$f$ is a bounded subset of $\ell^2$. For each $n\in\mathbb{N}$, it holds that
\begin{itemize}
\item[(1)] $f_n\in  L^2(\mathbb{C}^n, d\textbf{z}_n)\cap L^1(\mathbb{C}^n, d\textbf{z}_n)$ and $f_{n,\delta}\in L^2(\ell^2,P_r)$ for each $\delta\in(0,+\infty)$ and $\lim\limits_{\delta\to 0+}||f_{n,\delta}-f_n||_{L^2(\ell^2,P_r)}=0$;
\item[(2)] $\lim\limits_{\delta\rightarrow 0+} \int_{\ell^2}\big|| f_{n,\delta}|^{2}  -| f_n|^{2}\big|\,\mathrm{d}P_r =0;$
\item[(3)] For $g\in L^2(\ell^2,P_r)$ and supp$g\subset B_{r_2}$ for some $r_2\in(0,+\infty)$, we have
\begin{eqnarray*}
\int_{\ell^2}f_{n,\delta}g\,\mathrm{d}P_r=\int_{\ell^2}f\varphi_n^{-1}(g_{n}\varphi_n)_{n,\delta}\,\mathrm{d}P_r.
\end{eqnarray*}
\end{itemize}
\end{proposition}
\begin{proof}
(1) Since $f_{n,\delta}\in C_c^{\infty}(\mathbb{C}^n)$, $f_{n,\delta}\in L^2(\ell^2,P_r)$.
By assumption there exists $r_1>0$ such that supp$f\subset B_{r_1}$. By Definition \ref{Integration reduce dimension} and \ref{Convolution after reduce dimension}, we see that supp$f_{n}\subset \left\{ \textbf{z}_n\in \mathbb{C}^{n} :\left| \left| \textbf{z}_n\right|  \right|_{\mathbb{C}^{n} }  \leqslant r_1\right\}$ and
$$
\text{supp} f_{n,\delta }\subset \left\{ \textbf{z}_n\in \mathbb{C}^{n} :\left| \left| \textbf{z}_n\right|  \right|_{\mathbb{C}^{n} }  \leqslant r_1+\delta \right\} .
$$
Combining (1) of Proposition \ref{Reduce diemension}, we have $f_{n}\in L^2(\mathbb{C}^n,\mathrm{d}\textbf{z}_n)\cap L^1(\mathbb{C}^n,\mathrm{d}\textbf{z}_n)$ and
\begin{eqnarray*}
\int_{\ell^2}|f_{n,\delta}  -f_n|^{2}\,\mathrm{d}P_r
&=&\int_{\mathbb{C}^n} |  f_{n,\delta} - f_n |^2\,\mathrm{d}\mathcal{N}^{n,r}\\
&\leqslant &\frac{1}{\left( 2\pi r^2\right)^{n}  a^{2}_{1}a^{2}_{2}\cdots a^{2}_{n}} \int_{\mathbb{C}^{n} } \left| f_{n,\delta }\left(  \textbf{z}_n\right)  -f_{n}\left( \textbf{z}_n\right)  \right|^{2}  \mathrm{d}\textbf{z}_n.
\end{eqnarray*}
Since $\lim\limits_{\delta\to 0+}\int_{\mathbb{C}^{n} } \left| f_{n,\delta }\left(  \textbf{z}_n\right)  -f_{n}\left( \textbf{z}_n\right)  \right|^{2}  \mathrm{d}\textbf{z}_n=0$, $\lim\limits_{\delta\to 0+}||f_{n,\delta}-f_n||_{L^2(\ell^2,P_r)}=0$.\\
(2) Similar to the proof  of $(3)$ of Proposition \ref{Reduce diemension}.\\
(3) Note that
\begin{eqnarray*}
&&\int_{\ell^2 } f_{n,\delta }g\mathrm{d}P_r\\
&=&\int_{\ell^2} \left(\int_{\mathbb{C}^{n} } f_{n}\left( \textbf{z}_n'\right)  \chi_{n,\delta } \left( \textbf{z}_n-\textbf{z}_n'\right)  \mathrm{d}\textbf{z}_n'\right)g\left( \textbf{z}_n,\textbf{z}^n\right) \mathrm{d}P_r\left( \textbf{z}_n,\textbf{z}^n\right)  \\
&=&\int_{\ell^2}\int_{\ell^2} f\left( \textbf{z}_n',\widetilde{\textbf{z}}^n\right)  \chi_{n,\delta } \left( \textbf{z}_n-\textbf{z}_n'\right)  \varphi^{-1}_{n} \left( \textbf{z}_n'  \mathrm{d}P_r\left( \textbf{z}_n',\widetilde{\textbf{z}}^n\right)  \right) g\left( \textbf{z}_n,\textbf{z}^n\right) \mathrm{d}P_r\left(\textbf{z}_n,\textbf{z}^n\right) \\
&=&\int_{\ell^2}f\left( \textbf{z}_n',\widetilde{\textbf{z}}^n\right)\varphi^{-1}_{n} \left( \textbf{z}_n'\right)\int_{\ell^2}g\left( \textbf{z}_n,\textbf{z}^n\right) \chi_{n,\delta } \left( \textbf{z}_n-\textbf{z}_n'\right)\mathrm{d}P_r\textbf{z}_n,\textbf{z}^n\mathrm{d}P_r\left( \textbf{z}_n',\widetilde{\textbf{z}}^n\right)  \\
&=&\int_{\ell^2}f\left( \textbf{z}_n',\widetilde{\textbf{z}}^n\right)\varphi^{-1}_{n} \left( \textbf{z}_n'\right) \int_{\mathbb{C}^{n}  }g_n\left( \textbf{z}_n\right)\varphi_n(\textbf{z}_n) \chi_{n,\delta }  \textbf{z}_n-\textbf{z}_n'\mathrm{d}\textbf{z}_n\mathrm{d}P_r\left( \textbf{z}_n',\widetilde{\textbf{z}}^n\right)  \\
&=&\int_{\ell^2} f\varphi^{-1}_{n} \left( g_{n}\varphi_{n} \right)_{n,\delta } \mathrm{d}P_r,
\end{eqnarray*}
where the third equality follows from the Fubini's theorem. This completes the proof of Proposition \ref{convolution properties}.
\end{proof}
We also need the following simple approximation result (which should be known but we do not find an exact reference). Recall \eqref{20231110for6} for the definition of $C^{1 }_{S,b}( \ell^{2})$.
\begin{lemma}\label{F}
Assume that $f\in C^{1 }_{S,b}( \ell^{2})$, $f$ is a Borel measurable function on $\ell^2$ and supp$f$ is a bounded subset of $\ell^2$. Then there exists $\{h_{k }\}_{k=1}^{\infty}\subset  \mathscr{C}_c^{\infty}$ such that $\lim\limits_{k \rightarrow \infty } h_{k}=f$ and $\lim\limits_{k \rightarrow \infty } D_{x_{j}}h_{k}=D_{x_{j}}f$  and $\lim\limits_{k \rightarrow \infty } D_{y_{j}}h_{k}=D_{y_{j}}f$ almost everywhere respect to $P_r$ for each $j\in\mathbb{N}$. Meanwhile, for each $j,k\in\mathbb{N}$, the following inequalities hold on $\ell^2$,
\begin{equation}\label{6zqqqx}
\left\{
\begin{array}{ll}
\displaystyle\left|h_{k }\right|  \leqslant \sup_{\ell^{2} } \left| f\right|  ,\  \left| \partial_{j} h_{k}\right|  \leqslant \sup_{ \ell^{2}} \left| \partial_{j} f\right|  ,\left| \overline{\partial_{j} } h_{k }\right|  \leqslant \sup_{ \ell^{2}} \left| \overline{\partial_{j} } f\right|,  \\[2mm]
\displaystyle\left| D_{x_{j}}h_{k }\right|  \leqslant \sup_{\ell^{2} } \left| D_{x_{j}}f\right|\text{ and } \left| D_{y_{j}}h_{k}\right|  \leqslant \sup_{\ell^{2} } \left| D_{y_{j}}f\right|,\\[2mm]	
\displaystyle\sum^{\infty }_{j=1}a_j^2 \left| \partial_{j} h_{k}\right|^{2}  \leqslant \sup_{\ell^{2}} \sum^{\infty }_{j=1} a_j^2\left| \partial_{j} f\right|^{2} ,\quad\sum^{\infty }_{j=1}a_j^2 \left| \overline{\partial_{j} } h_{k}\right|^{2}  \leqslant \sup_{\ell^{2} } \sum^{\infty }_{j=1}a_j^2 \left| \overline{\partial_{j} } f\right|^{2}.
\end{array}\right.
\end{equation}
\end{lemma}

\begin{proof}
For each $n\in\mathbb{N}$, let
$$f_n(\textbf{z}_n)\triangleq \int_{\ell^2(\mathbb{N}_{2,\widehat{n} })} f(\textbf{z}_n,\textbf{z}^n)\,\mathrm{d}P_{n,r}(\textbf{z}^n),
 $$
where $\textbf{z}^n=(x_{i} +\sqrt{-1}y_{i})_{i=n+1}^\infty\in \ell^2(\mathbb{N}_{2,\widehat{n} })$ and $\textbf{z}_n=(x_{i} +\sqrt{-1}y_{i})_{i=1}^{n}\in \mathbb{C}^n$. It is easy to see that
$$
D_{x_j}f_n(\textbf{z}_n)=\int_{\ell^2(\mathbb{N}_{2,\widehat{n} })} D_{x_j}f(\textbf{z}_n,\textbf{z}^n)\,\mathrm{d}P_{n,r}(\textbf{z}^n)
$$
for $j=1,\cdots, n$. By the conclusion (2) of Proposition \ref{Reduce diemension}, we have
$$
\lim\limits_{n\rightarrow \infty } \int_{\ell^{2} } \left| f_{n}-f\right|^{2}  dP_r=0\, \text{and}\, \lim\limits_{n\rightarrow \infty } \int_{\ell^{2} } \left| D_{x_{j}} f_{n}  -D_{x_{j}} f\right|^{2}dP_r=0.
$$
Then, for each $\delta\in(0,+\infty)$ and $n,j\in\mathbb{N}$, we have $f_{n,\delta}\in \mathscr{C}_c^{\infty}$,
\begin{eqnarray*}
|f_{n,\delta}|&\leqslant &\sup\limits_{\mathbb{C}^n}|f_n|\leqslant \sup\limits_{\ell^2}|f|, \\
|D_{x_j}f_{n,\delta}|&\leqslant &\sup\limits_{\mathbb{C}^n}|D_{x_j}f_n|\leqslant \sup\limits_{\ell^2}|D_{x_j}f|,\\
|D_{y_j}f_{n,\delta}|&\leqslant &\sup\limits_{\mathbb{C}^n}|D_{y_j}f_n|\leqslant \sup\limits_{\ell^2}|D_{y_j}f|,\\
|\partial_{j}f_{n,\delta}|&\leqslant &\sup\limits_{\mathbb{C}^n}|\partial_{j}f_n|\leqslant \sup\limits_{\ell^2}|\partial_{j}f|,
\end{eqnarray*}
and
$$
\lim\limits_{\delta\rightarrow 0} \int_{\ell^{2} } \left| f_{n,\delta}-f_n\right|^{2}\,\mathrm{d}P_r=0,\qquad
\lim\limits_{\delta\rightarrow 0} \int_{\ell^{2} } \left| D_{x_{j}} f_{n,\delta}  -D_{x_{j}} f_n\right|^{2}\,\mathrm{d}P_r=0.
$$
By the Jensen inequality (See \cite[Theorem 3.3, p. 62]{Rud87}),
$$
\begin{array}{ll}
\displaystyle 		
\sum^{\infty }_{j=1}a_j^2 \left| \partial_{j} f_{n,\delta}( \textbf{z}_n)\right|^{2}
= \sum^{n}_{j=1} a_j^2\left| \partial_{j} f_{n,\delta}( \textbf{z}_n )\right|^{2}\\[3mm]
\displaystyle \leqslant
\int_{\mathbb{C}^n}\sum^{n}_{j=1} a_j^2\left| \partial_{j} f_{n }(\textbf{z}_n')\right|^{2}\gamma_{n,\delta }  ( \textbf{z}_n -\textbf{z}_n')\mathrm{d}\textbf{z}_n'\\[3mm]
\displaystyle \leqslant \sup_{\textbf{z}_n'\in\mathbb{C}^n}\sum^{n}_{j=1}a_j^2 \left| \partial_{j} f_{n }(\textbf{z}_n')\right|^{2}
= \sup_{\textbf{z}_n'\in\mathbb{C}^n}\sum^{n}_{j=1}a_j^2 \left| \int_{\ell^2(\mathbb{N}_{2,\widehat{n} })}\partial_{j} f (\textbf{z}_n',\textbf{z}^n)\,\mathrm{d}P_{n,r}(\textbf{z}^n)\right|^{2}\\[3mm]
\displaystyle \leqslant \sup_{\textbf{z}_n'\in\mathbb{C}^n} \int_{\ell^2(\mathbb{N}_{2,\widehat{n} })}\sum^{n}_{j=1} a_j^2\left|\partial_{j} f (\textbf{z}_n',\textbf{z}^n)\right|^{2}\,\mathrm{d}P_{n,r}(\textbf{z}^n)\\[3mm]
\displaystyle \leqslant \sup_{\textbf{z}\in\ell^2}  \sum^{n}_{j=1} a_j^2\left|\partial_{j} f (\textbf{z})\right|^{2}\leqslant \sup_{\textbf{z}\in\ell^2}  \sum^{\infty}_{j=1}a_j^2 \left|\partial_{j} f (\textbf{z})\right|^{2}.
\end{array}
$$
By the diagonal method we can pick a sequence of functions $\{g_k\}_{k=1}^{\infty}$ from $\{f_{n,\delta}:\;n\in\mathbb{N},\,\delta\in(0,+\infty)\}$ such that $\lim\limits_{k\rightarrow \infty } \int_{\ell^{2} } \left| g_{k}-f\right|^{2}  \,\mathrm{d}P_r=0$ and $\lim\limits_{k\rightarrow \infty } \int_{\ell^{2} } \left| D_{x_{j}} g_{k}  -D_{x_{j}} f\right|^{2}\,\mathrm{d}P_r=0$ for each $j\in\mathbb{N}$. By \cite[Theorem 6.3.1 (b) at pp. 171--172 and Section 6.5 at pp. 180--181]{Res} and using  the diagonal method again, we can find a subsequence $\{h_k\}_{k=1}^{\infty}$ of $\{g_k\}_{k=1}^{\infty}$ such that $\lim\limits_{k \rightarrow \infty } h_{k}=f$, $\lim\limits_{k \rightarrow \infty } D_{x_{j}}h_{k}=D_{x_{j}}f$ and $\lim\limits_{k \rightarrow \infty } D_{y_{j}}h_{k}=D_{y_{j}}f$  almost everywhere respect to $P_r$ for each $j\in\mathbb{N}$. Obviously, the sequence $\{h_k\}_{k=1}^{\infty}$ satisfies the conditions in (\ref{6zqqqx}).
This completes the proof of Lemma \ref{F}.
\end{proof}

\begin{definition}\label{definition of T S}
Suppose that $s,t\in\mathbb{N}_0$. Let
\begin{eqnarray*}
D_{T}&\triangleq &\left\{ u\in L^{2}_{(s,t)}(\ell^2,P_r)  :u\in D_{\overline{\partial }}\text{ and }\overline{\partial } u\in L^{2}_{(s,t+1)}(\ell^2,P_r)  \right\},\\
D_{S}&\triangleq &\left\{ f\in L^{2}_{\left( s,t+1\right)  }(\ell^2,P_r) :f\in D_{\overline{\partial }}\text{ and }\overline{\partial } f\in L^{2}_{(s,t+2)}(\ell^2,P_r)  \right\}.
\end{eqnarray*}
We define $Tu\triangleq  \overline{\partial } u,\,u\in D_T$ and $Sf\triangleq \overline{\partial } f,\,f\in D_S.$ Write
\begin{eqnarray*}
R_{T}&\triangleq &\left\{ Tu :u\in D_{T}  \right\},\\
N_{S}&\triangleq &\left\{ f :f\in D_{S}\text{ and }Sf=0  \right\},
\end{eqnarray*}
which are the range of $T$ and the kernel of $S$ respectively.
\end{definition}
Write
\begin{eqnarray*}
\mathscr{D}_{(s,t)}\triangleq \bigcup_{n=1}^{\infty}\Bigg\{f= \sum_{|I|=s,|J|=t,\,\max(I\cup J)\leq n}'f_{I,J}\,\mathrm{d}z^I\wedge \mathrm{d}\overline{z}^J:\;\,\,f_{I,J}\in \mathscr{C}_c^{\infty}\Bigg\}.
\end{eqnarray*}
 We have the following $L^2$ estimate for $(s,t+1)$-forms in $\mathscr{D}_{(s,t+1)}$:

\begin{theorem}\label{200207t1}
For any $f\in  \mathscr{D}_{(s,t+1)}$, it holds that $f\in D_{S}\cap D_{T^*}$ and
\begin{eqnarray}\label{202002071}
\frac{t+1}{2r^2}\cdot||f||_{L^2_{(s,t+1)}(\ell^2,P_r)}^2\leq ||T^*f||_{L^2_{(s,t)}(\ell^2,P_r)}^2
+||Sf||_{L^2_{(s,t+2)}(\ell^2,P_r)}^2.
\end{eqnarray}
\end{theorem}

\begin{proof} We borrow some idea from H\"omander \cite{Hor65}. Before proceeding, we need two simple equalities (Recall \eqref{230629def1} for $\delta_j$):
\begin{eqnarray}
\int_{\ell^2}\frac{\partial \varphi}{\partial \overline{z_j}}\cdot\overline{\psi} \,\mathrm{d}P_r
=-\int_{\ell^2}\varphi\cdot\overline{\delta_j\psi}\,\mathrm{d}P_r,\,\,\quad\forall\,j\in\mathbb{N},\,\,\,\varphi,\psi\in\mathscr{C}_c^{\infty},\label{22.1.8'}
\end{eqnarray}
and
\begin{eqnarray}
\bigg(\delta_k\frac{\partial  }{\partial \overline{z_j}}-\frac{\partial  }{\partial \overline{z_j}}\delta_k\bigg)w=\frac{w}{2r^2a_k^2}\frac{\partial\overline{z_k}}{\partial \overline{z_j}} ,\,\,\quad\forall\,k,j\in\mathbb{N},\,\,\,w\in \mathscr{C}_c^{\infty}.\label{22.1.8}
\end{eqnarray}

Suppose that $f\in \mathscr{D}_{(s,t+1)}$ and $u\in D_{T}(\subset L^2_{(s,t)}(\ell^2, P_r))$ are as follows:
\begin{eqnarray*}
f= \sum_{|I|=s,|J|=t+1}'f_{I,J}\,\mathrm{d}z^I\wedge \mathrm{d}\overline{z}^J,
\quad\,\,u= \sum_{|I|=s,|K|=t}'u_{I,K}\,\mathrm{d}z^I\wedge \mathrm{d}\overline{z}^K
\end{eqnarray*}
and $Tu= (-1)^s\sum\limits_{|I|=s,|J|=t+1}'g_{I,J}\,\mathrm{d}z^I\wedge \mathrm{d}\overline{z}^J$ such that
\begin{eqnarray}\label{2weak def of st form on Cn}
\int_{\ell^2}g_{I,J}\overline{\varphi}\,\mathrm{d}P_r=-\int_{\ell^2}\sum_{1\leq j<\infty}\sum_{|K|=t}'\varepsilon_{j, K}^{J}u_{I,K}\cdot\overline{\delta_j\varphi} \,\mathrm{d}P_r,\,\forall\;\varphi\in \mathscr{C}_c^{\infty}.
\end{eqnarray}
Then,
\begin{eqnarray*}
&&(f,T u)_{L^2_{(s,t+1)}(\ell^2)}\\
&=&(-1)^s\cdot\sum_{|I|=s,|J|=t+1}' \textbf{a}^{I,J} \cdot\int_{\ell^2}f_{I,J}\cdot\overline{g_{I,J}}\,\mathrm{d}P_r\\
&=&(-1)^{s-1}\cdot \sum_{|I|=s,|J|=t+1}' \textbf{a}^{I,J}\cdot \int_{\ell^2}\sum_{1\leq j<\infty}\sum_{|K|=t}'\varepsilon_{j, K}^{J}\cdot\delta_j(f_{I,J})\cdot\overline{u_{I,K}} \,\mathrm{d}P_r\\
&=&(-1)^{s-1}\cdot \sum_{|I|=s,|K|=t}'\textbf{a}^{I,K}\cdot\int_{\ell^2} \sum_{1\leq j<\infty}\sum_{|J|=t+1}'a_j^2\cdot\varepsilon_{j,K}^{J}\cdot \delta_j(f_{I,J})\cdot\overline{ u_{I,K}}\,\mathrm{d}P_r.
\end{eqnarray*}
Note that if $J\neq \{j\}\cup K$ then $\varepsilon_{j,K}^{J}=0$. Therefore, if $j\in K$ we let $f_{I,jK}\triangleq0$ and if $j\notin K$, there exists a unique multi-index $J$ with strictly order such that $|J|=t+1$ and $J= \{j\}\cup K$, we let $f_{I,jK}\triangleq\varepsilon_{j,K}^{J}\cdot f_{I,J}$. Therefore,
\begin{eqnarray*}
&&(f,T u)_{L^2_{(s,t+1)}(\ell^2)}\\
&=&(-1)^{s-1} \sum_{|I|=s,|K|=t}'\textbf{a}^{I,K}\cdot\int_{\ell^2} \sum_{1\leq j<\infty} a_j^2\delta_j(f_{I,jK})\cdot\overline{ u_{I,K}}\,\mathrm{d}P_r.
\end{eqnarray*}
Note that
\begin{eqnarray*}
&&\sum_{|I|=s,|K|=t}'\textbf{a}^{I,K}\int_{\ell^2}\left|\sum_{1\leq j<\infty}a_j^2\delta_j(f_{I,jK})\right|^2\,\mathrm{d}P_r\\
&=& \sum_{|I|=s,|K|=t}'\textbf{a}^{I,K}\int_{\ell^2} \sum_{1\leq j<\infty}a_j^2\delta_j(f_{I,jK})\cdot\overline{ \sum_{1\leq i<\infty}a_i^2\delta_i(f_{I,iK})}\,\mathrm{d}P_r\\
&=& \sum_{|I|=s,|K|=t}'\textbf{a}^{I,K} \sum_{1\leq i,j<\infty}\int_{\ell^2}a_j^2\delta_j(f_{I,jK})\cdot\overline{ a_i^2\delta_i(f_{I,iK})}\,\mathrm{d}P_r<\infty,
\end{eqnarray*}
which implies that $f\in D_{T^*}$,
\begin{eqnarray}\label{20231110for1}
T^*f&=&(-1)^{s-1}\sum_{|I|=s,|K|=t}'\sum_{1\leq j<\infty}a_j^2\delta_j(f_{I,jK})\mathrm{d}z^I\wedge \mathrm{d}\overline{z}^K,
\end{eqnarray}
 and
\begin{eqnarray*}
||T^*f||_{L^2_{(s,t)}(\ell^2,P_r)}^2
= \sum_{|I|=s,|K|=t}'\textbf{a}^{I,K} \sum_{1\leq i,j<\infty}\int_{\ell^2}a_j^2\delta_j(f_{I,jK})\cdot\overline{ a_i^2\delta_i(f_{I,iK})}\,\mathrm{d}P_r.
\end{eqnarray*}

On the other hand, combining (\ref{weak def of st form}), (\ref{d-dar-f-defi}) and (\ref{22.1.8'}) gives
\begin{eqnarray}\label{formula of d-bar}
Sf=(-1)^s\sum_{|I|=s,|M|=t+2}'\sum_{1\leq j<\infty}\sum_{|J|=t+1}'\varepsilon_{j,J}^{M}\cdot\frac{\partial f_{I,J}}{\partial \overline{z_j}}\,\mathrm{d}z^I\wedge \mathrm{d}\overline{z}^M.
\end{eqnarray}
Hence,
\begin{eqnarray*}
&&||Sf||_{L^2_{(s,t+2)}(\ell^2,P_r)}^2\\
&=&\sum_{|I|=s,|M|=t+2}'\textbf{a}^{I,M}\cdot\int_{\ell^2}\bigg(\sum_{1\leq j<\infty}\sum_{|K|=t+1}'\varepsilon_{j,K}^{M}\cdot\frac{\partial f_{I,K}}{\partial \overline{z_j}}\\
&&\qquad\qquad\qquad\qquad\qquad\qquad\qquad \times\;\overline{\sum_{1\leq i<\infty}\sum_{|L|=t+1}'\varepsilon_{i,L}^{M}\cdot\frac{\partial f_{I,L}}{\partial \overline{z_i}}}\;\bigg)\,\mathrm{d}P_r\\
&=&\sum_{|I|=s,|M|=t+2}'\sum_{1\leq i, j<\infty}\sum_{|K|=t+1}'\sum_{|L|=t+1}'\textbf{a}^{I,M}\int_{\ell^2}\varepsilon_{j,K}^{M}\varepsilon_{i,L}^{M}\frac{\partial f_{I,K}}{\partial \overline{z_j}}\cdot\overline{\frac{\partial f_{I,L}}{\partial \overline{z_i}}}\,\mathrm{d}P_r.
\end{eqnarray*}
Note that $\varepsilon_{j,K}^{M}\cdot\varepsilon_{i,L}^{M}\neq 0$ if and only if $j\notin K,\,i\notin L$ and $M=\{j\}\cup K=\{i\}\cup L$, in this case $\varepsilon_{j,K}^{M}\cdot\varepsilon_{i,L}^{M}=\varepsilon_{i,L}^{j,K}$. Thus we have
\begin{equation}\label{norm of Sf}
\begin{array}{ll}
\displaystyle||Sf||_{L^2_{(s,t+2)}(\ell^2,P_r)}^2\\[3mm]\displaystyle
=\sum_{|I|=s,|K|=t+1,|L|=t+1}'\sum_{1\leq i,j<\infty}\textbf{a}^{I,K}\cdot a_j^2\cdot\int_{\ell^2}\varepsilon_{i,L}^{j,K}\frac{\partial f_{I,K}}{\partial \overline{z_j}}\cdot\overline{\frac{\partial f_{I,L}}{\partial \overline{z_i}}}\,\mathrm{d}P_r.
\end{array}
\end{equation}
If $i=j$ then, $\varepsilon_{i,L}^{j,K}\neq 0$ if and only if $i\notin K$ and $K=L$. Hence, the corresponding part of (\ref{norm of Sf}) is
\begin{eqnarray*}
\sum_{|I|=s,|K|=t+1}'\sum_{ i\notin K}\textbf{a}^{I,K}\cdot a_i^2\cdot\int_{\ell^2}\bigg|\frac{\partial f_{I,K}}{\partial \overline{z_i}}\bigg|^2\,\mathrm{d}P_r.
\end{eqnarray*}
Also, if $i\neq j$ then, $\varepsilon_{i,L}^{j,K}\neq 0$ if and only if $i\in K,\,j\in L$ and $i\notin L,\,j\notin K$. Then there exists multi-index $J$ with strictly increasing order such that $|J|=t$ and $J\cup\{i,j\}=\{i\}\cup L=\{j\}\cup K$. Since $\varepsilon_{i,L}^{j,K}=\varepsilon_{i,L}^{i,j,J}\cdot\varepsilon_{i,j,J}^{j,i,J}\cdot\varepsilon_{j,i,J}^{j,K}
=-\varepsilon_{L}^{j,J}\cdot\varepsilon_{i,J}^{K}$, the rest part of (\ref{norm of Sf}) is
\begin{eqnarray*}
&&-\sum_{|I|=s,|K|=t+1,|L|=t+1,|J|=t}'\sum_{i\neq j,\,1\leq i,j<\infty}\textbf{a}^{I,K} a_j^2\int_{\ell^2}\varepsilon_{L}^{j,J}\varepsilon_{i,J}^{K}\frac{\partial f_{I,K}}{\partial \overline{z_j}}\overline{\frac{\partial f_{I,L}}{\partial \overline{z_i}}}\,\mathrm{d}P_r\\
&=&-\sum_{|I|=s,|K|=t+1,|L|=t+1,|J|=t}'\sum_{i\neq j,\,1\leq i,j<\infty}\textbf{a}^{I,K} a_j^2\int_{\ell^2}\varepsilon_{i,J}^{K}\frac{\partial f_{I,K}}{\partial \overline{z_j}}\overline{\varepsilon_{L}^{j,J}\frac{\partial f_{I,L}}{\partial \overline{z_i}}}\,\mathrm{d}P_r\\
&=&-\sum_{|I|=s,|J|=t}'\sum_{ i\neq j}\textbf{a}^{I,J} a_j^2 a_i^2\int_{\ell^2}\frac{\partial f_{I,iJ}}{\partial \overline{z_j}}\overline{\frac{\partial f_{I,jJ}}{\partial \overline{z_i}}}\,\mathrm{d}P_r.
\end{eqnarray*}
Therefore, we have
\begin{eqnarray*}
&&||T^*f||_{L^2_{(s,t)}(\ell^2,P_r)}^2
+
||Sf||_{L^2_{(s,t+2)}(\ell^2,P_r)}^2\\
&=& \sum_{|I|=s,|J|=t}'\sum_{1\leq i,j<\infty} \textbf{a}^{I,J}\cdot a_j^2\cdot a_i^2\cdot\int_{\ell^2}\delta_j(f_{I,jJ})\cdot\overline{\delta_i(f_{I,iJ})}\,\mathrm{d}P_r\\
&&+ \sum_{|I|=s,|K|=t+1}'\sum_{ i\notin K}\textbf{a}^{I,K}\cdot a_i^2\cdot\int_{\ell^2}\bigg|\frac{\partial f_{I,K}}{\partial \overline{z_i}}\bigg|^2\,\mathrm{d}P_r\\
&&- \sum_{|I|=s,|J|=t}'\sum_{ i\neq j}\textbf{a}^{I,J}\cdot a_j^2\cdot a_i^2\cdot\int_{\ell^2}\frac{\partial f_{I,iJ}}{\partial \overline{z_j}}\cdot\overline{\frac{\partial f_{I,jJ}}{\partial \overline{z_i}}}\,\mathrm{d}P\\
&=& \sum_{|I|=s,|J|=t}'\sum_{1\leq i,j<\infty} \textbf{a}^{I,J}\cdot a_j^2\cdot a_i^2\cdot\int_{\ell^2}\delta_j(f_{I,jJ})\cdot\overline{ \delta_i(f_{I,iJ})}\,\mathrm{d}P_r\\
&&+ \sum_{|I|=s,|K|=t+1}'\sum_{ 1\leq i<\infty}\textbf{a}^{I,K}\cdot a_i^2\cdot\int_{\ell^2}\bigg|\frac{\partial f_{I,K}}{\partial \overline{z_i}}\bigg|^2\,\mathrm{d}P_r\\
&&- \sum_{|I|=s,|J|=t}'\sum_{1\leq i,j<\infty}\textbf{a}^{I,J}\cdot a_j^2\cdot a_i^2\cdot\int_{\ell^2}\frac{\partial f_{I,iJ}}{\partial \overline{z_j}}\cdot\overline{\frac{\partial f_{I,jJ}}{\partial \overline{z_i}}}\,\mathrm{d}P_r\\
&=& \sum_{|I|=s,|J|=t}' \sum_{1\leq i,j<\infty}\textbf{a}^{I,J}\cdot a_j^2\cdot a_i^2\cdot\int_{\ell^2}\Bigg(\delta_j(f_{I,jJ})\cdot\overline{ \delta_i(f_{I,iJ})}\\
&&\qquad\qquad\qquad\qquad\qquad\qquad\qquad\qquad\qquad\qquad\qquad-\frac{\partial f_{I,jJ}}{\partial \overline{z_i}}\cdot\overline{\frac{\partial f_{I,iJ}}{\partial \overline{z_j}}}\Bigg)\,\mathrm{d}P_r\\
&&+ \sum_{|I|=s,|K|=t+1}'\sum_{ 1\leq i<\infty}\textbf{a}^{I,K}\cdot a_i^2\cdot\int_{\ell^2}\bigg|\frac{\partial f_{I,K}}{\partial \overline{z_i}}\bigg|^2\,\mathrm{d}P_r,
\end{eqnarray*}
where the second equality follows from
\begin{eqnarray*}
\sum_{|I|=s,|K|=t+1}'\sum_{ i\in K}\textbf{a}^{I,K}\cdot a_i^2\cdot\int_{\ell^2}\bigg|\frac{\partial f_{I,K}}{\partial \overline{z_i}}\bigg|^2\,\mathrm{d}P_r\\
=\sum_{|I|=s,|J|=t}'\sum_{1\leq i <\infty}\textbf{a}^{I,J}\cdot  a_i^4\cdot\int_{\ell^2}\bigg|\frac{\partial f_{I,iJ}}{\partial \overline{z_i}}\bigg|^2\,\mathrm{d}P_r,
\end{eqnarray*}
and the third equality follows from
\begin{eqnarray*}
&&-\sum_{|I|=s,|J|=t}'\sum_{1\leq i,j<\infty}\textbf{a}^{I,J}\cdot a_j^2\cdot a_i^2\cdot\int_{\ell^2}\frac{\partial f_{I,iJ}}{\partial \overline{z_j}}\cdot\overline{\frac{\partial f_{I,jJ}}{\partial \overline{z_i}}}\,\mathrm{d}P_r\\
&=&-\sum_{|I|=s,|J|=t}' \sum_{1\leq i,j<\infty}\textbf{a}^{I,J}\cdot a_j^2\cdot a_i^2\cdot\int_{\ell^2}\frac{\partial f_{I,jJ}}{\partial \overline{z_i}}\cdot\overline{\frac{\partial f_{I,iJ}}{\partial \overline{z_j}}}\,\mathrm{d}P_r.
\end{eqnarray*}
In view of $(\ref{22.1.8'})$ and $(\ref{22.1.8})$, we conclude that
\begin{eqnarray*}
&&||T^*f||_{L^2_{(s,t)}(\ell^2,P_r)}^2
+
||Sf||_{L^2_{(s,t+2)}(\ell^2,P_r)}^2\\
&=& \sum_{|I|=s,|J|=t}' \sum_{1\leq i,j<\infty}\textbf{a}^{I,J}\cdot a_j^2\cdot a_i^2\cdot\int_{\ell^2}\frac{1}{2r^2a_j^2}\cdot \frac{\partial\overline{z_j}}{\partial \overline{z_i}}\cdot f_{I,jJ}\cdot\overline{f_{I,iJ}}\,\mathrm{d}P_r\\
&&+ \sum_{|I|=s,|K|=t+1}'\sum_{ 1\leq i<\infty}\textbf{a}^{I,K}\cdot a_i^2\cdot\int_{\ell^2}\bigg|\frac{\partial f_{I,K}}{\partial \overline{z_i}}\bigg|^2\,\mathrm{d}P_r\\
&\geq& \frac{1}{2r^2}\sum_{|I|=s,|J|=t}' \sum_{1\leq i<\infty}\textbf{a}^{I,J}\cdot a_i^2\cdot\int_{\ell^2}|f_{I,iJ}|^2\,\mathrm{d}P_r\\
&=&\frac{t+1}{2r^2}\cdot||f||_{L^2_{(s,t+1)}(\ell^2,P_r)}^2,
\end{eqnarray*}
which gives the desired $L^2$ estimate \eqref{202002071}. This completes the proof of Theorem \ref{200207t1}.
\end{proof}
\begin{corollary}\label{20231108cor1}
$\mathscr{D}_{(s,t+1)}$  and $D_S$ are dense in $L^2_{(s,t+1)}(\ell^2,P_r)$.
\end{corollary}
\begin{proof}
For each
\begin{eqnarray*}
f= \sum_{|I|=s,|J|=t+1}'f_{I,J}\,\mathrm{d}z^I\wedge \mathrm{d}\overline{z}^J\in L^2_{(s,t+1)}(\ell^2,P_r),
\end{eqnarray*}
and given $\epsilon>0$. Since
\begin{eqnarray*}
\sum_{|I|=s,|J|=t+1}'  \textbf{a}^{I,J}\cdot \int_{\ell^2}|f_{I,J}|^2\,\mathrm{d}P_r<\infty,
\end{eqnarray*}
there exists $n_0\in\mathbb{N}$ and $r_0\in(0,+\infty)$ such that
\begin{eqnarray*}
\sum_{|I|=s,|J|=t+1,\max I\cup J>n_0}'  \textbf{a}^{I,J}\cdot \int_{\ell^2}|f_{I,J}|^2\,\mathrm{d}P_r<\frac{\epsilon^2}{9},
\end{eqnarray*}
and
\begin{eqnarray*}
\sum_{|I|=s,|J|=t+1,\max I\cup J\leq n_0}'  \textbf{a}^{I,J}\cdot \int_{\ell^2\setminus B_{r_0}}|f_{I,J}|^2\,\mathrm{d}P_r<\frac{\epsilon^2}{9},
\end{eqnarray*}
By the conclusion (1) of Proposition \ref{convolution properties}, there exists $n\geq n_0$ and $\delta\in(0,+\infty)$ such that
\begin{eqnarray*}
\sum_{|I|=s,|J|=t+1,\max I\cup J\leq n_0}'  \textbf{a}^{I,J}\cdot \int_{\ell^2 }|(f_{I,J}\cdot\chi_{B_{r_0}})_{n,\delta}-f_{I,J}\cdot\chi_{B_{r_0}}|^2\,\mathrm{d}P_r<\frac{\epsilon^2}{9}.
\end{eqnarray*}
Let
\begin{eqnarray*}
g\triangleq\sum_{|I|=s,|J|=t+1,\max I\cup J\leq n_0}'(f_{I,J}\cdot\chi_{B_{r_0}})_{n,\delta}.
\end{eqnarray*}
Then $g\in \mathscr{D}_{(s,t+1)}$ and $||f-g||_{L^2_{(s,t+1)}(\ell^2,P_r)}<\epsilon$. By Theorem \ref{200207t1}, $\mathscr{D}_{(s,t+1)}\subset D_S$ and this completes the proof of Corollary \ref{20231108cor1}.
\end{proof}

\subsection{Main Approximation Results}

\begin{proposition}\label{weak equality for more test functions}
For any $f=\sum\limits^{\prime }_{\left| I\right|  =s} \sum\limits^{\prime }_{\left| J\right|  =t} f_{I,J}\,\mathrm{d}z_{I}\wedge \,\mathrm{d}\overline{z_{J}}\in L^{2}_{\left( s,t\right)  }\left( \ell^2, P_r\right)$, the following assertions hold: If $f\in  D_{\overline{\partial } }$, then for all strictly increasing multi-indices $I$ and $K$ with $\left| I\right|=s$ and $\left| K\right|=t+1$,
\begin{equation}\label{230203e1}
\left( -1\right)^{s+1}  \int_{\ell^2} \sum^{\infty }_{i=1} \sum^{\prime }_{\left| J\right|  =t} \varepsilon^{K}_{iJ} f_{I,J}\overline{\delta_{i} \phi }\,\mathrm{d}P_r=\int_{\ell^2} \left( \overline{\partial } f\right)_{I,K}  \bar{\phi }\,\mathrm{d}P_r,
\end{equation}
for any Borel measurable function $\phi\in C_{S,b}^1(\ell^2)$ such that supp$\phi$ is a bounded subset of $\ell^2$.
\begin{proof}
By Lemma \ref{F}, there exists $\{\phi_{k}\}_{k=1}^{\infty}\subset  \mathscr{C}_c^{\infty}$ such that the conditions corresponding to (\ref{6zqqqx}) are satisfied (for which $f$ and $h_k$ therein are replaced by $\phi$ and $\phi_k$, respectively). For any $k\in\mathbb{N}$, by \eqref{weak def of st form}, we have
\begin{equation}\label{230204e4}
\left( -1\right)^{s+1}  \int_{\ell^2} \sum^{\infty }_{i=1} \sum^{\prime }_{\left| J\right|  =t} \varepsilon^{K}_{iJ} f_{I,J}\overline{\delta_{i} \left( \phi_{k} \right)  }\,\mathrm{d}P_r=\int_{\ell^2} \left( \overline{\partial } f\right)_{I,K}  \overline{\phi_{k} }\,\mathrm{d}P_r.
\end{equation}
Letting $k\rightarrow \infty $ in (\ref{230204e4}), we arrive at \eqref{230203e1}. The proof of Proposition \ref{weak equality for more test functions} is completed.
\end{proof}
 We need the following technical result.
\begin{lemma}\label{T*formula}
Assume that $ f=\sum\limits^{\prime }_{\left| I\right|  =s,\left| J\right|  =t+1} f_{I,J}dz_{I}\wedge d\overline{z_{J}} \in D_{T^{\ast }}$ and $\phi$ is a Borel measurable function $\phi\in C_{S,b}^1(\ell^2)$ such that supp$\phi$ is a bounded subset of $\ell^2$.
for some given strictly increasing multi-indices $K$ and $L$  with $\left| K\right|=s$ and $\left| L\right|=t$.
Then,
$$
\int_{\ell^2} \left( T^{\ast }f\right)_{K,L}  \phi  \,\mathrm{d}P_r=\left( -1\right)^{s}  \sum_{j=1}^{\infty}a_j^2 \int_{\ell^2}\partial_{j} \phi \cdot f_{K,jL} \,\mathrm{d}P_r.
$$
\end{lemma}
\begin{proof}
Let $g\triangleq \sum\limits^{\prime }_{\left| I_2\right|  =s,\left| L_2\right|  =t} g_{I_2,L_2}dz_{I_2}\wedge d\overline{z_{L_2}}$, where $g_{K,L}\triangleq \overline{\phi}$ and $g_{I_2,L_2}\triangleq 0$ for any other strictly increasing multi-indices $I_2$ and $L_2$ with $\left| I_2\right|=s$ and $\left| L_2\right|=t$.  Fix any $ \varphi \in \mathscr{C}_c^{\infty}$ and strictly increasing multi-indices $I_3$ and $J_3$ with $\left| I_3\right|=s$ and $\left| J_3\right|=t+1$. If $I_3\neq K$ or there does not exist $j_0\in\mathbb{N}$ such that $\{j_0\}\cup L=J_3$, then we have
$$
\left( -1\right)^{s+1}  \int_{\ell^2} \sum^{\infty }_{i=1} \sum^{\prime }_{\left| L_3\right|  =t} \varepsilon^{J_3}_{iL_3} g_{I_3,L_3}\overline{\delta_{i} \varphi }\,\mathrm{d}P_r=0.
$$
If $I_3=K$ and there exists $j_0\in\mathbb{N}$ such that $\{j_0\}\cup L=J_3$, then we have
$$
\begin{array}{ll}
\displaystyle \left( -1\right)^{s+1}  \int_{\ell^2} \sum^{\infty }_{i=1} \sum^{\prime }_{\left| L_3\right|  =t} \varepsilon^{J_3}_{iL_3} g_{K,L_3}\overline{\delta_{i} \varphi }\,\mathrm{d}P_r\\[3mm]
\displaystyle =\left( -1\right)^{s+1}  \int_{\ell^2} \varepsilon^{J_3}_{j_0L}\cdot \overline{\phi}\cdot\overline{\delta_{j_0} \varphi }\,\mathrm{d}P_r=\left( -1\right)^{s}  \int_{\ell^2} \varepsilon^{J_3}_{j_0L}\cdot \left(\overline{\partial_{j_0}}\overline{\phi}\right)\cdot\overline{ \varphi }\,\mathrm{d}P_r\\[3mm]
\displaystyle
=\left( -1\right)^{s}  \int_{\ell^2}\sum_{j=1}^{\infty} \varepsilon^{J_3}_{jL}\cdot \left(\overline{\partial_{j}}\overline{\phi}\right)\cdot\overline{ \varphi }\,\mathrm{d}P_r,
\end{array}
$$
where the second equality follows from Corollary \ref{integration by Parts or deltai}.
Since $\phi\in C_{S,b}^1(\ell^2)$, it is easy to see that $g\in L^{2}_{(s,t)}\left( \ell^2,P_r\right)$ and
$$
(-1)^s\sum^{\prime }_{\left| J\right|  =t+1} \sum_{j=1}^{\infty}\varepsilon^{J}_{jL}\overline{\partial_{j} }\overline{\phi} dz_{K}\wedge d\,\overline{z_{J}}\in L^{2}_{(s,t+1)}\left( \ell^2,P_r\right).
$$
By Definition \ref{definition of T S}, $g\in D_T$ and
$
Tg = (-1)^s\sum^{\prime }_{\left| J\right|  =t+1} \sum_{j=1}^{\infty}\varepsilon^{J}_{jL}\overline{\partial_{j} }\overline{\phi} dz_{K}\wedge d\,\overline{z_{J}}$.
By $(T^{\ast }f,g)_{L^2_{(s,t)}(\ell^2,P_r)}=
(f,Tg)_{L^2_{(s,t+1)}(\ell^2,P_r)}$, using the fact that $\varepsilon^{J}_{jL}=\varepsilon^{J}_{jL}|\varepsilon^{J}_{jL}|$, we have
$$
\begin{array}{ll}
\displaystyle
\textbf{a}^{K,L}\int_{\ell^2} \left( T^{\ast }f\right)_{K,L} \phi \,\mathrm{d}P_r
=\left( -1\right)^{s}  \sum^{\prime }_{\left| J\right|  =t+1} \sum_{j=1}^{\infty} \textbf{a}^{K,J}\varepsilon^{J}_{jL}\int_{\ell^2}f_{K,J}\cdot\overline{\overline{ \partial_{j}} \overline{\phi}}  \,\mathrm{d}P_r\\[3mm]
\displaystyle =\left( -1\right)^{s}  \sum^{\prime }_{\left| J\right|  =t+1} \sum_{j=1}^{\infty} \textbf{a}^{K,J}\varepsilon^{J}_{jL}\int_{\ell^2}f_{K,J}\cdot\partial_{j} \phi  \,\mathrm{d}P_r\\[3mm]
\displaystyle =\left( -1\right)^{s} \sum_{j=1}^{\infty} \textbf{a}^{K,L}a_j^2\int_{\ell^2}f_{K,jL}\cdot\partial_{j}\phi  \,\mathrm{d}P_r,
\end{array}
$$
which completes the proof of Lemma \ref{T*formula}.
\end{proof}
\begin{lemma}\label{20231111lem11}
Suppose that $f\in L^{2}\left( \ell^2, P_r\right)$ and for any Borel measurable function $\phi\in C_{S,b}^1(\ell^2)$ satifying $\phi=0$ on $B_{r+\frac{1}{2}}$, it holds that
\begin{equation*}
\int_{\ell^2} f  \phi  \,\mathrm{d}P_r=0.
\end{equation*}
Then we have supp$f\subset B_{r+1}$.
\end{lemma}
\begin{proof}
We only need to prove that $f\cdot \chi_{\ell^2\setminus B_{r+1}}=0$ almost everywhere respect to $P_r$. From the proof of Corollary \ref{20231108cor1},  we see that  $\mathscr{C}_c^{\infty}$ is dense in $L^{2}\left( \ell^2, P_r\right)$. Thus we only need to prove that
\begin{equation*}
\int_{\ell^2} f\cdot \chi_{\ell^2\setminus B_{r+1}}\cdot \varphi\,\mathrm{d}P_r=0,\quad\forall\,\varphi\in\mathscr{C}_c^{\infty}.
\end{equation*}
For each $n\in\mathbb{N}$, choosing $\gamma_n\in C^{\infty}(\mathbb{R})$ such that $\gamma_n(x)=1$ for all $x\leq (r+1-\frac{1}{4n})^2$, $\gamma_n(x)=0$ for all $x\geq (r+1-\frac{1}{8n})^2$ and $0\leq \gamma_n\leq 1$. Let $\psi_n(\textbf{z})=\gamma_n(||\textbf{z}||^2_{\ell^2}),\,\forall \,\textbf{z}\in\ell^2$. Then it is easy to see that $\{\psi_n\}_{n=1}^{\infty}\subset C_{S,b}^1(\ell^2)$, $\psi_n\cdot \varphi$ is Borel measurable, $\psi_n\cdot \varphi\in C_{S,b}^1(\ell^2)$, $\psi_n\cdot \varphi=0$ on $B_{r+\frac{1}{2}}$ and
\begin{equation*}
\lim_{n\to\infty}\psi_n(\textbf{z})=\chi_{\ell^2\setminus B_{r+1}}(\textbf{z}),\quad\forall\, \textbf{z}\in\ell^2,\,n\in\mathbb{N}.
\end{equation*}
These imply that
\begin{equation*}
\int_{\ell^2} f\cdot \chi_{\ell^2\setminus B_{r+1}}\cdot \varphi\,\mathrm{d}P_r
=\lim_{n\to\infty}\int_{\ell^2} f\cdot \psi_n \cdot \varphi\,\mathrm{d}P_r=0.
\end{equation*}
The proof of Lemma \ref{20231111lem11} is completed.
\end{proof}
\begin{lemma}\label{20231110lem1}
If $f\in D_{S}$ and supp$f\subset B_r$ for some $r\in(0,+\infty)$, then supp$Sf\subset B_{r+1}$; If $f\in D_{T^*}$ and supp$f\subset B_r$ for some $r\in(0,+\infty)$, then supp$T^*f\subset B_{r+1}$.
\end{lemma}
\begin{proof}
Suppose that
\begin{equation*}
f=\sum\limits^{\prime }_{\left| I\right|  =s} \sum\limits^{\prime }_{\left| J\right|  =t+1} f_{I,J}\,\mathrm{d}z_{I}\wedge \,\mathrm{d}\overline{z_{J}}\in L^{2}_{\left( s,t+1\right)  }\left( \ell^2, P_r\right).
\end{equation*}
For each Borel measurable function $\phi\in C_{S,b}^1(\ell^2)$ such that $\phi=0$ on $B_{r+\frac{1}{2}}$ and multi index $I$ and $K$ such that $|I|=s,|K|=t+2$, by Proposition \ref{weak equality for more test functions},
\begin{equation*}
\int_{\ell^2} \left(Sf\right)_{I,K}  \bar{\phi }\,\mathrm{d}P_r=\left( -1\right)^{s+1}  \int_{\ell^2} \sum^{\infty }_{i=1} \sum^{\prime }_{\left| J\right|  =t} \varepsilon^{K}_{iJ} f_{I,J}\overline{\delta_{i} \phi }\,\mathrm{d}P_r=0.
\end{equation*}
By Lemma \ref{20231111lem11}, supp$Sf\subset B_{r+1}$. Similarly, by Lemma \ref{T*formula} we can deduce that supp$T^*f\subset B_{r+1}$. The proof of Lemma \ref{20231110lem1} is completed.
\end{proof}
\end{proposition}
The following result will be useful later.
\begin{proposition}\label{general multiplitier}
Suppose that $\{X_{k}\}_{k=1}^{\infty}$ is the sequence of cut-off functions given in Theorem \ref{230215Th1}. Then for any $k\in \mathbb{N}$, it holds that
\begin{itemize}
		\item[$(1)$]  If $f\in D_{T}$, then $X_k\cdot f\in D_{T}$ and $T\left(X_k\cdot f\right)  =X_k\cdot (Tf)+(TX_k) \wedge f,$ where
$$
(TX_k)\wedge f\triangleq \left( -1\right)^{s}  \sum^{\prime }_{\left| I\right|  =s} \sum^{\prime }_{\left| K\right|  =t+1} \left(\sum^{\infty }_{i=1} \sum^{\prime }_{\left| J\right|  =t} \varepsilon^{K}_{iJ} f_{I,J}\overline{\partial_{i} } X_k \right)dz_{I}\wedge d\overline{z_{K}};
$$

\item[$(2)$] If $g\in D_{T^{\ast }}$, then $X_k\cdot g\in D_{{T}^{\ast }}$.
\end{itemize}
\end{proposition}
\begin{proof}
(1) Note that for any $k\in\mathbb{N}$, it is easy to see that $X_k$ is a Borel measurable function, $X_k\in C_{S,b}^1(\ell^2)$ and supp$X_k$ is a bounded subset of $\ell^2$. Hence, for any strictly increasing multi-indices $I$ and $K$ with $\left| I\right|  =s$ and $\left| K\right|  =t+1$ and $\phi\in \mathscr{C}_c^{\infty}$, the function $X_k\cdot \phi$ satisfies the assumptions in Proposition \ref{weak equality for more test functions}, by Proposition \ref{weak equality for more test functions}, we have
\begin{eqnarray*}
&&\left( -1\right)^{s+1}  \int_{\ell^2} \sum^{\infty }_{i=1} \sum^{\prime }_{\left| J\right|  =t} \varepsilon^{K}_{iJ} X_k f_{I,J}\overline{\delta_{i} \phi } \,\mathrm{d}P_r\\
&=&\left( -1\right)^{s+1}  \int_{\ell^2} \sum^{\infty }_{i=1} \sum^{\prime }_{\left| J\right|  =t} \varepsilon^{K}_{iJ} f_{I,J}\overline{(\delta_{i} \left( X_k  \phi \right)  -(\partial_{i}X_k )\cdot \phi }\,\mathrm{d}P_r\\
&=&\int_{\ell^2} X_k (T f)_{I,K}\overline{\phi } \,\mathrm{d}P_r+\left( -1\right)^{s}  \int_{\ell^2} \sum^{\infty }_{i=1} \sum^{\prime }_{\left| J\right|  =t} \varepsilon^{K}_{iJ} f_{I,J}(\overline{\partial_{i} } X_k) \cdot \overline{\phi } \,\mathrm{d}P_r\\
&=&\int_{\ell^2} X_k (T f)_{I,K}\overline{\phi }\,\mathrm{d}P_r+  \int_{\ell^2}((TX_k )\wedge f)_{I,K} \cdot \overline{\phi } \,\mathrm{d}P_r.
\end{eqnarray*}
$$
		\begin{array}{ll}
\displaystyle
			\end{array}
$$
Note that
\begin{equation}\label{general inequality for Teta wedege f}
\begin{array}{ll}
\displaystyle\left| \left|(T X_k) \wedge f\right|  \right|_{L^{2}_{\left( s,t+1\right)  }\left( \ell^2, P_r\right)}^2  \\[3mm]
\displaystyle\leq  \left( t+1\right)  \sum^{\prime }_{\left| I\right|  =s,\left| K\right|  =t+1} \textbf{a}^{I,K}\int_{\ell^2} \sum^{\infty }_{i=1} \sum^{\prime }_{\left| J\right|  =t} |\varepsilon^{K}_{iJ} f_{I,J}\overline{\partial_{i} } X_k |^{2}\,\mathrm{d}P_r \\[3mm]
\displaystyle\leq\left( t+1\right)  \sum^{\prime }_{\left| I\right|  =s}\int_{\ell^2} \sum^{\infty }_{i=1} \sum^{\prime }_{\left| J\right|  =t}a_i^2|\overline{\partial_{i} } X_k|^2\cdot \textbf{a}^{I,J}\cdot|f_{I,J}|^{2} \,\mathrm{d}P_r \\[3mm]
\displaystyle= \left( t+1\right) \cdot \int_{\ell^2} \left(\sum^{\infty }_{i=1}a_i^2 |\overline{\partial_{i} } X_k |^{2} \right)\cdot \left(\sum^{\prime }_{\left| I\right|  =s}\sum^{\prime }_{\left| J\right|  =t}\textbf{a}^{I,J}\cdot\left| f_{I,J}\right|^{2}\right)  \,\mathrm{d}P_r\\[3mm]
\displaystyle\leq \left( t+1\right) \cdot\left(\sup_{\ell^2} \sum^{\infty }_{i=1} a_i^2|\overline{\partial_{i} } X_k |^{2}\right)\cdot \left| \left| f\right|  \right|^{2}_{L^{2}_{\left( s,t\right)  }\left( \ell^2,P_r\right)  }\\[3mm]
\displaystyle \leq \left( t+1\right) \cdot C^2\cdot \left| \left| f\right|  \right|^{2}_{L^{2}_{\left( s,t\right)  }\left( \ell^2,P_r\right)  }<\infty,
\end{array}
\end{equation}
where the last inequality follows from the conclusion (2) of Theorem \ref{230215Th1} and the real number $C$ is given in the conclusion (2) of Theorem \ref{230215Th1}. Also, $||X_k\cdot(Tf)||_{L^{2}_{\left( s,t+1\right)  }\left( \ell^2,P_r\right)  }  \leqslant  ||Tf||_{L^{2}_{\left( s,t+1\right)  }\left( \ell^2,P_r\right)  }$. Therefore, $X_k \cdot f \in D_{T}$ and $T ( X_k\cdot f )  =X_k\cdot (Tf)+(TX_k) \wedge f.$

\medskip

(2) Since $\overline{X_k}=X_k$. For any $u\in D_{T}$, by the conclusion (1) that we have proved, it holds that $\overline{X_k}\cdot u\in D_T$ and $T\left( \overline{X_k}\cdot u\right)  =\overline{X_k}\cdot (Tu)+(T\overline{X_k}) \wedge u$ and hence
$$
\begin{array}{ll}
\displaystyle
			\left( Tu,X_k \cdot g\right)_{L^{2}_{(s,t+1)}\left( \ell^2,P_r\right)  }  =\left( \overline{X_k }\cdot (Tu),g\right)_{L^{2}_{(s,t+1)}\left( \ell^2, P_r\right)  } \\[3mm]
\displaystyle =\left(T\left( \overline{X_k}\cdot u\right),g\right)_{L^{2}_{\left( s,t+1\right)  }\left( \ell^2,P_r\right)  }  -\left( (T\overline{X_k}) \wedge u,g\right)_{L^{2}_{\left( s,t+1\right)  }\left( \ell^2,P_r\right)  } \\[3mm]
\displaystyle =\left( u,X_k\cdot T^{\ast }g\right)_{L^{2}_{\left( s,t\right)  }\left( \ell^2,P_{r}\right)  }  -\left( (T\overline{X_k}) \wedge u,g\right)_{L^{2}_{\left( s,t+1\right)  }\left( \ell^2,P_{r}\right)  }.
		\end{array}
$$
By \eqref{general inequality for Teta wedege f}, we have
\begin{eqnarray*}
&&| ( Tu,X_k\cdot g )_{L^{2}_{(s,t+1)}\left( \ell^2,P_{r}\right)  }  |\\
&&\leqslant |( u,X_k\cdot T^{\ast }g)_{L^{2}_{\left( s,t\right)  }\left(\ell^2,P_{r}\right)  }| +|( (T\overline{X_k}) \wedge u,g)_{L^{2}_{\left( s,t+1\right)  }\left(\ell^2,P_{r}\right)  }|\\
&&\leqslant \left(  || X_k\cdot T^{\ast }g||_{L^{2}_{\left( s,t\right)  }\left( \ell^2,P_{r}\right)  }  +\sqrt{t+1}\cdot C\cdot\left| \left| g\right|  \right|_{L^{2}_{\left( s,t+1\right)  }\left( \ell^2,P_{r}\right)  }  \right) \cdot \left| \left| u\right|  \right|_{L^{2}_{\left( s,t\right)  }\left( \ell^2,P_{r}\right)  }\\
&&\leqslant \left(  || T^{\ast }g ||_{L^{2}_{\left( s,t\right)  }\left( \ell^2,P_{r}\right)  }  +\sqrt{t+1}\cdot C  \cdot\left| \left| g\right|  \right|_{L^{2}_{\left( s,t+1\right)  }\left( \ell^2,P_{r}\right)  }  \right) \cdot \left| \left| u\right|  \right|_{L^{2}_{\left( s,t\right)  }\left( \ell^2,P_{r}\right)  }.
\end{eqnarray*}
By the definition of $T^*$ (e.g., \cite[Definition 4.2.2, p. 177]{Kra}), we have $X_k\cdot g\in D_{T^{\ast }}$ and this completes the proof of Proposition \ref{general multiplitier}.
\end{proof}
\begin{proposition}\label{230924prop1}
Suppose that $s,t\in\mathbb{N}_0$, $T,S$ are operators defined in Definition \ref{definition of T S} and $\{X_{k}\}_{k=1}^{\infty}$ is the sequence of cut-off functions given in Theorem \ref{230215Th1}. Then for each $f\in D_{S}\bigcap D_{T^{\ast }}$, it holds that $X_k\cdot f\in D_{S}\bigcap D_{T^{\ast }}$ for each $k\in\mathbb{N}$ and the following sequence tends to $0$ as $k\to\infty$,
$$
 \left| \left| S\left( X_{k}\cdot f\right)  -X_{k}\cdot(Sf)\right|  \right|_{3}  +\left| \left| T^{\ast }\left( X_{k}\cdot f\right)  -X_{k }\cdot(T^{\ast }f)\right|  \right|_{1}  +\left| \left| X_{k}\cdot f-f\right|  \right|_{2}
$$
where we denote $ \left| \left| \cdot \right|  \right|_{L^{2}_{\left( s,t+i-1\right)  }\left( \ell^2,P_r\right)  }$ briefly by $\left| \left| \cdot \right|  \right|_{i}$ for each $i=1,2,3$.
\end{proposition}
\begin{proof}
Similar to the proof of the conclusion (1) of Proposition \ref{general multiplitier}, one can show that $X_{k}\cdot f\in D_{S}$  for each $k\in\mathbb{N}$ and $S\left( X_{k}\cdot f\right)  =X_{k}\cdot (Sf)+(SX_{k}) \wedge f,$ where
$$
(SX_{k})\wedge f\triangleq \left( -1\right)^{s}  \sum^{\prime }_{\left| I\right|  =s} \sum^{\prime }_{\left| K\right|  =t+2} \left(\sum^{\infty }_{i=1} \sum^{\prime }_{\left| J\right|  =t+1} \varepsilon^{K}_{iJ} f_{I,J}\overline{\partial_{i} } X_{k} \right)dz_{I}\wedge d\overline{z_{K}}.
$$
By the conclusion (2) of Proposition \ref{general multiplitier}, we have $X_k\cdot f\in D_{T^{\ast }}$, and hence $X_k\cdot f\in D_{S}\bigcap D_{T^{\ast }}$, and
\begin{eqnarray*}
&&\left| \left| S\left( X_{k}\cdot f\right)  -X_{k}\cdot(Sf)\right|  \right|^{2}_{3 }\\
&=&\sum^{\prime }_{\left| I\right|  =s}\sum^{\prime }_{\left| K\right|  =t+2} \textbf{a}^{I,K}\int_{\ell^2} \left|\sum^{\infty }_{i=1} \sum^{\prime }_{\left| J\right|  =t+1} \varepsilon^{K}_{i,J} f_{I,J}\overline{\partial_{i} } X_{k }\right|^{2}\,\mathrm{d}P_r\\
&\leqslant &(t+2)\sum^{\prime }_{\left| I\right|  =s}\sum^{\prime }_{\left| K\right|  =t+2}\textbf{a}^{I,K}\int_{\ell^2}\sum^{\infty }_{i=1} \sum^{\prime }_{\left| J\right|  =t+1} \left|\varepsilon^{K}_{i,J} f_{I,J}\overline{\partial_{i} } X_{k }\right|^{2} \,\mathrm{d}P_r\\
&=&(t+2)\sum^{\prime }_{\left| I\right|  =s}\sum^{\prime }_{\left| K\right|  =t+2}\int_{\ell^2}\sum^{\infty }_{i=1} \sum^{\prime }_{\left| J\right|  =t+1} \frac{\textbf{a}^{I,K}\cdot|\varepsilon^{K}_{i,J}|}{\textbf{a}^{I,J}} \cdot \textbf{a}^{I,J}\cdot |f_{I,J}|^2\cdot|\overline{\partial_{i} }X_{k } |^{2} \,\mathrm{d}P_r\\
&\leq &(t+2)\sum^{\prime }_{\left| I\right|  =s}\int_{\ell^2}\sum^{\infty }_{i=1} \sum^{\prime }_{\left| J\right|  =t+1} a_i^2 \cdot \textbf{a}^{I,J}\cdot |f_{I,J}|^2\cdot|\overline{\partial_{i} }X_{k } |^{2} \,\mathrm{d}P_r\\
&\leq &(t+2)\cdot\sum^{\prime }_{\left| I\right|  =s}\sum^{\prime }_{\left| J\right|  =t+1}\textbf{a}^{I,J}\cdot\int_{\ell^2} |f_{I,J}|^2\cdot\left(\sum^{\infty }_{i=1} a_i^2|\overline{\partial_{i} }X_{k } |^{2} \right)\,\mathrm{d}P_r.
\end{eqnarray*}
Recalling Theorem \ref{230215Th1}, we have $\sum\limits^{\infty }_{i=1} |\overline{\partial_{i} } X_{k}|^{2}\leqslant C^2$ for all $k\in\mathbb{N}.$ Thus
$$
		 \sum^{\prime }_{\left| I\right|  =s,\left| J\right|  =t+1} \textbf{a}^{I,J} \left| f_{I,J}\right|^{2} \cdot\left(\sum^{\infty }_{i=1} a_i^2|\overline{\partial_{i} }X_{k } |^{2} \right) \leqslant C^2\cdot\sum^{\prime }_{\left| I\right|  =s,\left| J\right|  =t+1} \textbf{a}^{I,J}  \left| f_{I,J}\right|^{2}  .
$$
By Lebesgue's Dominated Convergence Theorem, it follows that
$$
		\lim_{k\rightarrow \infty }\sum^{\prime }_{\left| I\right|  =s}\sum^{\prime }_{\left| J\right|  =t+1}\textbf{a}^{I,J}\cdot\int_{\ell^2} |f_{I,J}|^2\cdot\left(\sum^{\infty }_{i=1} a_i^2|\overline{\partial_{i} }X_{k } |^{2} \right)\,\mathrm{d}P_r=0,
		$$
and hence $\lim\limits_{k\rightarrow \infty } \left| \left| S\left( X_{k}\cdot f\right)  -X_{k}\cdot(Sf)\right|  \right|^{2}_{3}  =0.$

On the other hand, for every $u\in D_{T}$ and $k>N_1$(Recall \eqref{230409eq1}), by Proposition \ref{general multiplitier}, we obtain that
\begin{eqnarray*}
&&\left|\left( T^{\ast }\left( X_{k}\cdot f\right)  -X_{k}\cdot(T^{\ast }f),u\right)_{L^{2}_{\left( s,t\right)  }\left( \ell^2,P_r\right)  }  \right|\\
&=&\left|\left(f, X_{k}\cdot(Tu)\right)_{L^{2}_{\left( s,t+1\right)  }\left( \ell^2,P_{r}\right)  }  -\left( f,T(X_{k}\cdot u)\right)_{L^{2}_{\left( s,t+1\right)  }\left( \ell^2,P_r\right)  }  \right|\\
&&=\left|\left( f,(T X_{k})\wedge u\right)_{L^{2}_{\left( s,t+1\right)  }\left(\ell^2,P_r\right)  }  \right|\\
&=&\left| \sum^{\prime }_{\left| I\right|  =s}\sum^{\prime }_{\left| J\right|  =t+1} \textbf{a}^{I,J}\int_{\ell^2} f_{I,J}\overline{((TX_{k})\wedge u)_{I,J}} \,\mathrm{d}P_r\right|  \\
&\leqslant& \sum^{\prime }_{\left| I\right|  =s}\sum^{\prime }_{\left| J\right|  =t+1} \textbf{a}^{I,J}\int_{\ell^2} \left|f_{I,J} \right|\cdot \left|((TX_{k})\wedge u)_{I,J}\right|\,\mathrm{d}P_r\\
&\leqslant &\int_{\ell^2} \left(\sum^{\prime }_{\left| I\right|  =s}\sum^{\prime }_{\left| J\right|  =t+1}\textbf{a}^{I,J} |f_{I,J}|^{2} \right)^{\frac{1}{2} }\\&&
\times\left(\sum^{\prime }_{\left| I\right|  =s}\sum_{\left| J\right|  =t+1}^{\prime }  \textbf{a}^{I,J}|((TX_{k})\wedge u)_{I,J}|^{2} \right)^{\frac{1}{2} }\,\mathrm{d}P_r\\
&\leqslant &\sqrt{(t+1)\cdot C^2}  \cdot\int_{\ell^2\setminus K_{k-N_1}} \left(\sum^{\prime }_{\left| I\right|  =s}\sum^{\prime }_{\left| J\right|  =t+1} \textbf{a}^{I,J} |f_{I,J}|^{2} \right)^{\frac{1}{2} }\\&&\times \left(\sum^{\prime }_{\left| I\right|  =s}\sum_{\left| L\right|  =t}^{\prime } \textbf{a}^{I,J}|u_{I,L}|^{2} \right)^{\frac{1}{2} }\,\mathrm{d}P_r\\
&&\leqslant \sqrt{(t+1)\cdot C^2} \cdot\left(\int_{\ell^2\setminus K_{k-N_1}} \sum^{\prime }_{\left| I\right|  =s}\sum^{\prime }_{\left| J\right|  =t+1} \textbf{a}^{I,J}|f_{I,J}|^{2} \,\mathrm{d}P_r\right)^{\frac{1}{2} }\\&&\times \left| \left| u\right|  \right|_{L^{2}_{\left( s,t\right)  }\left( \ell^2, P_r\right)  } ,
\end{eqnarray*}
where the third inequality follows from the facts that $((TX_{k})\wedge u)_{I,J}=0$ on $K_{k-N_1}$ for all $I,J$, and
$$
		\begin{array}{ll}
\displaystyle
			\sum_{\left| I\right|  =s}\sum^{\prime }_{\left| J\right|  =t+1}\textbf{a}^{I,J} |((T X_{k })\wedge u)_{I,J}|^{2} =\sum^{\prime }_{\left| I\right|  =s}\sum^{\prime }_{\left| J\right|  =t+1} \textbf{a}^{I,J}\left|\sum^{\infty }_{i=1} \sum^{\prime }_{\left| L\right|  =t} \varepsilon^{J}_{iL} u_{I,L}\overline{\partial_{i} } X_{k}\right|^{2} \\[3mm]
\displaystyle \leq \left( t+1\right)\cdot  \sum^{\prime }_{\left| I\right|  =s}\sum^{\prime }_{\left| J\right|  =t+1}\textbf{a}^{I,J} \sum^{\infty }_{i=1} \sum^{\prime }_{\left| L\right|  =t} |\varepsilon^{J}_{iL} u_{I,L}\overline{\partial_{i} } X_{k}|^{2} \\[3mm]
\displaystyle \leq \left( t+1\right)\cdot  \sum^{\prime }_{\left| I\right|  =s}\sum^{\prime }_{\left| J\right|  =t+1}\sum^{\infty }_{i=1} \sum^{\prime }_{\left| L\right|  =t} \frac{\textbf{a}^{I,J} |\varepsilon^{J}_{iL}|}{\textbf{a}^{I,L}}\cdot \textbf{a}^{I,L} \cdot|u_{I,L}|^2\cdot|\overline{\partial_{i} } X_{k}|^{2} \\[3mm]
\displaystyle \leq \left( t+1\right) \cdot \sum^{\prime }_{\left| I\right|  =s} \sum^{\infty }_{i=1} \sum^{\prime }_{\left| L\right|  =t} a_i^2\cdot \textbf{a}^{I,L} \cdot|u_{I,L}|^2\cdot|\overline{\partial_{i} } X_{k}|^{2} \\[3mm]
\displaystyle \leq \left( t+1\right)\cdot  \sum^{\prime }_{\left| I\right|  =s} \sum^{\prime }_{\left| L\right|  =t} \textbf{a}^{I,L}|u_{I,L}|^2\left(\sum^{\infty }_{i=1}a_i^2|\overline{\partial_{i} } X_{k}|^{2}\right)\\[3mm]
\displaystyle \leq \left( t+1\right)\cdot C^2\cdot  \sum^{\prime }_{\left| I\right|  =s} \sum^{\prime }_{\left| L\right|  =t} \textbf{a}^{I,L}|u_{I,L}|^2 .
		\end{array}
$$
Similar to the proof Corollary \ref{20231108cor1}, it is easy to see that $D_T$ is dense in $L^{2}_{\left( s,t\right)  }\left(\ell^2, P_r\right)$, from which we deduce that
\begin{eqnarray*}
&&\left| \left| T^{\ast }\left( X_{k}\cdot f\right)  -X_{k}\cdot( T^{\ast }f)\right|  \right|_{L^{2}_{(s,t)}\left( \ell^2, P_r\right)  }\\
&\leq&\sqrt{(t+1)\cdot C^2}  \left(\sum^{\prime }_{\left| I\right|  =s} \sum^{\prime }_{\left| J\right|  =t+1} \textbf{a}^{I,J}\int_{\ell^2\setminus K_{k-N_1}} |f_{I,J}|^{2} \,\mathrm{d}P_r\right)^{\frac{1}{2} }.
\end{eqnarray*}
By Lemma \ref{basic propeties of Knm}, we note that the right side of the above inequality tends to zero as $k\to\infty$, which implies that $\lim\limits_{k\rightarrow \infty } \left| \left| T^{\ast }\left( X_{k}\cdot f\right)  -X_{k} \cdot(T^{\ast }f)\right|  \right|_{L^{2}_{(s,t)}\left( \ell^2, P_r\right)  }  =0.$

Since $|X_{k}|\leq1$, we have
$$
\sum\limits^{\prime }_{\left| I\right|  =s}\sum\limits^{\prime }_{\left| J\right|  =t+1} \textbf{a}^{I,J}|\left( X_{k}-1\right)  f_{I,J}|^{2} \leq 4 \sum\limits^{\prime }_{\left| I\right|  =s}\sum\limits^{\prime }_{\left| J\right|  =t+1} \textbf{a}^{I,J}\left| f_{I,J}\right|^{2} .
$$
We also note that, by the conclusion (2) of Theorem \ref{230215Th1}, $\lim\limits_{k \rightarrow \infty }  X_{k } =1$ almost everywhere respect to $P_r$.
By Lebesgue's Dominated Convergence Theorem, we arrive at
$$
		\begin{array}{ll}
\displaystyle
			\lim_{k\rightarrow \infty } \left| \left| X_{k} \cdot f-f\right|  \right|^{2}_{L^{2}_{(s,t+1)}\left( \ell^2,P_r\right)  }\\[3mm]
\displaystyle=\lim_{k \rightarrow \infty } \int_{\ell^2} \sum^{\prime }_{\left| I\right|  =s}\sum^{\prime }_{\left| J\right|  =t+1} \textbf{a}^{I,J}|\left( X_{k}-1\right)  f_{I,J}|^{2} \,\mathrm{d}P_r \\[3mm]
\displaystyle = \int_{\ell^2} \lim_{k\rightarrow \infty }\sum^{\prime }_{\left| I\right|  =s}\sum^{\prime }_{\left| J\right|  =t+1} \textbf{a}^{I,J}|\left( X_{k}-1\right)  f_{I,J}|^{2} \,\mathrm{d}P_r=0 .
		\end{array}
$$
The proof of Proposition \ref{230924prop1} is completed.
\end{proof}
Now we need some notations for our furthermore arguments. For $f\in L^2_{(s,t+1)}(\ell^2,P_r)$ such that $f\in D_{T^*}\cap D_S$ and supp$f$ is a compact subset of $\ell^2$,
\begin{eqnarray*}
f= \sum_{|I|=s,|J|=t+1}'f_{I,J}\,\mathrm{d}z^I\wedge \mathrm{d}\overline{z}^J,
\end{eqnarray*}
where $f_{I,J}\in L^2(\ell^2,P_r)$ and for each $I$ and $J$. For each $n\in\mathbb{N}$ and $\delta\in(0,1)$, let
\begin{eqnarray}\label{f_{ndelta}}
f_{n,\delta}\triangleq \sum_{|I|=s,|J|=t+1,\max(I\cup J)\leqslant n}'(f_{I,J})_{n,\delta}\,\mathrm{d}z^I\wedge \mathrm{d}\overline{z}^J,
\end{eqnarray}
and
\begin{eqnarray}\label{f_n}
f_{n}\triangleq\sum_{|I|=s,|J|=t+1,\max(I\cup J)\leqslant n}'(f_{I,J})_n\,\mathrm{d}z^I\wedge \mathrm{d}\overline{z}^J,
\end{eqnarray}
and for multi index $I$ and $J$ such that $|I|=s,|J|=t+1,\max(I\cup J)\leqslant n$,
$(f_{I,J})_{n,\delta  }$ is defined in the Definition \ref{Convolution after reduce dimension}.
\begin{remark}\label{(f{n,delta}){I,J}}
For multi index $I$ and $J$ such that $|I|=s,|J|=t+1$, we use $(f_{n,\delta})_{I,J}$ to denote the $I,J$ component of $f_{n,\delta}$, then if $\max(I\cup J)\leqslant n$, $(f_{n,\delta})_{I,J}=(f_{I,J})_{n,\delta  }$ and if $\max(I\cup J)>n$, $(f_{n,\delta})_{I,J}=0.$
\end{remark}
\begin{theorem}\label{main approximation theoremhs}
Suppose that $f\in D_{S}\bigcap D_{T^{\ast }}$ and supp$f$ is a compact subset of $\ell^2$. It holds that $\{f_{n,\delta}:n\in\mathbb{N},\,\delta\in(0,+\infty)\}\subset \mathscr{D}_{(s,t+1)}\subset D_{T^*}\cap D_S$ and
\begin{itemize}
\item[$(1)$] $\lim\limits_{n\to\infty}f_n=f$ and $\lim\limits_{\delta\to 0}f_{n,\delta}=f_n$ in $L^2_{(s,t+1)}(\ell^2,P_r)$ for each $n\in\mathbb{N}$ ;
\item[$(2)$] $\lim\limits_{n\to\infty}(Sf)_n=Sf$ and $\lim\limits_{\delta\to 0}S(f_{n,\delta})=(Sf)_n$ in $L^2_{(s,t+2)}(\ell^2,P_r)$ for each $n\in\mathbb{N}$;
\item[$(3)$] $\lim\limits_{n\to\infty}(T^*f)_n=T^*f$ and $\lim\limits_{\delta\to 0}T^*(f_{n,\delta})=(T^*f)_n$ in $L^2_{(s,t)}(\ell^2,P_r)$ for each $n\in\mathbb{N}$.
\end{itemize}
\end{theorem}
\begin{proof}
Suppose that
\begin{eqnarray*}
f&=& \sum_{|I|=s,|J|=t+1}'f_{I,J}\,\mathrm{d}z^I\wedge \mathrm{d}\overline{z}^J,\\
Sf&=& (-1)^s\sum_{|I|=s,|K|=t+2}'g_{I,K}\,\mathrm{d}z^I\wedge \mathrm{d}\overline{z}^K.
\end{eqnarray*}
(1) Note that
\begin{eqnarray*}
&&||f_n-f||^2_{L^2_{(s,t+1)}(\ell^2)}\\
&=& \sum_{|I|=s,|J|=t+1,\max(I\cup J)\leq n}'\textbf{a}^{I,J}\cdot \int_{\ell^2}|(f_{I,J})_n-f_{I,J}|^2\,\mathrm{d}P_r+\\
&& \sum_{|I|=s,|J|=t+1,\max(I\cup J)> n}'\textbf{a}^{I,J}\cdot \int_{\ell^2}|f_{I,J}|^2\,\mathrm{d}P_r.
\end{eqnarray*}
Given $\epsilon>0$, there exists $n_0\in\mathbb{N}$ such that
\begin{eqnarray*}
\sum_{|I|=s,|J|=t+1,\max(I\cup J)> n_0}'4\cdot\textbf{a}^{I,J}\cdot \int_{\ell^2}|f_{I,J}|^2\,\mathrm{d}P_r<\frac{\epsilon}{2}.
\end{eqnarray*}
By (2) of Proposition \ref{convolution properties}, we have
\begin{eqnarray*}
\lim_{n\to\infty} \sum_{|I|=s,|J|=t+1,\max(I\cup J)\leq n_0}'\textbf{a}^{I,J}\cdot \int_{\ell^2}|(f_{I,J})_n-f_{I,J}|^2\,\mathrm{d}P_r=0.
\end{eqnarray*}
Then there exists $n_1\geq n_0$ such that for all $n\geq n_1$ it holds that
\begin{eqnarray*}
 \sum_{|I|=s,|J|=t+1,\max(I\cup J)\leq n_0}'\textbf{a}^{I,J}\cdot \int_{\ell^2}|(f_{I,J})_n-f_{I,J}|^2\,\mathrm{d}P_r<\frac{\epsilon}{2}
\end{eqnarray*}
and hence
\begin{eqnarray*}
&&||f_n-f||^2_{L^2_{(s,t+1)}(\ell^2)}\\
&=&\sum_{|I|=s,|J|=t+1,\max(I\cup J)\leq n_0}'\textbf{a}^{I,J}\cdot \int_{\ell^2}|(f_{I,J})_n-f_{I,J}|^2\,\mathrm{d}P_r+\\
&&\sum_{|I|=s,|J|=t+1,n_0<\max(I\cup J)\leq n}'\textbf{a}^{I,J}\cdot \int_{\ell^2}|(f_{I,J})_n-f_{I,J}|^2\,\mathrm{d}P_r+\\
&&\sum_{|I|=s,|J|=t+1,\max(I\cup J)> n}'\textbf{a}^{I,J}\cdot \int_{\ell^2}|f_{I,J}|^2\,\mathrm{d}P_r\\
&\leq&\sum_{|I|=s,|J|=t+1,\max(I\cup J)\leq n_0}'\textbf{a}^{I,J}\cdot \int_{\ell^2}|(f_{I,J})_n-f_{I,J}|^2\,\mathrm{d}P_r+\\
&&\sum_{|I|=s,|J|=t+1,\max(I\cup J)> n_0}'4\cdot\textbf{a}^{I,J}\cdot \int_{\ell^2}|f_{I,J}|^2\,\mathrm{d}P_r\\
&<&\epsilon.
\end{eqnarray*}
Therefore, $\lim\limits_{n\to\infty}f_n=f$ in $L^2_{(s,t+1)}(\ell^2, P_r)$. By Lemma \ref{20231110lem1}, supp$Sf$ and supp$T^*f$ are bounded subsets of $\ell^2$, similarly, one can prove that\\ $\lim\limits_{n\to\infty}(Sf)_n=Sf$ in $L^2_{(s,t+2)}(\ell^2, P_r)$ and
$\lim\limits_{n\to\infty}(T^*f)_n=T^*f$ in $L^2_{(s,t)}(\ell^2, P_r)$. Since supp$f$ is a compact subset of $\ell^2$, by Proposition \ref{convolution properties}, $\lim\limits_{\delta\to 0}(f_{n,\delta})_{I,J}=(f_n)_{I,J}$ in $L^2(\ell^2, P_r)$ for each multi index $I$ and $J$ such that $|I|=s,|J|=t+1,\max(I\cup J)\leq n$ and hence $\lim\limits_{\delta\to 0}f_{n,\delta}=f_n$ in $L^2_{(s,t+1)}(\ell^2, P_r)$.

(2) Since  $f_{n,\delta}\in \mathscr{D}_{(s,t+1)}$, by Theorem \ref{200207t1}, $f_{n,\delta}\in D_{S}$ and
\begin{eqnarray}\label{S(f_ndelta)}
S(f_{n,\delta})&=&(-1)^s\sum_{|I|=s,|K|=t+2,\max(I\cup K)\leq n}'\sum_{1\leq j\leq n}\sum_{|J|=t+1,\max J\leq n}'\varepsilon_{j, J}^{K}\\
&&\times\frac{\partial }{\partial \overline{z_j}}((f_{I,J})_n*\chi_{n,\delta})\,\mathrm{d}z^I\wedge \mathrm{d}\overline{z}^K.\nonumber
\end{eqnarray}
Note that for any $\textbf{z}=(z_1,\cdots,z_n)\in\mathbb{C}^n$ and $j=1,\cdots,n$, it holds that
\begin{eqnarray}
&&\frac{\partial }{\partial \overline{z_j}}((f_{I,J})_n*\chi_{n,\delta})(\textbf{z})\label{formula 15nt}\\
&=& \int_{\mathbb{C}^n}(f_{I,J})_n(\textbf{z}')\bigg(\frac{\partial }{\partial \overline{z_j}}\chi_{n,\delta}(\textbf{z}-\textbf{z}')
\bigg)
\varphi_n^{-1}(\textbf{z}')\varphi_n(\textbf{z}')
\,\mathrm{d}\textbf{z}'\nonumber\\
&=&\int_{\mathbb{C}^n}(f_{I,J})_n(\textbf{z}')\Bigg(\frac{\partial }{\partial \overline{z_j'}}\bigg(-\chi_{n,\delta}(\textbf{z}-\textbf{z}')
\varphi_n^{-1}(\textbf{z}')
\bigg)\nonumber\\
&&+\frac{z_j'\chi_{n,\delta}(\textbf{z}-\textbf{z}')\varphi_n^{-1}(\textbf{z}')}{2r^2a_j^2}
\Bigg)\varphi_n(\textbf{z}')
\,\mathrm{d}\textbf{z}'\nonumber\\
&=&-\int_{\mathbb{C}^n}(f_{I,J})_n(\textbf{z}')\overline{\delta_{j}\bigg(\chi_{n,\delta}(\textbf{z}-\textbf{z}')
\varphi_n^{-1}(\textbf{z}')
\bigg)}\varphi_n(\textbf{z}')
\,\mathrm{d}\textbf{z}'\nonumber\\
&=&-\int_{\mathbb{C}^n}(f_{I,J})_n(\textbf{z}')\overline{\delta_{j}\bigg(\chi_{n,\delta}(\textbf{z}-\textbf{z}')
\varphi_n^{-1}(\textbf{z}')
\bigg)}
\,\mathrm{d}\mathcal{N}^{n,r}(\textbf{z}')\nonumber\\
&=&-\int_{\mathbb{C}^n}\bigg(\int_{\ell^2(\mathbb{N}_{2,\widehat{n} })}f_{I,J}(\textbf{z}',\textbf{z}_n)
\mathrm{d}P_{n,r}(\textbf{z}^n)\bigg)\nonumber\\
&&\times\overline{\delta_{j}\bigg(\chi_{n,\delta}(\textbf{z}-\textbf{z}')
\varphi_n^{-1}(\textbf{z}')
\bigg)}
\,\mathrm{d}\mathcal{N}^{n,r}(\textbf{z}')\nonumber\\
&=&-\int_{\mathbb{C}^n}\int_{\ell^2(\mathbb{N}_{2,\widehat{n} })}f_{I,J}(\textbf{z}',\textbf{z}_n)
\overline{\delta_{j}\bigg(\chi_{n,\delta}(\textbf{z}-\textbf{z}')
\varphi_n^{-1}(\textbf{z}')
\bigg)}
\mathrm{d}P_{n,r}(\textbf{z}^n)\,\mathrm{d}\mathcal{N}^{n,r}(\textbf{z}')\nonumber\\
&=&-\int_{ \ell^2}f_{I,J}(\textbf{z}',\textbf{z}_n)\overline{\delta_{j}\bigg(\chi_{n,\delta}(\textbf{z}-\textbf{z}')
\varphi_n^{-1}(\textbf{z}')
\bigg)}\,\mathrm{d}P_r(\textbf{z}',\textbf{z}_n)\nonumber\\
&=&-\int_{\ell^2}f_{I,J}(\textbf{z}',\textbf{z}_n)\overline{\delta_{j}(h_{n,\delta,\textbf{z}}(\textbf{z}')
)}\,\mathrm{d}P_r(\textbf{z}',\textbf{z}_n)\nonumber\\
&=&-\int_{\ell^2}f_{I,J}\overline{\delta_{j}(h_{n,\delta,\textbf{z}}
)}\,\mathrm{d}P_r,\nonumber
\end{eqnarray}
where $\textbf{z}_n=(x_{i} +\sqrt{-1}y_{i})_{i=1}^{ n+1}\in \ell^2(\mathbb{N}_{2,\widehat{n} })$ and $h_{n,\delta,\textbf{z}}(\textbf{z}')\triangleq \chi_{n,\delta}(\textbf{z}-\textbf{z}')
\varphi_n^{-1}(\textbf{z}')
,$\\$\forall\,\textbf{z}'=(x_{i}' +\sqrt{-1}y_{i}')_{i=1}^{n}\in\mathbb{C}^n.$
Then $h_{n,\delta,\textbf{z}}\in  C_c^{\infty}(\mathbb{C}^n)$. By \eqref{weak def of st form}, for each multi index $I$ and $K$ such that $|I|=s,|K|=t+2,\max(I\cup K)\leq n$, it holds that
\begin{eqnarray}
&&\int_{\ell^2}g_{I,K}\overline{h_{n,\delta,\textbf{z}}}\,\mathrm{d}P_r\label{nt}\\
&=&-\int_{\ell^2}\bigg(\sum_{1\leq j<\infty}\sum_{|J|=t+1}'\varepsilon_{j, J}^{K}f_{I,J}\bigg)  \overline{\delta_{j}(h_{n,\delta,\textbf{z}})}  \,\mathrm{d}P_r.\nonumber
\end{eqnarray}
By (\ref{formula 15nt}) and (\ref{nt}), we have
\begin{eqnarray}
&&(-1)^s\int_{\mathbb{C}^n}(g_{I,K})_n(\textbf{z}')\chi_{n,\delta}(\textbf{z}-\textbf{z}')\,\mathrm{d}\textbf{z}'\label{sdfds dfsdnt}\\
&=&(-1)^s\int_{\mathbb{C}^n}(g_{I,K})_n\overline{h_{n,\delta,\textbf{z}}}\,\mathrm{d}\mathcal{N}^{n,r}\nonumber\\
&=&(-1)^s\int_{\ell^2}g_{I,K}\overline{h_{n,\delta,\textbf{z}}}\,\mathrm{d}P_r\nonumber\\
&=&(-1)^{s+1}\int_{\ell^2}\bigg(\sum_{1\leq j<\infty}\sum_{|J|=t+1}'\varepsilon_{j, J}^{K}f_{I,J}\bigg) (\overline{\delta_{j} h_{n,\delta,\textbf{z}}}) \,\mathrm{d}P_r\nonumber\\
&=&(-1)^s\sum_{1\leq j<\infty}\sum_{|J|=t+1}'\varepsilon_{j, J}^{K}\frac{\partial }{\partial \overline{z_j}}((f_{I,J})_n*\chi_{n,\delta})(\textbf{z})\nonumber\\
&=&(-1)^s\sum_{1\leq j\leq n}\sum_{|J|=t+1,\max J\leq n}'\varepsilon_{j, J}^{K}\frac{\partial }{\partial \overline{z_j}}((f_{I,J})_n*\chi_{n,\delta})(\textbf{z})\nonumber\\
&=&(S f_{n,\delta})_{I,K}(\textbf{z}),\nonumber
\end{eqnarray}
where the last equality follows from (\ref{S(f_ndelta)}). Recall the notations in (\ref{f_{ndelta}})-(\ref{f_n}), we have
\begin{eqnarray*}
Sf&=&(-1)^s\sum_{|I|=s,|K|=t+2}'g_{I,K} \,\mathrm{d}z^I\wedge \mathrm{d}\overline{z}^K,\\
(Sf)_{n}&=&(-1)^s\sum_{|I|=s,|K|=t+2,\max(I\cup K)\leq n}'(g_{I,K})_n\,\mathrm{d}z^I\wedge \mathrm{d}\overline{z}^K.
\end{eqnarray*}
Thus (\ref{sdfds dfsdnt}) implies that for multi index $I$ and $K$ such that $|I|=s,|K|=t+2,\max(I\cup K)\leq n$, we have
\begin{eqnarray}
 ((Sf)_n)_{I,K}\ast\chi_{n,\delta}  =(S f_{n,\delta})_{I,K} .\nonumber
\end{eqnarray}
By Lemma \ref{20231110lem1} and Proposition \ref{convolution properties}, it holds that
\begin{eqnarray*}
\lim_{\delta\to 0}(S f_{n,\delta})_{I,K}=((Sf)_n)_{I,K}
\end{eqnarray*}
in $L^2(\ell^2, P_r)$. Therefore,
\begin{eqnarray*}
\lim_{\delta\to 0}S f_{n,\delta}=(Sf)_n
\end{eqnarray*}
in $L^2_{(s,t+2)}(\ell^2, P_r)$.

(3) 
Obviously, $f_{n,\delta}\in \mathscr{D}_{(s,t+1)}$, by Theorem \ref{200207t1} and \eqref{20231110for1}, $f_{n,\delta}\in D_{T^*}$ and
\begin{eqnarray*}
&&T^*f_{n,\delta}\\
&=&\sum_{|I|=s,|K|=t}'\sum_{1\leq j<\infty}\sum\limits_{|J|=t+1}'(-1)^{s-1}a_j^2\varepsilon_{j, K}^{J}\delta_j((f_{n,\delta})_{I,J})\mathrm{d}z^I\wedge \mathrm{d}\overline{z}^K\\
&=&\sum_{|I|=s,|K|=t,\,\max(I\cup K)\leq n}'\sum_{1\leq j\leq n}\sum\limits_{|J|=t+1,\,\max J\leq n}'(-1)^{s-1}a_j^2\varepsilon_{j, K}^{J}\\
&&\qquad\qquad\qquad\qquad\times\delta_j((f_{n,\delta})_{I,J})\mathrm{d}z^I\wedge \mathrm{d}\overline{z}^K.
\end{eqnarray*}



Note that for multi index $I$ and $K$ such that $|I|=s,|K|=t,\max(I\cup K)\leq n$, we have
\begin{eqnarray*}
(T^*f_{n,\delta})_{I,K}
&=&(-1)^{s-1}\sum_{1\leq j\leq n}\sum\limits_{|J|=t+1,\,\max J\leq n}'a_j^2\varepsilon_{j,K}^{J}\cdot \delta_j(((f_{n,\delta})_{I,J})_n*\chi_{n,\delta}).
\end{eqnarray*}
Then for multi index $J$ such that $|J|=t+1,$ and $1\leq j\leq n$, we have
\begin{eqnarray}
&&\delta_j((f_{I,J})_n*\chi_{n,\delta})(\textbf{z})\label{20231110for3}\\
&=& \int_{\mathbb{C}^n}(f_{I,J})_n(\textbf{z}')\left(\delta_j\chi_{n,\delta}(\textbf{z}-\textbf{z}')
\right)
\varphi_n^{-1}(\textbf{z}')\varphi_n(\textbf{z}')
\,\mathrm{d} \textbf{z}'\nonumber\\
&=&\int_{\mathbb{C}^n}(f_{I,J})_n(\textbf{z}')\Bigg(\frac{\partial }{\partial z_j}\chi_{n,\delta}(\textbf{z}-\textbf{z}')
-\frac{\overline{z_j}}{2r^2a_j^2}\chi_{n,\delta}(\textbf{z}-\textbf{z}')
\Bigg) \varphi_n^{-1}(\textbf{z}')
\,\mathrm{d}\mathcal{N}^{n,r}(\textbf{z}')\nonumber\\
&=&\int_{\mathbb{C}^n}(f_{I,J})_n(\textbf{z}')\Bigg( \frac{\partial }{\partial z_j'}\bigg(-\chi_{n,\delta}(\textbf{z}-\textbf{z}')
\varphi_n^{-1}(\textbf{z}')\bigg)\nonumber\\
&&\qquad\qquad\qquad\qquad\qquad+\frac{\overline{z_j'}-\overline{z_j}}{2r^2a_j^2}\chi_{n,\delta}(\textbf{z}-\textbf{z}')
\varphi_n^{-1}(\textbf{z}')\Bigg)
\,\mathrm{d}\mathcal{N}^{n,r}(\textbf{z}')\nonumber\\
&=&\int_{\ell^2}f_{I,J}(\textbf{z}',\textbf{z}_n)\Bigg(\frac{\partial }{\partial z_j'}\bigg(-\chi_{n,\delta}(\textbf{z}-\textbf{z}')
\varphi_n^{-1}(\textbf{z}')\bigg)\nonumber
\\
&&\qquad\qquad\qquad\qquad\qquad+\frac{\overline{z_j'}-\overline{z_j}}{2r^2a_j^2}\chi_{n,\delta}(\textbf{z}-\textbf{z}')
\varphi_n^{-1}(\textbf{z}')\Bigg)
\,\mathrm{d}P_r(\textbf{z}',\textbf{z}_n),\nonumber
\end{eqnarray}
and
\begin{eqnarray*}
&&\textbf{a}^{I,K}\cdot(-1)^{s-1}\sum_{1\leq j\leq n}\sum\limits_{|L|=t+1,\,\max L\leq n}'a_j^2\varepsilon_{j,K}^{L}\\&&
\qquad\qquad\qquad\qquad\qquad\times\bigg(-\int_{\mathbb{C}^n}(f_{I,L})_n\overline{\frac{\partial }{\partial \overline{z_j}}(
h_{n,\delta,\textbf{z}} )}\,\mathrm{d}\mathcal{N}^{n,r}\bigg)\\
&=&(-1)^{s}\sum_{1\leq j\leq n}\sum\limits_{|L|=t+1,\,\max L\leq n}'\textbf{a}^{I,L}\cdot\varepsilon_{j,K}^{L}\cdot
\bigg(\int_{\ell^2} f_{I,L} \overline{\frac{\partial }{\partial \overline{z_j}}(
h_{n,\delta,\textbf{z}} )}\,\mathrm{d}P_r\bigg)\\
&=&(f,T(h_{n,\delta,\textbf{z}}\mathrm{d}z^I\wedge \mathrm{d}\overline{z}^K))_{L^2_{(s,t+1)}(\ell^2,P_r)}\\
&=&(T^{*} f,h_{n,\delta,\textbf{z}}\mathrm{d}z^I\wedge \mathrm{d}\overline{z}^K)_{L^2_{(s,t)}(\ell^2,P_r)}\\
&=&\textbf{a}^{I,K}\cdot\int_{\ell^2}(T^{*} f)_{I,K}h_{n,\delta,\textbf{z}}\,\mathrm{d}P_r\\
&=&\textbf{a}^{I,K}\cdot\int_{\mathbb{C}^n}((T^{*} f)_n)_{I,K}(\textbf{z}')\chi_{n,\delta}(\textbf{z}-\textbf{z}')\,\mathrm{d}\textbf{z}',
\end{eqnarray*}
where $\textbf{z}_n=(x_{i} +\sqrt{-1}y_{i})_{i=n+1}^{ \infty}\in \ell^2(\mathbb{N}_{2,\widehat{n} })$, $\textbf{z}=(x_{i} +\sqrt{-1}y_{i})_{ i=1}^{ n},\textbf{z}'=(x_{i}' +\sqrt{-1}y_{i}')_{ i=1}^{n}\in\mathbb{C}^n$ and the second equality follows from \eqref{formula of d-bar}. Therefore,
\begin{eqnarray}
&&(-1)^{s-1}\sum_{1\leq j\leq n}\sum\limits_{|J|=t+1,\,\max J\leq n}'a_j^2\varepsilon_{j,K}^{J}\cdot
\bigg(-\int_{\mathbb{C}^n}(f_{I,J})_n\overline{\frac{\partial }{\partial \overline{z_j}}(
h_{n,\delta,\textbf{z}} )}\,\mathrm{d}\mathcal{N}^{n,r}\bigg)\nonumber\\
&=&
\int_{\mathbb{C}^n}((T^{*}f)_n)_{I,K}(\textbf{z}')\chi_{n,\delta}(\textbf{z}-\textbf{z}')\,\mathrm{d}\textbf{z}'.\label{20231110for2}
 \end{eqnarray}
By Lemma \ref{20231110lem1} and Proposition \ref{convolution properties}, the above function in \eqref{20231110for2} tends to $((T^{*}f)_n)_{I,K}$ in $L^2(\ell^2,P_r)$ as $\delta\to 0$. Recall \eqref{20231110for3} that
\begin{eqnarray*}
&&\delta_j((f_{I,J})_n*\chi_{n,\delta})(\textbf{z})\\
&=&\int_{\mathbb{C}^n}(f_{I,J})_n(\textbf{z}')\Bigg(\frac{\partial }{\partial z_j'}\bigg(-\chi_{n,\delta}(\textbf{z}-\textbf{z}')
\varphi_n^{-1}(\textbf{z}')\bigg)\\
&&\qquad\qquad\qquad\qquad\qquad+\frac{\overline{z_j'}-\overline{z_j}}{2r^2a_j^2}\chi_{n,\delta}(\textbf{z}-\textbf{z}')
\varphi_n^{-1}(\textbf{z}')\Bigg)
\,\mathrm{d}\mathcal{N}^{n,r}(\textbf{z}')\\
&=&\int_{\mathbb{C}^n}(f_{I,J})_n(\textbf{z}')\Bigg(-\overline{\frac{\partial }{\partial \overline{z_j'}}(
h_{n,\delta,\textbf{z}} )}
+\frac{\overline{z_j'}-\overline{z_j}}{2r^2a_j^2}\chi_{n,\delta}(\textbf{z}-\textbf{z}')
\varphi_n^{-1}(\textbf{z}')\Bigg)
\,\mathrm{d}\mathcal{N}^{n,r}(\textbf{z}').
\end{eqnarray*}
Note that
\begin{eqnarray}
&&\int_{\mathbb{C}^n}(f_{I,J})_n(\textbf{z}')\left(
\frac{\overline{z_j'}-\overline{z_j}}{2r^2a_j^2}\chi_{n,\delta}(\textbf{z}-\textbf{z}')
\varphi_n^{-1}(\textbf{z}')\right)
\,\mathrm{d}\mathcal{N}^n(\textbf{z}')\label{20231110for4}\\
&=&
\int_{\mathbb{C}^n}
\frac{(f_{I,J})_n(\textbf{z}')\overline{z_j'}}{2r^2a_j^2}\chi_{n,\delta}(\textbf{z}-\textbf{z}')
\,\mathrm{d}\textbf{z}'
-\frac{\overline{z_j}}{2r^2a_j^2}\int_{\mathbb{C}^n}
(f_{I,J})_n(\textbf{z}')\chi_{n,\delta}(\textbf{z}-\textbf{z}')
\,\mathrm{d}\textbf{z}'\nonumber
\end{eqnarray}
and supp$(f_{I,J})_n$ is a bounded subset of $\mathbb{C}^n$. Therefore, the function in \eqref{20231110for4} tends to 0 in $L^2(\ell^2,P_r)$ as $\delta\to 0.$ Thus
\begin{eqnarray*}
\lim_{\delta\to 0}(T^*f_{n,\delta})_{I,K}=((T^{*}f)_n)_{I,K}
\end{eqnarray*}
in $L^2(\ell^2,P_r)$ and hence
\begin{eqnarray*}
\lim_{\delta\to 0}T^{*} f_{n,\delta}=(T^{*}f)_n
\end{eqnarray*}
in $L^2_{(s,t)}(\ell^2,P_r)$. The proof of Theorem \ref{main approximation theoremhs} is completed.
\end{proof}
Combining  Theorem \ref{200207t1}, Proposition \ref{230924prop1} and Theorem \ref{main approximation theoremhs}, we have the following inequality.
\begin{corollary}\label{20231110cor2}
For any $f\in D_{S}\cap D_{T^*}$, it holds that
\begin{eqnarray*}\label{20231110for5}
||T^*f||_{L^2_{(s,t)}(\ell^2,P_r)}^2
+||Sf||_{L^2_{(s,t+2)}(\ell^2,P_r)}^2
\geq \frac{t+1}{2r^2}\cdot||f||_{L^2_{(s,t+1)}(\ell^2,P_r)}^2.\nonumber
\end{eqnarray*}
\end{corollary}
Combining Corollary \ref{estimation lemma}, Corollary \ref{20231110cor2} and Lemma \ref{lem2.5.1}, we arrive at:
\begin{theorem}\label{existence theorem}
For any $f\in N_S$, there exists $u\in L^2_{(s,t)}(\ell^2,P_r)$ such that
$$
u\in D_T,\qquad Tu=f
$$
and
\begin{equation*}\label{200209e5}
||u||_{L^2_{(s,t)}(\ell^2,P_r)}\leq \sqrt{\frac{2r^2}{t+1}}\cdot||f||_{L^2_{(s,t+1)}(\ell^2,P_r)}.
 \end{equation*}
\end{theorem}

\newpage

\section{Infinite Dimensional C-K and Holmgren Type Theorems}
\label{20240111chapter3}

This section is mainly based on \cite{YZ-a}.

\subsection{Analyticity for Functions with Infinitely Many Variables}\label{section 2.2}
The history of the concept of analyticity for functions with infinitely many variables began with V. Volterra in 1887 (\cite{Vol1,Vol2,Vol3,Vol4,Vol5}) and was continued by H. von Koch in 1899 (\cite{Koc}), D. Hilbert in 1909 (\cite{Hil}), M. Fr\'{e}chet in 1900s (\cite{Fre1, Fre2}), and R. G\^{a}teaux in 1910s (\cite{Gat1,Gat2,Gat3}). Denote by $\mathbb{N}_0^{(\mathbb{N})}$ the set of all finitely supported sequences of nonnegative integers. This notation is borrowed from \cite[(1.2), p. 14]{DMP}. Then for $\textbf{x}=(x_i)_{i\in\mathbb{N}}\in \mathbb{R}^{\mathbb{N}}$ and $\alpha=(\alpha_i)_{i\in\mathbb{N}}\in \mathbb{N}_0^{(\mathbb{N})}$ with $\alpha_k=0$ for $k\geqslant n+1$ (for some $n\in \mathbb{N}$), write
\begin{equation}\label{220812e1}
\textbf{x}^{\alpha}\triangleq x_1^{\alpha_1}x_2^{\alpha_2}\cdots x_n^{\alpha_n},
\end{equation}
which is called a monomial on $\mathbb{R}^{\mathbb{N}}$. The following concept of analyticity is essentially from \cite{Hil, Koc} (see also \cite{DMP, Moo}).
\begin{definition}\label{def of analyticity}
Suppose $f$ is a real-valued function defined on a subset $D$ of $\mathbb{R}^{\mathbb{N}}$ and $\textbf{x}_0\in D$. If for each $\alpha\in \mathbb{N}_0^{(\mathbb{N})}$,  there exists a real number $c_{\alpha}$ (depending on $f$ and $x_0$) such that the series $\sum_{\alpha \in \mathbb{N}_0^{(\mathbb{N})}}  c_{\alpha}\textbf{v}^{\alpha}$
is absolutely convergent for some $\textbf{v}=(v_i)_{i\in\mathbb{N}}\in \mathbb{R}^{\mathbb{N}}$ with $v_i\neq 0$ for each $i\in \mathbb{N}$, and for each $\textbf{h}=(h_i)_{i\in\mathbb{N}}\in \mathbb{R}^{\mathbb{N}}$ with $|h_i|\leqslant |v_i|$ for all $i\in\mathbb{N}$, it holds that $\textbf{x}_0+\textbf{h}\in D$ and
\begin{eqnarray}\label{monomial expansion}
f(\textbf{x}_0+\textbf{h})=\sum_{\alpha \in \mathbb{N}_0^{(\mathbb{N})}}  c_{\alpha}\textbf{h}^{\alpha},
\end{eqnarray}
then we say $f$ is
analytic at $\textbf{x}_0$ (with the monomial expansion (\ref{monomial expansion}), or with the monomial expansion\index{monomial expansion} $f(\textbf{x})=\sum_{\alpha \in \mathbb{N}_0^{(\mathbb{N})}} c_{\alpha}(\textbf{x}-\textbf{x}_0)^{\alpha}$). In this case, we write
$$D_f^{\textbf{x}_0}\triangleq \left\{\textbf{h}\in \mathbb{R}^{\mathbb{N}}:\;\textbf{x}_0+\textbf{h}\in D, \;\sum_{\alpha \in \mathbb{N}_0^{(\mathbb{N})}}  |c_{\alpha}\textbf{h}^{\alpha}|<\infty \right\},
$$
and call the set $D_f^{\textbf{x}_0}$ the domain of convergence of the monomial expansion (\ref{monomial expansion}).

Further, if the above $\textbf{v}$ satisfies $\sum\limits_{i=1}^{\infty}\frac{1}{|v_i|^p}<\infty$ for some $p\in (0,\infty)$, then the monomial expansion (\ref{monomial expansion}) of $f$ at $\textbf{x}_0$ is called absolutely convergent at a point near $\infty$ (in the topology $\mathcal {T}^{p}$).
\end{definition}

Obviously, the closer (the above) {\it $\textbf{v}$} to $\infty$ is, the stronger the analyticity of $f$ at $\textbf{x}_0$ is. In other word, we may measure the analyticity of $f$ at $\textbf{x}_0$ by the ``distance" between {\it $\textbf{v}$} and $\infty$.
\begin{remark}\label{230820rem1}
The notations in the Definition \ref{def of analyticity} also works on any countable Cartesian product of $\mathbb{R}$ instead of $\mathbb{R}^{\mathbb{N}}$.
\end{remark}

\begin{example}
{(\textbf{Riemann Zeta Function})} Recall that the Riemann zeta function\index{Riemann zeta function} is defined by $\zeta(s)=\sum\limits_{n=1}^{\infty}\frac{1}{n^s}$ when $s\in \mathbb{C}$ with $\rm{Re}\,(s)>2$. Suppose that $\{p_n\}_{n=1}^{\infty}$ is the collection of all positive prime numbers. Then for any $s\in \mathbb{R}$ with $s>2$, one has
$$
\begin{array}{ll}
\displaystyle
\zeta(s)&\displaystyle=\prod_{n=1}^{\infty}\frac{1}{1-p_n^{-s}}\\[5mm]
&\displaystyle
=\sum_{\alpha \in \mathbb{N}_0^{(\mathbb{N})}}(p_1^{-s},\cdots,p_n^{-s},\cdots)^{\alpha}
=f(p_1^{-s},\cdots,p_n^{-s},\cdots),
\end{array}
$$
where $f(\textbf{x})\triangleq\sum\limits_{\alpha \in \mathbb{N}_0^{(\mathbb{N})}}\textbf{x}^{\alpha}$, defined on $D\triangleq\left\{\textbf{x}=(x_i)_{i\in\mathbb{N}}\in\mathbb{R}^{\mathbb{N}}:\;\sum\limits_{i=1}^{\infty}|x_i|<1\right\}$, is the function of geometric series of infinitely many variables.
\end{example}

The notion of majority function can be defined naturally as follows:

\begin{definition}\label{def of majority function}
Suppose $f$ and $F$ are analytic at $\textbf{x}_0\in \mathbb{R}^{\mathbb{N}}$ with monomial expansions $f(\textbf{x})=\sum\limits_{\alpha \in \mathbb{N}_0^{(\mathbb{N})}} c_{\alpha}(\textbf{x}-\textbf{x}_0)^{\alpha}$ and $F(\textbf{x})=\sum\limits_{\alpha \in \mathbb{N}_0^{(\mathbb{N})}} C_{\alpha}(\textbf{x}-\textbf{x}_0)^{\alpha}$, respectively,
where $c_{\alpha}\in\mathbb{R}$ and $C_{\alpha}\geqslant 0$ for each $\alpha \in \mathbb{N}_0^{(\mathbb{N})}$. If $|c_{\alpha}|\leqslant C_{\alpha}$ for all $\alpha \in \mathbb{N}_0^{(\mathbb{N})}$, then $F$ is called a majority function of $f$ at $\textbf{x}_0$.
\end{definition}

Let us recall definitions of polynomials and analyticity on a real Banach space $(X,||\cdot||)$, introduced by Fr\'{e}chet in 1909 (\cite{Fre1}). For each $n\in \mathbb{N}$, denote by $\prod_{i=1}^{n}X$ the $n$-fold Cartesian product of $X$, i.e., $\prod_{i=1}^{n}X\triangleq\overbrace{X\times
X\times \cdots\times X}^{n\hbox{ \tiny
times}}$. A function $g:\,X\to \mathbb{R}$ is called a continuous $n$-homogeneous polynomial if there exists a continuous $n$-linear map $T:\,\prod_{i=1}^{n}X\to\mathbb{R}$ such that $g(x)=T(x,\cdots,x)$ for each $x\in X$. For $n=0$, any function from $X$ into $\mathbb{R}$ with constant value is called a continuous $0$-homogeneous polynomial.
Suppose that $f$ is a real-valued function defined on a subset $D$  of $X$. If for each $\xi\in D$, one can find a sequence $\{P_{n}\}_{n=0}^{\infty}$ (where $P_n$ is a continuous $n$-homogeneous polynomials on $X$ for each $n\in\mathbb{N}$) and a radius $r>0$ such that $\xi+B(r)\subset D$ (where $B(r)\triangleq \{y\in X:\; ||y||<r\}$), and
\begin{eqnarray}\label{power series expansion}
f(\xi+\cdot)=\sum_{n=0}^{\infty} P_{n}(\cdot)\ \hbox{ uniformly in }B(r),
 \end{eqnarray}
then we say that $f$ is analytic on $D$ (with the power series expansion (\ref{power series expansion})).
In the above, each $P_n$ and $r$ depend on $f$ and $\xi$.


We emphasize that $\mathbb{R}^{\mathbb{N}}$ is NOT a Banach space. Hence, generally speaking, the above analyticities by monomial expansions and by power series expansions are two distinct notions.
Nevertheless, for every $\textbf{x}_0\in\mathbb{R}^{\mathbb{N}}$, $r\in(0,\infty)$ and $p\in(0,\infty]$, it holds that $\textbf{x}_0+B_r^{p,\mathbb{N}}\subset \mathbb{R}^{\mathbb{N}}$ and $B_r^{p,\mathbb{N}} \subset\ell^p $ (See (\ref{eq1}) for $B_r^{p,\mathbb{N}}$), which means that there exist many Banach sub-spaces in $\mathbb{R}^{\mathbb{N}}$. Motivated by this simple observation, we introduce the following concept.
\begin{definition}
Suppose that $S$ is a nonempty set and $f$ is a real-valued function defined on a subset $U$ of $\mathbb{R}^{S}$ and $\textbf{x}_0\in U$. We say that $f$ is Fr\'{e}chet differentiable with respect to $\ell^p(S)$ (for some $p\in[1,\infty]$) in a neighborhood of $\textbf{x}_0$ in the topology $\mathscr{T}_p$ if there exists $r\in(0,\infty)$ such that $\textbf{x}_0+B_r^{p,S} \subset U$, and the function defined by $g(\cdot)\triangleq f(\textbf{x}_0+\cdot)$ is Fr\'{e}chet differentiable in $B_r^{p,S}$ (with respect to the Banach space $\ell^p(S)$).
\end{definition}

In the sequel, we shall need to use the following two lemmas in the proof of Cauchy-Kowalevski type theorems on $\mathbb{R}^{\mathbb{N}}$.
\begin{lemma}\label{absolutely convergent lemma}
Suppose that $p\in(0,\infty)$ and $g$ is analytic at $\textbf{x}_0\in \mathbb{R}^{\mathbb{N}}$ with the following monomial expansion
$$
g(\textbf{x})=\sum_{\alpha \in \mathbb{N}_0^{(\mathbb{N})}} B_{\alpha}(\textbf{x}-\textbf{x}_0)^{\alpha}
$$
being absolutely convergent at a point near $\infty$ in the topology $\mathcal {T}^p$. Then there exists $\rho_0\in(0,+\infty)$, $s_i\in(0,1)$ for each $i\in \mathbb{N}$ such that
$$\sum_{j=1}^{\infty}s_j^p =1\quad\hbox{ and }\quad\sum_{\alpha\in \mathbb{N}_0^{(\mathbb{N})}}|B_{\alpha}|
\left(\frac{\rho_0 }{s_i}\right)_{i\in\mathbb{N}}^{\alpha}
<\infty.$$
\end{lemma}
\begin{proof}
By Definition \ref{def of analyticity} and our assumption, there exists $\textbf{v}=(v_i)\in \mathbb{R}^{\mathbb{N}}$ such that $v_i\neq 0$ for each $i\in \mathbb{N}$, $\sum\limits_{i=1}^{\infty}\frac{1}{|v_i|^p}<\infty$ and
\begin{equation}\label{220809e1}
\sum_{\alpha \in \mathbb{N}_0^{(\mathbb{N})}}  |B_{\alpha}|\cdot|\textbf{v}^{\alpha}|<\infty.
\end{equation}
Let
$$\rho_0=\bigg(\sum_{i=1}^{\infty}\frac{1}{|v_i|^p }\bigg)^{-\frac{1}{p}}\ \ \hbox{ and }\ \ s_i\triangleq\frac{\rho_0 }{|v_i|  }
\ \ \hbox{ for }\ \ i\in \mathbb{N}.
$$
Note that for each $i\in \mathbb{N}$ it holds that $0<s_i<1$ and $\sum\limits_{j=1}^{\infty}s_j^p =1.$
Hence, by (\ref{220809e1}), we have
$$\sum_{\alpha\in \mathbb{N}_0^{(\mathbb{N})}}|B_{\alpha}|
\left(\frac{\rho_0 }{s_i}\right)_{i\in\mathbb{N}}^{\alpha}=\sum_{\alpha\in \mathbb{N}_0^{(\mathbb{N})}}|B_{\alpha}|\cdot|\textbf{v}^{\alpha}|<\infty.$$
This completes the proof of Lemma \ref{absolutely convergent lemma}.
\end{proof}

 For any $p\in [1,\infty)$, denote by $p'$ the H\"older conjugate of $p$, i.e., \index{$p'$, the H\"older conjugate of $p$}
$$p'=\left\{\begin{array}{ll}
\frac{p}{p-1},\quad&p\in (1,\infty),\\
\infty,& p=1.\end{array}\right.$$

\begin{lemma}\label{Frechet lemma}
Suppose that $p\in [1,\infty)$ and there exist $r>0$, $r_i\in (0,1)$ for each $i\in \mathbb{N}$ and $C_{n,m}\geqslant 0$ for each $n,m\in \mathbb{N}_0$ such that $\sum\limits_{i=1}^{\infty}r_i^p<1$ and $\sum\limits_{n,m\in \mathbb{N}_0}C_{n,m}r^{m+n} <\infty.$
For each $(t,(x_i)_{i\in\mathbb{N}})\in B_r^{p',\mathbb{N}_0}$, let
\begin{eqnarray*}
f(t,(x_i)_{i\in\mathbb{N}})\triangleq\sum_{n,m\in \mathbb{N}_0 }C_{n,m}t^m \left(\sum_{i=1}^{\infty}r_ix_i \right)^n.
\end{eqnarray*}
Then $f$ is Fr\'{e}chet differentiable in $B_r^{p',\mathbb{N}_0}$ (with respect to the Banach space $\ell^{p'}(\mathbb{N}_0)$), and its Fr\'{e}chet derivative $Df$ is continuous.
\end{lemma}
\begin{proof}
By $(t,(x_i)_{i\in\mathbb{N}})\in B_r^{p',\mathbb{N}_0}$, there exists $r_0\in (0,r)$ such that $$
||(t,(x_i)_{i\in\mathbb{N}})||_{\ell^{p'}(\mathbb{N}_0)}<r_0
$$
and
\begin{eqnarray*}
\sum_{n,m\in \mathbb{N}_0}(m+n)C_{n,m}r_0^{m+n-2}(r_0+ (m+n-1)) <\infty.
\end{eqnarray*}
Suppose $(\Delta t,(\Delta x_i)_{i\in\mathbb{N}})\in \ell^{p'}(\mathbb{N}_0)$ such that
$$
||(t,(x_i)_{i\in\mathbb{N}})||_{\ell^{p'}(\mathbb{N}_0)}+||(\Delta t,(\Delta x_i)_{i\in\mathbb{N}})||_{\ell^{p'}(\mathbb{N}_0)}<r_0.
$$
For $n,m\in \mathbb{N}_0$
 and $s\in[0,1]$, let
\begin{eqnarray*}
g(s)&\triangleq &f(t+s\Delta t,(x_i+s\Delta x_i)_{i\in\mathbb{N}})-f(t,(x_i)_{i\in\mathbb{N}}),\\
g_{n,m}(s)&\triangleq &(t+s\Delta t)^m \bigg(\sum_{i=1}^{\infty}r_i(x_i+s\Delta x_i) \bigg)^n.
\end{eqnarray*}
By the classical mean value theorem, there exists $\theta_{n,m},\theta_{n,m}'\in(0,1)$ such that
$$g_{n,m}(1)-g_{n,m}(0)-g_{n,m}'(0)=\theta_{n,m}'\cdot g_{n,m}^{''}(\theta_{n,m}).
$$
Then,
\begin{eqnarray*}
&&g_{n,m}(1)-g_{n,m}(0)-g_{n,m}'(0)\\
&=&\theta_{n,m}'\cdot m(m-1)(\Delta t)^2(t+\theta_{n,m}\Delta t)^{m-2} \bigg(\sum_{i=1}^{\infty}r_i(x_i+\theta_{n,m}\Delta x_i )\bigg)^n\\
&&+\theta_{n,m}'\cdot mn(\Delta t)\bigg(\sum_{i=1}^{\infty}r_i\Delta x_i \bigg)(t+\theta_{n,m}\Delta t)^{m-1} \bigg(\sum_{i=1}^{\infty}r_i(x_i+\theta_{n,m}\Delta x_i ) \bigg)^{n-1}\\
&&+\theta_{n,m}'\cdot n(n-1)\bigg(\sum_{i=1}^{\infty}r_i\Delta x_i \bigg)^2(t+\theta_{n,m}\Delta t)^{m} \bigg(\sum_{i=1}^{\infty}r_i(x_i+\theta_{n,m}\Delta x_i)\bigg)^{n-2}.
\end{eqnarray*}
Thus, by $\sum\limits_{i=1}^{\infty}r_i^p<1$, we have
\begin{eqnarray*}
&&|g_{n,m}(1)-g_{n,m}(0)-g_{n,m}'(0)|\\
&\leqslant&\Bigg(m(m-1)(\Delta t)^2+mn(\Delta t)\bigg(\sum_{i=1}^{\infty}r_i\Delta x_i \bigg)+n(n-1)\bigg(\sum_{i=1}^{\infty}r_i\Delta x_i \bigg)^2\Bigg)r_0^{m+n-2}\\
&\leqslant& (m(m-1)+mn+n(n-1) )||(\Delta t,(\Delta x_i)_{i\in\mathbb{N}})||_{\ell^{p'}(\mathbb{N}_0)}^2\cdot r_0^{m+n-2}
\end{eqnarray*}
and
\begin{eqnarray*}
&&\bigg|g(1)-g(0)-\sum_{n,m\in \mathbb{N}_0 }C_{n,m}g_{n,m}'(0)\bigg|\\
&\leqslant&\sum_{n,m\in \mathbb{N}_0 }C_{n,m}|g_{n,m}(1)-g_{n,m}(0)-g_{n,m}'(0)|\\
&\leqslant& \sum_{n,m\in \mathbb{N}_0 }C_{n,m}(m(m-1)+mn+n(n-1))\cdot r_0^{m+n-2}\cdot||(\Delta t,(\Delta x_i)_{i\in\mathbb{N}})||_{\ell^{p'}(\mathbb{N}_0)}^2\\
&=&o(||(\Delta t,(\Delta x_i)_{i\in\mathbb{N}})||_{\ell^{p'}(\mathbb{N}_0)}),\quad\hbox{as }||(\Delta t,(\Delta x_i)_{i\in\mathbb{N}})||_{\ell^{p'}(\mathbb{N}_0)}\to0.
\end{eqnarray*}
Note that
\begin{eqnarray*}
\sum_{n,m\in \mathbb{N}_0 }C_{n,m}g_{n,m}'(0)
&=&(\Delta t)\cdot\bigg(\sum_{n,m\in \mathbb{N}_0 }C_{n,m}mt^{m-1} \bigg(\sum_{i=1}^{\infty}r_ix_i \bigg)^n\bigg)\\
&&+ \left( \sum_{j=1}^{\infty}(\Delta x_j)\cdot r_j \right)\cdot\bigg(\sum_{n,m\in \mathbb{N}_0 }C_{n,m}nt^m\bigg(\sum_{i=1}^{\infty}r_ix_i \bigg)^{n-1}\bigg)
\end{eqnarray*}
and denote the  following vector by $ Dg(0)$,
\begin{eqnarray*}
\left(\sum_{n,m\in \mathbb{N}_0 }C_{n,m}mt^{m-1} \bigg(\sum_{i=1}^{\infty}r_ix_i \bigg)^n,\left( \sum_{n,m\in \mathbb{N}_0 }r_j C_{n,m}nt^m\bigg(\sum_{i=1}^{\infty}r_ix_i \bigg)^{n-1}\right)_{j=1}^{\infty} \right).
\end{eqnarray*}
Then $ Dg(0)\in \ell^p(\mathbb{N}_0)$. Note that $\ell^p(\mathbb{N}_0)=(\ell^{p'}(\mathbb{N}_0))^*$ for $p\in (1,\infty)$ and $\ell^p(\mathbb{N}_0)\subset(\ell^{p'}(\mathbb{N}_0))^*$ for $p=1$. Therefore,
\begin{eqnarray*}
&&|f(t+\Delta t,(x_i+\Delta x_i)_{i\in\mathbb{N}})-f(t,(x_i)_{i\in\mathbb{N}})-  Dg(0) (\Delta t,(\Delta x_i)_{i\in\mathbb{N}})|\\[2mm]
&&=o(||(\Delta t,(\Delta x_i)_{i\in\mathbb{N}})||_{\ell^{p'}(\mathbb{N}_0)})\qquad\hbox{as }||(\Delta t,(\Delta x_i)_{i\in\mathbb{N}})||_{\ell^{p'}(\mathbb{N}_0)}\to0,
\end{eqnarray*}
which implies that $f$ is Fr\'{e}chet differentiable at $(t,(x_i)_{i\in\mathbb{N}})$ (with respect to the Banach space $\ell^{p'}(\mathbb{N}_0)$). Since $Df(t,(x_i)_{i\in\mathbb{N}})$ equals the following vector
\begin{eqnarray*}
\left(\sum_{n,m\in \mathbb{N}_0 }C_{n,m}mt^{m-1} \bigg(\sum_{i=1}^{\infty}r_ix_i \bigg)^n,\left( \sum_{n,m\in \mathbb{N}_0 }r_j C_{n,m}nt^m\bigg(\sum_{i\in\mathbb{N}}r_ix_i \bigg)^{n-1}\right)_{j=1}^{\infty}\right),
\end{eqnarray*}
it is easy to see that $Df$ is a continuous map from $B_r^{p',\mathbb{N}_0}$ into $\ell^{p}(\mathbb{N}_0)$. This completes the proof of Lemma \ref{Frechet lemma}.
\end{proof}

Convergence radius is a key tool to describe the analyticity of functions on $\mathbb{R}^n$. However, there exist no similar concepts for analytic functions on $\mathbb{R}^{\mathbb{N}}$. Instead, we shall introduce below a family of topologies on $\mathbb{R}^{\mathbb{N}}$ (which are finer than the usual product topology) to describe the analyticity of functions with infinite many variables. These topologies are metrizable, but $\mathbb{R}^{\mathbb{N}}$ endowed with them is no longer a topological vector space.

For any $p\in (0,\infty]$ and $(x_i)_{i\in\mathbb{N}},(y_i)_{i\in\mathbb{N}}\in \mathbb{R}^{\mathbb{N}}$, define
$$\text{d}_p((x_i)_{i\in\mathbb{N}},(y_i)_{i\in\mathbb{N}})\triangleq
\left\{
\begin{array}{ll}
	\displaystyle\min \left\{1,  \left(\sum_{i=1}^{\infty}|x_i-y_i|^p \right)^{\frac{1}{\max\{1,p\}}}  \right\},\quad &\hbox{if }p\in (0,\infty),\\[3mm]
	\displaystyle\min  \left\{1,\sup_{ i\in\mathbb{N} }|x_i-y_i| \right\},\quad &\hbox{if }p=\infty.
\end{array}
\right.
$$
Then,  $d_p$ is a complete metric on $\mathbb{R}^{\mathbb{N}}$. Denote by $\mathscr{T}_p$ the topology induced by the metric $d_p$. Let us consider the following family of subsets of $\mathbb{R}^{\mathbb{N}}$:
$$\mathscr{B}_p\triangleq\{(x_i)_{i\in\mathbb{N}}+B_r^{p,\mathbb{N}}:\; (x_i)_{i\in\mathbb{N}}\in \mathbb{R}^{\mathbb{N}},\,r\in(0,+\infty)\},$$
where for each nonempty set $S$,
\begin{equation}\label{eq1}
	B_r^{p,S}\triangleq
	\left\{
	\begin{array}{ll}
		\displaystyle\left\{(x_i)_{i\in S}\in \mathbb{R}^{S}: \; \left(\sum_{i
			\in S}|x_i|^p\right)^{\frac{1}{\max\{1,p\}}}<r\right\},\quad &\hbox{if }p\in (0,\infty),\\[5mm]
		\displaystyle \left\{(x_i)_{i\in S}\in \mathbb{R}^{S}: \;\sup_{ i\in S}|x_i|<r \right\},\quad &\hbox{if }p=\infty,
	\end{array}
	\right.
\end{equation}
$$
(x_i)_{i\in S}+B_r^{p,S}\triangleq \{(x_i+y_i)_{i\in S}:\;(y_i)_{i\in S}\in B_r^{p,S}\}.\qquad\qquad\qquad\qquad\qquad\qquad\;
$$
Clearly, the above topologies $\mathscr{T}_p$  are constructed based on the classical $\ell^p$. One can show the following simple result.

\begin{proposition}\label{prop 3}
	For any $p\in (0,\infty]$, the following assertions hold:
	\begin{itemize}
		\item[{\rm 1)}]  $(\mathbb{R}^{\mathbb{N}},\mathscr{T}_{p})$ is not a topological vector space;
		\item[{\rm 2)}]  $\mathscr{B}_p$ is a base for the topology $\mathscr{T}_p$,  and $(\mathbb{R}^{\mathbb{N}},d_{p})$ is a nonseparable complete metric space.
	\end{itemize}
\end{proposition}

In the above, we have introduced a family of topologies $\mathscr{T}_p$ (for each $p\in (0,\infty]$) on $\mathbb{R}^{\mathbb{N}}$. Denote by $\mathscr{T}$ the usual product topology on $\mathbb{R}^{\mathbb{N}}$. We have
the following result which reveals the relationship between these topologies.

\begin{proposition} \label{toplogies inclusion proposition}
	For any real numbers $p$ and $q$ with $0<p<q<\infty$, it holds that
	\begin{eqnarray*}
		\mathscr{T}\subsetneqq \mathscr{T}_{\infty}\subsetneqq \mathscr{T}_q\subsetneqq\mathscr{T}_p.
	\end{eqnarray*}
\end{proposition}

\begin{proof}
	We begin by proving that $\mathscr{T}\subsetneqq \mathscr{T}_{\infty}$. It is easy to see that $\mathscr{T}\subset\mathscr{T}_{\infty}$. To show the proper inclusion, we only need to note that for any $r>0$ there does not exist $U\in \mathscr{T}$ with $(0)\in U $ such that $U\subseteq B_r^{\infty,\mathbb{N}}$.
	
	We proceed to prove that $\mathscr{T}_{\infty} \subsetneqq \mathscr{T}_q$, where $ q\in [1,\infty)$. For any given $r\in(0, 1]$, $(x_i)_{i\in\mathbb{N}}\in \mathbb{R}^{\mathbb{N}}$ and
	$(y_i)_{i\in\mathbb{N}}\in (x_i)_{i\in\mathbb{N}}+B_r^{\infty,\mathbb{N}}$,
	let
	$$s_0\triangleq\max \left\{\sup_{  i\in\mathbb{N}}|y_i-x_i|,\frac{r}{2}\right\}<r\
	\hbox{ and }\ s_1\triangleq\frac{1}{2}\min\{r-s_0,s_0\}>0.
	$$
	Note that for any $(z_i)_{i\in\mathbb{N}}\in (y_i)_{i\in\mathbb{N}}+B_{s_1}^{q,\mathbb{N}}$, it holds that
	$$\sup_{ i\in\mathbb{N}}|z_i-y_i|\leqslant \left(\sum_{i=1}^{\infty}|z_i-y_i|^q \right)^{\frac{1}{q}}<s_1.$$
	Therefore,
	$$\sup_{ i\in\mathbb{N}}|z_i-x_i|\leqslant \sup_{ i\in\mathbb{N}}|z_i-y_i|+\sup_{ i\in\mathbb{N} }|y_i-x_i|<s_0+s_1<r.$$
	This gives $(z_i)_{i\in\mathbb{N}}\in (x_i)_{i\in\mathbb{N}}+B_r^{\infty,\mathbb{N}},$
	which implies that $(y_i)_{i\in\mathbb{N}}+B_{s_1}^{q,\mathbb{N}}\subseteq (x_i)_{i\in\mathbb{N}}+B_r^{\infty,\mathbb{N}}$ and hence $\mathscr{T}_{\infty} \subset \mathscr{T}_q$. To prove the proper inclusion we only need to note that there is no $r_1,r_2>0$ satisfying $B_{r_1}^{\infty,\mathbb{N}}\subseteq B_{r_2}^{q,\mathbb{N}}.$
	
	We now turn to prove that $\mathscr{T}_{p_2} \subsetneqq \mathscr{T}_{p_1}$ where $1 \leqslant p_1<p_2<\infty$. For any $r\in(0, 1]$, $(x_i)_{i\in\mathbb{N}}\in \mathbb{R}^{\mathbb{N}}$ and $(y_i)_{i\in\mathbb{N}}\in (x_i)_{i\in\mathbb{N}}+B_{r }^{p_2,\mathbb{N}},$
	let
	$$s_2\triangleq\left(\sum_{i=1}^{\infty}|y_i-x_i|^{p_2}\right)^{\frac{1}{p_2}}<r\leqslant 1\
	\hbox{ and } \ s_3\triangleq(r-s_2)^{\frac{p_2}{p_1}}.
	$$
	Note that for any $(z_i)_{i\in\mathbb{N}}\in (y_i)_{i\in\mathbb{N}}+B_{s_3}^{p_1,\mathbb{N}}$, it holds that
	$$\left(\sum_{ i=1}^{\infty}|z_i-y_i|^{p_1} \right)^{\frac{1}{p_1}}<s_3\leqslant 1,\quad\sum_{ i=1}^{\infty}|z_i-y_i|^{p_1}<1.
	$$
	Then  $\displaystyle 1>\sum_{ i=1}^{\infty}|z_i-y_i|^{p_1}\geqslant \sum_{ i=1}^{\infty}|z_i-y_i|^{p_2}$
	and hence
	\begin{eqnarray*}
		s_3>\left(\sum_{ i=1}^{\infty}|z_i-y_i|^{p_1}\right)^{\frac{1}{p_1}}\geqslant \left(\sum_{ i=1}^{\infty}|z_i-y_i|^{p_2}\right)^{\frac{1}{p_1}}
		= \left[\left(\sum_{ i=1}^{\infty}|z_i-y_i|^{p_2}\right)^{\frac{1}{p_2}}\right]^{\frac{p_2}{p_1}}.
	\end{eqnarray*}
	Therefore,
	\begin{eqnarray*}
		\left(\sum_{ i=1}^{\infty}|z_i-x_i|^{p_2}\right)^{\frac{1}{p_2}}&\leqslant &\left(\sum_{ i=1}^{\infty}|z_i-y_i|^{p_2}\right)^{\frac{1}{p_2}}+\left(\sum_{ i=1}^{\infty}|y_i-x_i|^{p_2}\right)^{\frac{1}{p_2}}\\
		&<& (s_3)^{\frac{p_1}{p_2}}+s_2
		=r,
	\end{eqnarray*}
	which proves that $(y_i)_{i\in\mathbb{N}}+B_{s_3}^{p_1,\mathbb{N}}\subseteq (x_i)_{i\in\mathbb{N}}+B_r^{p_2,\mathbb{N}}$ and hence $\mathscr{T}_{p_2} \subset \mathscr{T}_{p_1}$. Meanwhile, the proper inclusion follows from the fact that there is no $r_1,r_2>0$ such that
	$B_{r_1}^{p_2,\mathbb{N}}\subseteq B_{r_2}^{p_1,\mathbb{N}}.$
	To see this we only need to note that $\left(\frac{1}{i^{\frac{1}{p_1}}}\right)_{i\in\mathbb{N}}\in \ell^{p_2}\backslash \ell^{p_1}.$
	
	Further, we prove that $\mathscr{T}_{1} \subsetneqq \mathscr{T}_{p} $ where $0< p<1$. Given $0<r\leqslant 1$, $(x_i)_{i\in\mathbb{N}}\in \mathbb{R}^{\mathbb{N}}$ and $(y_i)_{i\in\mathbb{N}}\in (x_i)_{i\in\mathbb{N}}+B_r^{1,\mathbb{N}},$
	it holds that $\sum\limits_{ i=1}^{\infty}|x_i-y_i|   <r\leqslant 1.$
	Let
	$$r_0\triangleq\frac{1}{2}\min \left\{r-\sum_{ i=1}^{\infty}|x_i-y_i|,\;\max \left\{\frac{r}{2},\sum_{ i=1}^{\infty}|x_i-y_i|\right\} \right\}.
	$$
	Note that for any $(z_i)_{i\in\mathbb{N}}\in (y_i)_{i\in\mathbb{N}}+B_{r_0}^{p,\mathbb{N}}$, it holds that $\sum\limits_{ i=1}^{\infty}|z_i-y_i|^{p}<r_0<1.$
	Hence
	$$\sum_{ i=1}^{\infty}|z_i-y_i| \leqslant \sum_{ i=1}^{\infty}|z_i-y_i|^{p}<r_0.$$
	We thus get
	\begin{eqnarray*}
		\sum_{ i=1}^{\infty}|z_i-x_i| \leqslant  \sum_{ i=1}^{\infty}|z_i-y_i| + \sum_{ i=1}^{\infty}|y_i-x_i| < r_0+ \sum_{ i=1}^{\infty}|y_i-x_i|
		<r.
	\end{eqnarray*}
	which proves that $(y_i)_{i\in\mathbb{N}}+B_{r_0}^{p,\mathbb{N}}\subseteq (x_i)_{i\in\mathbb{N}}+B_r^{1,\mathbb{N}}$ and hence implies that $\mathscr{T}_{1} \subset\mathscr{T}_{p} $. To see the proper inclusion of $\mathscr{T}_{1} \subsetneqq \mathscr{T}_{p} $ we only need to note that there is no $r_1,r_2>0$ such that $B_{r_1}^{1,\mathbb{N}}\subseteq B_{r_2}^{p,\mathbb{N}},$ which follows from the fact that $\left(\frac{1}{i^{\frac{2}{p +1}}}\right)_{i\in\mathbb{N}}\in \ell^{1}\backslash \ell^{p}.$
	
	Finally, the above argument can be applied to prove that $\mathscr{T}_{p_2} \subsetneqq \mathscr{T}_{p_1}$, where $0< p_1<p_2<1$. 
	This completes the proof of Proposition \ref{toplogies inclusion proposition}.
\end{proof}


\medskip

Next, let us recall a simple fact: Suppose that $f(\cdot)$ is a real analytic function (on a suitable sub-interval of $\mathbb{R}$) with the power series expansion $f(x)=\sum\limits_{n=0}^{\infty}a_n x^n,$
where $\{a_n\}_{n=0}^{\infty}$ is a sequence of real numbers so that $\sum\limits_{n=0}^{\infty}|a_n|r^n<\infty$ for some $r\in (0,+\infty)$.
Then, the analyticity of $f(\cdot)$ can be ``measured" by three equivalent manners:
\begin{itemize}
	\item[1)] The larger $r$ is, the ``stronger" the analyticity of $f(\cdot)$ is;
	\item[2)] The closer $r$ to the infinity point is, the ``stronger" the analyticity of $f(\cdot)$ is;
	\item[3)] The  closer $\frac{1}{r}$ to 0 is, the ``stronger" the analyticity of $f(\cdot)$ is.
\end{itemize}
Let $\widetilde{\mathbb{R}^{\mathbb{N}} }\triangleq \mathbb{R}^{\mathbb{N}} \sqcup\{\infty\}$, with $\infty$ being a given point not belonging to $\mathbb{R}^{\mathbb{N}}$, which can be viewed as an infinity point of $\mathbb{R}^{\mathbb{N}}$. Motivated by the above analysis, three questions arise naturally:
\begin{itemize}
	\item[1)] What points in $\mathbb{R}^{\mathbb{N}}$ are large?
	\item[2)] What points in $\mathbb{R}^{\mathbb{N}}$ are close to $\infty$?
	\item[3)] What points in $\mathbb{R}^{\mathbb{N}}$ are close to $(0)$?
\end{itemize}
The situation in $\mathbb{R}^{\mathbb{N}}$ is much more complicated than that in the one (or even finite) dimensional setting. Indeed, the main difficulty is how to properly measure points in $\mathbb{R}^{\mathbb{N}}$ are close or far to $\infty$. For example,
$\sum\limits_{i=1}^{\infty}\frac{1}{i^p}<\infty$
for $1<p<\infty$ and $\sum\limits_{i=1}^{\infty}\frac{1}{i^p}=\infty$
for $0<p\leqslant 1$. In this sense, $\big(\frac{1}{\sqrt{i}}\big)_{i\in\mathbb{N}}$ is a ``large" point in $\mathscr{T}_{2}$, while $\big(\frac{1}{i}\big)_{i\in\mathbb{N}}$ is not a ``large" point in $\mathscr{T}_{2}$. Meanwhile, these two points are ``large" points in $\mathscr{T}_{1}$ and both of their $\ell^1$-norms are $\infty$ (and therefore such a norm does not distinguish them). Nevertheless, the third viewpoint in the above one dimensional case provides some useful idea for its infinite dimensional counterpart.
Precisely, we introduce the following family of subsets (in $\widetilde{\mathbb{R}^{\mathbb{N}} }$): for $0<p<\infty$, we use the notation $\mathscr{B}^p$ to denote the union of $\mathscr{B}_p$ and the collection of all sets of the following form:
\begin{eqnarray*}
	\left\{(x_i)_{i\in\mathbb{N}}\in \mathbb{R}^{\mathbb{N}}:\;x_i\neq 0 \text{ for each }i\in \mathbb{N}\,\text{ and } \,\,\sum_{i=1}^{\mathbb{N}}\frac{1}{|x_i|^p}<r \right\} \sqcup \{\infty\},
\end{eqnarray*}
where $r$ goes over positive numbers, and we use the notation $\mathscr{B}^{\infty}$ to denote the union of $\mathscr{B}_{\infty}$ and the collection of all sets of the following form:
\begin{eqnarray*}
	\left\{(x_i)_{i\in\mathbb{N}}\in \mathbb{R}^{\mathbb{N}}:\;x_i\neq 0 \text{ for each }i\in \mathbb{N}\,\text{ and } \,\,\sup_{i\in\mathbb{N}}\frac{1}{|x_i|}<r \right\}\sqcup\{\infty\},
\end{eqnarray*}
where $r$ goes over positive numbers. It is easy to check that, for any $p\in (0,\infty]$, $\mathscr{B}^p$ is a base for a topological space (on $\widetilde{\mathbb{R}^{\mathbb{N}} }$), denoted by $\mathcal {T}^p$.

\begin{proposition} \label{subspace topology}
	The subspace topology of $(\widetilde{\mathbb{R}^{\mathbb{N}} }, \mathcal {T}^p)$ restricted on $\mathbb{R}^{\mathbb{N}}$ is $\mathscr{T}_p$.
\end{proposition}

\begin{proof}
	We only consider the case that $p\in (0,\infty)$. It suffices to show that for any given $r>0$ and $(x_i)_{i\in\mathbb{N}}\in \mathbb{R}^{\infty}$ with $x_i\neq 0$ for each $i\in \mathbb{N}$ and
	$\sum\limits_{i=1}^{\infty}\frac{1}{|x_i|^p}<r$,
	there exists $\varepsilon>0$ such that for each $(y_i)_{i\in\mathbb{N}}\in \mathbb{R}^{\mathbb{N}}$ with
	$\sum\limits_{i=1}^{\infty} |y_i|^p <\varepsilon$, it holds that
	$x_i+y_i\neq 0$ for each $i\in \mathbb{N}$ and
	\begin{eqnarray*}
		\sum_{i=1}^{\infty}\frac{1}{|x_i+y_i|^p}<r.
	\end{eqnarray*}
	
	For the above $r>0$ and $(x_i)_{i\in\mathbb{N}}\in \mathbb{R}^{\mathbb{N}}$, clearly, $\lim\limits_{i\to\infty}|x_i|=+\infty$, and one can find $M>1$ such that
	\begin{eqnarray*}
		\bigg(\frac{M}{M-1}\bigg)^p\cdot\sum_{i=1}^{\infty}\frac{1}{|x_i|^p}<r.
	\end{eqnarray*}
	Thus there exists $N\in\mathbb{N}$ such that $ |x_i|\geqslant M$ for all $i\geqslant N$.
	For any $i\geqslant N$ and $y_i\in \mathbb{R}$ with $|y_i|<1$,
	\begin{eqnarray*}
		\frac{1}{|x_i+y_i|}&=&\frac{|x_i|}{|x_i+y_i|}\cdot \frac{1}{|x_i|}\leqslant \frac{|x_i|}{|x_i|-|y_i|}\cdot \frac{1}{|x_i|}\\
		&< &\frac{|x_i|}{|x_i|-1}\cdot \frac{1}{|x_i|}\leqslant \frac{M}{M-1}\cdot \frac{1}{|x_i|}.
	\end{eqnarray*}
	On the other hand,
	\begin{eqnarray*}
		\lim_{(y_1,y_2,\cdots,y_N)\to(0,0,\cdots,0)}\sum_{i=1}^{N}\frac{1}{|x_i+y_i|^p}=\sum_{i=1}^{N}\frac{1}{|x_i|^p}.
	\end{eqnarray*}
	Hence, one can find $\varepsilon>0$ so that for each $(y_i)_{i\in\mathbb{N}}\in \mathbb{R}^{\infty}$ with
	$\sum\limits_{i=1}^{\infty} |y_i|^p <\varepsilon$,
	\begin{eqnarray*}
		\sum_{i=1}^{\infty}\frac{1}{|x_i+y_i|^p}
		&=&\sum_{i=1}^{N}\frac{1}{|x_i+y_i|^p}+\sum_{i=N+1}^{\infty}\frac{1}{|x_i+y_i|^p}\\
		&\leqslant&\sum_{i=1}^{N}\frac{1}{|x_i+y_i|^p}+\bigg(\frac{M}{M-1}\bigg)^p\cdot\bigg(\sum_{i=N+1}^{\infty}\frac{1}{|x_i|^p}\bigg)\\
		&<&r.
	\end{eqnarray*}
	This completes the proof of Proposition \ref{subspace topology}.
\end{proof}

\begin{remark}
	Obviously, $\mathcal {T}^p$  is not the usual one-point compactification of $\mathscr{T}_p$, and therefore we use the notion $\widetilde{\mathbb{R}^{\mathbb{N}} }$, instead of $\widehat{\mathbb{R}^{\mathbb{N}} }$ for the usual one-point compactification of $(\mathbb{R}^{\mathbb{N}},\mathscr{T}_p)$.
	The above topologies and conclusions also hold on any countable Cartesian product of $\mathbb{R}$ instead of $\mathbb{R}^{\mathbb{N}}$.
\end{remark}

\subsection{The Cauchy-Kowalevski Type Theorem}\index{The Cauchy-Kowalevski Type Theorem}
In this subsection, we consider the following form of initial value problem:
\begin{equation}\label{one22-8-11}
\left\{
\begin{array}{ll}
 \displaystyle\partial^m_t u(t,\textbf{x})=f\big(t,\textbf{x},u, (\partial^{\beta}_{\textbf{x}}\partial^{j}_t  u)_{(\beta,j)\in \mathbb{I}_m} \big),\\[2mm]\displaystyle
 \partial^k_tu(t,\textbf{x})\mid_{t=0}=\phi_k(\textbf{x}),\ \ k=0,1,\cdots,m-1.
 \end{array}
\right.
\end{equation}
Here $m$ is a given positive integer, $\mathbb{I}_m\triangleq \{(\beta,j):\beta\in \mathbb{N}_0^{(\mathbb{N})},j\in\mathbb{N}_0, j<m, 1\leqslant|\beta|+j\leqslant m\}$, $t\in \mathbb{R}, \textbf{x}=(x_i)_{i\in\mathbb{N}}\in \mathbb{R}^{\mathbb{N}}$, $\partial^{\beta}_{\textbf{x}}\;{\buildrel \triangle \over =}\;\partial^{\beta_1}_{x_1 }\cdots\partial^{\beta_n}_{x_n }$  for
\begin{eqnarray*}
 \beta&=&(\beta_1,\cdots,\beta_n,0,\cdots)\in \mathbb{N}_0^{(\mathbb{N})} (\text{for some positive }n\in\mathbb{N}),\\
  |\beta|&\triangleq &\beta_1+\cdots+\beta_n,
\end{eqnarray*}
the unknown $u$ is a real-valued function depending on $t$ and $\textbf{x}$; $f$ is a non-linear real-valued function depending on $t,$ $\textbf{x}$, $u$ and all of its derivatives of the form $\partial^{\beta}_\textbf{x}\partial^{j}_t  u$ where $(\beta,j)\in \mathbb{I}_m$.

Note that the values $u(0,(0)_{i\in\mathbb{N}})$ and $\partial^{\beta}_\textbf{x}\partial^{j}_t  u(0,(0)_{i\in\mathbb{N}})$ (for $(\beta,j)\in \mathbb{I}_m$) can be determined by the equation (\ref{one22-8-11}). For simplicity, we write these values by
$$u_0\triangleq \phi_0((0)_{i\in\mathbb{N}}),\quad w_{\beta,j}^0\triangleq \partial^{\beta}_\textbf{x}\phi_j((0)_{i\in\mathbb{N}}),\ \hbox{ and }\ \textbf{w}^0\triangleq  (w_{\beta,j}^0)_{(\beta,j)\in \mathbb{I}_m}.$$
It is easy to see that $f$ is a function on $\mathbb{R}\times \mathbb{R}^{\mathbb{N}}\times \mathbb{R}\times \mathbb{R}^{\mathbb{I}_m}$, which is also a countable Cartesian product of $\mathbb{R}$. Hence, by Remark \ref{230820rem1}, we can use the same notations as in the Definition \ref{def of analyticity} when no confusion can arise. Suppose that $f$ is analytic at $(0,(0)_{i\in\mathbb{N}},u_0, \textbf{w}^0 )$ with the monomial expansion
$$
\begin{array}{ll}
\displaystyle f\big(t,(x_i)_{i\in\mathbb{N}},u, (\partial^{\beta}_{\textbf{x}}\partial^{j}_t  u)_{(\beta,j)\in \mathbb{I}_m} \big)\\[3mm]
\displaystyle=\sum_{\alpha\in\mathbb{N}_0^{(\mathbb{N})}}C_{\alpha}\big(t,(x_i)_{i\in\mathbb{N}},u-u_0, (\partial^{\beta}_\textbf{x}\partial^{j}_t  u-w_{\beta,j}^0)_{(\beta,j)\in \mathbb{I}_m} \big)^{\alpha}
\end{array}
$$
and let
$$F(t,\textbf{x},u, \textbf{w})\triangleq\sum_{\alpha\in \mathbb{N}_0^{(\mathbb{N})}}|C_{\alpha}| (t,\textbf{x},u, \textbf{w})^{\alpha},$$ where
$t,u\in\mathbb{R}$,\,$\textbf{w}= (w_{\beta,j})_{(\beta,j)\in \mathbb{I}_m}\in \mathbb{R}^{\mathbb{I}_m}$ and $\textbf{x}= (x_i)_{i\in\mathbb{N}}\in \mathbb{R}^{\mathbb{N}}.$ We also suppose that for each $0\leqslant k\leqslant m-1$, $\phi_k$ is analytic at $(0)_{i\in\mathbb{N}}$ with the monomial expansion
$$\phi_k(\textbf{x})=\sum_{\alpha\in \mathbb{N}_0^{(\mathbb{N})}}C_{\alpha,k}\textbf{x}^{\alpha}$$ and let
$$\Phi_k (t,\textbf{x},u, \textbf{w})\triangleq\sum_{\alpha\in \mathbb{N}_0^{(\mathbb{N})}}|C_{\alpha,k}|\textbf{x}^{\alpha}.$$

By means of the majorant method, we shall show below the following Cauchy-Kowalevski type theorem of infinitely many variables:

\begin{theorem}\label{Infinite Dimensional Cauchy-Kowalevski Theorem}
Suppose $1\leqslant p<\infty$ and that the monomial expansions of $\Phi_k$, $0\leqslant k\leqslant m-1$ and $F$ at $(0,(0)_{i\in\mathbb{N}},0,(0)_{(\beta,j)\in \mathbb{I}_m})$ are absolutely convergent at a point near $\infty$ in the topology $\mathcal {T}^p$(see Definition \ref{def of analyticity} and Remark \ref{230820rem1}). Then the Cauchy problem (\ref{one22-8-11}) admits a locally analytic solution (at $(0,(0)_{i\in\mathbb{N}})$), which is unique in the class of analytic functions under the topology $\mathscr{T}_{p'}$ where $p'$ is the usual H\"older conjugate of $p$. Furthermore, the solution $u$ is Fr\'{e}chet differentiable with respect to $\ell^{p'}(\mathbb{N}_0)$ in a neighborhood of $(0,(0)_{i\in\mathbb{N}}) $ in the topology $\mathcal {T}^{p'}$ and the corresponding Fr\'{e}chet derivative  $Du$ is continuous.
\end{theorem}
\begin{proof} Suppose that $u$ is analytic at $(0,(0)_{i\in\mathbb{N}})$ with the monomial expansion
$$u(t,\textbf{x})=\sum_{\alpha\in \mathbb{N}_0^{(\mathbb{N})}} c_{\alpha}(t,\textbf{x})^{\alpha}.$$
Then the coefficients $c_{\alpha}$'s are uniquely determined by (\ref{one22-8-11}) which shows the uniqueness of $u$.

We only need to prove the existence of locally analytic solutions to the Cauchy problem (\ref{one22-8-11}). Replacing respectively $f(\cdot)$ and $\phi_k(\cdot)$  in (\ref{one22-8-11}) by their majority functions ($F(\cdot)$ and $\Phi_k(\cdot)$),  we obtain a new Cauchy problem, called the majority equation
\begin{equation}\label{majority equation 1}
\left\{
\begin{array}{ll}
 \displaystyle\partial^m_t u(t,\textbf{x})=F\big(t,\textbf{x},u, (\partial^{\beta}_{\textbf{x}}\partial^{j}_t  u)_{(\beta,j)\in \mathbb{I}_m} \big),\\[2mm]\displaystyle
 \partial^k_tu(t,\textbf{x})\mid_{t=0}=\Phi_k(\textbf{x}),\ \ k=0,1,\cdots,m-1.
 \end{array}
\right.
\end{equation}
Note that the analytic solution to the majority equation (\ref{majority equation 1}) is also a majority function (see Definition \ref{def of majority function}) of the analytic solution to the Cauchy problem (\ref{one22-8-11}). This is equivalent to say that if the analytic solution $u$ of the Cauchy problem (\ref{one22-8-11}) has the following monomial expansion at $(0,(0)_{i\in\mathbb{N}})$
\begin{eqnarray}
 u(t,\textbf{x})=\sum_{ \alpha\in \mathbb{N}_0^{(\mathbb{N})}} a_{\alpha }(t,\textbf{x} )^{\alpha}.\label{power series solution of Cauchy problem}
\end{eqnarray}
Meanwhile, if the analytic solution $U$ of the majority equation (\ref{majority equation 1}) has the following monomial expansion at $(0,(0)_{i\in\mathbb{N}})$
\begin{eqnarray*}
 U(t,\textbf{x})=\sum_{\alpha\in \mathbb{N}_0^{(\mathbb{N})}} A_{\alpha }(t,\textbf{x})^{\alpha}.
\end{eqnarray*}
Then we have $|a_{\alpha}|\leqslant A_{\alpha}$ for each $\alpha\in  \mathbb{N}_0^{(\mathbb{N})}.$
Therefore, the solvability of the majority equation (\ref{majority equation 1}) implies the solvability of the Cauchy problem (\ref{one22-8-11})(the convergence of monomial expansion (\ref{power series solution of Cauchy problem})). Motivated by this observation, we only need to prove the solvability of one majority equation of (\ref{one22-8-11}).

Firstly, we need to establish the majority equation of the Cauchy problem (\ref{one22-8-11}). To do this, by the assumption that the monomial expansions of $\Phi_k$, $0\leqslant k\leqslant m-1$ and $F$ (at $(0,(0)_{i\in\mathbb{N}},0,(0)_{(\beta,j)\in \mathbb{I}_m})$) are absolutely convergent at a point near $\infty$ in the topology $\mathcal {T}^p$. By Lemma \ref{absolutely convergent lemma}, there exists $\rho>0, 0<t_1<1,0<u_1<1, \textbf{r}=(r_i)_{i\in\mathbb{N}}\in  \mathbb{R}^{\mathbb{N}}$ with $0<r_i<1$ for each $i\in \mathbb{N}$ and $\textbf{w}_1=  (\frac{\rho}{ w_{\beta,j}^1 } )_{(\beta,j)\in \mathbb{I}_m}$
with $0<w_{\beta,j}^1<1$ for each $(\beta,j)\in \mathbb{I}_m$ such that
\begin{eqnarray*}
 t_1^{p}+u_1^p+\sum_{i=1}^{\infty}r_i^p+\sum_{(\beta,j)\in \mathbb{I}_m} (w_{\beta,j}^1)^p =1
\end{eqnarray*}
and the monomial expansions of $\Phi_k$, $0\leqslant k\leqslant m-1$ and $F$ at\\ $(0,(0)_{i\in\mathbb{N}},0,(0)_{(\beta,j)\in \mathbb{I}_m})$ are absolutely convergent at $(\frac{\rho}{t_1},(\frac{\rho}{r_i})_{i\in\mathbb{N}},\frac{\rho}{u_1}, \textbf{w}_1)$. Then for $(y_i)_{i\in\mathbb{N}}\in \mathbb{R}^{\mathbb{N}}$ with $||(y_i)_{i\in\mathbb{N}}||_{\ell^{p'}(\mathbb{N})}<\rho$ it holds that
\begin{eqnarray}\label{230822e1}
\sum_{\alpha\in \mathbb{N}_0^{(\mathbb{N})}}|C_{\alpha}|\cdot  \rho^{|\alpha|}
\leqslant \sum_{\alpha\in \mathbb{N}_0^{(\mathbb{N})}}|C_{\alpha}|\cdot \bigg|\bigg(\frac{\rho}{t_1},\bigg(\frac{\rho}{r_i}\bigg)_{i\in\mathbb{N}},\frac{\rho}{u_1}, \textbf{w}_1\bigg)^{\alpha}\bigg|
<\infty,
\end{eqnarray}
\begin{eqnarray}\label{230822e2}
&&\sum_{\alpha\in \mathbb{N}_0^{(\mathbb{N})}}|C_{\alpha}|\cdot  \bigg| \bigg(\frac{\rho}{t_1},(y_i )_{i\in\mathbb{N}},\frac{\rho}{u_1}, \textbf{w}_1 \bigg) ^{\alpha}\bigg|\\
&= &\sum_{\alpha\in \mathbb{N}_0^{(\mathbb{N})}}|C_{\alpha}|\cdot \bigg|\bigg(\frac{\rho}{t_1},\bigg(\frac{|r_iy_i|}{r_i}\bigg)_{i\in\mathbb{N}},\frac{\rho}{u_1}, \textbf{w}_1\bigg)^{\alpha}\bigg|\nonumber\\
&\leqslant& \sum_{\alpha\in \mathbb{N}_0^{(\mathbb{N})}}|C_{\alpha}|\cdot \bigg|\bigg(\frac{\rho}{t_1},\bigg(\frac{\sum_{j=1}^{\infty}|r_j y_j|}{r_i}\bigg)_{i\in\mathbb{N}},\frac{\rho}{u_1}, \textbf{w}_1\bigg)^{\alpha}\bigg|\nonumber\\
&\leqslant &\sum_{\alpha\in \mathbb{N}_0^{(\mathbb{N})}}|C_{\alpha}|\cdot \bigg|\bigg(\frac{\rho}{t_1},\bigg(\frac{\rho}{r_i}\bigg)_{i\in\mathbb{N}},\frac{\rho}{u_1}, \textbf{w}_1\bigg)^{\alpha}\bigg|\nonumber\\
&<&\infty.\nonumber
\end{eqnarray}
Meanwhile, for each $0\leqslant k\leqslant m-1$, it holds that
\begin{eqnarray}\label{230822e3}
\sum_{\alpha\in \mathbb{N}_0^{(\mathbb{N})}}|C_{\alpha,k}| \cdot  \rho^{|\alpha |}\leqslant \sum_{\alpha\in \mathbb{N}_0^{(\mathbb{N})}}|C_{\alpha,k}|\cdot\bigg(\frac{\rho}{r_i}\bigg)_{i\in\mathbb{N}}^{\alpha}
 <\infty
\end{eqnarray}
and
\begin{eqnarray}
&&\sum_{\alpha\in \mathbb{N}_0^{(\mathbb{N})}}|C_{\alpha,k}|\cdot  |(y_i )_{i\in\mathbb{N}}^{\alpha}|\nonumber\\
&= &\sum_{\alpha\in \mathbb{N}_0^{(\mathbb{N})}}|C_{\alpha,k}|\cdot \bigg(\frac{|r_iy_i|}{r_i}\bigg)_{i\in\mathbb{N}}^{\alpha}\label{230822e4}\\
&\leqslant &\sum_{\alpha\in \mathbb{N}_0^{(\mathbb{N})}}|C_{\alpha,k}|\cdot \bigg(\frac{\sum_{j=1}^{\infty}|r_j y_j|}{r_i}\bigg)_{i\in\mathbb{N}}^{\alpha}\nonumber\\
&\leqslant &\sum_{\alpha\in \mathbb{N}_0^{(\mathbb{N})}}|C_{\alpha,k}|\cdot \bigg(\frac{\rho}{r_i}\bigg)_{i\in\mathbb{N}}^{\alpha}
<\infty.\nonumber
\end{eqnarray}
Then we get a majority equation
\begin{equation}\label{majority equation}
\left\{
\begin{array}{ll}
 \displaystyle \partial^m_t u(t,\textbf{x}) =F\bigg(t,\bigg(\frac{\sum_{j=1}^{\infty}r_j x_j}{r_i}\bigg)_{i\in\mathbb{N}},u-u_2, \left(\partial^{\beta}_\textbf{x}\partial^{j}_t  u-w_{\beta,j}^2\right)_{(\beta,j)\in \mathbb{I}_m} \bigg),\\[2mm]\displaystyle
\partial^k_tu\mid_{t=0}=\Phi_k\bigg(\bigg(\frac{\sum_{j=1}^{\infty}r_j x_j}{r_i}\bigg)_{i\in\mathbb{N}}\bigg),\,k=0,1,\cdots,m-1,
 \end{array}
\right.
\end{equation}
where $\textbf{x}=(x_i)_{i\in\mathbb{N}}\in\mathbb{R}^{\mathbb{N}},$
and $u_2,w_{\beta,j}^2,(\beta,j)\in \mathbb{I}_m$ are determined by the initial condition in (\ref{majority equation}). Precisely,
\begin{eqnarray*}
u_2 \triangleq \Phi_0((0)_{i\in\mathbb{N}}),\quad w_{\beta,j}^2\triangleq\partial^{\beta}_\textbf{x}\left(\Phi_j\bigg(\bigg(\frac{\sum_{k=1}^{\infty}r_k x_k}{r_i}\bigg)_{i\in\mathbb{N}}\bigg) \right)\Bigg|_{(x_i)_{i\in\mathbb{N}}=(0)_{i\in\mathbb{N}}},
\end{eqnarray*}
for any $(\beta,j)\in \mathbb{I}_m$. One can easily see that $\frac{w_{\beta,j}^2}{\textbf{r}^{\beta}}$ only depends on $|\beta|$ and $j$. Let $h_{|\beta|,j}^0\triangleq\frac{w_{\beta,j}^2}{\textbf{r}^{\beta}},(\beta,j)\in \mathbb{I}_m$. Now we consider solutions of the majority equation (\ref{majority equation}) of the following form
$$
u(t,(x_i)_{i\in\mathbb{N}})=U \left(t,\sum_{i=1}^{\infty}r_ix_i \right)=U(t,y),
$$
where $y=\sum\limits_{i=1}^{\infty}r_ix_i.$ If $U\bigg(t,\sum\limits_{i=1}^{\infty}r_ix_i \bigg)$ is a solution of (\ref{majority equation}), then it holds that
\begin{equation}\nonumber
\left\{
\begin{array}{ll}
 \displaystyle  \partial^m_tU\bigg(t,\sum_{j=1}^{\infty}r_j x_j \bigg) =F\Bigg(t,\bigg(\frac{\sum_{j=1}^{\infty}r_j x_j}{r_i}\bigg)_{i\in\mathbb{N}},U-u_2, \left(\textbf{r}^{\beta}\partial^{|\beta|}_y\partial^{j}_t  U-\textbf{r}^{\beta}h_{|\beta|,j}^0\right)_{(\beta,j)\in \mathbb{I}_m}\Bigg),\\[2mm]\displaystyle
 \partial^k_tU\mid_{t=0}=\Phi_k\Bigg(\bigg(\frac{\sum_{j=1}^{\infty}r_j x_j}{r_i}\bigg)_{i\in\mathbb{N}}\Bigg),\,k=0,1,\cdots,m-1.
 \end{array}
\right.
\end{equation}
Let
\begin{eqnarray*}
G\Big(t,y,u, (h_{i,j})_{(i,j)\in I_m}\Big)\triangleq F\Big(t,\Big(\frac{y}{r_i}\Big)_{i\in\mathbb{N}},u, \big(\textbf{r}^{\beta}h_{|\beta|,j}\big)_{(\beta,j)\in \mathbb{I}_m}\Big)
\end{eqnarray*}
and $\Psi_k(y)\triangleq \Phi_k((\frac{y}{r_i})_{i\in\mathbb{N}})\,$for each $0\leqslant k\leqslant m-1$, where $t,y,u\in\mathbb{R}$,\,$\textbf{r}=(r_i)_{i\in\mathbb{N}}$, $I_m\triangleq\{(i,j):i,j\in\mathbb{N}_0, 1\leqslant i+j\leqslant m,\,j<m\}$ and $h_{i,j}\in\mathbb{R}$ for each $(i,j)\in I_m$.
By \eqref{230822e1}-\eqref{230822e4}, it is easy to check that $G$ and $\Psi_k,\,k=0,1,\cdots,m-1$ are analytic functions of finite variables with positive convergent radium. Then the solvability of the following equation with initial values can imply the locally solvability of majority equation (\ref{majority equation}):
\begin{equation}\label{one22-8-12}
\left\{
\begin{array}{ll}
 \displaystyle \partial^m_t U(t,y) =G\big(t,y,U-u_2,(\partial^{i}_y\partial^{j}_t  U-h_{i,j}^0)_{_{(i,j)\in I_m}}\big),\\[2mm]\displaystyle
\partial^k_tU\mid_{t=0}=\Psi_k(y),\quad k=0,1,\cdots,m-1.
 \end{array}
\right.
\end{equation}
Note that $G$ and $\Psi_k$'s are analytic functions of finite variables with positive convergent radium, $u_2$ and $h_{i,j}^0,\,(i,j)\in I_m$, are determined by initial condition (\ref{one22-8-12}). Then by finite dimensional Cauchy-Kowalevski theorem, equation (\ref{one22-8-12}) has a locally analytic solution at $(0,0)$. From the above arguments we deduce that the solution $u$ of Theorem \ref{Infinite Dimensional Cauchy-Kowalevski Theorem} has a majority function
\begin{eqnarray*}
M(t,(x_i)_{i\in\mathbb{N}})\triangleq U\bigg(t,\sum_{i=1}^{\infty}r_ix_i \bigg)=\sum_{n,m\in \mathbb{N}_0}C_{n,m}t^m \bigg(\sum_{i=1}^{\infty}r_ix_i \bigg)^n,
\end{eqnarray*}
where $(t,(x_i)_{i\in\mathbb{N}})\in B_r^{p',\mathbb{N}_0}$, $C_{n,m}\geqslant 0,\forall\, n,m\in \mathbb{N}_0$, and there exists $r>0$ such that $$
\sum\limits_{n,m\in \mathbb{N}_0}C_{n,m}r^{m+n} <\infty.
$$
By Lemma \ref{Frechet lemma}, $M$ is Fr\'{e}chet differentiable respect to $\ell^{p'}(\mathbb{N}_0)$ in $B_r^{p',\mathbb{N}_0}$ and the corresponding Fr\'{e}chet derivative $DM$ is continuous. Suppose that $u$ is the analytic solution of the Cauchy problem (\ref{one22-8-11}). Then $M$ is a majority function of $u$, $u$ is also Fr\'{e}chet differentiable respect to $\ell^{p'}(\mathbb{N}_0)$ in $B_r^{p',\mathbb{N}_0}$ and the corresponding Fr\'{e}chet derivative $Du$ is also continuous. This finishes the proof Theorem \ref{Infinite Dimensional Cauchy-Kowalevski Theorem}.
\end{proof}

Now, let us consider the Cauchy problem of the following first order linear homogenous partial differential equation  of infinitely many variables:
\begin{equation}\label{one order-zx}
\left\{
\begin{array}{ll}
 \displaystyle\partial_t u(t,\textbf{x}) -\sum_{i=1}^{\infty}A_{i}(t,\textbf{x}) \partial_{ x_i} u(t,\textbf{x}) -A_0(t,\textbf{x}) u(t,\textbf{x})= 0,\\[2mm]
 \displaystyle u(t,\textbf{x})\mid_{t=0}=\Phi(\textbf{x}).
 \end{array}
\right.
 \end{equation}
Here $t\in \mathbb{R}, \textbf{x}=(x_i)_{i\in\mathbb{N}}\in \mathbb{R}^{\mathbb{N}}$, the unknown $u$ is a real-valued function depending on $t$ and $\textbf{x}$, and the data $A_i(t,\textbf{x})$'s, and $\Phi$ are analytic at $(0)_{i\in\mathbb{N}_0}$. Let
\begin{equation}\label{def of G}
G(t,\textbf{x}, \textbf{w})\triangleq\sum_{i=0}^{\infty}A_i(t,\textbf{x})w_i,\quad \textbf{w}= (w_{ j})_{j\in\mathbb{N}_0}\in \mathbb{R}^{\mathbb{N}_0}.
 \end{equation}
By modifying the proof of Theorem \ref{Infinite Dimensional Cauchy-Kowalevski Theorem}, we can show the following result.
\begin{corollary}\label{corollary 20}
Suppose $1\leqslant p<\infty$ and $p'$ is the usual H\"older conjugate of $p$, the monomial expansion of $G $ at $(0,(0)_{i\in\mathbb{N}},(0)_{i\in\mathbb{N}_0})$ is absolutely convergent at a point near $\infty$ in the topology $\mathcal {T}^{p}$ and $D_\Phi^{(0)_{i\in\mathbb{N}}}=\mathbb{R}^{\mathbb{N}}$. Then there exists $r\in (0,\infty)$, independent of $\Phi$, such that the equation (\ref{one order-zx}) admits locally an analytic solution $u$ at $(0)_{i\in\mathbb{N}_0}$ and $ B_r^{p',\mathbb{N}_0}\subset D_u^{(0)_{i\in\mathbb{N}_0}}$. Furthermore, $u$ is Fr\'{e}chet differentiable with respect to $\ell^{p'}(\mathbb{N}_0)$ in $B_r^{p',\mathbb{N}_0}$, and the corresponding Fr\'{e}chet derivative $Du$ is continuous.
\end{corollary}

\subsection{The Holmgren Type Theorem }
In this section, we will establish Holmgren type theorem\index{Holmgren type theorem} for a class of ``$C^1$" functions as following.

Denote by $\Xi$ the set of all real-valued functions $f$ satisfying the following conditions:
\begin{itemize}
\item[$(1)$] there exists an open neighborhood $O$ of $(0)_{i\in\mathbb{N}_0}$ in $\ell^2(\mathbb{N}_0)$;
\item[$(2)$] $f$ is a real-valued function defined on $O;$
\item[$(3)$] $\partial_{t} f$, $\partial_{x_i} f$ exists and $\partial_{t} f$, $\partial_{x_i} f$ are a continuous function on $O$ with $\ell^2(\mathbb{N}_0)$ topology for each $i\in\mathbb{N};$
\item[$(4)$] there exists $0<M<\infty$ such that $|f_{x_i}(\textbf{x})|\leqslant M$ for any $\textbf{x}\in O$ and $i\in\mathbb{N}.$
\end{itemize}


\begin{theorem}\label{Infinite Dimensional Holmgren Type Theorem}
Suppose the monomial expansion of $G$ (defined by (\ref{def of G})) at $(0,(0)_{i\in\mathbb{N}},(0)_{i\in\mathbb{N}_0})$ is absolutely convergent at a point near $\infty$ in the topology $\mathcal {T}^{\frac{1}{2}}$ and $\Phi\in \Xi$. Then the solution to (\ref{one order-zx}) is locally unique in the class $\Xi$.
\end{theorem}
\begin{proof}
We only need to prove that if $U\in \Xi$ and $U$ satisfies
\begin{eqnarray*}
 \partial_t U(t,\textbf{x}) -\sum_{i=1}^{\infty}A_{i}(t,\textbf{x}) \partial_{ x_i} U(t,\textbf{x}) -A_0(t,\textbf{x}) U(t,\textbf{x})= 0,\,\,U\mid_{t=0}=0,
\end{eqnarray*}
then $U\equiv0$ locally in the class $\Xi$.

The proof will be divided into three steps. Firstly, we will take transformation of variables to simplify the equation. Secondly, we will prove that  $\widetilde{U}\equiv 0$ on $H_{\lambda}$ for sufficiently small $\lambda$(Recall Example 2.1 for the notation $H_{\lambda}$).  Thirdly, we will see that $U\equiv0$ on a neighborhood of $(0)_{i\in\mathbb{N}_0}$ in $\ell^2(\mathbb{N}_0)$.

We come to take transformation of variables to simplify the equation. To do this, note that for each $i\in \mathbb{N}_0$, $A_{i}$ has a monomial expansion at $(0)_{i\in\mathbb{N}_0}$ as following $A_{i}(t,\textbf{x})=\sum_{\alpha\in \mathbb{N}_0^{(\mathbb{N})}}A_{i,\alpha}(t,\textbf{x})^{\alpha}$.
By assumption and Lemma \ref{absolutely convergent lemma}, there exists $\rho_0>0$, $0<s_j,t_i<1$ for each $i,j\in\mathbb{N}_0$
such that $\sum_{j=0}^{\infty} s_j^{\frac{1}{2}}+\sum_{i=0}^{\infty}t_i^{\frac{1}{2}}  =1$
and
\begin{eqnarray}\label{equation 1133}
\sum_{i=0}^{\infty}\Bigg(\sum_{\alpha\in \mathbb{N}_0^{(\mathbb{N})}}|A_{i,\alpha}|\bigg(\frac{\rho_0}{s_j}\bigg)^{\alpha}_{j\in\mathbb{N}_0}\Bigg)\frac{\rho_0}{t_i}<\frac{1}{2}.
\end{eqnarray}
Let $r_0\triangleq \min\left\{1,\frac{\rho_0}{1+a_0}\right\}$ and $a_{i }\triangleq\max\Big\{t_i^{\frac{1}{4}},s_i^{\frac{1}{4}}\Big\}$ for each $i\in \mathbb{N}_0$.
Then we have
\begin{eqnarray}\label{230822e5}
\sum_{i=0}^{\infty}a_{i}^2  < \sum_{j=0}^{\infty} s_j^{\frac{1}{2}}+\sum_{i=0}^{\infty}t_i^{\frac{1}{2}}  =1
\end{eqnarray}
and by (\ref{equation 1133}) we have
\begin{eqnarray}\label{230822e7}
\sum_{i=0}^{\infty}\Bigg(\sum_{\alpha\in \mathbb{N}_0^{(\mathbb{N})}}|A_{i,\alpha}|\bigg(\frac{\rho_0}{a_j^4}\bigg)_{j\in\mathbb{N}_0}^{\alpha}\Bigg)\frac{\rho_0}{a_i^4}<\frac{1}{2}.
\end{eqnarray}
Choose a neighborhood $O_0$ of $(0)_{i\in \mathbb{N}_0}$ in $\ell^2(\mathbb{N}_0)$ such that $O_0 \subset B_{r_0}^{2,\mathbb{N}_0}$, for $(t,(x_i)_{i\in\mathbb{N}})\in O_0$, let
\begin{eqnarray*}
t^{'}&\triangleq& t+\sum_{i=1}^{\infty} a_i^4 x_i^2,\quad
x_i^{'}\triangleq x_i,\,\,\forall\,i\in\mathbb{N},\\
\widetilde{U}(t^{'},(x_i^{'})_{i\in\mathbb{N}})&\triangleq& U\bigg(t^{'}-\sum_{i=1}^{\infty}  a_i^4(x_i^{'})^2,(x_i^{'})_{i\in\mathbb{N}}\bigg).
\end{eqnarray*}
Secondly, we will prove that $\widetilde{U}\equiv 0$ on $V_{\lambda}$(Recall Example 2.1 for the notation $V_{\lambda}$) for sufficiently small $\lambda$. For this purpose, note that $\widetilde{U}\in \Xi$ and it holds that
\begin{eqnarray}\label{transformation 1}
&& \left(1-2\sum_{i=1}^{\infty}a_i^4x_i^{'}A_{i}\bigg(t^{'}-\sum_{i=1}^{\infty}  a_i^4(x_i^{'})^2,\textbf{x}^{'}\bigg)\right)\partial_{t'} \widetilde{U}\\
 &=&\sum_{i=1}^{\infty}A_{i}\bigg(t^{'}-\sum_{i=1}^{\infty}  a_i^4(x_i^{'})^2,\textbf{x}^{'}\bigg)\partial_{x_i^{'}} \widetilde{U}+A_0\bigg(t^{'}-\sum_{i=1}^{\infty}  a_i^4(x_i^{'})^2,\textbf{x}^{'}\bigg)\widetilde{U},\nonumber
\end{eqnarray}
where $\textbf{x}^{'}=(x_i^{'})_{i\in\mathbb{N}}$ and $(t',\textbf{x}^{'})\in\ell^2(\mathbb{N}_0)$. For $(t^{'},\textbf{x}^{'})\in B_{r_0}^{2}(\mathbb{N}_0)$, it holds that
\begin{eqnarray*}
\bigg|2\sum_{i=1}^{\infty}a_i^4x_i^{'}A_{i}\bigg(t^{'}-\sum_{i=1}^{\infty}  a_i^4(x_i^{'})^2,\textbf{x}^{'}\bigg)\bigg|&\leqslant&
 2\sum_{i=1}^{\infty}r_0\bigg|A_{i}\bigg(t^{'}-\sum_{i=1}^{\infty} a_i^4(x_i^{'})^2,\textbf{x}^{'}\bigg)\bigg|\\
 &\leqslant&2\sum_{i=0}^{\infty}\Bigg(\sum_{\alpha\in \mathbb{N}_0^{(\mathbb{N})}}|A_{i,\alpha}|\bigg(\frac{\rho_0}{s_j}\bigg)_{j\in\mathbb{N}_0}^{\alpha}\Bigg)\frac{\rho_0}{t_i}\\
 &<&1.
\end{eqnarray*}
Thus (\ref{transformation 1}) can be written as
\begin{eqnarray}\label{230822e8}
\partial_{t'} \widetilde{U}
 =\sum_{i=1}^{\infty}\widetilde{A}_{i}(t^{'},\textbf{x}^{'})\partial_{x_i^{'}} \widetilde{U}+\widetilde{A}_0(t^{'},\textbf{x}^{'})\widetilde{U},
\end{eqnarray}
where
\begin{eqnarray*}
 \widetilde{A}_{i}(t^{'},\textbf{x}^{'})\triangleq\frac{A_{i}\left(t^{'}-\sum\limits_{i=1}^{\infty}  a_i^4(x_i^{'})^2,\textbf{x}^{'}\right)}{1-2\sum\limits_{i=1}^{\infty} a_i^4x_i^{'}A_{i}\left(t^{'}-\sum\limits_{i=1}^{\infty}   a_i^4(x_i^{'})^2,\textbf{x}^{'}\right)},\quad\forall\, i\in\mathbb{N}_0.
\end{eqnarray*}
For each $i\in \mathbb{N}_0$, writing $ \widetilde{A}_{i}(t^{'},\textbf{x}^{'})$ as their monomial expansion at $(0)_{i\in\mathbb{N}_0}$:\;
$$
\widetilde{A}_{i}(t^{'},\textbf{x}^{'})=\sum_{\alpha\in \mathbb{N}_0^{(\mathbb{N})}}\widetilde{A_{i,\alpha}}(t^{'},\textbf{x}^{'})^{ \alpha}.
$$
It is easily seen that for each $i\in\mathbb{N}_0$,
\begin{eqnarray}\label{convergence 1}
&&\sum_{i=0}^{\infty}\Bigg(\sum_{\alpha\in \mathbb{N}_0^{(\mathbb{N})}}|\widetilde{A_{i,\alpha}}|\bigg(\frac{r_0}{a_j }\bigg)_{j\in\mathbb{N}_0}^{\alpha}\Bigg)\frac{r_0}{a_i^3}\\
&\leqslant&\frac{\sum\limits_{i=0}^{\infty}\left(\sum\limits_{\alpha\in \mathbb{N}_0^{(\mathbb{N})}}|A_{i,\alpha}|\left(\frac{r_0}{a_0}+\left(\sum\limits_{i=1}^{\infty}   a_i^2\right)r_0^2,\left(\frac{r_0}{a_j  }\right)_{j\in\mathbb{N}}\right)^{\alpha}\right)\frac{r_0}{a_i^3 }}{1-2\sum\limits_{i=0}^{\infty}\left(\sum\limits_{\alpha\in \mathbb{N}_0^{(\mathbb{N})}}|A_{i,\alpha}|\left(\frac{r_0}{a_0}+\left(\sum\limits_{i=1}^{\infty}   a_i^2\right)r_0^2,\left(\frac{r_0}{a_j  }\right)_{j\in\mathbb{N}}\right)^{\alpha}\right)a_i^2r_0}\nonumber\\
&\leqslant&\frac{\sum\limits_{i=0}^{\infty}\left(\sum\limits_{\alpha\in \mathbb{N}_0^{(\mathbb{N})}}|A_{i,\alpha}|(\frac{\rho_0}{a_j  })_{j\in\mathbb{N}_0}^{\alpha}\right)\frac{\rho_0}{a_i^3 }}{1-2\sum\limits_{i=0}^{\infty}\left(\sum\limits_{\alpha\in \mathbb{N}_0^{(\mathbb{N})}}|A_{i,\alpha}|(\frac{\rho_0}{a_j  })_{j\in\mathbb{N}_0}^{\alpha}\right) \rho_0}\nonumber\\
&\leqslant&\frac{\sum\limits_{i=0}^{\infty}\left(\sum\limits_{\alpha\in \mathbb{N}_0^{(\mathbb{N})}}|A_{i,\alpha}|(\frac{\rho_0}{a_j^4  })_{j\in\mathbb{N}_0}^{\alpha}\right)\frac{\rho_0}{a_i^4 }}{1-2\sum\limits_{i=0}^{\infty}\left(\sum\limits_{\alpha\in \mathbb{N}_0^{(\mathbb{N})}}|A_{i,\alpha}|(\frac{\rho_0}{a_j^4  })_{j\in\mathbb{N}_0}^{\alpha}\right) \frac{\rho_0}{a_i^4}}<\infty,\nonumber
\end{eqnarray}
where the second inequality follows from the facts that \eqref{230822e5} and $r_0=\min\left\{1, \frac{\rho_0}{1+a_0}\right\}$ and the last inequality follows from \eqref{230822e7}. Let
\begin{eqnarray*}
A_0'(t^{'},\textbf{x}^{'})&\triangleq &-\widetilde{A_0}(t',\textbf{x}^{'})+\sum\limits_{i=1}^{\infty}\bigg(\partial_{x_i^{'}}(\widetilde{A}_{i}(t^{'},\textbf{x}^{'}))-\frac{x_i'}{a_i^2}\cdot\widetilde{A}_{i}(t^{'},\textbf{x}^{'})\bigg),\\
G' (t^{'},\textbf{x}', \textbf{w} )&\triangleq &\sum_{i=1}^{\infty}\widetilde{A}_{i}(t^{'},\textbf{x}')w_i+A_0'(t^{'},\textbf{x}')w_0,
\end{eqnarray*}
where $t'\in\mathbb{R},\,\textbf{w}= (w_{ j})_{j\in\mathbb{N}_0}\in \mathbb{R}^{\mathbb{N}_0}$ and $\textbf{x}^{'}= (x_i^{'})_{i\in\mathbb{N}}\in \mathbb{R}^{\mathbb{N}}.$ We will prove that the monomial expansion of $G' $ at $(0,(0)_{i\in\mathbb{N}},(0)_{i\in\mathbb{N}_0})$ is absolutely convergent at a point near $\infty$ in the topology $\mathcal {T}^{2}.$
By (\ref{convergence 1}), we have
\begin{eqnarray*}
\sum_{i=0}^{\infty}\Bigg(\sum_{\alpha\in \mathbb{N}_0^{(\mathbb{N})}}|\widetilde{A_{i,\alpha}}|\bigg(\frac{r_0}{a_j }\bigg)_{j\in\mathbb{N}_0}^{\alpha}\Bigg)\frac{r_0}{a_i^3}
&=&\sum_{n=0}^{\infty}\sum_{i=0}^{\infty}\sum_{\alpha\in \mathbb{N}_0^{(\mathbb{N})},|\alpha|=n} \frac{1}{a_i^3 }|\widetilde{A_{i,\alpha}}|\bigg(\frac{1}{a_j }\bigg)_{j\in\mathbb{N}_0}^{\alpha}r_0^{n+1}<\infty.
\end{eqnarray*}
Thus, for any $\rho\in (0,r_0)$, it holds that
\begin{eqnarray*}
\sum_{n=0}^{\infty}\sum_{i=0}^{\infty}\sum_{\alpha\in \mathbb{N}_0^{(\mathbb{N})},|\alpha|=n} \frac{1}{a_i^3 }|\widetilde{A_{i,\alpha}}|\bigg(\frac{1}{a_j }\bigg)_{j\in\mathbb{N}_0}^{\alpha}(n+1)\rho^{n}<\infty,
\end{eqnarray*}
and hence
\begin{eqnarray*}
&&\sum_{i=0}^{\infty}\bigg|\partial_{x_i^{'}}\bigg(\widetilde{A}_{i}\bigg(\bigg(\frac{\rho}{a_j}\bigg)_{j\in\mathbb{N}_0}\bigg)\bigg)\bigg|\\
&\leqslant&\sum_{i=0}^{\infty} \sum_{\alpha\in \mathbb{N}_0^{(\mathbb{N})}}\frac{a_i \alpha_i}{\rho }|\widetilde{A_{i,\alpha}}|\bigg(\frac{\rho }{a_j }\bigg)_{j\in\mathbb{N}_0}^{\alpha} \\
&=&\sum_{n=0}^{\infty}\sum_{i=0}^{\infty}\sum_{\alpha\in \mathbb{N}_0^{(\mathbb{N})},|\alpha|=n} a_i \alpha_i\rho^{n-1}|\widetilde{A_{i,\alpha}}|\bigg(\frac{1}{a_j}\bigg)_{j\in\mathbb{N}_0}^{\alpha}\\
&\leqslant&\frac{1}{\rho}\sum_{n=0}^{\infty}\sum_{i=0}^{\infty}\sum_{\alpha\in \mathbb{N}_0^{(\mathbb{N})},|\alpha|=n} \frac{1}{a_i^3 }(n+1)\rho^{n }|\widetilde{A_{i,\alpha}}|\bigg(\frac{1}{a_j }\bigg)_{j\in\mathbb{N}_0}^{\alpha}\\
&<&\infty,
\end{eqnarray*}
and
\begin{eqnarray*}
\sum_{i=0}^{\infty}\frac{\rho}{a_i^3}\bigg|\widetilde{A}_{i}\bigg(\bigg(\frac{\rho}{a_j}\bigg)_{j\in\mathbb{N}_0}\bigg)\bigg|
\leqslant\sum_{i=0}^{\infty}\Bigg(\sum_{\alpha\in \mathbb{N}_0^{(\mathbb{N})}}|\widetilde{A_{i,\alpha}}|\bigg(\frac{r_0}{a_j  }\bigg)_{j\in\mathbb{N}_0}^{\alpha}\Bigg)\frac{r_0 }{a_i^3 }
<\infty,
\end{eqnarray*}
which implies that for any $\rho\in (0,r_0)$, the monomial expansion of $G'$ at $(0,(0)_{j\in\mathbb{N}_0},(0)_{i\in\mathbb{N}_0})$  is absolutely convergent at $(\frac{\rho}{a_0},(\frac{\rho}{a_j})_{j\in\mathbb{N}_0},(\frac{\rho}{a_i})_{i\in\mathbb{N}_0})$
which is a point near $\infty$ in the topology $\mathcal {T}^{2}.$

Write
$$
G_1[\widetilde{U}]\triangleq\partial_{t'} \widetilde{U}
-\sum_{i=1}^{\infty}\widetilde{A}_{i}\partial_{x_i^{'}} \widetilde{U}-\widetilde{A}_0 \widetilde{U},
$$
and for $W\in \Xi$, let
\begin{eqnarray*}
G_2[W]&\triangleq&-\partial_{t'} W+\frac{t'}{a_0^2}W
+\sum_{i=1}^{\infty}\bigg(\partial_{x_i^{'}} (\widetilde{A}_{i} W)-\frac{x_i'}{a_i^2}\cdot\widetilde{A}_{i}\cdot W\bigg)-\widetilde{A}_0 W.
\end{eqnarray*}
Recall \eqref{230822e8}, we see that $G_1[\widetilde{U}]=0$. For sufficient small positive number $\lambda$,
one can see that $V_{\lambda}$ is an open subset of $\ell^2(\mathbb{N}_0)$ and $\widetilde{U}\mid_{I_{\lambda}\cup S_{2,\lambda}}=0$. Note that the equation $G_2[W]=0$ satisfies the assumptions in Corollary \ref{corollary 20} for $p=2$. Apply Corollary \ref{corollary 20} to any $n\in \mathbb{N}$ and $k_1,\cdots,k_n\in \mathbb{N}_0$, there exists analytic solution $W$ satisfying $G_2[W]=0$ and $W\mid_{t^{'}=\lambda}=(x_1^{'})^{k_1}\cdots(x_n^{'})^{k_n}$. From Corollary \ref{corollary 20}, it follows that $V_{\lambda}\subset D_W^{(0)_{i\in\mathbb{N}_0}}$ for sufficiently small $\lambda$ and $W$ is Fr\`{e}chet differentiable with respect to $\ell^2(\mathbb{N}_0)$ and the Fr\`{e}chet derivative is continuous in a neighborhood of $(0)_{i\in\mathbb{N}_0}$ in $\ell^2(\mathbb{N}_0)$ containing $V_{\lambda}$. Combining the definition of $\widetilde{U}$, we see that $W \widetilde{U}$ the assumptions in Corollary \ref{20241127cor2}. We also note that
\begin{eqnarray*}
&&WG_1 [\widetilde{U}]- \widetilde{U}G_2[W]\\
&=& \frac{\partial (W\widetilde{U})}{\partial_{t^{'}}}-\frac{t'}{a_0^2}W\widetilde{U}
-\sum_{i=1}^{\infty}\bigg( \frac{\partial (W\widetilde{U}\cdot\widetilde{A}_{i})}{\partial_{x_i^{'}}}-\frac{x_i'}{a_i^2}\cdot\widetilde{A}_{i}\cdot W\widetilde{U}\bigg).
\end{eqnarray*}
Similarly, one can obtain a Borel probability measure on $(\ell^2(\mathbb{N}_0),\mathscr{B}(\ell^2(\mathbb{N}_0))$ by restricting the product measure $\prod\limits_{i=0}^{\infty}\bn_{a_i}$ to $\ell^2(\mathbb{N}_0)$. For abbreviation, we use the same letter $P$ for it. Denote the restriction of the product measure $\prod\limits_{i\in \mathbb{N}}\bn_{a_i}$ on $\left(\ell^{2}(\mathbb{N}),\mathscr{B}\big(\ell^{2}(\mathbb{N})\big)\right)$ by $P^{\widehat{0}}_1$. Applying the counterparts of Corollary \ref{20241127cor2}, we have
\begin{eqnarray}
&&\int_{V_{\lambda}}(WG_1[\widetilde{U}]- \widetilde{U}G_2[W])\,\mathrm{d}P\label{formula 20}\\
&=&\int_{S_{1,\lambda}}\bigg( (W\widetilde{U})\cdot  n_0
-\sum_{i=1}^{\infty} (W\widetilde{U}\cdot\widetilde{A}_{i})\cdot  n_i \bigg)\cdot\frac{e^{-\frac{\lambda^2}{2a_{0}^2}}}{\sqrt{2\pi a_{0}^2}}\,\mathrm{d}P_1^{\widehat{0}}\nonumber\\
&&  +\int_{S_{2,\lambda}}\bigg( (W\widetilde{U})\cdot  n_0
-\sum_{i=1}^{\infty} (W\widetilde{U}\cdot\widetilde{A}_{i})\cdot  n_i \bigg)\,\mathrm{d}P_1^{S^2}\nonumber\\
&=&\int_{S_{1,\lambda}}\bigg( (W\widetilde{U})\cdot  n_0
-\sum_{i=1}^{\infty} (W\widetilde{U}\cdot\widetilde{A}_{i})\cdot  n_i \bigg)\cdot\frac{e^{-\frac{\lambda^2}{2a_{0}^2}}}{\sqrt{2\pi a_{0}^2}}\,\mathrm{d}P_1^{\widehat{0}}\nonumber \\
 &=& \int_{S_{1,\lambda}}W\widetilde{U} \cdot\frac{e^{-\frac{\lambda^2}{2a_{0}^2}}}{\sqrt{2\pi a_{0}^2}}\,\mathrm{d}P_1^{\widehat{0}} .\nonumber
\end{eqnarray}
Therefore, for sufficiently small $\lambda>0$, (\ref{formula 20}) implies that
\begin{eqnarray*}
\int_{S_{1,\lambda}}(x_1^{'})^{k_1}\cdots(x_n^{'})^{k_n}\widetilde{U} \,\mathrm{d}P_1^{\widehat{0}}=0
\end{eqnarray*}
for any $n\in \mathbb{N}$ and $k_1,\cdots,k_n\in \mathbb{N}_0$. Hence, we have $\widetilde{U}=0$ in $L^2(S_{1,\lambda},P_1^{\widehat{0}})$. Combining the continuity of $\widetilde{U}$, we deduce that $\widetilde{U}\equiv 0$ on $S_{1,\lambda}$ for sufficiently small $\lambda>0$ and by the definition of $\widetilde{U}$ we have $\widetilde{U}\mid_{\partial V_{\lambda}}=0$ for sufficiently small $\lambda>0$. Therefore, $\widetilde{U}\equiv 0$ on $V_{\lambda}$ for sufficiently small $\lambda>0$ which implies that $U\equiv 0$ on a neighborhood of $(0)_{i\in\mathbb{N}_0}$ in $\ell^2(\mathbb{N}_0)$ intersect $\{(t,\textbf{x})\in \ell^2(\mathbb{N}_0):\;t>0\}$. By the same way we can prove that $U\equiv 0$ on a neighborhood of $(0)_{i\in\mathbb{N}_0}$ in $\ell^2(\mathbb{N}_0)$ intersect $\{(t,\textbf{x})\in \ell^2(\mathbb{N}_0):\;t<0\}$.

At last, by the continuity of $U$ we have $U\equiv 0$ on a neighborhood of $(0)_{i\in\mathbb{N}_0}$ in $\ell^2(\mathbb{N}_0)$ . This completes the proof of Theorem \ref{Infinite Dimensional Holmgren Type Theorem}.
\end{proof}

\newpage

\section{Infinite Dimensional Sobolev Spaces}
\label{20240111chapter1}

This section is mainly based on \cite{YZ-b}.

\subsection{Definition of Sobolev Spaces}
Suppose that $O$ is an open subset of $\ell^2$.
\begin{definition}
Suppose $\alpha=(\alpha_1,\cdots,\alpha_k,0,\cdots)\in \mathbb{N}_0^{(\mathbb{N})}$ and $\varphi\in C_S^{\infty}(O)$.
We define $(D^{\alpha})^{*}\varphi$ by
\begin{eqnarray*}
(D^{\alpha})^{*}\varphi\triangleq(D_1^*)^{\alpha_1}\cdots (D_k^*)^{\alpha_k}\varphi,
\end{eqnarray*}
where $D_i^*\varphi\triangleq-D_i\varphi +\frac{x_i}{a_i^2}\cdot \varphi$ for $i=1,\cdots,k$.
\end{definition}
\begin{definition}\label{231013def1}
Assume that $f\in L^2(O,P)$ and $\alpha=(\alpha_1,\cdots,\alpha_k,0,\cdots)\in \mathbb{N}_0^{(\mathbb{N})}$. We define $D^{\alpha}f$ to be a function on $C_F^{\infty}(O)$ by
\begin{eqnarray*}
(D^{\alpha}f)(\varphi)\triangleq \int_{O}f D^{*}_{\alpha}\varphi\,\mathrm{d}P,\quad\forall\,\varphi\in C_F^{\infty}(O).
\end{eqnarray*}
If there exists $g\in  L^2(O,P)$ such that
\begin{eqnarray*}
\int_{O}g\varphi\,\mathrm{d}P= \int_{O}f D^{*}_{\alpha}\varphi\,\mathrm{d}P,\quad\forall\,\varphi\in C_F^{\infty}(O).
\end{eqnarray*}
We say that $D^{\alpha}f\in L^2(O,P)$ and $D^{\alpha}f=g$.
\end{definition}
\begin{definition}
Suppose that $m\in\mathbb{N}$. Denote by $W^{m,2}(O)$ the set of all $f\in L^2(O,P)$ such that $D^{\alpha}f\in L^2(O,P)$ for any multi-index $\alpha=(\alpha_1,\cdots,\alpha_k,0,\cdots)\in \mathbb{N}_0^{(\mathbb{N})}$ with $|\alpha|\triangleq\sum\limits_{i=1}^{k}\alpha_i\leq m$ and
\begin{eqnarray*}
\sum_{\alpha\in \mathbb{N}_0^{(\mathbb{N})},\,|\alpha|\leq m}\textbf{a}^{\alpha}\cdot\int_{O}|D^{\alpha}f|^2\,\mathrm{d}P<\infty,
\end{eqnarray*}
where $\textbf{a}^{\alpha}=\prod\limits_{i=1}^{\infty} a_i^{2\alpha_i}$ for each $\alpha\in \mathbb{N}_0^{(\mathbb{N})}$. There is a natural norm on $W^{m,2}(O)$ defined by
\begin{eqnarray}\label{20240201for9}
||f||_{W^{m,2}(O)}\triangleq\bigg(\sum_{\alpha\in \mathbb{N}_0^{(\mathbb{N})},\,|\alpha|\leq m}\textbf{a}^{\alpha}\cdot\int_{O}|D^{\alpha}f|^2\,\mathrm{d}P\bigg)^{\frac{1}{2}},
\end{eqnarray}
for any $f\in W^{m,2}(O)$. We abbreviate $W^{m,2}(O)$ to $H^{m}(O)$, and define a scalar product on it by
\begin{eqnarray}\label{20240201for10}
(u,v)_{H^{m}(O)}=\sum_{\alpha\in \mathbb{N}_0^{(\mathbb{N})},\,|\alpha|\leq m}\textbf{a}^{\alpha}\cdot\int_{O} D^{\alpha}u\overline{D^{\alpha}v}\,\mathrm{d}P,\quad \forall\, u,v\in H^{m}(O).
\end{eqnarray}
\end{definition}
\begin{remark}
It is easy to see that $H^{m}(O)$, with the scalar product defined by $(\ref{sobolev product})$ is a Hilbert space.
\end{remark}
\begin{definition}\label{20231123def1}
Assume that $n\in\mathbb{N}$, $\Omega$ is a non-empty open subset of $\mathbb{R}^n$, $f\in L^2(\Omega,\mathcal{N}^n)$ and $\alpha=(\alpha_1,\cdots,\alpha_n)\in \mathbb{N}_0^{n}$. We define $D^{\alpha}f$ to be a function on $C_c^{\infty}(\Omega)$ by
\begin{eqnarray*}
(D^{\alpha}f)(\varphi)\triangleq \int_{\Omega}f D^{*}_{\alpha}\varphi\,\mathrm{d}\mathcal{N}^n,\quad\forall\,\varphi\in C_c^{\infty}(\Omega).
\end{eqnarray*}
If there exists $g\in  L^2(\Omega,\mathcal{N}^n)$ such that
\begin{eqnarray*}
\int_{\Omega}g\varphi\,\mathrm{d}\mathcal{N}^n= \int_{\Omega}f D^{*}_{\alpha}\varphi\,\mathrm{d}\mathcal{N}^n,\quad\forall\,\varphi\in C_c^{\infty}(\Omega).
\end{eqnarray*}
We say that $D^{\alpha}f\in L^2(\Omega,\mathcal{N}^n)$ and $D^{\alpha}f=g$.
\end{definition}
\begin{definition}
Suppose that $m\in\mathbb{N}$ and we adapt the same notations as in Definition \ref{20231123def1}. Denote by $W^{m,2}(\Omega,\mathcal{N}^n)$ the set of all $f\in L^2(\Omega,\mathcal{N}^n)$ such that $D^{\alpha}f\in L^2(\Omega,\mathcal{N}^n)$ for any multi-index $\alpha=(\alpha_1,\cdots,\alpha_n)\in \mathbb{N}_0^{n}$ with $|\alpha|\triangleq\sum\limits_{i=1}^{n}\alpha_i\leq m$ and
\begin{eqnarray*}
\sum_{\alpha\in \mathbb{N}_0^{n},\,|\alpha|\leq m}\textbf{a}^{\alpha}\cdot\int_{\Omega}|D^{\alpha}f|^2\,\mathrm{d}\mathcal{N}^n<\infty,
\end{eqnarray*}
where $\textbf{a}^{\alpha}=\prod\limits_{i=1}^{n} a_i^{2\alpha_i}$ for each $\alpha\in \mathbb{N}_0^{n}$. There is a natural norm on\\ $W^{m,2}(\Omega, \mathcal{N}^n)$ defined by
\begin{eqnarray}
||f||_{W^{m,2}(\Omega,\mathcal{N}^n)}\triangleq\bigg(\sum_{\alpha\in \mathbb{N}_0^{n},\,|\alpha|\leq m}\textbf{a}^{\alpha}\cdot\int_{\Omega}|D^{\alpha}f|^2\,\mathrm{d}\mathcal{N}^n\bigg)^{\frac{1}{2}},
\end{eqnarray}
for any $f\in W^{m,2}(\Omega, \mathcal{N}^n)$. We abbreviate $W^{m,2}(\Omega,\mathcal{N}^n)$ to $H^{m}(\Omega, \mathcal{N}^n)$, and define a scalar product on it by
\begin{eqnarray}\label{sobolev product}
(u,v)_{H^{m}(\Omega,\mathcal{N}^n)}=\sum_{\alpha\in \mathbb{N}_0^{n},\,|\alpha|\leq m}\textbf{a}^{\alpha}\cdot\int_{\Omega} D^{\alpha}u\overline{D^{\alpha}v}\,\mathrm{d}\mathcal{N}^n,\quad \forall\, u,v\in H^{m}(\Omega,\mathcal{N}^n).
\end{eqnarray}
\end{definition}
\begin{proposition}\label{20231123prop2}
For each $n\in\mathbb{N}$, $C_c^{\infty}(\mathbb{R}^n)$ is dense in $W^{1,2}(\mathbb{R}^n,\mathcal{N}^n)$.
\end{proposition}
\begin{proof}
Suppose that $f\in W^{1,2}(\mathbb{R}^n,\mathcal{N}^n)$. Choosing $\psi\in C^{\infty}(\mathbb{R};[0,1])$ such that $\psi(x)=1$ for all $x\leq 1$ and $\psi(x)=0$ for all $x\geq 2$. For each $k\in\mathbb{N}$, let $\psi_k(\textbf{x}_n)\triangleq \psi\left(\frac{x_1^2+\cdots+x_n^2}{k}\right),\,\forall\,\textbf{x}_n=(x_1,\cdots,x_n)\in\mathbb{R}^n$.
Obviously,
\begin{eqnarray*}
\int_{\mathbb{R}^n}|\psi_k\cdot f|^2\,\mathrm{d}\mathcal{N}^n\leq \int_{\mathbb{R}^n}| f|^2\,\mathrm{d}\mathcal{N}^n<\infty,\qquad\forall\,k\in\mathbb{N}.
\end{eqnarray*}
Note that for each $\varphi\in C_c^{\infty}(\mathbb{R}^n),\,k\in\mathbb{N}$ and $i=1,\cdots,n$, we have
\begin{eqnarray*}
\int_{\mathbb{R}^n}\psi_k\cdot f\cdot D_i^*\varphi\,\mathrm{d}\mathcal{N}^n
&=&\int_{\mathbb{R}^n}\psi_k\cdot f\cdot \left(-D_i\varphi+\frac{x_i}{a_i^2}\cdot \varphi\right)\,\mathrm{d}\mathcal{N}^n\\
&=&\int_{\mathbb{R}^n} f\cdot \left((D_i\psi_k)\cdot\varphi-D_i(\varphi\cdot\psi_k)+\frac{x_i}{a_i^2}\cdot \varphi\cdot\psi_k\right)\,\mathrm{d}\mathcal{N}^n\\
&=&\int_{\mathbb{R}^n} f\cdot \left((D_i\psi_k)\cdot\varphi+D_i^*(\varphi\cdot\psi_k) \right)\,\mathrm{d}\mathcal{N}^n\\
&=&\int_{\mathbb{R}^n} \left(f\cdot(D_i\psi_k)+(D_if) \cdot\psi_k \right)\cdot\varphi\,\mathrm{d}\mathcal{N}^n,
\end{eqnarray*}
where the last equality follows from the fact that $\varphi\cdot\psi_k\in C_c^{\infty}(\mathbb{R}^n)$ and Definition \ref{20231123def1}. Since
\begin{eqnarray*}
&&\int_{\mathbb{R}^n} |f\cdot(D_i\psi_k)+(D_if) \cdot\psi_k|^2\,\mathrm{d}\mathcal{N}^n\\
&\leq &2\int_{\mathbb{R}^n} (|f\cdot(D_i\psi_k)|^2+|(D_if) \cdot\psi_k|^2)\,\mathrm{d}\mathcal{N}^n\\
&\leq&2\cdot \sup|D_i\psi_k|^2\cdot\int_{\mathbb{R}^n} |f|^2\,\mathrm{d}\mathcal{N}^n+2\int_{\mathbb{R}^n}|D_if|^2\,\mathrm{d}\mathcal{N}^n<\infty.
\end{eqnarray*}
These implies that $\{\psi_k\cdot f\}_{k=1}^{\infty}\subset W^{1,2}(\mathbb{R}^n,\mathcal{N}^n)$ and $D_i(\psi_k\cdot f)=f\cdot(D_i\psi_k)+(D_if) \cdot\psi_k$ for any $k\in\mathbb{N}$ and $i=1,\cdots,n.$ Then we have
\begin{eqnarray*}
&&||\psi_k\cdot f-f||_{W^{1,2}(\mathbb{R}^n,\mathcal{N}^n)}^2\\
&=&\int_{\mathbb{R}^n}|( \psi_k-1)f|^2\,\mathrm{d}\mathcal{N}^n+\sum_{i=1}^{n}a_i^2\cdot\int_{\mathbb{R}^n}\bigg|(\psi_k-1)\cdot (D_i f)+f \cdot (D_i \psi_k)\bigg|^2\,\mathrm{d}\mathcal{N}^n\\
&\leq&\int_{\mathbb{R}^n}|(\psi_k-1)f|^2\,\mathrm{d}\mathcal{N}^n\\
&&+2\sum_{i=1}^{n}a_i^2\cdot\int_{\mathbb{R}^n}\bigg(|(\psi_k-1)\cdot (D_i f)|^2+|f \cdot(D_i \psi_k)|^2\bigg)\,\mathrm{d}\mathcal{N}^n\\
&=&\int_{\mathbb{R}^n}|(\psi_k-1)f|^2\,\mathrm{d}\mathcal{N}^n+2\int_{\mathbb{R}^n}|\psi_k-1|^2\cdot\left( \sum_{i=1}^{n}a_i^2\cdot|D_i f|^2\right)\,\mathrm{d}\mathcal{N}^n\\
&&+2\cdot\int_{\mathbb{R}^n}\left(\sum_{i=1}^{n}a_i^2\cdot|D_i \psi_k|^2\right)\cdot |f|^2 \,\mathrm{d}\mathcal{N}^n,\qquad\forall\,k\in\mathbb{N}.
\end{eqnarray*}
Therefore, $\lim\limits_{k\to\infty}||\psi_k\cdot f-f||_{W^{1,2}(\mathbb{R}^n,\mathcal{N}^n)}^2=0$.
Now we fix a $k\in\mathbb{N}$. Then there exists $r\in(0,+\infty)$ such that
\begin{eqnarray}\label{20231124for2}
\int_{||\textbf{x}_n||_{\mathbb{R}^n}>r}|(\psi_k\cdot f)(\textbf{x}_n)|^2\,\mathrm{d}\mathcal{N}^n(\textbf{x}_n)=0.
\end{eqnarray}
Let
\begin{eqnarray}\label{20231124for1}
G_n(\textbf{x}_n)\triangleq \prod_{i=1}^{n}\frac{1}{\sqrt{2\pi a_i^2}}e^{-\frac{x_i^2}{2a_i^2}},\qquad\forall\,\textbf{x}_n=(x_1,\cdots,x_n)\in\mathbb{R}^n,
\end{eqnarray}
which is the Gaussian weight function of the measure $\mathcal{N}^n$, i.e., $\mathrm{d}\mathcal{N}^n(\textbf{x}_n)=G_n(\textbf{x}_n)\,\mathrm{d}\textbf{x}_n$ where $\mathrm{d}\textbf{x}_n$ is the Lebesgue measure on $\mathbb{R}^n$. Choosing $\chi_n\in C_c^{\infty}(\mathbb{R}^n;[0,1])$ such that $\int_{\mathbb{R}^n}\chi_n(\textbf{x}_n)\mathrm{d}\textbf{x}_n=1$ and $\chi_n(\textbf{x}_n)=0$ for all $||\textbf{x}_n||_{\mathbb{R}^n}>1$. Write $\chi_{n,\delta}(\textbf{x}_n)\triangleq \frac{1}{\delta^n}\cdot\chi_{n}(\frac{1}{\delta}\cdot\textbf{x}_n),\,\,\forall\,\textbf{x}_n\in\mathbb{R}^n$ and $\delta\in(0,+\infty)$. By \eqref{20231124for2}, we have
\begin{eqnarray*}
&&\int_{ \mathbb{R}^n }|(\psi_k\cdot f)(\textbf{x}_n)|^2\,\mathrm{d}\mathcal{N}^n(\textbf{x}_n)\\
&=&\int_{||\textbf{x}_n||_{\mathbb{R}^n}>r}|(\psi_k\cdot f)(\textbf{x}_n)|^2\,\mathrm{d}\mathcal{N}^n(\textbf{x}_n)\\
&=&\int_{||\textbf{x}_n||_{\mathbb{R}^n}>r}|(\psi_k\cdot f)(\textbf{x}_n)|^2\cdot G_n(\textbf{x}_n)\,\mathrm{d}\textbf{x}_n\\
&\geq&\left(\prod_{i=1}^{n}\frac{1}{\sqrt{2\pi a_i^2}}e^{-\frac{r^2}{2a_i^2}}\right)\cdot\int_{||\textbf{x}_n||_{\mathbb{R}^n}>r}|(\psi_k\cdot f)(\textbf{x}_n)|^2  \,\mathrm{d}\textbf{x}_n\\
&=&\left(\prod_{i=1}^{n}\frac{1}{\sqrt{2\pi a_i^2}}e^{-\frac{r^2}{2a_i^2}}\right)\cdot\int_{\mathbb{R}^n}|(\psi_k\cdot f)(\textbf{x}_n)|^2  \,\mathrm{d}\textbf{x}_n\\
\end{eqnarray*}
which implies that $\psi_k\cdot f\in L^2(\mathbb{R}^n,\mathrm{d}\textbf{x}_n)\cap  L^1(\mathbb{R}^n,\mathrm{d}\textbf{x}_n)$. Then let
\begin{eqnarray*}
f_{k,n,\delta}(\textbf{x}_n)\triangleq  \int_{\mathbb{R}^n}(\psi_k\cdot f)(\textbf{y}_n)\chi_{n,\delta}(\textbf{x}_n-\textbf{y}_n)\,\mathrm{d}\textbf{y}_n,\quad
\forall\,\,\textbf{x}_n\in\mathbb{R}^n,\,\delta\in(0,+\infty).
\end{eqnarray*}
Then $f_{k,n,\delta}\in C_c^{\infty}(\mathbb{R}^n)$, supp$f_{k,n,\delta}\subset \{\textbf{x}_n\in\mathbb{R}^n:||\textbf{x}_n||_{\mathbb{R}^n}\leq r+\delta\}$ and
\begin{eqnarray*}
&&\int_{\mathbb{R}^n}|f_{k,n,\delta}(\textbf{x}_n)- (\psi_k\cdot f)(\textbf{x}_n)|^2\,\mathrm{d}\mathcal{N}^n(\textbf{x}_n)\\
&=&\int_{\mathbb{R}^n}|f_{k,n,\delta}(\textbf{x}_n)- (\psi_k\cdot f)(\textbf{x}_n)|^2G_n(\textbf{x}_n)\,\mathrm{d}\textbf{x}_n\\
&\leq&\left(\prod_{i=1}^{n}\frac{1}{\sqrt{2\pi a_i^2}} \right)\cdot\int_{\mathbb{R}^n}|f_{k,n,\delta}(\textbf{x}_n)- (\psi_k\cdot f)(\textbf{x}_n)|^2 \,\mathrm{d}\textbf{x}_n,
\end{eqnarray*}
which implies that $\lim\limits_{\delta\to 0}\int_{\mathbb{R}^n}|f_{k,n,\delta}(\textbf{x}_n)- (\psi_k\cdot f)(\textbf{x}_n)|^2\,\mathrm{d}\mathcal{N}^n(\textbf{x}_n)=0.$

Let $\chi_{n,\delta,\textbf{x}_n}(\textbf{y}_n)\triangleq\chi_{n,\delta}(\textbf{x}_n-\textbf{y}_n),\,\forall\,\textbf{x}_n,\textbf{y}_n\in\mathbb{R}^n$. For each $i=1,\cdots,n$, note that
\begin{eqnarray*}
&&D_if_{k,n,\delta}(\textbf{x}_n)\\
&=&\int_{\mathbb{R}^n}(\psi_k\cdot f)(\textbf{y}_n)D_i\chi_{n,\delta}(\textbf{x}_n-\textbf{y}_n)\,\mathrm{d}\textbf{y}_n\\
&=&-\int_{\mathbb{R}^n}(\psi_k\cdot f)(\textbf{y}_n)\cdot(D_i\chi_{n,\delta,\textbf{x}_n})(\textbf{y}_n)\,\mathrm{d}\textbf{y}_n\\
&=&-\int_{\mathbb{R}^n}(\psi_k\cdot f)(\textbf{y}_n)\cdot(D_i\chi_{n,\delta,\textbf{x}_n})(\textbf{y}_n)\cdot\frac{1}{G_n(\textbf{y}_n)}\,\mathrm{d}\mathcal{N}^n(\textbf{y}_n)\\
&=&\int_{\mathbb{R}^n}(\psi_k\cdot f)(\textbf{y}_n)\left(- \left(D_i\frac{\chi_{n,\delta,\textbf{x}_n} }{G_n}\right)(\textbf{y}_n)+\frac{y_i}{a_i^2}\cdot\frac{\chi_{n,\delta,\textbf{x}_n} }{G_n}(\textbf{y}_n)\right)\,\mathrm{d}\mathcal{N}^n(\textbf{y}_n)\\
&=&\int_{\mathbb{R}^n}(\psi_k\cdot f)\cdot \left(D_i^*\frac{\chi_{n,\delta,\textbf{x}_n}}{G_n} \right)\,\mathrm{d}\mathcal{N}^n\\
&=&\int_{\mathbb{R}^n}D_i(\psi_k\cdot f)\cdot \frac{\chi_{n,\delta,\textbf{x}_n}}{G_n} \,\mathrm{d}\mathcal{N}^n\\
&=&\int_{\mathbb{R}^n}D_i(\psi_k\cdot f)(\textbf{y}_n)\cdot  \chi_{n,\delta,\textbf{x}_n}(\textbf{y}_n)  \,\mathrm{d}\textbf{y}_n \\
&=&\int_{\mathbb{R}^n}D_i(\psi_k\cdot f)(\textbf{y}_n)\cdot  \chi_{n,\delta}(\textbf{x}_n-\textbf{y}_n)  \,\mathrm{d}\textbf{y}_n,
\end{eqnarray*}
where the sixth equality follows from the fact that $\frac{\chi_{n,\delta,\textbf{x}_n}}{G_n}\in C_c^{\infty}(\mathbb{R}^n)$ and Definition \ref{20231123def1}. It is easy to see that $\int_{||\textbf{x}_n||_{\mathbb{R}^n}>r}|D_i(\psi_k\cdot f)|^2\,\mathrm{d}\mathcal{N}^n=0$. These imply that
\begin{eqnarray*}
\lim_{\delta\to 0}\int_{\mathbb{R}^n}|D_if_{k,n,\delta} - D_i(\psi_k\cdot f)|^2\,\mathrm{d}\mathcal{N}^n =0.
\end{eqnarray*}
Hence, $\lim\limits_{\delta\to 0}||f_{k,n,\delta}-\psi_k\cdot f||_{W^{1,2}(\mathbb{R}^n,\mathcal{N}^n)}^2=0$.  By diagonal process, we complete the proof of Proposition \ref{20231123prop2}.
\end{proof}
For each $t\in\mathbb{R}$, set $\tau_t(\textbf{x} )\triangleq(x_1+1,(x_i)_{i\in\mathbb{N}\setminus\{1\}}),\quad\forall\,\textbf{x}=(x_i)_{i\in\mathbb{N}}\in\ell^2$. Obviously, $\tau_t$ is a bijection from $\ell^2$ into itself. A natural problem is that:
\begin{problem}\label{20131123pro1}
For each $f\in H^{1}(\ell^2)$, does $f(\tau_t)\in H^{1}(\ell^2)$?
\end{problem}
To answer this question, we need the following two lemmas.
\begin{lemma}\label{231013th1}
There is no positive number $C$ such that
\begin{eqnarray*}
C^{-1}||u||_{H^{1}(\ell^2)}\leq ||u(\tau_1)||_{H^{1}(\ell^2)} \leq C||u||_{H^{1}(\ell^2)},\quad \forall\,u\in \mathscr {C}_c^{\infty}.
\end{eqnarray*}
\end{lemma}
\begin{proof}
Choose $f\in C_c^{\infty}(\mathbb{R})$ such that $f(x)=0$ for any $x\notin (0,1)$ and $f(1/2)>0$. For $n\in\mathbb{N}$, set $f_n( \textbf{x} )\triangleq f(x_1+n),\quad\forall\,\textbf{x}=(x_i)_{i\in\mathbb{N}}\in\ell^2.$ Note that, for each $n\in\mathbb{N}$,
\begin{eqnarray*}
&&||f_{n}(\tau_1)||_{H^{1}(\ell^2)}\\
 &=&\bigg(\int_{\mathbb{R}}  |f(x+n+1)|^2\,e^{-\frac{x^2}{2a_1^2}}\mathrm{d}x+ a_1^{2}\cdot \int_{\mathbb{R}}  |f^{'}(x+n+1)|^2\,e^{-\frac{x^2}{2a_1^2}}\mathrm{d}x\bigg)^{\frac{1}{2}}\\
 &=&\bigg(\int_{0}^1 |f(x)|^2\,e^{-\frac{(x-n-1)^2}{2a_1^2}}\mathrm{d}x+ a_1^{2}\cdot \int_{0}^{1} |f^{'}(x)|^2\,e^{-\frac{(x-n-1)^2}{2a_1^2}}\mathrm{d}x\bigg)^{\frac{1}{2}}\\
 &=&e^{-\frac{1}{4a_1^2}}\bigg( \int_{0}^{1}\left(|f(x)|^2+ a_1^{2}\cdot|f^{'}(x)|^2\right)\,e^{-\frac{(x-n)^2}{2a_1^2}}e^{\frac{x-n }{ a_1^2}}\mathrm{d}x\bigg)^{\frac{1}{2}},
\end{eqnarray*}
and hence
\begin{eqnarray*}
e^{ -\frac{1}{4a_1^2}-\frac{n}{ 2a_1^2}} \leq \frac{||f_n(\tau_1)||_{H^{1}(\ell^2)}}{||f_n||_{H^{1}(\ell^2)}}\leq e^{ \frac{1}{4a_1^2}-\frac{n }{ 2a_1^2}} .
\end{eqnarray*}
This completes the proof of Lemma \ref{231013th1}.
\end{proof}

\begin{lemma}\label{20231123lem1}
For any Borel measurable function $f$ on $\ell^2$, it holds that
\begin{eqnarray*}
\int_{\ell^2}|f(\tau_t)|^2\,\mathrm{d}P
&=& e^{-\frac{t^2}{2a_1^2}} \int_{\ell^2}|f|^2e^{\frac{t x_1}{a_1^2}}\,\mathrm{d}P.
\end{eqnarray*}
\end{lemma}
\begin{proof}
Note that
\begin{eqnarray*}
\int_{\ell^2}|f(\tau_t)|^2\,\mathrm{d}P
&=&\int_{\mathbb{R}\times\ell^2(\mathbb{N}\setminus\{1\})}|f(\tau_t(x_1,\textbf{x}^1))|^2\,\mathrm{d}\mathcal{N}_{a_1^2}(x_1)\mathrm{d}P^{\widehat{1}}(\textbf{x}^1)\\
&=&\int_{\mathbb{R}\times\ell^2(\mathbb{N}\setminus\{1\})}|f(x_1+t,\textbf{x}^1))|^2\,\mathrm{d}\mathcal{N}_{a_1^2}(x_1)\mathrm{d}P^{\widehat{1}}(\textbf{x}^1)\\
&=&\frac{1}{\sqrt{2\pi a_1^2}}\int_{\mathbb{R}}\int_{\ell^2(\mathbb{N}\setminus\{1\})}|f(x_1+t,\textbf{x}^1))|^2e^{-\frac{x_1^2}{2a_1^2}}\,\mathrm{d}x_1\mathrm{d}P^{\widehat{1}}(\textbf{x}^1)\\
&=&\frac{1}{\sqrt{2\pi a_1^2}}\int_{\mathbb{R}}\int_{\ell^2(\mathbb{N}\setminus\{1\})}|f(x_1,\textbf{x}^1))|^2e^{-\frac{(x_1-t)^2}{2a_1^2}}\,\mathrm{d}x_1\mathrm{d}P^{\widehat{1}}(\textbf{x}^1)\\
&=&e^{-\frac{t^2}{2a_1^2}}\int_{\mathbb{R}}\int_{\ell^2(\mathbb{N}\setminus\{1\})}|f(x_1,\textbf{x}^1))|^2e^{-\frac{x_1^2}{2a_1^2}}e^{\frac{t x_1}{a_1^2}}\,\mathrm{d}x_1\mathrm{d}P^{\widehat{1}}(\textbf{x}^1)\\
&=&e^{-\frac{t^2}{2a_1^2}}\int_{\ell^2}|f(\textbf{x}))|^2e^{\frac{t x_1}{a_1^2}}\,\mathrm{d}P(\textbf{x}),
\end{eqnarray*}
which completes the proof of Lemma \ref{20231123lem1}.
\end{proof}
The following proposition gives a negative answer for Problem \ref{20131123pro1}.
\begin{proposition}\label{20231123prop1}
There exists $f\in H^{1}(\ell^2)$ such that $f(\tau_1)\notin H^{1}(\ell^2)$.
\end{proposition}
\begin{proof}
Suppose that for each $f\in H^{1}(\ell^2)$, it holds that $f(\tau_t)\in H^{1}(\ell^2)$. Then the following mapping is a linear mapping from $H^{1}(\ell^2)$ into itself:
\begin{eqnarray}\label{20231123for1}
f\mapsto f(\tau_1),\quad\forall\, f\in H^{1}(\ell^2).
\end{eqnarray}
If $\{f_n\}_{n=1}^{\infty}\subset  H^{1}(\ell^2)$ and $f,g\in  H^{1}(\ell^2)$ such that $\lim\limits_{n\to\infty}f_n=f$ and $\lim\limits_{n\to\infty}f_n(\tau_1)=g$ in $H^{1}(\ell^2)$. For each $r\in(0,+\infty)$ and $n\in\mathbb{N}$, by Lemma \ref{20231123lem1}, we have
\begin{eqnarray*}
\int_{B_r}|f_n(\tau_1)-f(\tau_1)|^2\mathrm{d}P
&=&\int_{\ell^2}|(\chi_{\tau_{1}(B_r)}\cdot f_n)(\tau_1)-(\chi_{\tau_{1}(B_r)}\cdot f)(\tau_1)|^2\mathrm{d}P\\
&=&e^{-\frac{1}{2a_1^2}} \int_{\ell^2}|\chi_{\tau_{1}(B_r)}\cdot f_n-\chi_{\tau_{1}(B_r)}\cdot f|^2e^{\frac{ x_1}{a_1^2}}\,\mathrm{d}P\\
&=&e^{-\frac{1}{2a_1^2}} \int_{\tau_{1}(B_r)}| f_n- f|^2e^{\frac{ x_1}{a_1^2}}\,\mathrm{d}P\\
&\leq &e^{-\frac{1}{2a_1^2}} \cdot\left(\sup_{\tau_{t}(B_r)}e^{\frac{ x_1}{a_1^2}}\right)\int_{\ell^2}| f_n- f|^2\,\mathrm{d}P,
\end{eqnarray*}
which implies that
\begin{eqnarray*}
\lim_{n\to\infty}\int_{B_r}|f_n(\tau_1)-f(\tau_1)|^2\mathrm{d}P=0.
\end{eqnarray*}
Since
\begin{eqnarray*}
\int_{B_r}|f_n(\tau_1)-g|^2\mathrm{d}P\leq \int_{\ell^2}|f_n(\tau_1)-g|^2\mathrm{d}P,
\end{eqnarray*}
\begin{eqnarray*}
\lim_{n\to\infty}\int_{B_r}|f_n(\tau_1)-g|^2\mathrm{d}P=0.
\end{eqnarray*}
Therefore,
\begin{eqnarray*}
\int_{B_r}|f(\tau_1)-g|^2\mathrm{d}P=0.
\end{eqnarray*}
Letting $r\to\infty$ and we obtain
\begin{eqnarray*}
\int_{\ell^2}|f(\tau_1)-g|^2\mathrm{d}P=0,
\end{eqnarray*}
which implies that $f(\tau_1)=g$. By the closed graph theorem, the mapping defined by \eqref{20231123for1} is a bounded linear operator which contradicts to Lemma \ref{20231123lem1}. This completes the proof of Proposition \ref{20231123prop1}.
\end{proof}

\begin{lemma}\label{20231111lem1}
Suppose that $f\in W^{1,2}(\ell^2)$. Then for each $\varphi\in \mathscr {C}_c^{\infty}$, it holds that
\begin{eqnarray}\label{20231111for1}
\int_{\ell^2} f D^{*}_i\varphi\,\mathrm{d}P
=\int_{\ell^2} (D_if)\varphi\,\mathrm{d}P.
\end{eqnarray}
\end{lemma}
\begin{proof}
Choosing $\psi\in C^{\infty}(\mathbb{R};[0,1])$ such that $\psi(x)=1$ for all $x\leq 1$ and $\psi(x)=0$ for all $x\geq 2$.
 Let $\varphi_k(\textbf{x})\triangleq \psi \left(\frac{||\textbf{x}||^2_{\ell^2}}{k}\right),\,\forall\,k\in\mathbb{N},\,\textbf{x}\in\ell^2$. Then for each $k\in\mathbb{N}$, $\varphi_k\cdot \varphi\in C_F^{\infty}(\ell^2)$ and
 $$
 D^{*}_i(\varphi_k\cdot\varphi)=-(D_i\varphi)\cdot \varphi_k-(D_i\varphi_k)\cdot \varphi +\frac{x_i}{a_i^2}\cdot \varphi_k\cdot \varphi.
 $$
By Definition \ref{231013def1}, we have
\begin{eqnarray*}
\int_{\ell^2} f \cdot \left(-(D_i\varphi)\cdot \varphi_k-(D_i\varphi_k)\cdot \varphi +\frac{x_i}{a_i^2}\cdot \varphi_k\cdot \varphi\right)\,\mathrm{d}P
=\int_{\ell^2} (D_if)\cdot\varphi_k\cdot \varphi\,\mathrm{d}P.
\end{eqnarray*}
Letting $k\to\infty$ and we arrive at \eqref{20231111for1} which completes the proof of Lemma \ref{20231111lem1}.
\end{proof}

Recall the sequence $\{X_n\}_{n=1}^{\infty}$ as in Theorem \ref{230215Th1}.
\begin{lemma}\label{lemma21gc}
Suppose that $f\in H^{1}(O)$. Then $\{X_n\cdot f\}_{n=1}^{\infty}\subset H^{1}(O)$ and
\begin{eqnarray*}
 \lim_{n\to\infty}||X_n\cdot f-f||_{H^{1}(O)}=0.
\end{eqnarray*}
\end{lemma}
\begin{proof}
Note that for any $\varphi\in C_F^{\infty}(O)$ and $i\in\mathbb{N}$, it holds that
\begin{eqnarray*}
&& \int_{O}X_n\cdot f\cdot( D_i^*\varphi)\,\mathrm{d}P\\
 &=&-\int_{O}X_n\cdot f\cdot\bigg(D_i \varphi-\frac{x_i}{a_i^2}\cdot \varphi\bigg)\,\mathrm{d}P\\
 &=&-\int_{O}\bigg(X_n\cdot f\cdot(D_i \varphi)-X_n\cdot f\cdot\frac{x_i}{a_i^2}\cdot \varphi\bigg)\,\mathrm{d}P\\
 &=&-\int_{O}\bigg(f\cdot D_i(\varphi\cdot X_n) -f\cdot\varphi \cdot(D_i X_n)-X_n\cdot f\cdot\frac{x_i}{a_i^2}\cdot \varphi\bigg)\,\mathrm{d}P.
\end{eqnarray*}
Note that $\varphi\cdot X_n\in C_F^{\infty}(O)$ and by Definition \ref{231013def1}, we have
\begin{eqnarray*}
&&-\int_{O}\bigg(f\cdot D_i(\varphi\cdot X_n) -f\cdot\varphi \cdot(D_i X_n)-X_n\cdot f\cdot\frac{x_i}{a_i^2}\cdot \varphi\bigg)\,\mathrm{d}P\\
&=&\int_{O}\bigg(f\cdot D_i^*(\varphi\cdot X_n) +f\cdot\varphi \cdot(D_i X_n)\bigg)\,\mathrm{d}P\\
&=&\int_{O}\bigg( (D_i f)\cdot\varphi \cdot X_n+f\cdot\varphi \cdot (D_iX_n)\bigg)\,\mathrm{d}P\\
&=&\int_{O}\bigg( X_n\cdot (D_i f)+ (D_i X_n)\cdot f \bigg)\cdot\varphi \,\mathrm{d}P.
\end{eqnarray*}
Therefore, we arrive at
\begin{eqnarray*}
 \int_{O}X_n\cdot f\cdot( D_i^*\varphi)\,\mathrm{d}P
&=&\int_{O}\bigg( X_n\cdot (D_i f)+ (D_i X_n)\cdot f \bigg)\cdot\varphi \,\mathrm{d}P.
\end{eqnarray*}
Observe that
\begin{eqnarray*}
&&\sum_{i=1}^{\infty}a_i^2\cdot\int_{O}\left| X_n\cdot (D_i f)+ (D_i X_n)\cdot f \right|^2\,\mathrm{d}P\\
&\leq&2\sum_{i=1}^{\infty}a_i^2\cdot\int_{O} \left(| X_n\cdot (D_i f)|^2+ |(D_i X_n)\cdot f  |^2\right)\,\mathrm{d}P\\
&\leq&2\sum_{i=1}^{\infty}a_i^2\cdot\int_{O} |D_i f|^2 \,\mathrm{d}P+
2\cdot\sup_{\ell^2}\left(\sum_{i=1}^{\infty}a_i^2\cdot|D_i X_n|^2\right)\cdot\int_{O}  |f  |^2 \,\mathrm{d}P<\infty.
\end{eqnarray*}
Thus $X_n\cdot f\in H^{1}(O)$,\,$D_i(X_n\cdot f)=X_n\cdot (D_i f)+ (D_i X_n)\cdot f $ for each $i\in\mathbb{N}$ and
\begin{eqnarray*}
&&||X_n\cdot f-f||_{H^{1}(O)}^2\\
&=&\int_{O}|( X_n-1)f|^2\,\mathrm{d}P+\sum_{i=1}^{\infty}a_i^2\cdot\int_{O}\bigg|(X_n-1)\cdot (D_i f)+f \cdot (D_i X_n)\bigg|^2\,\mathrm{d}P\\
&\leq&\int_{O}|(X_n-1)f|^2\,\mathrm{d}P\\
&&\qquad\qquad\qquad+2\sum_{i=1}^{\infty}a_i^2\cdot\int_{O}\bigg(|(X_n-1)\cdot (D_i f)|^2+|f \cdot(D_i X_n)|^2\bigg)\,\mathrm{d}P\\
&=&\int_{O}|(X_n-1)f|^2\,\mathrm{d}P+2\int_{O}|X_n-1|^2\cdot\left( \sum_{i=1}^{\infty}a_i^2\cdot|D_i f|^2\right)\,\mathrm{d}P\\
&&+2\cdot\int_{O}\left(\sum_{i=1}^{\infty}a_i^2\cdot|D_i X_n|^2\right)\cdot |f|^2 \,\mathrm{d}P.
\end{eqnarray*}
Combining Theorem \ref{230215Th1}, we have $\lim\limits_{n\to\infty}||X_n\cdot f-f||_{H^{1}(O)}=0$ which completes the proof of Lemma \ref{lemma21gc}.
\end{proof}

Suppose $f\in L^1(\ell^2, P)$ and $n\in\mathbb{N}$. Recall the definition of $f_n$ as in Definition \ref{Integration reduce dimension}.
\begin{lemma}\label{20231120lem3}
Suppose that $f\in C_F^{\infty}(\ell^2)$. Then $f_{n}\in C_c^{\infty}(\mathbb{R}^n)$.
\end{lemma}
\begin{proof}
Suppose that $\{\textbf{x}_n^k=(x_1^k,\cdots,x_n^k )\}_{k=0}^{\infty}\subset \mathbb{R}^n$, and $\lim\limits_{k\to\infty}\textbf{x}_n^k=\textbf{x}_n^0$ in $\mathbb{R}^n$. Since $f\in C_F^{\infty}(\ell^2)$,
 \begin{itemize}
\item[(1)]there exists $r\in (0,+\infty)$ such that $f(\textbf{x})=0$ for all $\textbf{x}\in\ell^2$ satisfying  $||\textbf{x}||_{\ell^2}>r$;
\item[(2)]
$\lim\limits_{k\to\infty}f(\textbf{x}_n^k,\textbf{x}^n)=f(\textbf{x}_n^0,\textbf{x}^n),\qquad \forall\,\,\textbf{x}^n\in \ell^{2}(\mathbb{N}\setminus\{1,\cdots,n\});$
\item[(3)] there exists $M\in(0,+\infty)$ such that $|f(\textbf{x})|\leq M$ for all $\textbf{x}\in\ell^2$.
\end{itemize}
By the Bounded Convergence Theorem, we have
\begin{eqnarray*}
&&\lim_{k\to\infty}\int_{\ell^2(\mathbb{N}\setminus\{1,\cdots,n\})} f(\textbf{x}_n^k,\textbf{x}^n)\,\mathrm{d}P^{\widehat{1,\cdots,n}}(\textbf{x}^n)\\
&=&\int_{\ell^2(\mathbb{N}\setminus\{1,\cdots,n\})} f(\textbf{x}_n^0,\textbf{x}^n)\,\mathrm{d}P^{\widehat{1,\cdots,n}}(\textbf{x}^n),
\end{eqnarray*}
which implies that $f_n$ is a continuous function on $\mathbb{R}^n$.

For each $\textbf{x}_n\in\mathbb{R}^n$ such that $||\textbf{x}_n||_{\mathbb{R}^n}>r$. Then for any $\textbf{x}^n\in \ell^{2}(\mathbb{N}\setminus\{1,\cdots,n\})$, we have $f(\textbf{x}_n,\textbf{x}^n)=0$ and hence
\begin{eqnarray*}
f_n(\textbf{x}_n)=\int_{\ell^2(\mathbb{N}\setminus\{1,\cdots,n\})} f(\textbf{x}_n,\textbf{x}^n)\,\,\mathrm{d}P^{\widehat{1,\cdots,n}}(\textbf{x}^n)=0,
\end{eqnarray*}
which implies that supp$f_n\subset\{\textbf{x}_n\in\mathbb{R}^n:||\textbf{x}_n||_{\mathbb{R}^n}\leq r\}$. At last, it is easy to see that $f_n\in C^{\infty}(\mathbb{R}^n)$ which completes the proof Lemma \ref{20231120lem3}.
\end{proof}

\begin{lemma}
\label{20231123lem2}
For each $n\in\mathbb{N}$, it holds that $W^{1,2}(\mathbb{R}^n,\mathcal{N}^n)\subset W^{1,2}(\ell^2)$. Precisely, for each $f\in W^{1,2}(\mathbb{R}^n,\mathcal{N}^n)$,
\begin{itemize}
\item[(1)] $f\in W^{1,2}(\ell^2)$;
\item[(2)] $D_if=0$ in the sense of $W^{1,2}(\ell^2)$,\,\,$\forall\, i>n$;
\item[(3)]  $D_if=D_if$, where the left ``$D_i$'' is in the sense of $W^{1,2}(\mathbb{R}^n,\mathcal{N}^n)$ and the right ``$D_i$'' is in the sense of $W^{1,2}(\ell^2)$,\,\,$\forall\,\, 1\leq i\leq n$.
\end{itemize}
\end{lemma}
\begin{proof}
For any $f\in W^{1,2}(\mathbb{R}^n,\mathcal{N}^n)$, we view $f$ as a cylinder function on $\ell^2$ that only depends the first $n$ variables. Firstly, we have
\begin{eqnarray*}
\int_{\ell^2}|f|^2\,\mathrm{d}P=\int_{\mathbb{R}^n}|f|^2\,\mathrm{d}\mathcal{N}^n.
\end{eqnarray*}
By finite dimensional result, there exists $\{g_k\}_{k=1}^{\infty}\subset C_c^{\infty}(\mathbb{R}^n)$ such that $\lim\limits_{k\to\infty}||g_k-f||_{W^{1,2}(\mathbb{R}^n,\mathcal{N}^n)}=0.$
For each $i>n$ and $\varphi\in C_F^{\infty}(\ell^2)$, it holds that
\begin{eqnarray*}
(D_if)(\varphi)
&=& \int_{\ell^2}fD_i^*\varphi\,\mathrm{d}P
=\lim_{k\to\infty}\left( \int_{\ell^2}g_k D_i^*\varphi\,\mathrm{d}P\right)\\
&=&\lim_{k\to\infty}\left(\int_{\ell^2}(D_ig_k)\varphi\,\mathrm{d}P\right)=0,
\end{eqnarray*}
which implies that $D_if=0$.

For $i=1,\cdots, n$ and $\varphi\in C_F^{\infty}(\ell^2)$, we have
\begin{eqnarray*}
(D_if)(\varphi)
&=&\int_{\ell^2}fD_i^*\varphi\,\mathrm{d}P\\
&=&\int_{\mathbb{R}^n\times \ell^2(\mathbb{N}\setminus\{1,2,\cdots,n\})}f(\textbf{x}_n)
\bigg(-D_i\varphi(\textbf{x}_n,\textbf{x}^n)\\
&&+\frac{x_i}{a_i^2}\cdot\varphi(\textbf{x}_n,\textbf{x}^n)\bigg)\,\mathrm{d}\mathcal{N}^n(\textbf{x}_n)\,\mathrm{d}P^{\widehat{1,\cdots,n}}(\textbf{x}^n)\\
&=&\int_{\mathbb{R}^n}f(\textbf{x}_n)
\left(-D_i\varphi_n(\textbf{x}_n)+\frac{x_i}{a_i^2}\cdot\varphi_n(\textbf{x}_n)\right)\,\mathrm{d}\mathcal{N}^n(\textbf{x}_n) \\
&=&\int_{\mathbb{R}^n}f
\left( D_i^*\varphi_n\right)\,\mathrm{d}\mathcal{N}^n,
\end{eqnarray*}
where $\textbf{x}_n=(x_1,\cdots,x_n)\in \mathbb{R}^n,\textbf{x}^n=(x_i)_{i\in\mathbb{N}\setminus\{1,\cdots,n\}}\in \ell^2(\mathbb{N}\setminus\{1,2,\cdots,n\})$, $P^{\widehat{1,\cdots,n}}$ is the product measure without the $1,2,\cdots,n$-th components, i.e., it is the restriction of the product measure $\Pi_{j\in\mathbb{N}\setminus\{1,\cdots,n\}}\bn_{a_j}$ on \\ $\left(\ell^{2}(\mathbb{N}\setminus\{1,\cdots,n\}),\mathscr{B}\big(\ell^{2}(\mathbb{N}\setminus\{1,\cdots,n\})\big)\right)$ and
\begin{eqnarray*}
\varphi_n(\textbf{x}_n) &=&\int_{\ell^2(\mathbb{N}\setminus\{1,2,\cdots,n\})}\varphi(\textbf{x}_n,\textbf{x}^n) \,\mathrm{d}P^{\widehat{1,\cdots,n}}(\textbf{x}^n).
\end{eqnarray*}
By Lemma \ref{20231120lem3}, we see that $\varphi_{n}\in C_c^{\infty}(\mathbb{R}^n)$ and hence
\begin{eqnarray*}
(D_if)(\varphi)=\int_{\mathbb{R}^n}f
\left( D_i^*\varphi_n\right)\,\mathrm{d}\mathcal{N}^n
=\int_{\mathbb{R}^n}(D_if)
 \varphi_n \,\mathrm{d}\mathcal{N}^n
=\int_{\ell^2(\mathbb{N})}(D_if)
 \varphi \,\mathrm{d}P,
\end{eqnarray*}
which implies that $D_if$ exists in $W^{1,2}(\ell^2)$ and equals $D_if$ in $W^{1,2}(\mathbb{R}^n,\mathcal{N}^n)$. This completes the proof of Lemma \ref{20231120lem2}.
\end{proof}
\begin{lemma}\label{global approximation}
Suppose that $f\in W^{1,2}(\ell^2)$,  then there exists $\{\varphi_n\}_{n=1}^{\infty}\subset \mathscr{C}_c^{\infty}$ such that
\begin{eqnarray*}
 \lim_{n\to\infty}||\varphi_n-f||_{W^{1,2}(\ell^2)}=0.
\end{eqnarray*}
\end{lemma}
\begin{proof}
For each $f\in W^{1,2}(\ell^2)$, $n\in\mathbb{N}$, $\varphi\in C_c^{\infty}(\mathbb{R}^n)$ and $i=1,\cdots,n$, by  Lemma \ref{20231111lem1}, we have
\begin{eqnarray*}
 \int_{\mathbb{R}^n}f_n D_i^*\varphi\,\mathrm{d}\mathcal{N}^n= \int_{\ell^2} f D_i^*\varphi\,\mathrm{d}P
=\int_{\ell^2} (D_if)\varphi\,\mathrm{d}P
= \int_{\mathbb{R}^n}(D_if)_n\varphi\,\mathrm{d}\mathcal{N}^n
\end{eqnarray*}
and hence $ D_i(f_n)=(D_if)_n$. Thus $f_n\in  W^{1,2}(\mathbb{R}^n,\mathcal{N}^n)$. By Lemma \ref{20231123lem2}, $f_n\in W^{1,2}(\ell^2)$ and
\begin{eqnarray*}
&&||f_n-f||_{W^{1,2}(\ell^2)}^2\\
&=&\int_{\ell^2}|f_n-f|^2\,\mathrm{d}P
+\sum_{i=1}^{n}a_i^2\int_{\ell^2}|(D_if)_n-D_if|^2\,\mathrm{d}P
+\sum_{i=n+1}^{\infty}a_i^2\int_{\ell^2}|D_if|^2\,\mathrm{d}P.
\end{eqnarray*}
For any given $\epsilon>0$, by the definition of $W^{1,2}(\ell^2)$ and Proposition \ref{Reduce diemension}, there exists $n_1\in\mathbb{N}$ such that for any $n\geq n_1$, we have
\begin{eqnarray*}
\int_{\ell^2}|f_n-f|^2\,\mathrm{d}P&<&\frac{\epsilon}{3},\\
4\sum_{i=n_1+1}^{\infty}a_i^2\int_{\ell^2}|D_if|^2\,\mathrm{d}P&<&\frac{\epsilon}{3}.
\end{eqnarray*}
By Proposition \ref{Reduce diemension} again, there exists $n_2>n_1$ such that for any $n\geq n_2$, we have
\begin{eqnarray*}
\sum_{i=1}^{n_1}a_i^2\int_{\ell^2}|(D_if)_n-D_if|^2\,\mathrm{d}P&<&\frac{\epsilon}{3}.
\end{eqnarray*}
Then for any $n\geq n_2$, we have
\begin{eqnarray*}
&&||f_n-f||_{W^{1,2}(\ell^2)}^2\\
&=&\int_{\ell^2}|f_n-f|^2\,\mathrm{d}P
+\sum_{i=1}^{n_1}a_i^2\int_{\ell^2}|(D_if)_n-D_if|^2\,\mathrm{d}P\\
&&+\sum_{i=n_1+1}^{n}a_i^2\int_{\ell^2}|(D_if)_n-D_if|^2\,\mathrm{d}P
+\sum_{i=n+1}^{\infty}a_i^2\int_{\ell^2}|D_if|^2\,\mathrm{d}P\\
&\leq&\int_{\ell^2}|f_n-f|^2\,\mathrm{d}P
+\sum_{i=1}^{n_1}a_i^2\int_{\ell^2}|(D_if)_n-D_if|^2\,\mathrm{d}P\\
&&+2\sum_{i=n_1+1}^{n}a_i^2\int_{\ell^2}(|(D_if)_n|^2+|D_if|^2)\,\mathrm{d}P
+\sum_{i=n+1}^{\infty}a_i^2\int_{\ell^2}|D_if|^2\,\mathrm{d}P\\
&\leq&\int_{\ell^2}|f_n-f|^2\,\mathrm{d}P
+\sum_{i=1}^{n_1}a_i^2\int_{\ell^2}|(D_if)_n-D_if|^2\,\mathrm{d}P\\
&&+4\sum_{i=n_1+1}^{\infty}a_i^2\int_{\ell^2}|D_if|^2\,\mathrm{d}P<\epsilon,
\end{eqnarray*}
where the second inequality follows from the conclusion (1) of Proposition \ref{Reduce diemension}.
 Thus $\lim\limits_{n\to\infty}||f_n-f||_{W^{1,2}(\ell^2)}=0$. By Proposition \ref{20231123prop2}, for each $n\in\mathbb{N}$, there exists $\{\varphi_m\}_{m=1}^{\infty}\subset C_c^{\infty}(\mathbb{R}^n)$ such that
\begin{eqnarray*}
 \lim_{m\to\infty}||\varphi_m-f_n||_{W^{1,2}(\mathbb{R}^n,\mathcal{N}^n)}=0.
\end{eqnarray*}
Diagonal process completes the proof of Lemma \ref{global approximation}.
\end{proof}
\begin{definition}
Let $\mathbb{H}\triangleq\{\textbf{x}=(x_i)_{i\in\mathbb{N}}\in\ell^2:x_1>0\}$. We call $\mathbb{H}$ the half space of $\ell^2$. For each $n\in\mathbb{N}$, let $\mathbb{H}_n\triangleq\{\textbf{x}_n=(x_1,\cdots,x_n)\in \mathbb{R}^n:x_1>0\}$. We call $\mathbb{H}_n$ the half space of $\mathbb{R}^n$. Then it holds that $\mathbb{H}=\mathbb{H}_n\times \ell^2(\mathbb{N}\setminus\{1,2,\cdots,n\})$
\end{definition}
\begin{proposition}\label{20231123prop21}
For each $n\in\mathbb{N}$, $C_c^{\infty}(\mathbb{R}^n)$ is dense in $W^{1,2}(\mathbb{H}_n,\mathcal{N}^n)$.
\end{proposition}
\begin{proof}
Suppose that $f\in W^{1,2}(\mathbb{R}^n,\mathcal{N}^n)$.

\textbf{Step 1.}
Choosing $\psi\in C^{\infty}(\mathbb{R};[0,1])$ such that $\psi(x)=1$ for all $x\leq 1$ and $\psi(x)=0$ for all $x\geq 2$. For each $k\in\mathbb{N}$, let $\psi_k(\textbf{x}_n)\triangleq \psi\left(\frac{x_1^2+\cdots+x_n^2}{k}\right),\,\forall\,\textbf{x}_n=(x_1,\cdots,x_n)\in\mathbb{R}^n$.
Obviously,
\begin{eqnarray*}
\int_{\mathbb{H}_n}|\psi_k\cdot f|^2\,\mathrm{d}\mathcal{N}^n\leq \int_{\mathbb{H}_n}| f|^2\,\mathrm{d}\mathcal{N}^n<\infty,\qquad\forall\,k\in\mathbb{N}.
\end{eqnarray*}
Note that for each $\varphi\in C_c^{\infty}(\mathbb{R}^n),\,k\in\mathbb{N}$ and $i=1,\cdots,n$, we have
\begin{eqnarray*}
\int_{\mathbb{H}_n}\psi_k\cdot f\cdot D_i^*\varphi\,\mathrm{d}\mathcal{N}^n
&=&\int_{\mathbb{H}_n}\psi_k\cdot f\cdot \left(-D_i\varphi+\frac{x_i}{a_i^2}\cdot \varphi\right)\,\mathrm{d}\mathcal{N}^n\\
&=&\int_{\mathbb{H}_n} f\cdot \left((D_i\psi_k)\cdot\varphi-D_i(\varphi\cdot\psi_k)+\frac{x_i}{a_i^2}\cdot \varphi\cdot\psi_k\right)\,\mathrm{d}\mathcal{N}^n\\
&=&\int_{\mathbb{H}_n} f\cdot \left((D_i\psi_k)\cdot\varphi+D_i^*(\varphi\cdot\psi_k) \right)\,\mathrm{d}\mathcal{N}^n\\
&=&\int_{\mathbb{H}_n} \left(f\cdot(D_i\psi_k)+(D_if) \cdot\psi_k \right)\cdot\varphi\,\mathrm{d}\mathcal{N}^n,
\end{eqnarray*}
where the last equality follows from the fact that $\varphi\cdot\psi_k\in C_c^{\infty}(\mathbb{H}_n)$ and Definition \ref{20231123def1}. Since
\begin{eqnarray*}
&&\int_{\mathbb{H}_n} |f\cdot(D_i\psi_k)+(D_if) \cdot\psi_k|^2\,\mathrm{d}\mathcal{N}^n\\
&\leq &2\int_{\mathbb{H}_n} (|f\cdot(D_i\psi_k)|^2+|(D_if) \cdot\psi_k|^2)\,\mathrm{d}\mathcal{N}^n\\
&\leq&2\cdot \sup_{\mathbb{H}_n}|D_i\psi_k|^2\cdot\int_{\mathbb{H}_n} |f|^2\,\mathrm{d}\mathcal{N}^n+2\int_{\mathbb{H}_n}|D_if|^2\,\mathrm{d}\mathcal{N}^n<\infty.
\end{eqnarray*}
These implies that $\{\psi_k\cdot f\}_{k=1}^{\infty}\subset W^{1,2}(\mathbb{H}_n,\mathcal{N}^n)$ and $D_i(\psi_k\cdot f)=f\cdot(D_i\psi_k)+(D_if) \cdot\psi_k$ for any $k\in\mathbb{N}$ and $i=1,\cdots,n.$ Then we have
\begin{eqnarray*}
&&||\psi_k\cdot f-f||_{W^{1,2}(\mathbb{H}_n,\mathcal{N}^n)}^2\\
&=&\int_{\mathbb{H}_n}|( \psi_k-1)f|^2\,\mathrm{d}\mathcal{N}^n+\sum_{i=1}^{n}a_i^2\cdot\int_{\mathbb{H}_n}\bigg|(\psi_k-1)\cdot (D_i f)+f \cdot (D_i \psi_k)\bigg|^2\,\mathrm{d}\mathcal{N}^n\\
&\leq&\int_{\mathbb{H}_n}|(\psi_k-1)f|^2\,\mathrm{d}\mathcal{N}^n\\
&&+2\sum_{i=1}^{n}a_i^2\cdot\int_{\mathbb{H}_n}\bigg(|(\psi_k-1)\cdot (D_i f)|^2+|f \cdot(D_i \psi_k)|^2\bigg)\,\mathrm{d}\mathcal{N}^n\\
&=&\int_{\mathbb{H}_n}|(\psi_k-1)f|^2\,\mathrm{d}\mathcal{N}^n+2\int_{\mathbb{H}_n}|\psi_k-1|^2\cdot\left( \sum_{i=1}^{n}a_i^2\cdot|D_i f|^2\right)\,\mathrm{d}\mathcal{N}^n\\
&&+2\cdot\int_{\mathbb{H}_n}\left(\sum_{i=1}^{n}a_i^2\cdot|D_i \psi_k|^2\right)\cdot |f|^2 \,\mathrm{d}\mathcal{N}^n,\qquad\forall\,k\in\mathbb{N}.
\end{eqnarray*}
Therefore, $\lim\limits_{k\to\infty}||\psi_k\cdot f-f||_{W^{1,2}(\mathbb{H}_n,\mathcal{N}^n)}^2=0$.

\textbf{Step 2.} Now we fix a $k\in\mathbb{N}$. Then there exists $r\in(0,+\infty)$ such that
\begin{eqnarray}\label{20231124for21}
\int_{||\textbf{x}_n||_{\mathbb{R}^n}>r,\textbf{x}_n\in\mathbb{H}_n}|(\psi_k\cdot f)(\textbf{x}_n)|^2\,\mathrm{d}\mathcal{N}^n(\textbf{x}_n)=0.
\end{eqnarray}
Choosing $\chi_n\in C_c^{\infty}(\mathbb{R}^n;[0,1])$ such that $\int_{\mathbb{R}^n}\chi_n(\textbf{x}_n)\mathrm{d}\textbf{x}_n=1$ and supp$\chi_n\subset\{\textbf{x}_n=(x_1,\cdots,x_n)\in\mathbb{R}^n:1<x_1<2\,\,\text{and}\,\,||\textbf{x}_n||_{\mathbb{R}^n}<2\}$. Write $\chi_{n,\delta}(\textbf{x}_n)\triangleq \frac{1}{\delta^n}\cdot\chi_{n}(\frac{1}{\delta}\cdot\textbf{x}_n),\,\,\forall\,\textbf{x}_n\in\mathbb{R}^n$ and $\delta\in(0,+\infty)$. Then $x_{n,\delta}\in C_c^{\infty}(\mathbb{R}^n)$ and supp$\chi_{n,\delta}\subset\{\textbf{x}_n=(x_1,\cdots,x_n)\in\mathbb{R}^n:-2\delta<x_1<-\delta\,\,\text{and}\,\,||\textbf{x}_n||_{\mathbb{R}^n}<2\delta\}$. By \eqref{20231124for21}, we have
\begin{eqnarray*}
&&\int_{\mathbb{H}_n}|(\psi_k\cdot f)(\textbf{x}_n)|^2\,\mathrm{d}\mathcal{N}^n(\textbf{x}_n)\\
&=&\int_{||\textbf{x}_n||_{\mathbb{R}^n}>r,\textbf{x}_n\in\mathbb{H}_n}|(\psi_k\cdot f)(\textbf{x}_n)|^2\,\mathrm{d}\mathcal{N}^n(\textbf{x}_n)\\
&=&\int_{||\textbf{x}_n||_{\mathbb{R}^n}>r,\textbf{x}_n\in\mathbb{H}_n}|(\psi_k\cdot f)(\textbf{x}_n)|^2\cdot G_n(\textbf{x}_n)\,\mathrm{d}\textbf{x}_n\\
&\geq&\left(\prod_{i=1}^{n}\frac{1}{\sqrt{2\pi a_i^2}}e^{-\frac{r^2}{2a_i^2}}\right)\cdot\int_{||\textbf{x}_n||_{\mathbb{R}^n}>r,\textbf{x}_n\in\mathbb{H}_n}|(\psi_k\cdot f)(\textbf{x}_n)|^2  \,\mathrm{d}\textbf{x}_n\\
&=&\left(\prod_{i=1}^{n}\frac{1}{\sqrt{2\pi a_i^2}}e^{-\frac{r^2}{2a_i^2}}\right)\cdot\int_{ \mathbb{H}_n}|(\psi_k\cdot f)(\textbf{x}_n)|^2  \,\mathrm{d}\textbf{x}_n,
\end{eqnarray*}
where $G_n$ is defined at \eqref{20231124for1} and this implies that $\psi_k\cdot f\in L^2(\mathbb{H}_n,\mathrm{d}\textbf{x}_n)\cap  L^1(\mathbb{H}_n,\mathrm{d}\textbf{x}_n)$. We can extend $f$ to a function on $\mathbb{R}^n$ by zero extension to $\mathbb{R}^n\setminus\mathbb{H}_n$. Then let
\begin{eqnarray*}
f_{k,n,\delta}(\textbf{x}_n)\triangleq  \int_{\mathbb{R}^n}(\psi_k\cdot f)(\textbf{y}_n)\chi_{n,\delta}(\textbf{x}_n-\textbf{y}_n)\,\mathrm{d}\textbf{y}_n,
\end{eqnarray*}
for any $\textbf{x}_n\in\mathbb{R}^n\,\,\text{and}\,\,\delta\in(0,+\infty)$. Then $f_{k,n,\delta}\in C_c^{\infty}(\mathbb{R}^n)$, supp$f_{k,n,\delta}\subset \{\textbf{x}_n\in\mathbb{R}^n:||\textbf{x}_n||_{\mathbb{R}^n}\leq r+2\delta\}$ and
\begin{eqnarray*}
&&\int_{\mathbb{H}_n}|f_{k,n,\delta}(\textbf{x}_n)- (\psi_k\cdot f)(\textbf{x}_n)|^2\,\mathrm{d}\mathcal{N}^n(\textbf{x}_n)\\
&=&\int_{\mathbb{H}_n}|f_{k,n,\delta}(\textbf{x}_n)- (\psi_k\cdot f)(\textbf{x}_n)|^2G_n(\textbf{x}_n)\,\mathrm{d}\textbf{x}_n\\
&\leq&\left(\prod_{i=1}^{n}\frac{1}{\sqrt{2\pi a_i^2}} \right)\cdot\int_{\mathbb{R}^n}|f_{k,n,\delta}(\textbf{x}_n)- (\psi_k\cdot f)(\textbf{x}_n)|^2 \,\mathrm{d}\textbf{x}_n,
\end{eqnarray*}
which implies that $\lim\limits_{\delta\to 0}\int_{\mathbb{H}_n}|f_{k,n,\delta}(\textbf{x}_n)- (\psi_k\cdot f)(\textbf{x}_n)|^2\,\mathrm{d}\mathcal{N}^n(\textbf{x}_n)=0.$

Let $\chi_{n,\delta,\textbf{x}_n}(\textbf{y}_n)\triangleq\chi_{n,\delta}(\textbf{x}_n-\textbf{y}_n),\,\forall\,\textbf{x}_n,\textbf{y}_n\in\mathbb{R}^n$. If $\textbf{x}_n=(x_1,\cdots,x_n)\in\mathbb{H}_n$, then
$$
\text{supp}\chi_{n,\delta,\textbf{x}_n}\subset \{\textbf{y}_n=(y_1,\cdots,y_n)\in\mathbb{R}^n:||\textbf{y}_n||_{\mathbb{R}^n}\leq||\textbf{x}_n||_{\mathbb{R}^n}+2\delta,\,y_1\geq  x_1+\delta\},
$$
which is a compact subset of $\mathbb{H}_n$ which implies that $\chi_{n,\delta,\textbf{x}_n}\in C_c^{\infty}(\mathbb{H}_n)$. For each $i=1,\cdots,n$ and $\textbf{x}_n=(x_1,\cdots,x_n)\in\mathbb{H}_n$, note that
\begin{eqnarray*}
&&D_if_{k,n,\delta}(\textbf{x}_n)\\
&=&\int_{\mathbb{R}^n}(\psi_k\cdot f)(\textbf{y}_n)D_i\chi_{n,\delta}(\textbf{x}_n-\textbf{y}_n)\,\mathrm{d}\textbf{y}_n\\
&=&\int_{\mathbb{H}_n}(\psi_k\cdot f)(\textbf{y}_n)D_i\chi_{n,\delta}(\textbf{x}_n-\textbf{y}_n)\,\mathrm{d}\textbf{y}_n\\
&=&-\int_{\mathbb{H}_n}(\psi_k\cdot f)(\textbf{y}_n)\cdot(D_i\chi_{n,\delta,\textbf{x}_n})(\textbf{y}_n)\,\mathrm{d}\textbf{y}_n\\
&=&-\int_{\mathbb{H}_n}(\psi_k\cdot f)(\textbf{y}_n)\cdot(D_i\chi_{n,\delta,\textbf{x}_n})(\textbf{y}_n)\cdot\frac{1}{G_n(\textbf{y}_n)}\,\mathrm{d}\mathcal{N}^n(\textbf{y}_n)\\
&=&\int_{\mathbb{H}_n}(\psi_k\cdot f)(\textbf{y}_n)\left(- \left(D_i\frac{\chi_{n,\delta,\textbf{x}_n} }{G_n}\right)(\textbf{y}_n)+\frac{y_i}{a_i^2}\cdot\frac{\chi_{n,\delta,\textbf{x}_n} }{G_n}(\textbf{y}_n)\right)\,\mathrm{d}\mathcal{N}^n(\textbf{y}_n)\\
&=&\int_{\mathbb{H}_n}(\psi_k\cdot f)\cdot \left(D_i^*\frac{\chi_{n,\delta,\textbf{x}_n}}{G_n} \right)\,\mathrm{d}\mathcal{N}^n\\
&=&\int_{\mathbb{H}_n}D_i(\psi_k\cdot f)\cdot \frac{\chi_{n,\delta,\textbf{x}_n}}{G_n} \,\mathrm{d}\mathcal{N}^n\\
&=&\int_{\mathbb{H}_n}D_i(\psi_k\cdot f)(\textbf{y}_n)\cdot  \chi_{n,\delta,\textbf{x}_n}(\textbf{y}_n)  \,\mathrm{d}\textbf{y}_n \\
&=&\int_{\mathbb{H}_n}D_i(\psi_k\cdot f)(\textbf{y}_n)\cdot  \chi_{n,\delta}(\textbf{x}_n-\textbf{y}_n)  \,\mathrm{d}\textbf{y}_n,
\end{eqnarray*}
where the seventh equality follows from the fact that $\frac{\chi_{n,\delta,\textbf{x}_n}}{G_n}\in C_c^{\infty}(\mathbb{H}_n)$ and Definition \ref{20231123def1}. It is easy to see that $\int_{||\textbf{x}_n||_{\mathbb{R}^n}>r}|D_i(\psi_k\cdot f)|^2\,\mathrm{d}\mathcal{N}^n=0$. These imply that
\begin{eqnarray*}
\lim_{\delta\to 0}\int_{\mathbb{H}_n}|D_if_{k,n,\delta} - D_i(\psi_k\cdot f)|^2\,\mathrm{d}\mathcal{N}^n =0.
\end{eqnarray*}
Hence, $\lim\limits_{\delta\to 0}||f_{k,n,\delta}-\psi_k\cdot f||_{W^{1,2}(\mathbb{H}_n,\mathcal{N}^n)}^2=0$.  By diagonal process, we complete the proof of Proposition \ref{20231123prop21}.
\end{proof}
\begin{lemma}
\label{20231120lem2}
For each $n\in\mathbb{N}$, it holds that $W^{1,2}(\mathbb{H}_n,\mathcal{N}^n)\subset W^{1,2}(\mathbb{H})$. Precisely, for each $f\in W^{1,2}(\mathbb{H}_n,\mathcal{N}^n)$,
\begin{itemize}
\item[(1)] $f\in W^{1,2}(\mathbb{H})$;
\item[(2)] $D_if=0$ in the sense of $W^{1,2}(\mathbb{H})$,\,\,$\forall\, i>n$;
\item[(3)]  $D_if=D_if$, where the left ``$D_i$'' is in the sense of $W^{1,2}(\mathbb{H}_n,\mathcal{N}^n)$ and the right ``$D_i$'' is in the sense of $W^{1,2}(\mathbb{H})$,\,\,$\forall\,\, 1\leq i\leq n$.
\end{itemize}
\end{lemma}
\begin{proof}
For any $f\in W^{1,2}(\mathbb{H}_n,\mathcal{N}^n)$, we view $f$ as a cylinder function on $\mathbb{H}$ that only depends the first $n$ variables. Firstly, we have
\begin{eqnarray*}
\int_{\mathbb{H}}|f|^2\,\mathrm{d}P=\int_{\mathbb{H}_n}|f|^2\,\mathrm{d}\mathcal{N}_n.
\end{eqnarray*}
By Proposition \ref{20231123prop21}, there exists $\{g_k\}_{k=1}^{\infty}\subset C_c^{\infty}(\mathbb{R}^n)$ such that $\lim\limits_{k\to\infty}||g_k-f||_{W^{1,2}(\mathbb{H}_n)}=0.$
For each $i>n$ and $\varphi\in C_F^{\infty}(\mathbb{H})$, it holds that
\begin{eqnarray*}
(D_if)(\varphi)
&=& \int_{\mathbb{H}}fD_i^*\varphi\,\mathrm{d}P
=\lim_{k\to\infty}\left( \int_{\mathbb{H}}g_k D_i^*\varphi\,\mathrm{d}P\right)\\
&=&\lim_{k\to\infty}\left(\int_{\mathbb{H}}(D_ig_k)\varphi\,\mathrm{d}P\right)=0,
\end{eqnarray*}
which implies that $D_if=0$.

For $1\leq i\leq n$ and $\varphi\in C_F^{\infty}(\mathbb{H})$, we have
\begin{eqnarray*}
(D_if)(\varphi)
&=&\int_{\mathbb{H}}fD_i^*\varphi\,\mathrm{d}P\\
&=&\int_{\mathbb{H}_n\times \ell^2(\mathbb{N}\setminus\{1,2,\cdots,n\})}f(\textbf{x}_n)
\bigg(-D_i\varphi(\textbf{x}_n,\textbf{x}^n)\\
&&+\frac{x_i}{a_i^2}\cdot\varphi(\textbf{x}_n,\textbf{x}^n)\bigg)\,\mathrm{d}\mathcal{N}^n(\textbf{x}_n)\,\mathrm{d}P^{\widehat{1,\cdots,n}}(\textbf{x}^n)\\
&=&\int_{\mathbb{H}_n}f(\textbf{x}_n)
\left(-D_i\varphi_n(\textbf{x}_n)+\frac{x_i}{a_i^2}\cdot\varphi_n(\textbf{x}_n)\right)\,\mathrm{d}\mathcal{N}^n(\textbf{x}_n) \\
&=&\int_{\mathbb{H}_n}f
\left( D_i^*\varphi_n\right)\,\mathrm{d}\mathcal{N}^n,
\end{eqnarray*}
where $\textbf{x}_n=(x_1,\cdots,x_n)\in \mathbb{H}_n,\textbf{x}^n=(x_i)_{i\in\mathbb{N}\setminus\{1,\cdots,n\}}\in \ell^2(\mathbb{N}\setminus\{1,2,\cdots,n\})$, $P^{\widehat{1,\cdots,n}}$ is the product measure without the $1,2,\cdots,n$-th components, i.e., it is the restriction of the product measure $\prod_{j\in\mathbb{N}\setminus\{1,\cdots,n\}}\bn_{a_j}$ on \\ $\left(\ell^{2}(\mathbb{N}\setminus\{1,\cdots,n\}),\mathscr{B}\big(\ell^{2}(\mathbb{N}\setminus\{1,\cdots,n\})\big)\right)$ and
\begin{eqnarray*}
\varphi_n(\textbf{x}_n) &=&\int_{\ell^2(\mathbb{N}\setminus\{1,2,\cdots,n\})}\varphi(\textbf{x}_n,\textbf{x}^n) \,\mathrm{d}P^{\widehat{1,\cdots,n}}(\textbf{x}^n).
\end{eqnarray*}
Similar to the proof of Lemma \ref{20231120lem3}, we see that $\varphi_{n}\in C_c^{\infty}(\mathbb{H}_n)$ and hence
\begin{eqnarray*}
(D_if)(\varphi)=\int_{\mathbb{H}_n}f
\left( D_i^*\varphi_n\right)\,\mathrm{d}\mathcal{N}^n
=\int_{\mathbb{H}_n}(D_if)
 \varphi_n \,\mathrm{d}\mathcal{N}^n
=\int_{\mathbb{H}}(D_if)
 \varphi \,\mathrm{d}P,
\end{eqnarray*}
which implies that $D_if$ exists in $W^{1,2}(\mathbb{H})$ and equals $D_if$ in $W^{1,2}(\mathbb{H}_n)$. This completes the proof of Lemma \ref{20231120lem2}.
\end{proof}
\begin{lemma}
\label{20231120lem1}
Suppose that $f\in W^{1,2}(\mathbb{H})$,  then there exists $\{\varphi_n\}_{n=1}^{\infty}\subset  \mathscr{C}_c^{\infty}$ such that
\begin{eqnarray*}
 \lim_{n\to\infty}||\varphi_n-f||_{W^{1,2}(\mathbb{H})}=0.
\end{eqnarray*}
\end{lemma}
\begin{proof}
Note that for any $\varphi\in C_c^{\infty}(\mathbb{H}_n)$, $n\in\mathbb{N}$ and $1\leq i\leq n$, we have
\begin{eqnarray*}
\int_{\mathbb{H}_n}(D_if)_n\varphi\,\mathrm{d}\mathcal{N}^n
=\int_{\mathbb{H} }(D_if)\varphi\,\mathrm{d}P
=\int_{\mathbb{H} }fD_i^*\varphi\,\mathrm{d}P
=\int_{\mathbb{H}_n }f_nD_i^*\varphi\,\mathrm{d}\mathcal{N}^n,
\end{eqnarray*}
where the second equality follows from the fact that $\varphi\in C_F^{\infty}(\mathbb{H})$. This implies that $(D_if)_n=D_i(f_n)$ in $W^{1,2}(\mathbb{H}_n,\mathcal{N}^n)$ and hence $f_n\in W^{1,2}(\mathbb{H}_n,\mathcal{N}^n)$. By Lemma \ref{20231120lem2}, $f_n\in W^{1,2}(\mathbb{H})$ and
\begin{eqnarray*}
 ||f_n-f||_{W^{1,2}(\mathbb{H})}^2&=&\int_{\mathbb{H}}|f_n-f|^2\,\mathrm{d}P
 +\sum_{i=1}^{n}a_i^2\int_{\mathbb{H}}|(D_if)_n-D_if|^2\,\mathrm{d}P\\
&&+\sum_{i=n+1}^{\infty}a_i^2\int_{\mathbb{H}}|D_if|^2\,\mathrm{d}P.
\end{eqnarray*}
By Proposition \ref{Reduce diemension},
\begin{eqnarray*}
\lim_{n\to\infty}\int_{\mathbb{H}}|f_n-f|^2\,\mathrm{d}P=0,\qquad\lim_{n\to\infty}\int_{\mathbb{H}}|(D_if)_n-D_if|^2\,\mathrm{d}P,\quad\,\forall\,i\in\mathbb{N},
\end{eqnarray*}
and $\sum\limits_{i=1}^{\infty}a_i^2\int_{\mathbb{H}}|D_if|^2\,\mathrm{d}P<\infty$. For any given $\epsilon>0$, there exists $n_1\in\mathbb{N}$ such that for any $n\geq n_1$, we have
\begin{eqnarray*}
 \int_{\mathbb{H}}|f_n-f|^2\,\mathrm{d}P<\frac{\epsilon}{3},\qquad \sum\limits_{i=n_1+1}^{\infty}4a_i^2\int_{\mathbb{H}}|D_if|^2\,\mathrm{d}P<\frac{\epsilon}{3}.
\end{eqnarray*}
Note that
\begin{eqnarray*}
 \lim_{n\to\infty}\left( \sum\limits_{i=1}^{ n_1}4a_i^2\int_{\mathbb{H}}|(D_if)_n-D_if|^2\,\mathrm{d}P\right)=0,
\end{eqnarray*}
and hence there exists $n_2>n_1$ such that for any $n\geq n_2$, it holds that
\begin{eqnarray*}
  \sum\limits_{i=1}^{ n_1}4a_i^2\int_{\mathbb{H}}|(D_if)_n-D_if|^2\,\mathrm{d}P<\frac{\epsilon}{3}.
\end{eqnarray*}
Then for any any $n\geq n_2$, we have
\begin{eqnarray*}
&&||f_n-f||_{W^{1,2}(\mathbb{H})}^2\\
&=&\int_{\mathbb{H}}|f_n-f|^2\,\mathrm{d}P
+\sum_{i=1}^{n_1}a_i^2\int_{\mathbb{H}}|(D_if)_n-D_if|^2\,\mathrm{d}P\\
&&+\sum_{i=n_1+1}^{n}a_i^2\int_{\mathbb{H}}|(D_if)_n-D_if|^2\,\mathrm{d}P
+\sum_{i=n+1}^{\infty}a_i^2\int_{\mathbb{H}}|D_if|^2\,\mathrm{d}P\\
&\leq&\int_{\mathbb{H}}|f_n-f|^2\,\mathrm{d}P
+\sum_{i=1}^{n_1}a_i^2\int_{\mathbb{H}}|(D_if)_n-D_if|^2\,\mathrm{d}P\\
&&+\sum_{i=n_1+1}^{\infty}4a_i^2\int_{\mathbb{H}}|D_if|^2\,\mathrm{d}P
<\frac{\epsilon}{3}+\frac{\epsilon}{3}+\frac{\epsilon}{3}=\epsilon,
\end{eqnarray*}
which implies that $\lim\limits_{n\to\infty}||f_n-f||_{W^{1,2}(\mathbb{H})}=0$. We also note that for each $n\in\mathbb{N}$, $C_c^{\infty}(\mathbb{R}^n)$ is dense in $W^{1,2}(\mathbb{H}_n,\mathcal{N}^n)$. Since $C_c^{\infty}(\mathbb{R}^n)\subset  \mathscr{C}_c^{\infty}$, combining  Lemma \ref{20231120lem2}, we see that
$C_c^{\infty}(\mathbb{R}^n)$ is dense in $W^{1,2}(\mathbb{H})$. The proof of Lemma \ref{20231120lem1} is completed.
\end{proof}
\begin{lemma}
\label{20231120lem4}
Suppose that $O$ is an open subset of $\ell^2$ and $\textbf{x}_0\in(\partial \mathbb{H})\cap O$. If $f\in W^{1,2}(O\cap\mathbb{H})$ and there exists $r_1,r_2\in(0,+\infty)$ such that $r_1<r_2$, $\chi_{\ell^2(\mathbb{N})\setminus B_{r_1}(\textbf{x}_0)}\cdot f=0$ and $B_{r_2}(\textbf{x}_0)\cap\mathbb{H}\subset O$. Then it holds that
\begin{itemize}
\item[(1)] $f\in W^{1,2}(\mathbb{H})$;
\item[(2)] $D_if=D_if$, where the left ``$D_i$'' is in the sense of $W^{1,2}(O\cap\mathbb{H})$ and the right ``$D_i$'' is in the sense of $W^{1,2}(\mathbb{H})$,\,\,$\forall\,\, i\in\mathbb{N}$;
\item[(3)]  there exists $\{\varphi_n\}_{n=1}^{\infty}\subset \mathscr{C}_c^{\infty}$ such that
\begin{eqnarray*}
 \lim_{n\to\infty}||\varphi_n-f||_{W^{1,2}(\mathbb{H}\cap O)}=0.
\end{eqnarray*}
\end{itemize}
\end{lemma}
\begin{proof}
We view $f,\,D_if,\,i\in\mathbb{N}$ as functions on $\mathbb{H}$ by zero extension to $\mathbb{H}\setminus (O\cap\mathbb{H})$. It is easy to see that $\chi_{\ell^2\setminus B_{r_1}(\textbf{x}_0)}\cdot (D_if)=0$ holds for each $i\in\mathbb{N}$. Firstly, we have
\begin{eqnarray*}
\int_{\mathbb{H}}|f|^2\,\mathrm{d}P=\int_{O\cap\mathbb{H}}|f|^2\,\mathrm{d}P.
\end{eqnarray*}
Choosing $\psi\in C^{\infty}(\mathbb{R};[0,1])$ such that $\psi(x)=1$ if $x\leq r_1^2$ and  $\psi(x)=0$ if $x>r_2^2$. Let $\Psi(\textbf{x})\triangleq \psi(||\textbf{x}||_{\ell^2}^2),\,\,\forall\,\textbf{x}\in \ell^2$. For any $i\in\mathbb{N}$ and $\varphi\in C_F^{\infty}(\mathbb{H})$, we have
\begin{eqnarray*}
(D_if)(\varphi)
&=&\int_{\mathbb{H}}f(D_i^*\varphi)\,\mathrm{d}P\\
&=&\int_{\mathbb{H}\cap O}f(D_i^*\varphi)\,\mathrm{d}P\\
&=&\int_{\mathbb{H}\cap O}f\cdot \Psi\cdot(D_i^*\varphi)\,\mathrm{d}P\\
&=&\int_{\mathbb{H}\cap O}f\cdot D_i^*(\Psi\cdot\varphi)\,\mathrm{d}P
+\int_{\mathbb{H}\cap O}(f\cdot D_i(\Psi))\cdot\varphi\,\mathrm{d}P\\
&=&\int_{\mathbb{H}\cap O}f\cdot D_i^*(\Psi\cdot\varphi)\,\mathrm{d}P\\
&=&\int_{\mathbb{H}\cap O}(D_if)\cdot \Psi\cdot\varphi\,\mathrm{d}P\\
&=&\int_{\mathbb{H}\cap O}(D_if)\cdot \varphi\,\mathrm{d}P\\
&=&\int_{\mathbb{H}}(D_if) \cdot\varphi\,\mathrm{d}P,
\end{eqnarray*}
where the third equality follows from the fact that $f\cdot \Psi=f$, the fourth equality follows from the fact $f\cdot D_i(\Psi)=0$, the sixth equality follows from the fact that $\Psi\cdot\varphi\in C_F^{\infty}(\mathbb{H}\cap O)$, the seventh equality follows from the fact that $(D_if)\cdot \Psi=f$. Finally, the conclusion (3) follows from Lemma \ref{20231120lem1}. This completes the proof of Lemma \ref{20231120lem4}.
\end{proof}

\subsection{Support Contained in a Neighborhood}
\begin{lemma}\label{lemma21hs}
Suppose $f\in W^{1,2}(O)$ and
$$
\text{supp}f\subset B_r(\textbf{x}_0)\subset B_{r'}(\textbf{x}_0)\subset O
$$
for some $r,r'\in(0,+\infty)$ satisfying $r'>r$. Then there exists $\{\varphi_n\}_{n=1}^{\infty}\subset \mathscr{C}_c^{\infty}$ such that
\begin{eqnarray*}
 \lim_{n\to\infty}||\varphi_n-f||_{W^{1,2}(O)}=0.
\end{eqnarray*}
\end{lemma}
\begin{proof}
We extend the values of $f, D_if,\,i\in\mathbb{N}$ to $\ell^2\setminus O$ to be zero. For any $\varphi\in C_F^{\infty}(\ell^2)$, choosing $\psi\in C^{\infty}(\mathbb{R};[0,1])$ such that $\psi(x)=1$ if $x\leqslant r^2$ and $\psi(x)=0$ if $x\geqslant (r')^2$. Let $\Psi(\textbf{x})\triangleq \psi(||\textbf{x}||_{\ell^2}^2),\,\forall\,\textbf{x}\in\ell^2$. Then $\Psi,\Psi\cdot \varphi\in  C_F^{\infty}(\ell^2)$ and
\begin{eqnarray*}
\int_{\ell^2(\mathbb{N})}(D_if)\cdot\varphi\,\mathrm{d}P=\int_{\ell^2}(D_if)\cdot\varphi\cdot \Psi\,\mathrm{d}P=\int_{O}(D_if)\cdot\varphi\cdot \Psi\,\mathrm{d}P.
\end{eqnarray*}
Since supp$(\varphi\cdot \Psi)\stackrel{\circ}{\subset} O$, we can view $\varphi\cdot \Psi$ as a member of $C_F^{\infty}(O)$. By Definition \ref{231013def1},
\begin{eqnarray*}
\int_{O}(D_if)\cdot\varphi\cdot \Psi\,\mathrm{d}P&=&\int_{O}f\cdot D_i^*(\varphi\cdot \Psi)\,\mathrm{d}P\\
&=&\int_{O}f\cdot (D_i^*\varphi)\,\mathrm{d}P\\
&=&\int_{\ell^2}f\cdot (D_i^*\varphi)\,\mathrm{d}P,
\end{eqnarray*}
where the second equality follows from the fact that supp$f\subset \{\textbf{x}\in\ell^2:||\textbf{x}-\textbf{x}_0||_{\ell^2}<r\}$. These imply that $f\in W^{1,2}(\ell^2)$ and by Lemma \ref{global approximation}, there exists $\{\varphi_n\}_{n=1}^{\infty}\subset \mathscr{C}_c^{\infty}$ such that $\lim\limits_{n\to\infty}||\varphi_n-f||_{W^{1,2}(\ell^2)}=0$ and hence $\lim\limits_{n\to\infty}||\varphi_n-f||_{W^{1,2}(O)}=0$. This completes the proof of Lemma \ref{lemma21hs}.
\end{proof}

Suppose that $O$ is an open subset of $\ell^2$, $g$ is a real-valued $C^1$ Fr\'{e}chet differentiable function on $\ell^2$ such that
\begin{itemize}
\item[] $\partial O=\{ \textbf{x} :g(\textbf{x})=0\}$ and $O=\{ \textbf{x} :g(\textbf{x})>0\}$;
\item[] $Dg(\textbf{x})\neq 0$ for each $\textbf{x}\in \partial O$.
\end{itemize}
For $\textbf{x}_0\in \partial O$, without lose of generality we assume that $D_1g(\textbf{x}_0) >0,$ and there exists $r\in(0,+\infty)$ such that $\inf\limits_{\textbf{x}\in B_r(\textbf{x}_0)}D_1g(\textbf{x})>0$ and  $\sup\limits_{\textbf{x}\in B_r(\textbf{x}_0)}D_1g(\textbf{x})<+\infty$. Let $U_0\triangleq B_r(\textbf{x}_0)$. Define
\begin{eqnarray*}
\Psi(\textbf{x})&\triangleq &(g(\textbf{x}),(x_i)_{i\in\mathbb{N}\setminus\{1\}}), \quad\forall\,\textbf{x}=(x_i)_{i\in\mathbb{N}}\in \ell^2(\mathbb{N}),\\
\psi(\textbf{x})&\triangleq & \Psi(\textbf{x}),\qquad \,\forall \textbf{x}\in U_0.
\end{eqnarray*}
Then $\psi(U_0)$ is an open subset of $\ell^2$, $\psi(U_0\cap O)=\psi(U_0)\cap \mathbb{H}$ and $\psi$ is a homeomorphism. For each $ \widehat{\textbf{x}} \in \psi(U_0)$ there exists a unique $\textbf{x}=(x_i)_{i\in\mathbb{N}}\in U_0$ such that $\widehat{\textbf{x}} =\psi(\textbf{x})$, let $h(\widehat{\textbf{x}})\triangleq x_1$. Suppose that $h$ is a real-valued $C^1$ Fr\'{e}chet differentiable function on $\psi(U_0)$. Define
\begin{eqnarray*}
\tau(\widehat{\textbf{x}})\triangleq (h(\widehat{\textbf{x}}),(\widehat{x_i})_{i\in\mathbb{N}\setminus\{1\}} ),\quad\forall\,\widehat{\textbf{x}}=(\widehat{x_i})_{i\in\mathbb{N}}\in \psi(U_0).
\end{eqnarray*}
Then $\tau=\psi^{-1}$.
\begin{condition}\label{20231118cond1}
 Suppose that $h\in C_F^{\infty}(\psi(U_0)),\,g\in C_F^{\infty}(U_0)$ and $\widehat{\textbf{x}}\in \eta(U_0)$, and there exists $\delta \in(0,+\infty)$ such that $-\delta< \widehat{x_1 }<\delta,\,-\delta<  x_1  <\delta$ and
\begin{eqnarray*}
&&|h(\widehat{\textbf{x}})|<\delta,\,\,\,|g( \textbf{x})|<\delta,\,\frac{1}{\delta}<D_1h(\widehat{\textbf{x}})<\delta,\\
&&\frac{1}{\delta}<D_1 g( \textbf{x} )<\delta,\,\,\,
\sum\limits_{i=1}^{\infty}a_i^2\left|D_i h(\widehat{\textbf{x}}) \right|^2<\delta,
\end{eqnarray*}
hold for all $\widehat{\textbf{x}}=(\widehat{x_i})_{i\in\mathbb{N}}\in \psi (U_0)$ and $\textbf{x}=(x_i)_{i\in\mathbb{N}}\in U_0$. For any $E\stackrel{\circ}{\subset}U_0\cap O$ and $E'\stackrel{\circ}{\subset}\psi(U_0)\cap \mathbb{H}$, it holds that $\psi(E)\stackrel{\circ}{\subset}\psi(U_0)\cap \mathbb{H}$ and $\tau(E')\stackrel{\circ}{\subset} U_0\cap O$.
\end{condition}
\begin{example}{}
If $O=\mathbb{H}$, then $g(\textbf{x})=x_1$ for any $\textbf{x}=(x_i)_{i\in\mathbb{N}}\in\ell^2$. For any $\textbf{x}_0=(x_i^0)_{i\in\mathbb{N}}\in\partial\mathbb{H}$ and $U_0=B_r(\textbf{x}_0)$ where $r\in(0,+\infty),\,\psi(\textbf{x})=\textbf{x}$ and $\tau(\widehat{\textbf{x}})=\widehat{\textbf{x}}$ for any $\textbf{x}\in U_0\cap\mathbb{H}$ and $\widehat{\textbf{x}}\in  U_0\cap\mathbb{H}$. It holds that
 $-r< \widehat{x_1 }<r,\,-r<  x_1  <r$ and
$$
D_1h(\widehat{\textbf{x}})=1,\quad \sum\limits_{i=1}^{\infty}a_i^2\left|D_i h(\widehat{\textbf{x}}) \right|^2=a_1^2,\quad\forall\,\widehat{\textbf{x}}\in  U_0\cap\mathbb{H}.
$$
Let $\delta\triangleq \max\{1,r,a_1^2\}$. Then Condition \ref{20231118cond1} holds.
\end{example}
\begin{example}{}
If $O=B_1$, then
$$
g(x)=\sqrt{\sum_{i=1}^{\infty}x_i^2}-1,\qquad\forall\,\textbf{x}=(x_i)_{i\in\mathbb{N}}\in \ell^2.
$$
Let $\textbf{x}_0=(\delta_{1,i})_{i\in\mathbb{N}}\in \partial(B_1)$ and $B_{\frac{1}{3}}(\textbf{x}_0)$, which is a convex open bounded neighborhood of $\textbf{x}_0$. Let
\begin{eqnarray*}
\psi(\textbf{x})&\triangleq &\left(\sqrt{\sum_{i=1}^{\infty}x_i^2}-1,(x_i)_{i\in\mathbb{N}\setminus\{1\}}\right),\qquad\forall\,\textbf{x}=(x_i)_{i\in\mathbb{N}}\in U_0,\\
h(\widehat{\textbf{x}})&\triangleq & \sqrt{(\widehat{x_1}+1)^2- \sum_{i=2}^{\infty}(\widehat{x_i})^2},\qquad\forall\,\widehat{\textbf{x}}=(\widehat{x_i})_{i\in\mathbb{N}}\in \psi(U_0),\\
\tau(\textbf{x})&\triangleq& \left( \sqrt{(\widehat{x_1}+1)^2- \sum_{i=2}^{\infty}(\widehat{x_i})^2},(\widehat{x_i})_{i\in\mathbb{N}\setminus\{1\}}\right),\qquad\forall\,\widehat{\textbf{x}}=(\widehat{x_i})_{i\in\mathbb{N}}\in \psi(U_0).
\end{eqnarray*}
Note that for any $\widehat{\textbf{x}}=(\widehat{x_i})_{i\in\mathbb{N}}\in \psi(U_0)$ and $\textbf{x}=(x_i)_{i\in\mathbb{N}}\in U_0$, we have
\begin{eqnarray*}
D_1g(\textbf{x})&\triangleq &\frac{x_1}{\sum_{i=1}^{\infty}x_i^2},\\
D_1h(\widehat{\textbf{x}})&\triangleq &\frac{2(\widehat{x_1}+1)}{\sqrt{(\widehat{x_1}+1)^2- \sum_{i=2}^{\infty}(\widehat{x_i})^2}} ,\\
\sum\limits_{i=1}^{\infty}a_i^2\left|D_i h(\widehat{\textbf{x}}) \right|^2&\triangleq& \frac{4a_1^2(\widehat{x_1}+1)^2+\sum_{i=2}^{\infty}a_i^2(\widehat{x_i})^2}{ (\widehat{x_1}+1)^2- \sum_{i=2}^{\infty}(\widehat{x_i})^2 },
\end{eqnarray*}
and
\begin{eqnarray*}
\frac{2}{3}\leq x_1\leq \frac{4}{3},&&\frac{2}{3}\leq \sqrt{\sum_{i=1}^{\infty}x_i^2}\leq \frac{4}{3},\\
0\leq |\widehat{x_1}|\leq \frac{1}{3},&&\frac{2}{3}\leq  \widehat{x_1}+1\leq \frac{4}{3},\\
0\leq \sum_{i=2}^{\infty}|\widehat{x_i}|^2\leq \frac{1}{9},&&\frac{16a_1^2}{9}\leq  4a_1^2(\widehat{x_1}+1)^2+\sum_{i=2}^{\infty}a_i^2(\widehat{x_i})^2\leq \frac{64a_1^2}{9}+\frac{\sum_{i=2}^{\infty}a_i^2}{9}.
\end{eqnarray*}
Hence $\frac{2}{3}\leq x_1\leq \frac{4}{3},\,-\frac{1}{3}\leq g(\textbf{x})\leq \frac{1}{3}$ and
\begin{eqnarray*}
\frac{1}{2}\leq D_1g(\textbf{x})\leq 2, \quad 1\leq D_1h(\widehat{\textbf{x}})\leq 8, \quad \sum\limits_{i=1}^{\infty}a_i^2\left|D_i h(\widehat{\textbf{x}}) \right|^2\leq \frac{64a_1^2}{9}+\frac{\sum_{i=2}^{\infty}a_i^2}{9}.
\end{eqnarray*}
Let
\begin{eqnarray*}
\delta\triangleq \max\left\{\frac{1}{3},\frac{4}{3},1,2,8,\frac{64a_1^2}{9}+\frac{\sum_{i=2}^{\infty}a_i^2}{9}\right\}.
\end{eqnarray*}
Then Condition \ref{20231118cond1} holds.

For any $\textbf{x}=(x_{i})_{i\in\mathbb{N}}\in \partial(B_1)$, there exists $i_0\in\mathbb{N}$ such that $x_{i_0}\neq 0$. We assume that $x_{i_0}>0 $.

Let $U_0\triangleq B_{\frac{x_{i_0}}{3}}(\textbf{x}_0)$, which is a convex open bounded neighborhood of $\textbf{x}_0$. Let
\begin{eqnarray*}
\psi(\textbf{x})&\triangleq &\left(\sqrt{\sum_{i=1}^{\infty}x_i^2}-1,(x_i)_{i\in\mathbb{N}\setminus\{1\}}\right),\qquad\forall\,\textbf{x}=(x_i)_{i\in\mathbb{N}}\in U_0,\\
h(\widehat{\textbf{x}})&\triangleq & \sqrt{(\widehat{x_1}+1)^2- \sum_{i=2}^{\infty}(\widehat{x_i})^2},\qquad\forall\,\widehat{\textbf{x}}=(\widehat{x_i})_{i\in\mathbb{N}}\in \psi(U_0),\\
\tau(\textbf{x})&\triangleq& \left( \sqrt{(\widehat{x_1}+1)^2- \sum_{i=2}^{\infty}(\widehat{x_i})^2},(\widehat{x_i})_{i\in\mathbb{N}\setminus\{1\}}\right),\qquad\forall\,\widehat{\textbf{x}}=(\widehat{x_i})_{i\in\mathbb{N}}\in \psi(U_0).
\end{eqnarray*}
Note that
\begin{eqnarray*}
D_{i_0}g(\textbf{x})&\triangleq &\frac{x_{i_0}}{\sum_{i=1}^{\infty}x_i^2},\\
D_1h(\widehat{\textbf{x}})&\triangleq &\frac{2(\widehat{x_1}+1)}{\sqrt{(\widehat{x_1}+1)^2- \sum_{i=2}^{\infty}(\widehat{x_i})^2}} ,\\
\sum\limits_{i=1}^{\infty}a_i^2\left|D_i h(\widehat{\textbf{x}}) \right|^2&\triangleq& \frac{4a_1^2(\widehat{x_1}+1)^2+\sum_{i=2}^{\infty}a_i^2(\widehat{x_i})^2}{ (\widehat{x_1}+1)^2- \sum_{i=2}^{\infty}(\widehat{x_i})^2 },
\end{eqnarray*}
\end{example}

 \begin{lemma}\label{20231121lem1}
Suppose that Condition \ref{20231118cond1} holds. Then, it holds that
\begin{eqnarray*}
\{f(\psi):f\in C_F^{\infty}(\psi(U_0)\cap\mathbb{H})\}&=&C_F^{\infty}( U_0\cap O),\\
\{g(\tau):g\in C_F^{\infty}(U_0\cap O)\}&=&C_F^{\infty}(\psi(U_0)\cap\mathbb{H}).
\end{eqnarray*}
\end{lemma}
\begin{proof}
Note that for any $f\in C_F^{\infty}(\psi(U_0)\cap\mathbb{H})$ and $E\stackrel{\circ}{\subset} U_0\cap O$, we have $\psi(E)\stackrel{\circ}{\subset} \psi(U_0)\cap \mathbb{H}$ and
\begin{eqnarray*}
&& \sup_E\left(|f(\psi)|+\sum_{i=1}^{\infty}a_i^2|D_i(f(\psi)|^2\right)\\
&=& \sup_E\Bigg(|f(\psi)|+a_1^2|(D_1f)(\psi)\cdot D_1g|^2\\
&&\qquad\qquad\qquad\qquad+\sum_{i=2}^{\infty}a_i^2|(D_1f)(\psi)\cdot D_ig+(D_if)(\psi)|^2\Bigg)\\
&\leq& \sup_E\bigg(|f(\psi)|+a_1^2|(D_1f)(\psi)\cdot D_1g|^2+2\sum_{i=2}^{\infty}a_i^2|(D_1f)(\psi)\cdot D_ig|^2\\
&&+|(D_if)(\psi) |^2)\bigg)\\
&\leq&\sup_E|f(\psi)|+2\left(\sup_E|(D_1f)(\psi)|^2\right)\cdot\sup_E\left(\sum_{i=1}^{\infty}a_i^2\cdot |D_ig|^2\right)\\
&&+2\sup_E\left(\sum_{i=2}^{\infty}a_i^2|(D_if)(\psi)|^2\right)\\
&\leq&\sup_{\psi(E)}|f|+2\left(\sup_{\psi(E)}|(D_1f)|^2\right)\cdot\sup_{E}\left(\sum_{i=1}^{\infty}a_i^2\cdot |D_ig|^2\right)\\
&&+2\sup_{\psi(E)}\left(\sum_{i=2}^{\infty}a_i^2|(D_if)|^2\right)<\infty.
\end{eqnarray*}
Since supp$f\stackrel{\circ}{\subset}\psi(U_0)\cap \mathbb{H}$, supp$f(\psi)=\tau(\text{supp}f)\stackrel{\circ}{\subset}(U_0\cap O)$. Obviously, $f\in C_S^{\infty}(U_0\cap O)$. Therefore, $f(\psi)\in C_F^{\infty}(U_0\cap O)$ and hence $\{f(\psi):f\in C_F^{\infty}(\psi(U_0)\cap \mathbb{H})\}\subset C_F^{\infty}(U_0\cap O)$. Similarly, we can prove that $\{g(\tau):g\in C_F^{\infty}( U_0 \cap O)\}\subset C_F^{\infty}(\psi(U_0)\cap \mathbb{H})$. Then for any $f\in C_F^{\infty}(U_0\cap O)$, we have $f=f(\tau(\psi))$ and $f(\tau)\in C_F^{\infty}(\psi(U_0)\cap \mathbb{H})$ which implies that $\{f(\psi):f\in C_F^{\infty}(\psi(U_0)\cap \mathbb{H})\}= C_F^{\infty}(U_0\cap O)$. Similarly, we can prove that $\{g(\tau):g\in C_F^{\infty}(U_0\cap O)\}=C_F^{\infty}(\psi(U_0)\cap\mathbb{H}).$ This completes the proof of Lemma \ref{20231121lem1}.
\end{proof}
\begin{lemma}\label{02231118lem1}
Suppose that Condition \ref{20231118cond1} holds. Then for each \\ $f\in L^1(U_0,P)$, it holds that
\begin{eqnarray*}
\int_{U_0} f\,\mathrm{d}P=\int_{\eta(U_0)} f(\tau(\widehat{\textbf{x}}) )\cdot\left|D_1 h(\widehat{\textbf{x}}) \right| \cdot e^{\frac{\widehat{x_1}^2-|h(\widehat{\textbf{x}})|^2}{2a_1^2}}\,\mathrm{d}P(\widehat{\textbf{x}}).
\end{eqnarray*}
For each $F\in L^1(\eta(U_0),P)$, it holds that
\begin{eqnarray*}
\int_{\psi(U_0)} F\,\mathrm{d}P=\int_{ U_0 } F(\psi( \textbf{x} )) \cdot\left|D_1 g( \textbf{x} ) \right| \cdot e^{\frac{ x_1^2-|g( \textbf{x} )|^2}{2a_1^2}}\,\mathrm{d}P( \textbf{x} ) .
\end{eqnarray*}
\end{lemma}
\begin{proof}
Note that
\begin{eqnarray*}
\int_{U_0} f\,\mathrm{d}P&=&\int_{\psi(U_0)} f(\tau(\widehat{\textbf{x}}))\,\mathrm{d}P(\tau(\widehat{\textbf{x}})),\\
 \int_{\psi(U_0)} F\,\mathrm{d}P&=&\int_{ U_0 }F(\psi(\widehat{\textbf{x}}))\,\mathrm{d}P(\psi(\widehat{\textbf{x}})).
\end{eqnarray*}
We only need to prove that
\begin{eqnarray}\label{20231118for1}
\mathrm{d}P(\tau(\widehat{\textbf{x}}))=\left|D_1 h(\widehat{\textbf{x}})\right| \cdot e^{\frac{\widehat{x_1}^2-|h(\widehat{\textbf{x}})|^2}{2a_1^2}}\,\mathrm{d}P(\widehat{\textbf{x}}),
\end{eqnarray}
and
\begin{eqnarray}\label{20231118for2}
\mathrm{d}P(\psi(\textbf{x}))= |D_1 g( \textbf{x} ) | \cdot e^{\frac{ x_1^2-|g( \textbf{x})|^2}{2a_1^2}}\,\mathrm{d}P(\textbf{x}).
\end{eqnarray}
Note that
\begin{eqnarray*}
\mathrm{d}P(\tau(\widehat{\textbf{x}}))=\mathrm{d}\mathcal{N}_{a_1^2}(h(\widehat{\textbf{x}}))\times\mathrm{d}P^{\widehat{1}}( \widehat{\textbf{x}}^1),
\end{eqnarray*}
where $ \widehat{\textbf{x}}^1=(\widehat{x_i})_{i\in\mathbb{N}\setminus\{1\}}$, $\mathrm{d}\mathcal{N}_{a_1^2}(x)=\frac{1}{\sqrt{2\pi a_1^2}}e^{-\frac{x^2}{2a_1^2}}\mathrm{d}x$ and $P^{\widehat{1}}$ is the product measure without the $1$-th component, i.e., it is the restriction of the product measure $\prod\limits_{i\not= 1}\bn_{a_i}$ on $\left(\ell^{2}(\mathbb{N}\setminus\{1\}),\mathscr{B}\left(\ell^{2}(\mathbb{N}\setminus\{1\})\right)\right)$. Therefore,
\begin{eqnarray*}
\mathrm{d}P(\tau(\widehat{\textbf{x}}))
&=&\mathrm{d}\mathcal{N}_{a_1^2}(h(\widehat{\textbf{x}}))\times\mathrm{d}P^{\widehat{1}}( \widehat{\textbf{x}}^1)\\
&=&\frac{1}{\sqrt{2\pi a_1^2}}\cdot |D_1 h(\widehat{\textbf{x}})| \cdot e^{\frac{-|h(\widehat{\textbf{x}})|^2}{2a_1^2}}\mathrm{d}\widehat{x_1}\times\mathrm{d}P^{\widehat{1}}( \widehat{\textbf{x}}^1)\\
&=&|D_1 h(\widehat{\textbf{x}})| \cdot e^{\frac{|\widehat{x_1}|^2-|h(\widehat{\textbf{x}})|^2}{2a_1^2}}
\cdot \mathrm{d}\mathcal{N}_{a_1^2}(\widehat{x_1})\times\mathrm{d}P^{\widehat{1}}( \widehat{\textbf{x}}^1)\\
&=&|D_1 h(\widehat{\textbf{x}})| \cdot e^{\frac{|\widehat{x_1}|^2-|h(\widehat{\textbf{x}})|^2}{2a_1^2}}
 \mathrm{d}P( \widehat{\textbf{x}}),
\end{eqnarray*}
where the second equality follows from the one dimensional Jacobi Transformation Formula and this proves \eqref{20231118for1}. The proof of \eqref{20231118for2} is similar which completes the proof of Lemma \ref{02231118lem1}.
\end{proof}
\begin{corollary}\label{20231118cor1}
For each $\textbf{x}\in U_0$, we have $ |D_1 h  (\psi( \textbf{x} ))\big|\cdot |D_1 g( \textbf{x} ) | =1.$
\end{corollary}
\begin{proof}
For each $f\in L^1(U_0,P)$, by  Lemma \ref{02231118lem1}, it holds that
\begin{eqnarray*}
&&\int_{U_0} f\,\mathrm{d}P\\
&=&\int_{\psi(U_0)} f(\tau(\widehat{\textbf{x}} )) \cdot |D_1 h(\widehat{\textbf{x}})| \cdot e^{\frac{\widehat{x_1}^2-|h(\widehat{\textbf{x}})|^2}{2a_1^2}}\,\mathrm{d}P(\widehat{\textbf{x}})\\
&=&\int_{ U_0 } f (\tau(\psi(\textbf{x}) )) \cdot|D_1 h(\psi(\textbf{x}))| \cdot e^{\frac{|g(\textbf{x})|^2-|h(\psi( \textbf{x} ))|^2}{2a_1^2}}\\
&&\qquad\qquad\qquad\qquad\qquad\times |D_1 g( \textbf{x} ) |\cdot e^{\frac{ x_1^2-|g( \textbf{x} )|^2}{2a_1^2}}\,\mathrm{d}P(\textbf{x})\\
&=&\int_{ U_0 } f (\textbf{x}) \cdot|D_1 h(\psi(\textbf{x}))|\cdot e^{\frac{|g(\textbf{x})|^2-x_1^2}{2a_1^2}}\cdot |D_1 g( \textbf{x} ) | \cdot e^{\frac{ x_1^2-|g( \textbf{x} )|^2}{2a_1^2}}\,\mathrm{d}P(\textbf{x})\\
&=&\int_{ U_0 } f (\textbf{x})\cdot|D_1 h(\psi(\textbf{x}))| \cdot|D_1 g( \textbf{x} ) | \,\mathrm{d}P(\textbf{x}),
\end{eqnarray*}
which implies the $ |D_1h (\psi( \textbf{x} )) |\cdot |D_1 g( \textbf{x} )| =1$ almost everywhere respect to $P$ on $U_0$. By the continuity of $ |  h (\psi) |\cdot |D_1 g|$, we see that $ |D_1 h (\psi( \textbf{x} )) |\cdot |D_1g( \textbf{x} )| =1$ for each $\textbf{x}\in U_0$, which completes the proof of Corollary \ref{20231118cor1}.
\end{proof}
Motivated by Lemma \ref{02231118lem1}, we have the following notations.
\begin{definition}
Let
\begin{eqnarray*}
\mathcal {J}(\widehat{\textbf{x}})&\triangleq  & |D_1 h(\widehat{\textbf{x}})| \cdot e^{\frac{\widehat{x_1}^2-|h(\widehat{\textbf{x}})|^2}{a_1^2}},\qquad\forall\,\widehat{\textbf{x}}=(\widehat{x_i})_{i\in\mathbb{N}}\in \psi(U_0),\\
\mathcal {J}_1(\textbf{x})&\triangleq  & |D_1 g( \textbf{x} )| \cdot e^{\frac{ x_1^2-|g( \textbf{x} )|^2}{a_1^2}},\qquad\forall\, \textbf{x} =(x_i)_{i\in\mathbb{N}}\in U_0,\\
\widehat{f}(\widehat{\textbf{x}} )
&\triangleq & f( \tau(\widehat{\textbf{x}})),\qquad\forall\,\widehat{\textbf{x}}\in \psi(U_0),\,f\in L^1(U_0,P).
\end{eqnarray*}
\end{definition}
Obviously, we have the following inequalities for $\mathcal {J}$ and $\mathcal {J}_1$.
\begin{lemma}\label{20231208lem1}
Suppose the Condition \ref{20231118cond1} holds. Let
\begin{eqnarray*}
C_1\triangleq \frac{1}{\delta}\cdot e^{-\frac{\delta^2}{a_1^2}},\qquad C_2\triangleq \delta\cdot e^{\delta^2}.
\end{eqnarray*}
Then it holds that
\begin{eqnarray*}
C_1\leq \mathcal {J}(\widehat{\textbf{x}})\leq C_2,\qquad C_1\leq\mathcal {J}_1(\textbf{x})\leq   C_2,\qquad\forall\,\,\textbf{x}\in U_0,\,\,\widehat{\textbf{x}}\in \psi(U_0).
\end{eqnarray*}
\end{lemma}
\begin{proposition}\label{weakly change of variable formula}
Suppose that $1\leq p\leq \infty$ and Condition \ref{20231118cond1} holds. If $f,\,D_1 f \in L^p(U_0\cap O,P)$. Then $\widehat{f},\,D_1\widehat{f} \in L^p(\psi(U_0)\cap\mathbb{H},P)$  and
\begin{eqnarray*}
D_1 \widehat{f}
= (  D_1  f  )(\tau) \cdot(D_1h ).
\end{eqnarray*}
For $i\geq 2$, if $f,\,D_1 f  ,\,D_i f \in L^p(U_0\cap O,P)$. Then $\widehat{f},\,D_i\widehat{f} \in L^p(\psi(U_0)\cap\mathbb{H},P)$  and
\begin{eqnarray*}
D_i\widehat{f}
= (D_i  f)(\tau) + (D_1 f)(\tau) \cdot(D_i h).
\end{eqnarray*}
Conversely, if $g,D_1g\in L^p(\psi(U_0)\cap\mathbb{H},P)$. Then $g(\psi),D_1(g(\psi))\in L^p(U_0\cap O,P)$ and
\begin{eqnarray*}
D_1(g(\psi))=(D_1g)(\psi)\cdot \frac{1}{(D_1h)(\psi)}.
\end{eqnarray*}
For $i\geq 2$, if $g,\,D_1 g  ,\,D_ig \in L^p(\psi(U_0)\cap \mathbb{H},P)$. Then $g(\psi),\,D_i(g(\psi)) \in L^p( U_0 \cap O,P)$  and
\begin{eqnarray*}
D_i(g(\psi))
= (D_i  g)(\psi) -(D_1 g)(\psi) \cdot\frac{(D_i h)(\psi)}{(D_1 h)(\psi)}.
\end{eqnarray*}
\end{proposition}
\begin{proof}
For simplicity, we only prove that
\begin{eqnarray*}
D_1 \widehat{f}
= (  D_1 f)(\tau) \cdot(D_1 h).
\end{eqnarray*}
For each $\varphi\in C_F^{\infty}(\psi(U_0))$, recall Definition \ref{231013def1}, we have
\begin{eqnarray*}
(D_1\widehat{f})(\varphi)
&=&-\int_{\psi(U_0)}\widehat{f}\cdot\bigg(D_1 \varphi -\frac{\widehat{x_1}}{a_1^2}\cdot \varphi\bigg)\,\mathrm{d}P\\
&=&-\int_{ U_0\cap O } \widehat{f}(\psi) \cdot\bigg((D_1 \varphi)(\psi) -\frac{ g }{a_1^2}\cdot (\varphi(\psi))\bigg) \cdot \mathcal {J}_1\,\mathrm{d}P\\
&=&-\int_{ U_0\cap O } f \cdot\bigg(D_1 (\varphi(\psi))\cdot\bigg((D_1h)(\psi )\bigg) -\frac{  g }{a_1^2}\cdot (\varphi(\psi))\bigg)\cdot \mathcal {J}_1\,\mathrm{d}P\\
&=& \int_{ U_0\cap O } f \cdot D_1^*\bigg(  (\varphi(\psi)) \cdot\bigg((D_1 h)(\psi) \bigg) \cdot \mathcal {J}_1\bigg)\,\mathrm{d}P\\
&&+\int_{ U_0\cap O } f \cdot (\varphi(\psi)) \cdot\bigg(-D_1^*\bigg( \bigg((D_1h)  (\psi) \bigg) \cdot \mathcal {J}_1\bigg)+\frac{  g }{a_1^2}\bigg)\,\mathrm{d}P,
\end{eqnarray*}
where the second equality follows from Lemma \ref{02231118lem1}, the third equality follows from the facts that $\varphi=(\varphi(\psi))(\tau)$ and $D_1\varphi=\big(D_1(\varphi(\psi))\big)(\tau)\cdot (D_1h)$. Note that for any $\eta\in C_F^{\infty}(U_0)$, it holds that
\begin{eqnarray*}
&&\int_{ U_0\cap O } \eta  \cdot\bigg(-D_1^*\bigg( \bigg((D_1h)(\psi)\bigg) \cdot \mathcal {J}_1\bigg)+\frac{  g }{a_1^2}\cdot \mathcal {J}_1\bigg)\,\mathrm{d}P\\
&=&-\int_{ U_0\cap O  } (D_1 \eta ) \cdot \bigg((D_1 h)(\psi) \bigg) \cdot \mathcal {J}_1 \,\mathrm{d}P
+\int_{ U_0\cap O  } \eta \cdot\frac{g }{a_1^2} \cdot \mathcal {J}_1 \,\mathrm{d}P\\
&=&-\int_{ \psi(U_0)\cap \mathbb{H} } (D_1 \eta)(\tau)  \cdot (D_1 h) \,\mathrm{d}P
+\int_{ \psi(U_0) }(\eta(\tau))\cdot \frac{\widehat{x_1}}{a_1^2}  \,\mathrm{d}P\\
&=&-\int_{ \psi(U_0)\cap \mathbb{H} } D_1 (\eta(\tau))    \,\mathrm{d}P
+\int_{\psi(U_0)}(\eta(\tau))\cdot\frac{\widehat{x_1}}{a_1^2}  \,\mathrm{d}P\\
&=& \int_{ \psi(U_0)\cap \mathbb{H} }D_1^*(\eta(\tau))\,\mathrm{d}P\\
&=& \int_{ \psi(U_0)  }D_1^*(\eta(\tau))\,\mathrm{d}P=0,
\end{eqnarray*}
where the first and the last equality follow from Gauss-Green type theorem \cite{WYZ} and the second equality follows from Lemma \ref{02231118lem1}.
Therefore,
\begin{eqnarray*}
-D_1^*\bigg(  (D_1 h )(\psi) \bigg) \cdot \mathcal {J}_1\bigg)+\frac{  g }{a_1^2}\cdot \mathcal {J}_1=0,\quad\text{on}\,\,U_0\cap O,
\end{eqnarray*}
and hence
\begin{eqnarray*}
(D_1\widehat{f})(\varphi)
&=& \int_{ U_0\cap O } f \cdot D_1^*\bigg(  (\varphi(\psi)) \cdot\bigg((D_1 h)(\psi) \bigg) \cdot \mathcal {J}_1\bigg)\,\mathrm{d}P\\
&=&\int_{ U_0\cap O }(D_1 f)\cdot  (\varphi(\psi)) \cdot\bigg((D_1 h)(\psi) \bigg) \cdot \mathcal {J}_1\,\mathrm{d}P\\
&=&\int_{ \psi(U_0)\cap \mathbb{H} }\bigg((D_1 f)(\tau)\bigg)\cdot\varphi \cdot (D_1 h)  \,\mathrm{d}P,
\end{eqnarray*}
where the second equality follows from the assumption that $D_1 f\in L^p(U_0\cap O,P)$, Lemma \ref{20231121lem1} and the fact that the function $(\varphi(\psi)) \cdot\left((D_1 h )(\psi) \right) \cdot \mathcal {J}_1\in  C_F^{\infty}(U_0)$ and the third equality follows from Lemma \ref{02231118lem1}. This implies that $D_1\widehat{f}
=\left(  (D_1 f)(\tau)\right)\cdot (D_1 h)$. Combining Condition \ref{20231118cond1}, Lemma \ref{02231118lem1} and Lemma \ref{20231208lem1}, we see that  $\widehat{f},\,D_1\widehat{f} \in L^p(\psi(U_0)\cap\mathbb{H},P)$.
This completes the proof of Proposition \ref{weakly change of variable formula}.
\end{proof}

\begin{theorem}\label{equivalence of norm}
Suppose that Conditions \ref{20231118cond1} holds. Then there exists constant $C\in(0,+\infty)$ such that: If $f\in W^{1,2}(U_0\cap O)$. Then $f(\tau)\in W^{1,2}(\psi(U_0)\cap \mathbb{H})$ and
\begin{eqnarray}\label{20231119for1}
C^{-1}||f||_{W^{1,2}(U_0\cap O)}\leq ||f(\tau)||_{W^{1,2}(\psi(U_0)\cap \mathbb{H})} \leq C||f||_{W^{1,2}(U_0\cap O)}.
\end{eqnarray}
Conversely, if $g\in W^{1,2}(\psi(U_0)\cap \mathbb{H})$. Then $g(\psi)\in W^{1,2}(U_0\cap O)$ and
\begin{eqnarray}\label{20231208for1}
C^{-1}||g||_{W^{1,2}(\psi(U_0)\cap \mathbb{H})}\leq ||g(\psi)||_{W^{1,2}(U_0\cap O)} \leq C||g||_{W^{1,2}(\psi(U_0)\cap \mathbb{H})}.
\end{eqnarray}
\end{theorem}
\begin{proof}
Suppose that $f\in W^{1,2}(U_0\cap O)$, by Proposition \ref{weakly change of variable formula},  $\widehat{f}, D_i\widehat{f}\in L^2(\psi(U_0)\cap \mathbb{H},P)$ for any $i\in\mathbb{N}$ and
\begin{eqnarray*}
&&\left|\left|\widehat{f}\right|\right|_{W^{1,2}(\psi(U_0)\cap\mathbb{H})}^2\\
&=&\int_{\psi(U_0)\cap\mathbb{H}} \left|\widehat{f} \right|^2\,\mathrm{d}P+\sum_{i=1}^{\infty}a_i^{2}\cdot\int_{\psi(U_0)\cap\mathbb{H}} \left|D_i\widehat{f} \right|^2\,\mathrm{d}P\\
&=&\int_{\psi(U_0)\cap\mathbb{H}} \left|\widehat{f} \right|^2\,\mathrm{d}P+\sum_{i=2}^{\infty}a_i^{2}\cdot\int_{\psi(U_0)\cap\mathbb{H}}\left|D_i\widehat{f} \right|^2\,\mathrm{d}P+
a_1^{2}\cdot\int_{\psi(U_0)\cap \mathbb{H}}\left|D_1\widehat{f} \right|^2\,\mathrm{d}P\\
&=&\int_{\psi(U_0)\cap\mathbb{H}} \left|\widehat{f} \right|^2\,\mathrm{d}P+\sum_{i=2}^{\infty}a_i^{2}\cdot\int_{\psi(U_0)\cap\mathbb{H}}|(D_i f )(\tau) +(D_1 f)( \tau)\cdot(D_i h)|^2\,\mathrm{d}P\\
&&+a_1^{2}\cdot\int_{\psi(U_0)\cap\mathbb{H}} | (D_1 f)(\tau) \cdot (D_1 h) |^2\,\mathrm{d}P\\
&=&\int_{ U_0 \cap O} \left|f \right|^2\cdot  \mathcal {J}_1\,\mathrm{d}P+\sum_{i=2}^{\infty}a_i^{2}\cdot\int_{ U_0 \cap O}|(D_i f )  +(D_1 f) \cdot(D_i h)(\psi)|^2\cdot  \mathcal {J}_1\,\mathrm{d}P \\
&&+a_1^{2}\cdot\int_{ U_0\cap O } | (D_1 f)  \cdot (D_1 h)(\psi) |^2\cdot  \mathcal {J}_1\,\mathrm{d}P\\
&\leq&C_2\cdot\int_{ U_0 \cap O} \left|f \right|^2 \,\mathrm{d}P+2C_2\cdot\sum_{i=2}^{\infty}a_i^{2}\cdot\int_{ U_0 \cap O}|D_i f |^2  \,\mathrm{d}P\\
&&+C_2\cdot \int_{ U_0\cap O } | D_1 f|^2  \cdot \left(\sum_{i=1}^{\infty}a_i^{2}\cdot|(D_i h)(\psi) |^2\right) \,\mathrm{d}P\\
&\leq&2C_2\left(1+\frac{\delta }{a_1^2}\right)\cdot||f||_{W^{1,2}(U_0)}^2,
\end{eqnarray*}
where the first equality follows from Lemma \ref{20231208lem1} and the second inequality follows from Condition \ref{20231118cond1}. On the other side, suppose that $g\in W^{1,2}(\psi(U_0)\cap \mathbb{H})$, by Proposition \ref{weakly change of variable formula},  $g(\psi), D_i(g(\psi))\in L^2( U_0 \cap O,P)$ for any $i\in\mathbb{N}$ and
\begin{eqnarray*}
&&||g(\psi)||_{W^{1,2}(U_0\cap O)}^2\\
&=&\int_{U_0\cap O} |g(\psi) |^2\,\mathrm{d}P+\sum_{i=1}^{\infty}a_i^{2}\cdot\int_{U_0\cap O} |D_i (g(\psi)) |^2\,\mathrm{d}P\\
&=&\int_{U_0\cap O} |g(\psi) |^2\,\mathrm{d}P+\sum_{i=2}^{\infty}a_i^{2}\cdot\int_{U_0\cap O}|D_i (g(\psi)) |^2\,\mathrm{d}P\\
&&+a_1^{2}\cdot\int_{U_0\cap O} |D_1 (g(\psi)) |^2\,\mathrm{d}P\\
&=&\int_{\psi(U_0)\cap\mathbb{H}} |g |^2\cdot\mathcal {J}\,\mathrm{d}P+\sum_{i=2}^{\infty}a_i^{2}\cdot\int_{\psi(U_0)\cap\mathbb{H}} \left|D_ig-(D_1g)\cdot \frac{D_ih}{D_1h} \right|^2\cdot\mathcal {J}\,\mathrm{d}P\\
&&+a_1^{2}\cdot\int_{\psi(U_0)\cap\mathbb{H}}\bigg|(D_1g)\cdot\frac{1}{D_1h}\bigg|^2\cdot\mathcal {J}\,\mathrm{d}P\\
&\leq&C_2\int_{\psi(U_0)\cap\mathbb{H}} |g |^2 \,\mathrm{d}P\\
&&+2C_2\sum_{i=2}^{\infty}a_i^{2}\cdot\int_{\psi(U_0)\cap\mathbb{H}} \left(|D_ig|^2+\left|(D_1g)\cdot \frac{D_ih}{D_1h}\right|^2 \right)  \,\mathrm{d}P\\
&&+C_2\cdot a_1^{2}\cdot\int_{\psi(U_0)\cap\mathbb{H}}\bigg|(D_1g)\cdot\frac{1}{D_1h}\bigg|^2\,\mathrm{d}P\\
&\leq&C_2\int_{\psi(U_0)\cap\mathbb{H}} |g |^2 \,\mathrm{d}P\\
&&+2C_2\sum_{i=2}^{\infty}a_i^{2}\cdot\int_{\psi(U_0)\cap\mathbb{H}} \left(|D_ig|^2+|(D_1g)\cdot (D_ih)|^2\cdot {\delta}^2 \right)  \,\mathrm{d}P\\
&&+C_2\cdot a_1^{2}\cdot {\delta}^2\cdot\int_{\psi(U_0)\cap\mathbb{H}}|D_1g|^2\,\mathrm{d}P\\
&\leq&C_2\int_{\psi(U_0)\cap\mathbb{H}} |g |^2 \,\mathrm{d}P+2C_2(1+\delta^2)\sum_{i=2}^{\infty}a_i^{2}\cdot\int_{\psi(U_0)\cap\mathbb{H}}|D_ig|^2\,\mathrm{d}P\\
&&+2C_2\cdot  {\delta}^2\cdot\int_{\psi(U_0)\cap\mathbb{H}}|D_1g|^2\cdot\left(\sum_{i=2}^{\infty}a_i^{2}\cdot|D_ig|^2\right)\,\mathrm{d}P\\
&\leq&C_2\int_{\psi(U_0)\cap\mathbb{H}} |g |^2 \,\mathrm{d}P+2C_2(1+\delta^2)\sum_{i=2}^{\infty}a_i^{2}\cdot\int_{\psi(U_0)\cap\mathbb{H}}|D_ig|^2\,\mathrm{d}P\\
&&+2C_2\cdot  {\delta}^3\cdot\int_{\psi(U_0)\cap\mathbb{H}}|D_1g|^2\,\mathrm{d}P\\
&\leq &2C_2\bigg(1+ \delta^2+ \frac{\delta^3}{a_1^2}\bigg)\cdot ||g||_{W^{1,2}(\psi(U_0)\cap \mathbb{H})}^2.
\end{eqnarray*}
Let
\begin{eqnarray*}
C\triangleq \max\left\{\sqrt{2C_2\bigg(1+ \delta^2+ \frac{\delta^3}{a_1^2}\bigg)},\sqrt{2C_2\left(1+\frac{\delta }{a_1^2}\right)}\right\}.
\end{eqnarray*}
Obviously, \eqref{20231119for1} holds and this completes the proof of Theorem \ref{equivalence of norm}.
\end{proof}

\begin{proposition}\label{lemma21boudrary}
Suppose $f\in W^{1,2}(O)$. If there exists $x_0\in\partial O$ and $\{r_1,r_2,r_3,r_4\}\subset(0,+\infty)$ such that
\begin{itemize}
\item[(1)] Condition \ref{20231118cond1} holds for $U_0=B_r(\textbf{x}_0)$;
\item[(2)]  $r_1<r_2$ and $r_3<r_4$;
\item[(3)] $\psi(B_{r_1}(\textbf{x}_0))\subset B_{r_3}(\psi(\textbf{x}_0))\subset B_{r_4}(\psi(\textbf{x}_0))\subset \psi(B_{r_2}(\textbf{x}_0))$;
\item[(4)] supp$f\subset B_{r_1}(\textbf{x}_0)$.
\end{itemize}
 Then there exists $\{\Phi_n\}_{n=1}^{\infty}\subset C_{S,b}^{\infty}(\ell^2)\cap C_{F}^{\infty}(\ell^2)$ such that
\begin{eqnarray*}
 \lim_{n\to\infty}||\Phi_n-f||_{W^{1,2}(O)}=0.
\end{eqnarray*}
\end{proposition}
\begin{proof}
By assumption, $\psi(U_0\cap O)=\psi(U_0)\cap \mathbb{H}$ is an open subset of the half space $\mathbb{H}$. Since supp$f\subset B_{r_1}(\textbf{x}_0)\subset U_0,\,f\in W^{1,2}(U\cap O)$ and by Theorem \ref{equivalence of norm}, $f( \tau)\in W^{1,2}(\psi(U_0)\cap \mathbb{H})$. Note that supp$f(\tau)\subset B_{r_3}(\psi(\textbf{x}_0))\subset B_{r_4}(\psi(\textbf{x}_0))\subset \psi(U_0)$. By Lemma \ref{20231120lem4}, there exists $\{\varphi_n\}_{n=1}^{\infty}\subset \mathscr{C}_c^{\infty}$ such that
\begin{eqnarray*}
 \lim_{n\to\infty}||\varphi_n-f( \tau)||_{W^{1,2}(\psi(U_0)\cap \mathbb{H})}=0.
\end{eqnarray*}
Apply Theorem \ref{equivalence of norm} again we have
\begin{eqnarray*}
 \lim_{n\to\infty}||\varphi_n(\psi)-f||_{W^{1,2}( U_0\cap O)}=0.
\end{eqnarray*}
Choosing $\phi\in C^{\infty}(\mathbb{R};[0,1])$ such that $\phi(x)=1$ for any $x\leq r_1^2$ and $\phi(x)=0$ for any $x\geq r_2^2$. Let $\Phi(\textbf{x})\triangleq \phi(||\textbf{x}-\textbf{x}_0||^2_{\ell^2}),\,\forall \,\textbf{x}\in\ell^2$. Then we have $\Phi\cdot f=f$ and
\begin{eqnarray*}
 \lim_{n\to\infty}||\varphi_n(\psi)-f||_{W^{1,2}( U_0\cap O)}= \lim_{n\to\infty}||(\varphi_n(\psi))\cdot \Phi-f||_{W^{1,2}( U_0\cap O)}=0.
\end{eqnarray*}
Note that
\begin{eqnarray*}
||(\varphi_n(\psi))\cdot \Phi-f||_{W^{1,2}( U_0\cap O)}=||(\varphi_n(\Psi))\cdot \Phi-f||_{W^{1,2}( O)},\qquad \forall\,\,n\in\mathbb{N},
\end{eqnarray*}
which implies that
\begin{eqnarray*}
 \lim_{n\to\infty}||(\varphi_n(\Psi))\cdot \Phi-f||_{W^{1,2}( O)}=0.
\end{eqnarray*}
Write $\Phi_n\triangleq (\varphi_n(\Psi))\cdot \Phi,  \,\,\forall\,\,n\in\mathbb{N}$. Obviously, $\{\Phi_n\}_{n=1}^{\infty}\subset C_{S,b}^{\infty}(\ell^2)\cap C_{F}^{\infty}(\ell^2)$ and this completes the proof of Proposition \ref{lemma21boudrary}.
\end{proof}
\subsection{General Case}
\begin{lemma}\label{20231209lem1}
Suppose that $f\in W^{1,2}(O)$, $\theta$ is a real-valued Fr\'{e}chet differentiable on $\ell^2$ and $\theta\in C_{S,b}^{\infty}(\ell^2)\cap C_{F}^{\infty}(\ell^2)$. Then $\theta\cdot f\in W^{1,2}(O)$.
\end{lemma}
\begin{proof}
Note that for any $\varphi\in  C_F^{\infty}(O)$ and $i\in\mathbb{N}$, we have
\begin{eqnarray*}
(D_i(\theta\cdot f))(\varphi)&=&\int_O(D_i^*\varphi)\cdot\theta\cdot f\\
&=&\int_O(D_i^*(\varphi \cdot\theta))\cdot f+\int_O\varphi\cdot (D_i\theta)\cdot f\,\mathrm{d}P\\
&=&\int_O \varphi \cdot\theta \cdot(D_i f)+\int_O\varphi\cdot (D_i\theta)\cdot f\,\mathrm{d}P\\
&=&\int_O (\theta \cdot(D_i f)+(D_i\theta)\cdot f)\cdot \varphi \,\mathrm{d}P,
\end{eqnarray*}
which implies that $D_i(\theta\cdot f)=\theta \cdot(D_i f)+(D_i\theta)\cdot f$. We also note that
\begin{eqnarray*}
&&\sum_{i=1}^{\infty}a_i^2\int_O |\theta \cdot(D_i f)+(D_i\theta)\cdot f|^2\,\mathrm{d}P\\
&\leq &2\sum_{i=1}^{\infty}a_i^2\int_O( |\theta \cdot(D_i f)|^2+|(D_i\theta)\cdot f|^2)\,\mathrm{d}P\\
&\leq &2\left(\sup_O|\theta|^2 \right)\cdot\sum_{i=1}^{\infty}a_i^2\int_O |D_i f|^2 \,\mathrm{d}P\\
&&\qquad\qquad\qquad+2\left(\sup_O\sum_{i=1}^{\infty}a_i^2|D_i\theta|^2\right)\cdot \int_O|f|^2\,\mathrm{d}P<\infty
\end{eqnarray*}
and
\begin{eqnarray*}
 \int_O|\theta\cdot f|^2\,\mathrm{d}P\leq  \left(\sup_O|\theta|^2 \right)\cdot\int_O|f|^2\,\mathrm{d}P<\infty.
\end{eqnarray*}
These imply that $\theta\cdot f\in W^{1,2}(O)$ and this completes the proof of Lemma \ref{20231209lem1}.
\end{proof}
\begin{theorem}\label{general approximation}
Suppose that $f\in W^{1,2}(O)$. Then there exists $\{\varphi_n\}_{n=1}^{\infty}\subset C_{S,b}^{\infty}(\ell^2)\cap C_{F}^{\infty}(\ell^2)$ such that
\begin{eqnarray*}
 \lim_{n\to\infty}||\varphi_n-f||_{W^{1,2}(O)}=0.
\end{eqnarray*}
\end{theorem}
\begin{proof}
By lemma \ref{lemma21gc}, we can assume that the support of $f$ is compact. For $\textbf{x}\in O$, there exists a neighborhood $U_{\textbf{x}}$ of $\textbf{x}$ such that $U_{\textbf{x}}$ satisfies the conditions in lemma \ref{lemma21hs}($B_r(\textbf{x}_0)$ in Lemma \ref{lemma21hs} ).

For $\textbf{x}\in \partial O$, there exists a neighborhood $U_{\textbf{x}}$ of $\textbf{x}$ such that $U_{\textbf{x}}$ satisfies the conditions in Proposition \ref{lemma21boudrary}($B_{r_1}(\textbf{x}_0)$ in Proposition \ref{lemma21boudrary}).   For $\textbf{x}\notin \partial O$, there exists a neighborhood $U_{\textbf{x}}$ of $\textbf{x}$ such that $U_{\textbf{x}}\cap  O=\emptyset.$ Then $\{U_{\textbf{x}}:\textbf{x}\in\ell^2\}$ is an open covering of $\ell^2$. By the partition of unity for $C^{\infty}$ function on $\ell^2.$ There exists a sequence of functions $\{f_n\}_{n=1}^{\infty}$ such that:
\begin{itemize}
\item[$\bullet$] $f_n$ is Fr\'{e}chet differentiable on $\ell^2$ for each $n\in\mathbb{N}$;
\item[$\bullet$] $0\leq f_n(\textbf{x})\leq 1$ for each $n\in\mathbb{N}$ and $\textbf{x}\in\ell^2;$
\item[$\bullet$] supp$f_n\subset U_{\textbf{x}}$ for some $\textbf{x}\in \ell^2$ and $\{f_n\}_{n=1}^{\infty}\subset C_{S,b}^{\infty}(\ell^2)\cap C_{F}^{\infty}(\ell^2)$;
\item[$\bullet$] $\sum\limits_{n=1}^{\infty}f_n(\textbf{x})\equiv 1$,\,$\forall\,\textbf{x}\in \ell^2$;
\item[$\bullet$] For each $\textbf{x}_0\in\ell^2$, there exists a neighborhood $U$ of $\textbf{x}_0$ such that $f_n(\textbf{x})\equiv 0$,\,$\forall\,\textbf{x}\in U$ except finite many $n\in\mathbb{N}$.
\end{itemize}
Then
\begin{eqnarray*}
 f=\sum_{n=1}^{\infty}f_n\cdot f.
\end{eqnarray*}
Since supp$f$ is compact, there exists $n_0\in\mathbb{N}$ such that
\begin{eqnarray*}
 f=\sum_{n=1}^{n_0}f_n\cdot f.
\end{eqnarray*}
For each $1\leq n\leq n_0$, by Lemma \ref{lemma21hs} and Proposition \ref{lemma21boudrary}, there exists $\{\varphi_n^k\}_{k=1}^{\infty}\subset C_{S,b}^{\infty}(\ell^2)\cap C_{F}^{\infty}(\ell^2)$ such that
\begin{eqnarray*}
 \lim_{k\to\infty}||\varphi_n^k-f\cdot f_n||_{W^{1,2}(O)}=0.
\end{eqnarray*}
Therefore,
\begin{eqnarray*}
 \lim_{k\to\infty}\bigg|\bigg|\sum_{n=1}^{n_0}\varphi_n^k-f\bigg|\bigg|_{W^{1,2}(O)}=0.
\end{eqnarray*}
Let $\varphi_n\triangleq\sum\limits_{n=1}^{n_0}\varphi_n^k$ for each $n\in\mathbb{N}$. This completes the proof of Theorem \ref{general approximation}.
\end{proof}



\end{document}